\documentclass{gtpart}
\usepackage{amsthm,amsfonts,amsmath,amssymb,amscd}

\usepackage{graphicx}
\usepackage{amstext}
\usepackage{amsfonts}
\usepackage{subfig}
\usepackage{hyperref}
\usepackage{mathrsfs}
\usepackage{mathtools}
\usepackage{comment}
\usepackage{tikz}
\usetikzlibrary{trees}
\usetikzlibrary{arrows}
\usetikzlibrary{matrix}
\usepackage{ifthen}
\usetikzlibrary{decorations.markings, shapes.geometric, patterns} 
\usetikzlibrary{decorations.pathmorphing}

\newtheorem{thm}{Theorem}
\newtheorem{lem}[thm]{Lemma}
\newtheorem{prop}[thm]{Proposition}
\newtheorem{defi}[thm]{Definition}
\newtheorem{cor}[thm]{Corollary}
\newtheorem{ex}[thm]{Example}
\newtheorem{rem}[thm]{Remark}
\newtheorem{conj}[thm]{Conjecture}

\newcommand{\rhom}{\mathrm{RHom}}
\renewcommand{\k}{\mathbf{k}}
\newcommand{\K}{\mathbb{K}}
\newcommand{\m}{\mathfrak{m}}
\newcommand{\n}{\mathfrak{n}}
\newcommand{\e}{\mathfrak{e}}
\renewcommand{\t}{\mathfrak{t}}
\newcommand{\A}{\mathscr{A}}
\newcommand{\Abar}{\overline{\mathscr{A}}}
\newcommand{\Cbar}{\overline{\mathscr{C}}}
\newcommand{\Cobar}{\Omega}
\renewcommand{\Bar}{\mathrm{B}}
\newcommand{\B}{\mathscr{B}}

\renewcommand{\C}{\mathscr{C}}
\renewcommand{\D}{\mathscr{D}}
\newcommand{\F}{\mathcal{F}}
\newcommand{\I}{\mathscr{I}}
\renewcommand{\J}{\mathscr{J}}
\renewcommand{\Z}{\mathbb{Z}}
\renewcommand{\R}{\mathbf{R}}
\newcommand{\Cc}{\mathbf{C}}

\newcommand{\id}{\mathbf{1}}
\renewcommand{\co}{\mathrm{co}}
\newcommand{\fl}{\mathrm{fi}}
\newcommand{\sy}{\mathrm{sy}}
\newcommand{\pb}{\mathrm{pb}}

\newcommand{\ind}{\operatorname{index}}
\renewcommand{\parallel}{{\mkern3mu\vphantom{\perp}\vrule depth 0pt\mkern2mu\vrule depth
0pt\mkern3mu}}

\title{Duality between Lagrangian and Legendrian invariants}

\author{Tobias Ekholm}
\address{Uppsala University, Box 480, 751 06 Uppsala, Sweden\newline
	\indent Insitute Mittag-Leffler, Aurav 17, 182 60 Djursholm, Sweden}
\email{tobias.ekholm@math.uu.se}
\author{Yank\i\ Lekili}
\address{Imperial College London, South Kensington, London, UK}
\email{y.lekili@imperial.ac.uk}

\begin{document}

\begin{abstract}
Consider a pair $(X,L)$, of a Weinstein manifold $X$ with an exact Lagrangian submanifold $L$, with
    ideal contact boundary $(Y,\Lambda)$, where $Y$ is a contact manifold and $\Lambda\subset Y$ is
    a Legendrian submanifold. We introduce the Chekanov-Eliashberg DG-algebra, $CE^{\ast}(\Lambda)$,
    with coefficients in chains of the based loop space of $\Lambda$ and study its relation to the
    Floer cohomology $CF^{\ast}(L)$ of $L$.  Using the augmentation induced by $L$,
    $CE^{\ast}(\Lambda)$ can be expressed as the Adams cobar
    construction $\Omega$ applied to a Legendrian coalgebra, $LC_{\ast}(\Lambda)$. We define a twisting cochain: 
\[ 
    \t\colon LC_{\ast}(\Lambda) \to \Bar (CF^*(L))^\# 
\] 
via holomorphic curve counts,
where $\Bar$ denotes the bar construction and $\#$ the graded linear dual. We show under simply-connectedness assumptions that the corresponding Koszul complex is acyclic which then implies that $CE^*(\Lambda)$ and $CF^{\ast}(L)$ are Koszul dual. In particular, $\t$ induces a quasi-isomorphism between $CE^*(\Lambda)$ and the cobar of the Floer homology of $L$, $\Omega CF_*(L)$. 

This generalizes the classical Koszul duality result between $C^*(L)$ and $C_{-*}(\Omega L)$ for $L$ a simply-connected manifold, where $\Omega L$ is the based loop space of $L$, and provides the geometric ingredient explaining the computations given in \cite{EtLe} in the case when $X$ is a plumbing of cotangent bundles of 2-spheres (where an additional weight grading ensured Koszulity of $\t$).  

We use the duality result to show that under certain connectivity and locally finiteness
    assumptions, $CE^*(\Lambda)$ is quasi-isomorphic to $C_{-*}(\Omega L)$ for any Lagrangian filling $L$ of $\Lambda$.    
    
Our constructions have interpretations in terms of wrapped Floer cohomology after versions of
    Lagrangian handle attachments. In particular, we outline a proof that $CE^{\ast}(\Lambda)$ is
    quasi-isomorphic to the wrapped Floer cohomology of a fiber disk $C$ in the Weinstein domain
    obtained by attaching $T^{\ast}(\Lambda\times[0,\infty))$ to $X$ along $\Lambda$ (or, in the
    terminology of \cite{sylvan} the wrapped Floer cohomology of $C$ in $X$ with wrapping stopped by $\Lambda$). Along the way, we give a definition of wrapped Floer cohomology via holomorphic buildings that avoids the use of Hamiltonian perturbations which might be of independent interest.

\end{abstract}
\maketitle

\section{Introduction}
In this introduction we first give an overview of our results. The overview starts with a review of
well known counterparts of our constructions in algebraic topology. We then introduce our Legendrian
and Lagrangian invariants in Sections \ref{ssec:partwrap} and \ref{ssec:infinteswrap}, respectively,
and discuss the connection between them and applications thereof in Section \ref{ssec:connections}. Among these the most central role is played by the \emph{Chekanov-Eliashberg algebra with based loop space coefficients} denoted as $CE^*$. As we show, any other invariant can be obtained from $CE^*$ by algebraic manipulation. Finally, in Section \ref{ssec:Hopflink} we give detailed calculations of the invariants introduced in the simple yet illustrative example of the Legendrian Hopf link filled by two Lagrangian disks intersecting transversely in one point.
  
The starting point for our study is a construction in classical topology. 
Consider the following augmented DG-algebras over a field $\mathbb{K}$ associated to a based, connected, topological space $(M,\mathrm{pt})$: 
\[ C^*(M) \to \K , \ \ C_{-*}(\Omega M) \to \K, \]
where $C^*(M)$ is the singular cochain complex equipped with the cup-product and $C_{-*}(\Omega M)$
is the singular chain complex of the based (Moore) loop space of $M$ equipped with the Pontryagin product. (We use
cohomologically graded complexes throughout the paper so that all differentials increase the grading by 1.) In the case of
singular cohomology, the inclusion $i\colon \mathrm{pt} \to M$  gives the augmentation $i^*\colon C^*(M) \to C^*(\mathrm{pt}) = \K$ and
in the case of the based loop space, the augmentation is given by the trivial local system $\pi_1(M,\mathrm{pt}) \to \K$. 

If $M$ is of finite-type (for example, a finite CW-complex), then it is well known that one can
recover the augmented DG-algebra $C^*(M)$ from the augmented DG-algebra $C_{-*}(\Omega M)$ by
the Eilenberg-Moore equivalence:
\[ C^*(M) \simeq \mathrm{RHom}_{C_{-*}(\Omega M)} (\K, \K). \] 
In the other direction, if $M$ is \emph{simply connected}, then the Adams construction gives a quasi-isomorphism 
\[ C_{-*}(\Omega M) \simeq \mathrm{RHom}_{C^*(M)} (\K,\K), \]
and in this case $C^*(M)$ and $C_{-*}(\Omega M)$ are
said to be \emph{Koszul dual} DG-algebras. 
Koszul duality is sometimes abbreviated and simply called
\emph{duality}. 
For more general $M$, using the method of acyclic models, Brown \cite{brown} constructed a twisting cochain
\[ \t \colon C_{-*}(M) \to C_{-*}(\Omega M). \]
This is a degree 1 map that induces a DG-algebra map $\Omega C_{-*}(M) \to C_{-*}(\Omega M)$.
where $\Omega C_{-*}(M)$ is the cobar construction applied to chains on $M$ (see Section \ref{barcobarsec}). By definition, t$\t$ is a
quasi-isomorphism when duality holds, and this can be detected by an associated \emph{Koszul complex}, which
is acyclic if and only if duality holds. In the general case, $\Omega C_{-*}(M)$ is a certain completion of $C_{-*}(\Omega M)$ and consequently $C_{-*}(\Omega M)$ is a more refined invariant of $M$ than $\Omega C_{-*}(M)$. 

In this paper, we pursue this idea in the context of invariants
associated to Lagrangian and Legendrian submanifolds. Here the role played by simple connectivity in the above discussion has two natural counterparts, one corresponds to a generalized notion of simply-connectedness for intersecting Lagrangian submanifolds and the other is the usual notion of simply-connectedness for Legendrian submanifolds. 

We start with the geometric data of a Liouville domain $X$ with convex boundary $Y$ and an exact Lagrangian submanifold $L \subset X$ with Legendrian boundary $\Lambda \subset Y$. We assume that $c_{1}(X)=0$, that the Maslov class of $L$ vanishes (for grading purposes) and that $L$ is relatively spin (to orient certain moduli spaces of holomorphic disks). 
Assume that $L$ is subdivided into embedded components intersecting transversely $L = \bigcup_{v \in \Gamma} L_v$ and that $\Lambda$ is subdivided into connected components $\Lambda =
\bigsqcup_{v \in \Gamma} \Lambda_v$. To avoid notational complications, we take both parametrized by the same finite set $\Gamma$ and assume that the boundary of $L_v$ is $\Lambda_v$. We use a base field $\mathbb{K}$ and define the semi-simple 
ring
\[ 
\k = \bigoplus_{v \in \Gamma} \mathbb{K} e_v, 
\]  
generated by mutually orthogonal idempotents $e_v$. Also, we fix a partition 
\[ 
\Gamma = \Gamma^{+} \cup \Gamma^{-} 
\]
into two disjoint sets, and choose a base-point $p_v \in \Lambda_v$ for each $v \in \Gamma^{+}$. 

For simplicity, let us restrict, in this introduction, to the following situation: 
\begin{itemize}
    \item $X$ is a subcritical Liouville domain, 
    \item If $v \in \Gamma^{-}$ then the corresponding
Legendrian $\Lambda_v$ is an embedded \emph{sphere}. 
\end{itemize} 
From a technical point of view, these restrictions are unnecessary. We make them in order to facilitate the explanation of our constructions from the perspective of Legendrian surgery. (Note that, the topology of $\Lambda_v$ is unrestricted when $v \in \Gamma^{+}$.) 

We write $X_{\Lambda}$ for the completion of the Liouville sector obtained from $X$ by attaching critical Weinstein
handles along $\Lambda_v$ for each $v \in \Gamma^-$ and cotangent cones $T^*(\Lambda_v
\times [0,\infty))$ along $\Lambda_v$ for each $v \in
\Gamma^+$. If $\Gamma^+= \varnothing$, $X_\Lambda$ is an ordinary Liouville manifold. In this case Gromov compactness is ensured by convexity of the boundary. When $\Gamma^+ \neq \varnothing$, we also have part of the boundary that can be identified with a neighborhood of the zero section in the cotangent bundle $\bigcup_{v\in\Gamma^{+}} T^{\ast}(\Lambda_{v}\times [T,\infty))$, for some $T>0$. This is a geometrically bounded manifold, hence Gromov compactness still holds \cite{Gromov}, and holomorphic curve theory is well-behaved. 

In $X_\Lambda$, for $v\in \Gamma^{-}$, there is a closed exact Lagrangian submanifold $S_v = L_v \cup D_v$, the union of the Lagrangian $L_v$
in $X$ and the \emph{Lagrangian core disk} $D_v$ of the Weinstein handle attached to $\Lambda_v$, and for $v \in \Gamma^+$, there is a non-compact Lagrangian obtained by attaching the cylindrical boundary
$\Lambda_v \times [0,\infty)$ to $L_v$ for $v \in \Gamma^+$, which we will still denote by $L_v$, by
abuse of notation, even when we view them now in $X_\Lambda$. Dually, for each $v \in \Gamma^-$, we obtain (non-compact) exact Lagrangian disks $C_v$, the \emph{Lagrangian cocore disks} of the Weinstein handles attached to $\Lambda_v$ on $X$, and for each $v \in
\Gamma^+$, we construct \emph{dual} Lagrangians disks $C_v$ intersecting $L_v$ once and asymptotic to a Legendrian meridian of $L_v$ (these can be constructed
as the cotangent fiber at the point $(p_{v},t)$, $t>0$, in $T^*(\Lambda_v \times [0,\infty)) \subset X_\Lambda$, where $p_v$ is the base points on $\Lambda_v$). 

The invariants that we will construct are associated to the union of Lagrangian submanifolds 
\[ 
L_\Lambda := \bigcup_{v \in \Gamma^{+}} L_v \ \cup \ \bigcup_{v \in \Gamma^{-}} S_v  \quad\text{\ \ \ and\ \ \   }\quad
C_\Lambda := \bigcup_{v \in \Gamma} C_v. 
\] 
The Lagrangian $L_\Lambda$ will be referred to as a \emph{Lagrangian skeleton} of $X_\Lambda$, it
is a union of Lagrangian submanifolds which intersect transversely. The dual Lagrangian
$C_\Lambda$ is the union of Lagrangian disks which can be locally identified with cotangent fibers to irreducible components of $L_\Lambda$. 

We will study two algebraic invariants associated to $(X_\Lambda, L_\Lambda, C_\Lambda)$. The first
one is the \emph{Legendrian $A_\infty$-algebra}, $LA^*$. It corresponds to the endomorphism algebra of
$L_\Lambda$ considered in the infinitesimal Fukaya category of $X_{\Lambda}$ (Thm. \ref{ainfiso}). The second one is
the \emph{Chekanov-Eliashberg DG-algebra}, $CE^*$. It corresponds to the endomorphism algebra of $C_\Lambda$ considered in the partially wrapped Fukaya category of $X_{\Lambda}$ (Thm. \ref{l:BEEdiagonal}). However, 
we will take the pre-surgery perspective as in \cite{BEE} and construct all these invariants by
studying Legendrian invariants of $\Lambda\subset X$ rather than Floer cohomology in $X_\Lambda$. From this perspective, the case $\Gamma^+ \neq \varnothing$ is a new construction, that generalizes the theory from \cite{BEE} in a way analogous to how partially wrapped Fukaya categories \cite{sylvan} generalizes wrapped Fukaya categories \cite{AbouzSeidel}. 

The invariants $LA^*$ and $CE^*$ come equipped with canonical augmentations to the semi-simple ring
$\k$ and it is easy to see by construction that $LA^*$ is determined by $CE^*$ via the equivalence:
\[ LA^* \simeq \mathrm{RHom}_{CE^*} (\k, \k). \]
The duality which would recover $CE^{\ast}$ from $LA^{\ast}$ holds in the ``simply-connected'' case (see Section \ref{algkoszul}). In the topological case discussed above, this is analogous to the
simply-connectedness assumption on $M$, which makes the augmented algebras $C^*(M)$ and
$C_{-*}(\Omega M)$ Koszul dual. In fact, the topological case is a special case of our study for
the Weinstein manifold $T^*M$, with the Lagrangian skeleton $L_\Lambda = M$ given by the
0-section, and the dual Lagrangian $C_\Lambda$ given by a cotangent fiber $T^*_{p} M$. This is because the wrapped Floer cohomology complex of a cotangent fiber is quasi-isomorphic to $C_{-*}(\Omega M)$ by \cite{abloop} and the Floer cohomology complex of the zero section is quasi-isomorphic to $C^*(M)$ (\cite{fukayaoh}) as augmented $A_\infty$ algebras. 

We next sketch the definition of our version of the Chekanov-Eliashberg DG-algebra without any assumption of simply-connectedness (See Section \ref{maininvariants} for details). This is the DG-algebra over $\k$ called
$CE^*$ above. Its underlying $\k$-bimodule is the unital $\k$-algebra generated by Reeb chords between components of
$\Lambda$ and chains in $C_{-*}(\Omega_{p_v} \Lambda_v)$ for $v \in \Gamma^+$. (This is the crucial distinctions between $\Gamma^+$ and $\Gamma^-$.)

We use the cubical chain complex (cf. \cite{serre}) $C_{-*}(\Omega_{p_v} \Lambda_v)$ for $v \in \Gamma^{+}$, see Section \ref{Coefficients} for a discussion of other possible choices of chain models, 
to express $CE^*$ as a free algebra over $\k$ generated by Reeb chords $c$ and generators 
of $C_{-*}(\Omega_{p_v} \Lambda_v)$ for $v \in \Gamma^{+}$. The differential on $CE^*$ is determined by its action on generators. On a generator 
of $C_{-*}(\Omega_{p_v} \Lambda_v)$ we simply apply the usual differential.
On a generator $c_0$ which is a Reeb chord, the
differential is determined by moduli spaces of holomorphic disks in the symplectization $\mathbb{R}
\times Y$ which asymptotically converge to $c_0$ on the positive end and chords $c_1, \ldots, c_i$
at the negative end as follows.
\begin{figure}[h!]
    \centering
    \begin{tikzpicture}[scale=1]
    \tikzset{->-/.style={decoration={ markings,
                mark=at position #1 with {\arrow{>}}},postaction={decorate}}}
    \draw [blue, thick=1.5] (2,3) to[in=90,out=270] (-0.5,0);
    \draw [blue, thick=1.5] (0,0) to[in=90,out=90] (0.5,0);         
    \draw [red, thick=1.5] (1,0) to[in=90,out=90] (1.5,0);         
    \draw [blue, thick=1.5] (2,0) to[in=90,out=90] (2.5,0);         
    \draw [blue, thick=1.5] (3,0) to[in=90,out=90] (3.5,0);         
    \draw [red, thick=1.5] (4,0) to[in=90,out=90] (4.5,0);         
    \draw [red, thick=1.5] (2.5,3) to[in=90,out=270] (5,0);
    \draw [black, thick=1, ->-=.5] (0,0) to (-0.5,0); 
            \draw [black, thick=1, ->-=.5] (1,0) to (0.5,0); 
    \draw [black, thick=1, ->-=.5] (2,0) to (1.5,0); 
    \draw [black, thick=1, ->-=.5] (3,0) to (2.5,0); 
    \draw [black, thick=1, ->-=.5] (4,0) to (3.5,0); 
    \draw [black, thick=1, ->-=.5] (5,0) to (4.5,0); 
    \draw [black, thick=1, ->-=.5] (2.5,3) to (2,3); 

        \node at (2.3, 3.3) {\footnotesize{$c_0$}};
    \node at (-0.25,-0.3) {\footnotesize{$c_1$}}  ;
    \node at (0.75,-0.3) {\footnotesize{$c_2$}}  ;
\node at (1.75,-0.3) {\footnotesize{$c_3$}}  ;
\node at (2.75,-0.3) {\footnotesize{$c_4$}}  ;
\node at (3.75,-0.3) {\footnotesize{$c_5$}}  ;
\node at (4.75,-0.3) {\footnotesize{$c_6$}}  ;

        \node at (1,2) {\footnotesize{$\sigma_0$}}; 
        \node at (0.25, 0.3) {\footnotesize{$\sigma_1$}};
 \node at (2.25, 0.3) {\footnotesize{$\sigma_2$}};
 \node at (3.25, 0.3) {\footnotesize{$\sigma_3$}};
    \end{tikzpicture}
    \caption{The differential in $CE^*$, the word $\sigma_0 c_1 \sigma_1 c_2 c_3 \sigma_2 c_4
    \sigma_3 c_5 c_6$ appears in $dc_0$.}
    \label{diff}
\end{figure}
Consider the moduli space of $J$-holomorphic maps  $u\colon D \to \mathbb{R} \times Y$, where $D$ is a disk with $k+1$ boundary
punctures $z_j \in \partial D = S^1$ that are mutually distinct and $(z_0,z_1,\ldots z_k)$ respects
the counter-clockwise cyclic order of
    $S^1$, and $u$ sends the boundary component $(z_{j-1},z_{j})$  of $S^1\setminus\{z_0,\ldots,
        z_k \}$ to $\mathbb{R} \times \Lambda$, and is asymptotic to $c_j$ near the puncture at
        $z_j$ for $j=1,\ldots, {k}$ and to $c_0$ near the puncture at $z_0$ (as usual these disks may be anchored in $X$). The moduli space,
        which is naturally a stratified space with manifold strata that carries a fundamental chain, comes with evaluation maps to $\Omega_{p_v}
        \Lambda_v$ for $v \in \Gamma^{+}$. The image of the fundamental chain determines a word in our chain model of
        $C_{-*}(\Omega_{p_v} \Lambda_v)$. Reading these together with the Reeb chords in order gives
        the differential of $c_0$. 
        
We remark that loop space coefficients were used in the
        context of Lagrangian Floer cohomology before \cite{barcor,Fuk}. See also \cite{abloop},
        \cite{CL} for uses of high-dimensional moduli spaces in Floer theory.

While $CE^*$ with loop space coefficients is a powerful invariant, it is in general hard to compute
as it involves high-dimensional moduli spaces of disks. As mentioned above, duality in the Legendrian $\Lambda$ will also play a role. More precisely, we define
another DG-algebra $CE^{\ast}_{\parallel}$ related to $CE^*$ via a Morse theoretic version of Adams
cobar construction whose definition involves taking parallel copies of $\Lambda$ but uses only
0-dimensional moduli spaces (see Section \ref{ssec:parallelcopies}). In fact, we prove that the two
DG-algebras are quasi-isomorphic when $\Lambda_{v}$ for $v \in \Gamma^{+}$ are simply-connected. 

\begin{thm}\label{t:parallelintro} 
	There exists a DG-algebra map 
    \[ CE^* \to CE^*_{\parallel} \]
    which is a quasi-isomorphism when $\Lambda_v$ is simply-connected for all $v \in \Gamma^{+}$.
\end{thm} 

Theorem \ref{t:parallelintro} is restated and proved as Theorem \ref{t:parallel=loops} in Section \ref{ssec:parallelcopies}.

\subsection{Partially wrapped Fukaya categories by surgery}\label{ssec:partwrap} 
Let $\Lambda=\bigsqcup_{v\in\Gamma}\Lambda_{v}$ be a Legendrian submanifold and $\Gamma=\Gamma^{+}\cup\Gamma^{-}$ as above. Furthermore, we use the notation above for co-core disks and write $CE^{\ast}=CE^{\ast}(\Lambda)$.  An important result that is implicit in \cite[Remark 5.9]{BEE} is the following: 

\begin{thm}   \label{srgry}  
    Suppose $\Gamma = \Gamma^{-}$, then there exists a surgery map defined via a holomorphic disk count that gives an 
	$A_\infty$-quasi isomorphism
    between the wrapped Floer cochain complex    
    $CW^* := \bigoplus_{v,w \in \Gamma^{-} } CW^*(C_v, C_w)$
    and the Legendrian DG-algebra $CE^*$.
\end{thm} 

We prove Theorem \ref{srgry} in Appendix
\ref{ssec:CWBEE} following \cite{BEE}, referring to \cite{Esurgerycurves} for the necessary technical results omitted there. Appendix \ref{sec:CWnoHam}, also contains a
construction of wrapped Floer $A_\infty$-algebras that uses only purely holomorphic disks (without Hamiltonian perturbation), and a proof that this agrees with the more standard version
defined in \cite{AbouzSeidel} which uses Hamiltonian perturbations. 

One of the main guiding principles for the results in this paper is that Theorem \ref{srgry} remains
true when $\Gamma^{+}$ is non-empty, provided the Lagrangians
$C_v$ are considered as objects of the \emph{partially wrapped} Fukaya category of
$X_\Lambda$, where the non-capped Legendrians $\Lambda_v$ for $v \in \Gamma^{+}$ serve as
\emph{stops} (cf. \cite{sylvan}). The full proof of this result when $\Gamma^+$ is non-empty can be reduced to the standard surgery result, Theorem \ref{srgry}, and will appear elsewhere. Here we give an outline of a somewhat different and more topological proof, see Appendix \ref{ssec:CWpwrap}. 
We will use the geometric intuition provided by this viewpoint and our constructions of Legendrian invariants provide a rigorous ``working definition'' of
$CE^*$ even in the case that $\Gamma^{+}$ is non-empty, and also a starting point for the study of
``partially wrapped Fukaya categories'' via Legendrian surgery (extending the scope of \cite{BEE} considerably). For future reference, we state this result as a conjecture:

\begin{conj}\label{conjintro} 
	There exists a surgery map defined via moduli spaces of holomorphic disks which gives an $A_\infty$-quasi isomorphism between the partially wrapped Floer
    cochain complex $CW^* := \bigoplus_{v,w \in \Gamma} CW^*(C_v,C_w)$ and the DG-algebra 
    $CE^*$. 
\end{conj}

While writing this paper, we learned that Z. Sylvan independently considered a similar conjecture \cite{syltalk} in relation with his theory of partially wrapped Fukaya categories \cite{sylvan}.

\subsection{Augmentations and infinitesimal Fukaya categories}\label{ssec:infinteswrap} 
We keep the notation above and now consider an exact Lagrangian filling $L$ in $X$ of $\Lambda$. Such a filling gives an augmentation 
\[ 
\epsilon_{L} \colon CE^*  \to \k.  
\]
For chords on components $\Lambda_{v}$, $v\in\Gamma^{-}$ this is well known and given by a count of
rigid disks with one positive puncture and boundary on $L_{v}$. 

For components $\Lambda_{v}$, $v\in\Gamma^{+}$ we define an augmentation using the same disks. More formally, we define a chain map 
\[ 
\beta_{L} \colon CE^*  \to \oplus_{v \in \Gamma^+} C_{-\ast}(\Omega L_v),  
\]
which acts on chains in $\oplus_{v \in \Gamma^+} C_{-\ast}(\Omega\Lambda_v)$ by the inclusion and on Reeb chords $c$ as the chain of loops carried by the moduli spaces of holomorphic disks with boundary on $L_v$ (for each $v$) and a positive puncture at $c$. The augmentation $\epsilon_{L}$ is then this map followed by the augmentation on $\oplus_{v \in \Gamma^+} C_{-\ast}(\Omega L_v)\to\k$ that takes higher dimensional chains to zero and takes any loop in $L_{v}$ to $e_{v}$.

This allows us to write 
\[ CE^* = \Omega LC_* \]
for an $A_\infty$-coalgebra $LC_*=LC_{\ast}(\Lambda)$ that we call the \emph{Legendrian $A_\infty$-coalgebra} (which depends on $\epsilon_L$). Here $\Omega$ is the Adams cobar
construction. Writing $LA^{\ast}:= (LC_*)^{\#}$ for the $A_{\infty}$ algebra which is the linear dual of $LC_{\ast}$, the following result recovers the Floer cochain complex of $L$ in $X_\Lambda$.
\begin{thm} \label{generation} There exists an $A_\infty$-quasi isomorphism between the Floer cochain complex
    in the infinitesimal Fukaya category of $X_\Lambda$,
    $CF^* := CF^*(L_\Lambda)$ 
and the $A_\infty$-algebra $LA^*$. 
\end{thm}

Note that, by the general properties of bar-cobar constructions (see Section \ref{barcobarsec}), the
algebra $\mathrm{RHom}_{\Omega LC_*} (\k, \k )$ is quasi-isomorphic to the graded
$\k$-dual of the bar construction on the algebra $\Omega LC_*$, which can be computed as 
\begin{equation}\label{eq:barcobarandAug}
(\Bar \Cobar LC_*)^\#\cong (LC_*)^\# = LA^*.
\end{equation}

\begin{rem}\label{r:Aug_pm}
In case $\Gamma^{+}$ is empty, the $A_\infty$-algebra $LA^*$ is obtained from
the construction in \cite{CEKSW} and \cite{BC}, known as $\mathrm{Aug}_-$ category, by adding a copy of $\k$ making it unital. 

In case $\Gamma^{-}$ is empty, the $A_\infty$-algebra $LA^*\approx (\Bar CE^{\ast})^{\#}$, see \eqref{eq:barcobarandAug}, is the endomorphism algebra of $\Lambda$ with the augmentation $\epsilon_{L}$ in the $\mathrm{Aug}_+$ category of \cite{NRSSZ}. In the setting of microlocal sheaves, a related result was obtained by Nadler \cite[Theorem 1.6]{nadler}.
\end{rem}

\subsection{Duality between partially wrapped and infinitesimal Fukaya categories}\label{ssec:connections} 

We study duality in the setting of the two categories described above: the partially wrapped Fukaya category and the infinitesimal Fukaya category of
$X_\Lambda$ (after surgery), or equivalently, the augmented DG-algebra $\Omega LC_*$ and the
augmented $A_\infty$-algebra $LA^*$ (before surgery). 

As we have seen in Theorem \ref{generation}, the augmented
DG-algebra $\Omega LC_*$ determines the augmented (unital) $A_\infty$-algebra
$CF^*$. Now, a natural question is to what extent the quasi-isomorphism type of the $A_\infty$-algebra
$CF^*$ determines the quasi-isomorphism type of the augmented Legendrian DG-algebra $\Omega LC_*$. 

We emphasize here the phrase ``quasi-isomorphism type'': even though it is possible
to construct chain models of the $A_\infty$-algebra  $LA^*$ (which is $A_\infty$-quasi isomorphic to
$CF^*$) and the DG-algebra $\Omega LC_*$ by counting ``the same'' holomorphic disks interpreted in
different ways, the two algebras are considered with respect to different
equivalence relations, and the resulting equivalence classes can be very different. In particular, it is \emph{not} generally true that $\mathfrak{f}\colon \C \to \mathscr{D}$ being a
quasi-isomorphism of $A_\infty$-coalgebras implies that $\Omega
\mathfrak{f}\colon \Omega \C \to \Omega \mathscr{D}$ is a quasi-isomorphism. 

We will study this question by (geometrically) constructing a \emph{twisting cochain} (see
Section \ref{twisting}): 
\[ \t \colon LC_* \to (\Bar CF^*)^\#,  \] 
where $\Bar$ stands for the bar construction and $\#$ is the graded $\k$-dual.
This twisting cochain induces a map of DG-algebras:
\[ \Omega LC_* \to \mathrm{RHom}_{CF^*}( \k, \k), \]
which is a quasi-isomorphism if and only if $\t$ is a \emph{Koszul} twisting cochain. For
example, we will prove the following result. 

\begin{thm} Suppose that $LC_*$ is a locally finite, simply-connected $\k$-bimodule, then
    the natural map $\Omega LC_* \to \mathrm{RHom}_{CF^*}(\k, \k)$ is a quasi-isomorphism.
\end{thm}

This is an instance of \emph{Koszul duality} between the $A_\infty$-algebras $\Omega LC_*$ and
$CF^*$. It has many useful implications. For example, it implies an isomorphism between Hochschild cohomologies: 
\[
HH^*(\Omega LC_*, \Omega LC_*) \cong HH^*(CF^*, CF^*).
\]
When $\Gamma^+ = \varnothing$, an isomorphism defined via a surgery map \cite{BEE} was described between \emph{symplectic cohomology}, $SH^{\ast}=SH^*(X_{\Lambda})$, and the Hochschild cohomology $HH^*(\Omega LC_*, \Omega LC_*)$. Therefore, when duality holds (i.e. $\t$ induces an isomorphism), we obtain a more economical way of computing $SH^*$. 

In the case of cotangent bundles $T^{\ast}M$ of simply-connected manifolds $M$, this recovers a classical result due to
Jones \cite{jones}, which gives:
\[ H_{n-*}(\mathcal{L} M) \cong HH^*(C_{-*}(\Omega M), C_{-*}(\Omega M)) \cong HH^*(CF^*(M),
CF^*(M)), \] 
where $M$ is a simply-connected manifold of dimension $n$ and $\mathcal{L} M$ denotes the free loop space of $M$.

In Section \ref{ssec:applications}, we give several concrete examples where the duality holds beyond the
case of cotangent bundles. For example, the duality holds for plumbings of simply-connected cotangent bundles
according to an arbitrary plumbing tree, see Theorem \ref{plumbs}. 

In another direction, combining duality and Floer cohomology with local coefficients, we establish the following result for relatively spin exact Lagrangian fillings $L\subset X$ with vanishing Maslov class of a Legendrian submanifold $\Lambda\subset Y$. 

\begin{thm}
	Let $\Gamma=\Gamma^{-}$ and assume that $SH^*(X)=0$ and that $\Lambda$ is simply-connected. If $CE^{\ast}(\Lambda)$ is supported in degrees $<0$, then $L$ is simply connected. Moreover, if $\Lambda$ is a sphere then $CE^{\ast}(\Lambda)$ is isomorphic to $C_{-\ast}(\Omega \overline{L})$, where $\overline{L}=L\cup_{\Lambda} D$, for a disk $D$ with boundary $\partial D=\Lambda$.
\end{thm}

In general, duality between $\Omega LC_{\ast}$ and $CF^{\ast}$ does not hold (as can be seen for example by looking at
cotangent bundles of non-simply connected manifolds or, letting $\Lambda$ be the
standard Legendrian trefoil knot in $S^3$ filled by a punctured torus). However, there are cases
when duality holds even if $LC_*$ is not simply-connected, for instance because of the existence of
an auxiliary weight grading, see \cite{EtLe}, or for an example in the 1-dimensional case see
\cite{LePol}. It is a very interesting open question to find a geometric characterization of when duality holds.

\begin{rem}
Constructions of Legendrian and Lagrangian holomorphic curve invariants require the use of perturbations to achieve transversely cut out moduli spaces. 
For our main invariant $CE^{\ast}$, all moduli spaces used can be shown to be transverse by classical techniques (see Theorem \ref{thm:mdlitv}) except for the rigid holomorphic planes in $X_\Lambda$ with a single positive end that are used to anchor the disks in the terminology of \cite{BEE}.  These are also relevant for defining the wrapped Floer cochain complex $CW^*$ without Hamiltonian perturbations and for constructing the surgery map. In all cases, there is a distinguished boundary puncture in the main disk that determines a asymptotic markers on the split off planes. Taking this marker into account removes symmetries of the planes and a specific perturbation scheme for transversality of the resulting moduli spaces was constructed in \cite{Esurgerycurves}. We will use that perturbation scheme here, see Section \ref{ssec:perturbationsanchor} for details. 

%
%

\end{rem}

\subsection{An example -- the Hopf Link}\label{ssec:Hopflink}

In this section, we study the example of the Hopf link in order to illustrate our results in a simple and computable example. Some of the algebraic constructions used here are
explained in detail only later, see Section \ref{algebra}. 

\begin{figure}[h!]
    \centering
    \begin{tikzpicture}[scale=1]

    \tikzset{->-/.style={decoration={ markings,
                mark=at position #1 with {\arrow{>}}},postaction={decorate}}}
            
\draw [blue, thick=1.5, ->-=.7] (-1.5,1) to[in=90,out=190] (-2.3,0);
    \draw [blue, thick=1.5] (-1.5,1) to[in=135,out=10] (-0.3,0.3);         
    \draw [blue, thick=1.5]  (-1.5,-1)to[in=270,out=170] (-2.3,0);
    \draw [blue, thick=1.5] (-1.5,-1)to[in=225,out=350] (-0.3,-0.3)  ;
    
    \draw [blue, thick=1.5] (1.5,1) to[in=45,out=170] (0.3,0.3);         
    
    \draw [blue, thick=1.5]  (1.5,-1)to[in=270,out=10] (2.3,0);
    \draw [blue, thick=1.5] (1.5,1) to[in=140,out=350] (1.9,0.8);
    \draw [blue, thick=1.5] (2.05,0.65) to[in=90,out=320] (2.3,0);

    \draw [blue, thick=1.5] (1.5,-1)to[in=315,out=190] (0.3,-0.3)  ;

    \draw [blue, thick=1.5] (-0.3,0.3) to (0.3,-0.3);
    \draw [blue, thick=1.5] (-0.3,-0.3) to (-0.1,-0.1);
    \draw [blue, thick=1.5] (0.1,0.1) to (0.3,0.3);

    \draw [red, thick=1.5] (2.5,1) to[in=90,out=190] (1.7,0);

    \draw [red, thick=1.5] (2.5,-1) to[in=320,out=170] (2.1,-0.8);
    \draw [red, thick=1.5] (1.95,-0.65) to[in=270,out=140] (1.7,0);

    \draw [red, thick=1.5] (2.5,1) to[in=135,out=10] (3.7,0.3);         
    \draw [red, thick=1.5] (2.5,-1)to[in=225,out=350] (3.7,-0.3)  ;
    
    \draw [red, thick=1.5] (5.5,1) to[in=90,out=350] (6.3,0);
    \draw [red, thick=1.5] (5.5,1) to[in=45,out=170] (4.3,0.3);         
    \draw [red, thick=1.5, ->-=.7]  (6.3,0)to[in=10,out=270] (5.5,-1);
    \draw [red, thick=1.5] (5.5,-1)to[in=315,out=190] (4.3,-0.3)  ;

    \draw [red, thick=1.5] (3.7,0.3) to (4.3,-0.3);
    \draw [red, thick=1.5] (3.7,-0.3) to (3.9,-0.1);
    \draw [red, thick=1.5] (4.1,0.1) to (4.3,0.3);

    \node at (0,0.3) {\footnotesize{$c_1$}}; 
        \node at (2,1) {\footnotesize{$c_{21}$}}; 
        \node at (2,-1) {\footnotesize{$c_{12}$}}; 
    \node at (4,0.3) {\footnotesize{$c_2$}}; 

    \end{tikzpicture}
    \caption{Hopf link when both $\Gamma^{+}$ and $\Gamma^{-}$ are
    non-empty, the blue component lies in $\Gamma^{+}$ and the red in $\Gamma^{-}$.}
    \label{hopf}
\end{figure}

Let $\Lambda \subset S^3$ be the standard Legendrian Hopf link. We work over $\k =
\mathbb{K} e_1 \oplus \mathbb{K} e_2$ and with the Lagrangian filling $L$ given by two disks in $D^4$ that intersect transversely in a single point.
We choose the partition $\Lambda = \Lambda^+ \cup \Lambda^-$. This means that after attaching a Weinstein 2-handle to $\Lambda^{-}$ and $T^{\ast}(S^{1}\times[0,\infty))$ to $\Lambda^{+}$, we obtain the symplectic manifold $X_{\Lambda}$ with Lagrangian skeleton 
\[ L_\Lambda = S^2 \cup T_{\mathrm{pt}}^* S^2 \subset T^{\ast} S^{2}, \]
or in the terminology of \cite{sylvan}, $X_{\Lambda}$ is $T^{\ast}S^{2}$ with wrapping stopped by a Legendrian fiber sphere.   
The DG-algebra $CE^{\ast}=CE^*(\Lambda)$ of $\Lambda$ has coefficients in 
\[C_{-*}(\Omega \Lambda^+) e_1 \oplus \mathbb{K} e_2 \cong \K[t,t^{-1}] e_1 \oplus \K e_2. \]
A free model for $\K[t,t^{-1}]$ is given by the tensor algebra $\K\langle s_1,
t_1, k_1, l_1, u_1 \rangle$ with $|s_1|=|t_1|=0$, $|k_1|=|l_1|=-1$, $|u_1|=-2$ and the differential  
\begin{align*}
    dk_1 &= e_1 - s_1 t_1 \\
    dl_1 &= e_1 - t_1 s_1 \\
    du_1 &= k_1 s_1 - s_1 l_1 
\end{align*}
The natural map $\K\langle s_1,t_1,k_1,l_1,u_1\rangle \to \K[t,t^{-1}]$ sending $t_1 \to t$ and $s_1 \to t^{-1}$ is a quasi-isomorphism. The subscripts indicate that as $\k$-module generators
$s_1, t_1, k_1, l_1, u_1$ are annihilated by the idempotent $e_2$. 

Next, incorporating the Reeb chords, with notation as in Figure \ref{hopf}, we get the free algebra 
\[  \k \langle c_{12} , c_{21} , c_{1}, c_{2},
s_1, t_1, k_1, l_1, u_1 \rangle \] 
with gradings\[|u_1|=-2, |c_{1}|=|c_{2}| = |k_1|= |l_1|=  -1,\ |c_{12}|=|c_{21}|=|s_1|=|t_1|=0  \] and differential:
\begin{align*} 
    dc_1 &= e_1 + s_1 + c_{12}c_{21}, \\
    dc_2 &= c_{21} c_{12}, \\
    dk_1 &= e_1 - s_1 t_1, \\
    dl_1 &= e_1 - t_1 s_1, \\
    du_1 &= k_1 s_1 - s_1 l_1. 
\end{align*}
The only augmentation to $\k$ is given by $\epsilon(s_1) = \epsilon(t_1) = -e_1$, and
$\epsilon(c_1) = \epsilon(c_2) = \epsilon(c_{12})
=\epsilon(c_{21})=\epsilon(k_1)= \epsilon(l_1)= \epsilon(u_1) = 0$. After change of
variables, $s_1 \to s_1 - e_1$ and $t_1 \to t_1 - e_1$, we obtain the free algebra 
\[ \k \langle c_{12} , c_{21} , c_{1}, c_{2},
s_1, t_1, k_1, l_1, u_1 \rangle \] 
with non-zero differential on generators: 
\begin{equation}
\begin{aligned} 
    dc_1 &= s_1 + c_{12}c_{21}, \\
    dc_2 &= c_{21} c_{12}, \\
    dk_1 &= s_1 +t_1 - s_1 t_1, \\
    dl_1 &= s_1 +t_1  - t_1 s_1, \\
    du_1 &= l_1 -k_1 + k_1 s_1 - s_1 l_1. 
\end{aligned}
\end{equation}
On the other hand, we can compute the Floer cochains $CF^{\ast}=CF^{\ast}(L_{\Lambda})$ of $L_\Lambda$ as 
\[ CF^* = \k \oplus \mathbb{K} a_{12} \oplus \K a_{21} \oplus \K a_2, \
|a_{2}|=2,\ |a_{12}|=|a_{21}|=1 \] 
The cohomology level
computation follows easily from the geometric picture and general properties of Floer cohomology:  $L_\Lambda$ is a union of a disk $D^2$ and
a sphere $S^2$ that intersect transversely in one point and we have
\begin{align*}
    HF^*(D^2,D^2) &= \K e_1 \\
    HF^*(S^2,S^2) &= \K e_2 \oplus \K a_2 \\
    HF^*(D^2,S^2) &= \K a_{12} \\
    HF^*(S^2,D^2) &= \K a_{21} 
\end{align*}
The only non-trivial product that does not involve idempotents is given by $\m_2(a_{12},a_{21}) = a_2$. For degree reasons, the only possible non-trivial higher products are:
\[ \m_{2k} (a_{12},a_{21}, \ldots, a_{12},a_{21}) = ca_2  \text{\ for some  } k >1 \text{ and } c
\in \K.\]
It turns out that one can take $c=0$ for all $k>1$. Indeed, assuming that the $A_\infty$-structure
is strictly unital (which can be arranged up to quasi-isomorphism), consider the $A_\infty$-relation that involves the term
\[ \m_2( m_{2k}(a_{12},a_{21},\ldots,a_{12},a_{21}), e_2). \]
By induction on $k>1$, this term has to vanish, which implies
$m_{2k}(a_{12},a_{21},\ldots,a_{12},a_{21})$ has to vanish for all $k>1$. Let us confirm this by
using the quasi-isomorphism:
\[ CF^*  \cong \mathrm{RHom}_{CE^*}(\k,\k).\] 
We introduce the counital $A_\infty$-coalgebra
\[ LC_* = \k \oplus \mathbb{K} c_{12} \oplus \K
c_{21} \oplus \K c_{1} \oplus \K c_{2} \oplus \K s_1 \oplus \K t_1 \oplus \K k_1 \oplus \K l_1
\oplus \K u_1 \] 
with $|u_1|=-3, |c_{1}|=|c_{2}| = |k_1|= |l_1|=  -2,\ |c_{12}|=|c_{21}|=|s_1|=|t_1|=-1$ 
for which $\Delta_{i}=0$ except for $i=1$ or 2 where there are the following non-zero terms:  
\begin{align*} 
    \Delta_1 (c_1) &= s_1, \\
    \Delta_1 (k_1) &= s_1 + t_1, \\
    \Delta_1 (l_1) &= s_1 + t_1, \\
    \Delta_1 (u_1) &= l_1 - k_1.
\end{align*}
Write $\Delta_2 (x) = 1 \otimes_\k x + x \otimes_\k 1 +
\overline{\Delta}_2(x)$. Then
\begin{align*}
    \overline{\Delta}_2(c_1) &= c_{12} c_{21}, \\
    \overline{\Delta}_2(c_{2}) &= c_{21} c_{12}, \\ 
    \overline{\Delta}_2(k_1) &= - s_1 t_1, \\ 
    \overline{\Delta}_2(l_1) &= - t_1 s_1, \\
    \overline{\Delta}_2(u_1) &=  k_1 s_1 - s_1 l_1,
\end{align*} 
where the $A_\infty$ coalgebra operations on $LC_* $ is defined so that $\Omega LC_*$ is isomorphic to $CE^*$. Thus, $\mathrm{RHom}_{CE^*}(\k,\k)$ can be computed as the graded dual of $LC_*$ which is the $A_\infty$-algebra: 
\[ LA^*= \k \oplus \K c_{12}^\vee \oplus \K 
c_{21}^\vee \oplus \K c_{1}^\vee \oplus \K c_{2}^\vee \oplus \K s_1^\vee \oplus \K t_1^\vee \oplus
\K k_1^\vee \oplus \K l_1^\vee \oplus \K u_1^\vee,  \]
with gradings  
\[ 
|u_1^\vee|=3, |c_{1}^\vee|=|c_{2}^\vee| = |k_1^\vee|=|l_1^\vee| =  2,\
|c_{12}^\vee|=|c_{21}^\vee|=|s_1^\vee|=|t_1^\vee|=1, 
\]
where $c^{\vee}$ is the linear dual of the generator $c$ of $LC_{\ast}$.
The differential is  \[ \m_1(s_1^\vee) = c_1^\vee + k_1^\vee + l_1^\vee,
\m_1(t_1^\vee)=k_1^\vee + l_1^\vee, \m_1 (k_1^\vee) = -u_1, \m_1(l_1^\vee) = u_1, \] and the
products that do not involve
idempotents are  
\begin{align*} &\m_2(c_{12}^\vee,c_{21}^\vee) = c_2^\vee, \m_2(c_{21}^\vee, c_{12}^\vee) = c_1^\vee ,
    \m_2(t_1^\vee,s_1^\vee) = -k_1^\vee,\\ &\m_2(s_1^\vee,t_1^\vee) = -l_1^\vee,  \m_2(k_1^\vee,s_1^\vee) =
u_1^\vee, \m_2(s_1^\vee,l_1^\vee) = -u_1^\vee.   
\end{align*} 

All the higher
products vanish. 
We claim that that this $A_\infty$-algebra is quasi-isomorphic to the algebra
\[  \k \oplus \mathbb{K} a_{12} \oplus \K a_{21} \oplus \K a_2, \
|a_{2}|=2,\ |a_{12}|=|a_{21}|=1  \]
with the only non-trivial product (not involving idempotents) given by \[ \m_2(a_{12},a_{21}) =a_2.\]  
Indeed, it is easy to show that the map defined by:
\[ c^\vee_{12} \to a_{12}, c_{21}^\vee \to a_{21}, c_2^\vee \to a_2 , \text{\ and \ } c_1^\vee, s_1^\vee,
t_1^\vee, k_1^\vee, l_1^\vee, u_1^\vee\to 0 \]
is a DG-algebra (hence also $A_\infty$-algebra) map, which induces an isomorphism at the level of
cohomology. 

Dually, we can construct a $DG$-algebra map 
\[ CE^* \to \mathrm{RHom}_{CF^*}(\k,\k). \]
The Floer cochain complex $CF^*$ has a unique augmentation given by projection to $\k$ and we compute 
\[ \mathrm{RHom}_{CF^*}(\k,\k) \cong \Cobar CF_*, \]
where $CF_*$ is the coalgebra dual to $CF^{\ast}$. This is the free coalgebra
\[  \k \langle a^\vee_{12} , a^\vee_{21} , a^\vee_{2} \rangle \] 
with $|a_{12}^\vee|=|a_{21}^\vee|=0$ and $|a^\vee_{2}|=-1$ and the only non-trivial differential not involving counits is
\[ \Delta_{2}(a^\vee_{2}) = a_{21}^\vee a_{12}^\vee. \] 
We have a twisting cochain 
\[ \t\colon LC_* \to \Omega CF_* \]
given by  \begin{align*}
    \t (c_2) &= a^\vee_2,\ \t (c_{12}) = a^\vee_{12},\ \t (c_{21}) = a^\vee_{21},\\
    \t(c_1) &=0,\ \t (s_1) = -a_{12}^\vee a_{21}^\vee ,\ \t (t_1) = a_{12}^\vee a_{21}^\vee,\ \\  \t
    (l_1) &= \t(k_1)=  a_{12}^\vee
    a_2^\vee a_{21}^\vee, \t(u_1) = a_{12}^\vee a_2^\vee a_2^\vee a_{21}^\vee.
\end{align*}
This means that $\t$ satisfies the following equations:
\begin{align*} 
    d\t(c_1) &= \t(s_1) + \t(c_{12}) \t(c_{21}), \\
    d\t(c_2) &= \t(c_{21}) \t(c_{12}), \\
    d\t(k_1) &= \t(s_1) + \t(t_1) - \t(s_1) \t(t_1), \\
    d\t(l_1) &= \t(s_1) + \t(t_1) - \t(t_1) \t(s_1), \\
    d\t(u_1) &= \t(l_1) -\t(k_1) + \t(k_1) \t(s_1) - \t(s_1) \t(l_1).  
\end{align*}
Hence, it induces a DG-algebra map 
\[  \Omega LC_*  \to \Omega  CF_*. \]
We have not checked whether this is a quasi-isomorphism, or equivalently whether $\t$ is a
\emph{Koszul} twisting cochain. Note, however, the DG-algebra map $\Omega CF_* \to \Omega LC_* $ defined by
\[ a_2^\vee \to c_2, a_{12}^\vee \to c_{12}, a_{21}^\vee  \to c_{21}, \]
shows that $\t$ is a retraction, and $\Omega CF_* $ is a retract of $\Omega LC_*$.

\vspace{.5cm}

\noindent
{\bf Acknowledgments.} T.E. is supported by the Knut and Alice Wallenberg Foundation and by the
Swedish Research Council. Y.L. is supported in part by the Royal Society (URF) and the NSF grant
DMS-1509141. Both authors would like to thank the Mittag-Leffler institute for hospitality and
excellent working conditions. We also thank Lenny Ng for providing the example in Section
\ref{lenny}, and Zack Sylvan and Paolo Ghiggini for helpful comments.

\section{Algebraic preliminary}
\label{algebra}

In this section, we review the homological algebra we use in our study of various invariants
associated to Legendrian submanifolds and their Lagrangian fillings. Most of this material is well
established, see \cite{L-H} and also
\cite{Keller}, \cite{proute}, \cite{Posit}, \cite{herscovich}, \cite{LV}, \cite{LPWZ}. Note though that our sign conventions follow \cite{seidelbook}, see Remark \ref{rem:Seidelsign}. 

\subsection{$A_\infty$-algebras and $A_\infty$-coalgebras}

In this section we will discuss the basic algebraic objects we use. These are modules over a ground ring $\k$ of the following form.  
Fix a coefficient field $\mathbb{K}$ (of arbitrary characteristic) and let $\k$ be a semi-simple ring of the form:
\[ 
\k = \oplus_{v\in\Gamma} \,\mathbb{K} e_v 
\]
where $e_v^2= e_v$ and $e_v e_w =0$ for $v \neq w$, and where the index set $\Gamma$ is finite.

We will use $\Z$-graded $\k$-bimodules. If $M= \bigoplus_i M^{i}$ is such a module then we call $M$ \emph{connected} if $M^0 \cong \k$ and
either $M^i =0$ for all $i>0$, or $M^{i}=0$ for all $i<0$, and we call $M$ \emph{simply-connected} if, in addition, in the former case $M^{-1}=0$ and in the latter $M^{1}=0$. Further, we say that $M$ is \emph{locally finite} if each $M^i$ is finitely generated as a $\k$-bimodule.  

We have the usual shifting and tensor product operations on modules. If $M = \bigoplus_{i} M^{i} $ is a graded $\k$-bimodule and $s$ is an integer then we let the corresponding shifted module $M[s]=\bigoplus_{i}M[s]^{i}$ be the module with graded components  
\[ 
M[s]^{i} = M^{i+s}. 
\]
If $N=\bigoplus_{i} N^{i}$ is another graded $\k$-bimodule then
$M\otimes_{\k}N=\bigoplus_{k}\left(M\otimes_{\k}N\right)^{k}$ is naturally a graded $\k$-bimodule with
\[ \left(M\otimes_{\k}N\right)^{k}=\bigoplus_{i+j=k}M^{i}\otimes_{\k} N^{j}.
\]
For iterated tensor products we write 
\[  
M^{\otimes_{\k}r}=\underbrace{M\otimes_{\k}\ldots \otimes_{\k} M}_{r}.
\]

Our modules will often have further structure as $\Z$-graded $A_\infty$-algebras and $A_\infty$-coalgebras over $\k$, see Sections \ref{ssec:Ainftyalg} and \ref{ssec:Ainftycoalg}. The modules are then in particular chain complexes with a differential and we will use \emph{cohomological} grading throughout. That is, the differential \emph{increases} the grading by 1. For
example, if $L$ is a topological space then its cohomology complex $C^*(L)$ is supported in
non-negative grading, while the homology complex $C_{-*}(L)$ is supported in non-positive degrees.
To be consistent with this, we denote the grading as a subscript (resp. superscript) when the underlying chain
complex has a coalgebra (resp. algebra) structure.

\subsubsection{$A_{\infty}$-algebras}\label{ssec:Ainftyalg} An \emph{$A_\infty$-algebra} over $\k$
is a $\Z$-graded $\k$-module $\A$ with a collection of grading preserving $\k$-linear maps \[ \m_i
\colon \A^{\otimes_\k i } \to \A [2-i],  \] for all integers $i \geq 1$, that satisfies the
$A_\infty$-relations: \begin{equation} \label{ainf} \sum_{i,j} (-1)^{|a_1|+\ldots + |a_j|-j}
\m_{d-i+1}\left(a_d,\ldots, a_{j+i+1},\m_i(a_{j+i},\ldots, a_{j+1}), a_j,\ldots, a_1\right) = 0,
\end{equation} for all $d$. 

\begin{rem}\label{rem:Seidelsign}
	We follow the sign conventions of \cite{seidelbook}. Note that
	even though, $\m_i$ is written on the left of $(a_{j+i},\ldots, a_{j+1})$, the sign
	convention is so that $\m_i$ acts from the right. To be consistent, we will insist that all our
	operators act on the right independently of how they are written. This convention and the usual
	Koszul sign exchange rule applied with respect to the shifted grading $\mathscr{A}[1]$ determine the signs that appear in our formulas. 
\end{rem}

A \emph{DG-algebra} over $\k$ is an $A_\infty$-algebra $\A$ such that $\m_i =0$ for $i \geq 3$. In this
case, we call the first two operations the \emph{differential} and the \emph{product}, respectively, and use the following adjustments to obtain an (ordinary) differential graded algebra: 
\begin{equation} \label{left} 
    d a = (-1)^{|a|} \m_1(a) \quad \text{ and } \quad a_2 a_1 = (-1)^{|a_1|} \m_2(a_2,a_1). 
\end{equation}
In particular, the product is then associative and the graded Leibniz rule for $d$ holds:
\begin{equation} \label{Leib} 
    d(a_2 a_1) = (d a_2) a_1 + (-1)^{|a_2|} a_2 (da_1). 
\end{equation}
An \emph{$A_{\infty}$-map} $\e\colon\A\to\B$ between $A_{\infty}$-algebras $\A$ and $\B$ over $\k$, with operations $\m_i$ and $\n_{i}$, $i\ge 1$, respectively, is a collection of $\k$-linear grading preserving maps 
\[ 
\e_i \colon \A^{\otimes_\k i } \to \B [1-i],  
\]
$i\ge 1$, that satisfies the relations
\begin{align*} 
\sum_{i,j} (-1)^{|a_1|+\ldots + |a_j|-j} 
\e_{d-i+1}\left(a_d,\ldots, a_{j+i+1},\m_i(a_{j+i},\ldots, a_{j+1}), a_j,\ldots, a_1\right) \ & =  \\ 
\sum_{\begin{smallmatrix}
1\le j\le d, \\
0<i_{1}<i_{2}<\dots<i_{j}<d
\end{smallmatrix}}  
\n_{j}\left(\e_{d-i_{j}}(a_d,\ldots,a_{d-i_{j}}),\ldots,\e_{i_2-i_1}(a_{i_2},\ldots,a_{i_1+1}), \e_{i_1}(a_{i_1},\ldots, a_{1})\right). 
\end{align*}

An $A_\infty$-map $\e \colon \A \to \B$ is called an \emph{$A_\infty$-quasi isomorphism} if the map on
cohomology $H^*(\A) \to H^*(\B)$ induced by $\e^1$ is an isomorphism.

We say that an $A_{\infty}$-algebra $\A$ is \emph{strictly unital} if there is
an element $1_\A \in\A$ such that $\m_1(1_\A)=0$, $\m_2(1_\A,a)=\m_{2}(a,1_\A)=a$ for any $a\in\A$, and such that
$\m_{i}$, $i > 2$ annihilates any monomial containing $1_\A$ as a factor. Any $A_\infty$-algebra $\A$ which has a cohomological unit,i.e. a cocycle representing the identity element in $H^*(\A)$, is quasi-isomorphic to a strictly unital $A_\infty$-algebra (\cite[Section 7.2]{Posit}).

An \emph{augmentation} of a strictly unital $A_\infty$-algebra is an $A_\infty$-map $\epsilon \colon \A \to \k$,
where $\k$ is considered as a strictly unital $A_\infty$-algebra in degree 0 with trivial
differential and higher $A_\infty$-products, and such
that $\epsilon_1 (1_\A)=1_\k$, $\epsilon_i$ for $i>1$ annihilates any monomial containing $1_\A$.
An augmentation is called $\emph{strict}$ if $\epsilon_{i}=0$ for $i>1$. The category of
augmented, strictly unital $A_\infty$-algebras is equivalent to the category of strictly
augmented, strictly unital $A_\infty$-algebras (see \cite[Section 7.2]{Posit}). 

\subsubsection{$A_{\infty}$-coalgebras}\label{ssec:Ainftycoalg}
An \emph{$A_\infty$-coalgebra} $\C$ over $\k$ is a $\Z$-graded $\k$-module with a collection of
$\k$-linear grading preserving maps:
\[ 
\Delta_i\colon \C \to \C^{\otimes_\k i}[2-i],
\]
for all integers $i \geq 1$, with the following properties. The maps satisfy the co-$A_\infty$-relations:
\begin{equation} \label{coainf} 
\sum_{i=1}^d \sum_{j=0}^{d-i} ( \id^{\otimes_\k (d-i-j)} \otimes_\k \Delta_i \otimes_\k
\id^{\otimes_\k\, j} ) \Delta_{d-i+1} = 0, 
\end{equation}
where,  
\begin{align*} 
    \id^{\otimes_\k (d-i-j)}&\otimes_\k \Delta_i\otimes_\k \id^{\otimes_\k j}\,(c_{d-i+1},\ldots, c_{1}) =\\
    &(-1)^{|c_{1}| +\ldots+ |c_{j}|-j} (c_{d-i+1} , \ldots, c_{j+2})\otimes_\k \Delta_i(c_{j+1}) \otimes_\k (c_{j} ,
    \ldots , c_{1})\\
    & \in \C^{\otimes_\k (d-i-j)} \otimes_\k \C^{\otimes_\k i} \otimes_\k
    \C^{\otimes_\k j}. 
\end{align*}
Furthermore, the degree 1 map
\[ 
\C[-1] \to \prod_{i=1}^\infty \C[-1]^{\otimes_\k i}, 
\]
with $i^{\rm th}$ component equal to $\Delta_i$, factorizes through the natural inclusion
\[ 
\bigoplus_{i=1}^\infty \C[-1]^{\otimes_\k i} \to \prod_{i=1}^\infty \C[-1]^{\otimes_\k i}, 
\]
of the direct sum into the direct product.    
    
A \emph{DG-coalgebra} over $\k$ is an $A_\infty$-coalgebra such that $\Delta_i=0$ for $i\geq 3$.
In this case, we call the first two operations the \emph{differential} and the \emph{coproduct},
respectively, and use the following adjustments to obtain an (ordinary) differential
graded coalgebra:
\begin{equation} \label{coleft}  
    \theta c = (-1)^{|c|} \Delta_1(c) \quad \text{ and } \quad \Delta(c) = 
    \sum (-1)^{|c_{(2)}|} c_{(1)} \otimes_\k c_{(2)}, 
\end{equation}
where we write $\Delta_2(c) = \sum c_{(1)} \otimes_\k c_{(2)}$. 

In particular, the coproduct is coassociative (i.e. $(\Delta \otimes_\k \id)\circ \Delta = (\id
\otimes_\k \Delta) \circ \Delta$) and the graded co-Leibniz rule holds:
\begin{equation} \label{coLeib}
\Delta \theta(c) = \sum (-1)^{|c_{(1)}|} c_{(1)} \otimes_{\k} \theta(c_{(2)}) + \theta(c_{(1)})
\otimes_\k c_{(2)}. 
\end{equation}

An \emph{$A_{\infty}$-comap} $\mathfrak{f} \colon\C\to\D$ between $A_{\infty}$-coalgebras $\C$ and $\D$ over
$\k$, with operations $\Delta_i$ and $\Theta_{i}$, $i\ge 1$, respectively, is a collection of $\k$-linear grading preserving maps 
\[ 
    \mathfrak{f}_i \colon \C \to \D^{\otimes_\k i}[1-i], 
\]
$i\ge 1$, that satisfies the relations:
\begin{align*} 
    \sum_{i=1}^{d} \sum_{j=0}^{d-i} \left(\id^{\otimes_\k (d-i-j)} \otimes_\k \Theta_i \otimes_\k
\id^{\otimes_\k j} \right) \mathfrak{f}_{d-i+1}  \ &= \\ 
\sum_{\begin{smallmatrix}
1\le j\le d, \\
0<i_{1}<i_{2}<\dots<i_{j}<d
\end{smallmatrix}}  
\left( \mathfrak{f}_{d-i_j} \otimes_\k \cdots \otimes_\k \mathfrak{f}_{i_2-i_1} \otimes_\k
\mathfrak{f}_{i_1} \right) \Delta_j.
\end{align*}
where, 
\begin{align*} 
    \id^{\otimes_\k (d-i-j)}\otimes_\k &\Theta_i\otimes_\k \id^{\otimes_\k j}\,(d_{d-i+1},\ldots, d_{1}) =\\
    &(-1)^{|d_{1}|+\ldots+|d_{j}|-j} (d_{d-i+1} , \ldots, d_{j+2})\otimes_\k \Theta_i(d_{j+1}) \otimes_\k (d_{j} ,
    \ldots , d_{1})\\
    & \in \D^{\otimes_\k (d-i-j)} \otimes_\k \D^{\otimes_\k i} \otimes_\k
    \D^{\otimes_\k j}. 
\end{align*}

Furthermore, the degree 0 map 
\[ \C[-1] \to \prod_{i=1}^\infty \D[-1]^{\otimes_\k i} \]
with $i^{th}$ component equal to $\mathfrak{f}_i$, factorizes through the natural inclusion:
\begin{equation} \label{factor} 
\bigoplus_{i=1}^\infty \D[-1]^{\otimes_\k i} \to \prod_{i=1}^\infty \D[-1]^{\otimes_\k i}, 
\end{equation}
of the direct sum into the direct product.    

An $A_\infty$-comap $\mathfrak{f} \colon \C \to \D$ is called an \emph{$A_\infty$-quasi isomorphism} if the map on
cohomology $H^*(\C) \to H^*(\D)$ induced by $\mathfrak{f}^1$ is an isomorphism.

We say that an $A_\infty$-coalgebra is \emph{strictly counital} if there exists a $\k$-linear map
$\epsilon\colon \C \to \k$ such that $(\epsilon \otimes \id) \Delta_2 = (\id
\otimes \epsilon) \Delta_2 = \id$ and $(\id^{\otimes_\k (i-j)} \otimes_\k \epsilon \otimes_\k
\id^{\otimes_\k\, j-1}) \Delta_i =0$ for all $i \neq 2$ and $j$. 
 Any $A_\infty$-coalgebra $\C$ which has a cohomological counit i.e a cocycle representing the counit in $H^*(\C)$, is quasi-isomorphic to a strictly counital $A_\infty$-coalgebra (see \cite[Section 7.5]{Posit}).

A \emph{coaugmentation} of a strictly counital $A_\infty$-coalgebra $\C$ is an $A_\infty$-comap $\eta \colon\k \to
\C$, where $\k$ is considered as a vector space in degree 0 with the trivial
$A_\infty$-coalgebra structure, such that $\epsilon \eta_1 = 1_\k$ and 
$(\id^{\otimes_\k (i-j)} \otimes_\k \epsilon \otimes_\k \id^{\otimes_\k\, j-1}) \eta_i =0$ for all
$i>1$ and $j$. The coaugmentation is called \emph{strict} if $\eta_i=0$ for $i \geq 2$. 

A DG-coalgebra $\C$ is called \emph{conilpotent} (also called
\emph{cocomplete})  if for any $c\in \C$,
there exists an $n \geq 2$ such that $c$ is in the kernel of the iterated comultiplication map defined recursively by $\Delta^{(2)}=\Delta$ and 
$\Delta^{(n)}=(\id^{\otimes_\k (n-2)} \otimes_\k \Delta) \circ \Delta^{(n-1)}$, for $n>2$.  When considering coaugmented DG-coalgebras, conilpotency is enforced only on the coaugmentation ideal $\mathrm{coker}(\eta)$.

\subsubsection{Graded dual}\label{ssec:gradeddual}
We next discuss the graded dual of a graded $\k$-module. Since we are working with bimodules over the ring $\k$, there are two $\k$-linear duals (cf. \cite{bgs}). 

If $\A$ is a graded $\k$-bimodule, $\A=\bigoplus_{i}\A_{i}$, then the $\emph{graded duals}$
$\A^{\#}=\bigoplus_{i}(\A^{\#})_{i}$ and $^{\#}\!\A =\bigoplus_{i}(^{\#}\!\A)_{i}$ are defined as
follows. The graded components $(\A^{\#})_{i}$ of $\A^{\#}$ are left $\k$-module maps:
\[ 
\mathrm{hom}_{\k-}(\A_{-i},\k),
\] 
and the $\k$-bimodule structure on $\A^{\#}$ is given as follows: if $e_v,e_w\in\k$, $a\in(\A^{\#})_{i}$, and $c\in \A_{-i}$, then
\begin{equation} \label{leftkbimodule} 
(e_v \cdot a \cdot e_w) (c) = a(ce_v ) e_w. 
\end{equation}
The graded components $(^{\#}\!\A)_{i}$ of $^\#\!\A$ in degree $i$ are right $\k$-module maps,
which we write as:
\[ 
\mathrm{hom}_{-\k}(\A_{-i},\k),
\] 
and the $\k$-bimodule structure is given by: 
if $e_v,e_w\in\k$, $a\in(^{\#}\!\A)_{i}$, and $c\in \A_{-i}$, then
\begin{equation} \label{rightkbimodule}
(e_v \cdot a \cdot e_w) (c) = e_v a(e_w c). 
\end{equation}
Both canonical maps $\A \to \!^\#(\A^\#)$ and $\A \to (^\#\!\A)^\#$ are $\k$-bimodule maps, which are isomorphisms if $\A$ is locally finite.

If $V_1,V_2,\ldots, V_n$ are $\k$-bimodules, there is a natural map 
\[V_n^\# \otimes_\k  V_{n-1}^\# \otimes_\k \cdots \otimes_\k V_1^\# \to (V_1 \otimes_\k V_2 \otimes_\k \cdots \otimes_\k V_n)^\# \]
given  by:
\begin{equation}\label{composition} 
(a_n \otimes a_{n-1} \otimes \ldots \otimes a_1) (c_1 \otimes c_2 \otimes \ldots\otimes c_n ) :=
    a_1(c_1 a_2(c_2 \ldots a_n(c_n) \ldots)).  
\end{equation}
Similarly, there is a natural map 
\[^\# V_n \otimes_\k\!^\# V_{n-1} \otimes_\k \cdots \otimes_\k \!^\# V_1 \to \!^\#(V_1 \otimes_\k V_2
\otimes_\k \cdots \otimes_\k V_n), \]
given  by:
\begin{equation}\label{composition2} 
    (a_n \otimes a_{n-1} \otimes \ldots \otimes a_1) (c_1 \otimes c_2 \otimes \ldots\otimes c_n ) := a_n (\ldots a_2(a_1(c_1)c_2)\ldots c_n).  
\end{equation}

These give the graded duals $\C^\#$ and $\!^\#\C$ of a coaugmented $A_\infty$-coalgebra $\C$ the structure of augmented $A_\infty$-algebras with structure maps defined by:
\begin{equation}\label{dualainf} \m_i(a_i,\ldots,a_1) (c) := (-1)^{|c|} (a_i \otimes \ldots \otimes a_1)
\Delta_i (c). \end{equation}
Note that to get a non-zero product, we must have $|\m_i(a_i,\ldots,a_1)|=|c|$, hence the sign $(-1)^{|c|}$ equals the sign $(-1)^{|a_1|+\ldots+|a_i|-i}$.

In general, there is no natural way of equipping the graded dual of an augmented $A_\infty$-algebra with an $A_\infty$-coalgebra structure. However, if the grading on $\A$ is locally finite (i.e.
$\A_i$ are finitely generated as $\k$-bimodules), then it follows that
\begin{align*}
    \A^\# \otimes_{\k} \A^\# \otimes_\k \cdots \otimes_\k \A^{\#} &\cong (\A         \otimes_\k \A
    \otimes_\k
\cdots \otimes_\k \A)^\#, \\
    \!^\#\A\! \otimes_{\k} \!^\#\!\A \otimes_\k \cdots \otimes_\k \!^{\#}\!\A &\cong \!^\# (\A         \otimes_\k \A \otimes_\k \cdots \otimes_\k \A).
\end{align*}  
Using these isomorphisms, the graded duals $\A^\#$ and $^\#\!\A$ of an augmented
        $A_\infty$-algebra $\A$ with locally finite grading can be naturally equipped with the structure of a coaugmented $A_\infty$-coalgebras by using the formulae
        \[ \Delta_i(c)(a_i \otimes_\k \cdots \otimes_\k a_1) = (-1)^{|c|} c(\m_i(a_i,\ldots, a_1))\] 
\subsubsection{Twisting cochains}
\label{twisting} 

Let $(\mathscr{C}, \Delta_{\bullet})$ be an $A_\infty$-coalgebra and let $(\mathscr{A},
\m_1, \m_2)$ be a DG-algebra. A \emph{twisting cochain} is a $\k$-linear map $\mathfrak{t}\colon
\mathscr{C} \to \mathscr{A}$ of degree 1 that satisfies:
\begin{equation}\label{eq:twcochain}
    \m_1 \circ \mathfrak{t} - \mathfrak{t} \circ \Delta_1 + \sum_{d \geq 2} (-1)^d \m_2^{(d)}
\circ \mathfrak{t}^{\otimes_\k d} \circ \Delta_d = 0, 
\end{equation}
where $\m_2^{(2)} := \m_2$, and $\m_2^{(d)} := \m_2\circ
(\mathrm{Id}_\mathscr{A} \otimes_\k \m_2^{(d-1)})$. Note that for
$c \in \mathscr{C}$, $\Delta_i(c) \neq 0$ for only finitely many $i$, and hence the potentially infinite sum in \eqref{eq:twcochain} is actually finite when it acts on $c$.

If the coalgebra $\C$ is coaugmented by $\eta\colon \k \to \C$ and the algebra $\A$ is augmented
$\epsilon\colon \A \to \k$, we require in addition that its  twisting cochains $\t$ are compatible in the sense that  \begin{equation}
\t \circ \eta = \epsilon \circ \t =0.
\end{equation}
We denote the set of twisting cochains from $\C$ to $\A$ by $\mathrm{Tw}(\C,\A)$.

Let $\t\in\mathrm{Tw}(\C,\A)$ be a twisting cochain. Consider the twisted tensor product $\A
\otimes^\mathfrak{t}_\k \C$ as a chain complex with 
differential $d^{\mathfrak{t}}\colon \A\otimes^\mathfrak{t}_\k \C\to\A\otimes^\mathfrak{t}_\k \C$ defined as follows,
\begin{equation}\label{koszulcomplex} 
d^\mathfrak{t} = \m_1 \otimes_\k \mathrm{Id}_\C + \mathrm{Id}_\A \otimes_\k \Delta_1 +
\sum_{d \geq 2} (\m_2^{(d)} \otimes \mathrm{Id}_\C) \circ (\mathrm{Id}_\A \otimes_\k
\mathfrak{t}^{\otimes_{\k} d-1}
\otimes_\k \mathrm{Id}_\C ) \circ (\mathrm{Id}_\A \otimes_\k \Delta_{d}). 
\end{equation}
Here the differential squares to zero, $d^\mathfrak{t} \circ d^\mathfrak{t}=0$, since
$\mathfrak{t}$ satisfies \eqref{eq:twcochain}. This complex is the \emph{Koszul complex}
associated with $\t$. It is called acyclic if the projection to $\k$ is a quasi-isomorphism. 

Note that one also has an analogous complex of the form $\C
\otimes_{\k}^{\t} \A$.  

The $\mathbb{K}$-vector space of $\k$-bimodule morphisms $\hom_{\k-\k}(\C,\A)$ carries an $A_\infty$-algebra
structure with operations $\mathfrak{n}_d$, $d\ge 1$ given by
\[ \mathfrak{n}_1 (t) = \m_1 \circ t + (-1)^{|t|} t \circ \Delta_1 \]
and 
\[ 
\mathfrak{n}_d (t_d, t_{d-1}, \ldots, t_1) = (-1)^{d(|t_d|+\ldots+|t_1|)} \m_2^{(d)} \circ (t_d \otimes t_{d-1} \otimes
\ldots \otimes t_1) \circ \Delta_d ,  \ \ d\geq 2. 
\] 
where the composition $(t_d \otimes t_{d-1} \otimes \ldots \otimes t_1) \circ \Delta_d$ is defined
component-wise. Thus, if $\Delta_d(c)= c_d \otimes \ldots \otimes c_1$, then 
\[ (t_d \otimes_\k t_{d-1} \otimes_\k \ldots, t_1)\Delta_d(c) = (-1)^\dagger \t_d(c_d)\otimes_\k
\t_{d-1}(c_{d-1}) \ldots  \otimes_\k t_1(c_1), \] 
where $\dagger = \sum_{j=2}^d \sum_{i=1}^{j-1} |c_i||t_j|$.
In this setting, a twisting cochain $\mathfrak{t}\colon \C \to \A$ corresponds to a solution of the Maurer-Cartan equation:
\begin{equation}\label{eq:MCeq} 
\sum_{i \geq 1} \mathfrak{n}_i (\mathfrak{t}, \mathfrak{t}, \ldots, \mathfrak{t}) =0. 
\end{equation} 
(As before, this sum is effectively finite since for any $c \in \C$, $\Delta_i(c)\neq 0$ only for finitely many $i$.). 

In particular, a twisting cochain $\mathfrak{t}\colon \mathscr{C} \to \mathscr{A}$ defines a twisted $A_\infty$-structure on $\hom_\k (\C, \A)$, with operations $\mathfrak{n}^{\mathfrak{t}}_{d}$ given by
\[ 
\n^\mathfrak{t}_d (t_d,t_{d-1},\ldots, t_1) = \sum_{l_i \geq 0} \n_{d+l_0+l_1+\ldots+l_d}
(\overbrace{\mathfrak{t},\ldots, \mathfrak{t}}^{l_d},
t_d, \overbrace{\mathfrak{t},\ldots, \mathfrak{t}}^{l_{d-1}}, t_{d-1},\;\ldots\;, 
t_1,\overbrace{\mathfrak{t},\ldots, \mathfrak{t}}^{l_0}).  
\]
We will denote this twisted $A_\infty$-structure $\hom^\mathfrak{t}_{\k}(\C,\A)$.

There are direct analogs of the above construction if we instead consider a DG-coalgebra $(\mathscr{C}, \Delta_1, \Delta_2)$ and an $A_\infty$-algebra $\mathscr{A}$ with operations $\m_i$. 
The module $\hom_{\k-\k} (\mathscr{C},\mathscr{A})$ has the structure of an $A_\infty$-algebra with operations $\mathfrak{n}_d$ given by:
\[ 
\mathfrak{n}_1 (t) = \m_1 \circ t + (-1)^{|t|} t \circ \Delta_1 
\]
and 
\[ 
\mathfrak{n}_d (t_d, t_{d-1}, \ldots, t_1) = \m_d \circ (t_d \otimes t_{d-1} \otimes
\ldots \otimes t_1) \circ \Delta_2^{(d)} ,  \ \ d\geq 2. 
\] 
To make sense of the twisting cochain equation \eqref{eq:MCeq}, one needs to make additional assumptions to ensure the convergence of the infinite sum. This holds, for example, if $\C$ is conilpotent.

We remark that if both $\C$ and $\A$ are $A_\infty$-(co)algebras, then defining a twisting
cochain is a more complicated matter (cf. \cite[Introduction]{proute}). We will not need this here.

\subsection{Bar-cobar duality for $A_\infty$-(co)algebras}
In this section we first introduce the bar and cobar constructions and then discuss basic relations between them.

\subsubsection{Bar and cobar constructions} 
\label{barcobarsec}
Let $(\A, \{\m_j\}_{j \geq 1})$ be a strictly unital $A_\infty$-algebra with a strict
augmentation $\epsilon\colon \A \to \k$. Define the \emph{augmentation ideal} $\Abar =\mathrm{ker}
(\epsilon)$. Note that if we are given a non-unital $A_\infty$-algebra $\Abar$, then we can turn it into a strictly unital
$A_\infty$-algebra $\A := \k \oplus \Abar$ with an augmentation given by projection to $\k$. 

We next recall the construction of the (reduced) \emph{bar construction} $\Bar \A$. For any augmented $A_\infty$-algebra $\A$, $\Bar \A$ is a coaugmented conilpotent DG-coalgebra. As a coaugmented coalgebra $\Bar \A$ is defined as
\[ 
\Bar \A = \k \oplus \Abar [1] \oplus \Abar [1]^{\otimes_\k 2} \oplus \ldots,  
\]
where $[1]$ denotes the downwards shift by 1. We write a typical monomial using Eilenberg and
Maclane's notation: \[ [a_d|a_{d-1}|\ldots |a_1] = sa_d \otimes_\k sa_{d-1} \otimes_\k \ldots
\otimes_\k sa_1, \]
where for $a \in \Abar$, $sa \in \Abar[1]$ denotes the corresponding element in $\Abar[1]$ with degree shifted down by 1. 

The differential $b\colon \Bar \A \to \Bar \A$ is defined to vanish on $\k\subset\A$, $b_{|\k}=0$, and as follows on monomials:
\[ 
b([a_d|a_{d-1}|\ldots|a_1]) = \sum_{i,j} (-1)^{|a_1|+\ldots+|a_j|-j}
[a_d|\ldots|a_{j+i+1}|\m_i(a_{j+i},\ldots,a_{j+1})|a_{j}|\ldots|a_1], 
\]

The coproduct $\Delta_2 \colon \Bar \A \to \Bar \A \otimes_\k \Bar \A $ is defined by
\[ 
\Delta_2 ([a_d|a_{d-1}|\ldots |a_1]) = \sum_{i=0}^d (-1)^{|a_i|+\ldots+|a_1|-i}  [a_d|a_{d-1}|\ldots|a_{i+1}] \otimes_\k [a_{i}|
    a_{i-1}|\ldots | a_1]. 
\]
The slightly unusual sign $(-1)^{|a_i|+\ldots+|a_1|-i}$ appears as a consequence of the following two facts:
\begin{itemize}
\item[$(\mathrm{i})$]
    The equation $b^2=0$ is equivalent to the $A_\infty$-relations \eqref{ainf} for
$(\m_i)_{i\geq 1}$, 
\item[$(\mathrm{ii})$]
The pair $(b,\Delta_2)$ satisfies the co-$A_\infty$-relations
\eqref{coainf}. 
\end{itemize}
Redefining $(b,\Delta_2)$ to $(\theta, \Delta) $
using \eqref{coleft} removes the sign in $\Delta_2$, and $(\theta, \Delta)$ becomes a (usual)
coassociative DG-coalgebra, where the co-Leibniz rule \eqref{coLeib} holds. The coaugmentation $\eta\colon \k \to \Bar \A$ is defined by letting $\eta_{1}$ be the inclusion of $\k$ and $\eta_{i}=0$ for $i>0$.

There is an increasing, exhaustive, and bounded below (hence, complete Hausdorff) filtration on the complex  $\Bar \A$:
\[ 
\k = \F^{0}\Bar\A \subset \F^1 \Bar\A \subset \cdots \subset \Bar \A, 
\]
where
\[ 
\F^p \Bar \A := \k \oplus \Abar[1] \oplus \cdots \oplus \Abar[1]^{\otimes_\k p}.  
\]
This induces the \emph{word-length spectral sequence} with
\[ E_1^{p,q} = H^{p+q} (\F^{p} \Bar \A / \F^{p-1} \Bar \A) \]
converging strongly to
\[ E_\infty^{p,q} = \F^{p} H^{p+q} (\Bar \A) / \F^{p-1} H^{p+q} (\Bar \A) \]
by the classical convergence theorem (\cite[Theorem 5.5.1]{weibel}). It can be proved using this spectral sequence that if an $A_\infty$-map $\e \colon \A \to \B$ is a
quasi-isomorphism then the naturally induced DG-coalgebra map $\Bar\e \colon \Bar \A \to \Bar \B$ is a quasi-isomorphism, see \cite[Proposition 2.2.3]{LV}. 

There is a universal twisting cochain $\t_\A \colon \Bar \A \to \A$ which is non-zero only on
$\Abar[1] \subset \Bar \A$ and is given by the inclusion map $\Abar[1] \to \A$. The twisting cochain $\t_\A$ gives rise to a free $\A$-bimodule resolution of $\A$ obtained as a
twisted tensor product
\[ \A \otimes_{\k}^{\t_\A} \Bar \A \otimes_{\k}^{\t_\A} \A, \]
with the differential $d$ given by the following formula:
\begin{align} \label{bimodulediff} 
	d&=\m_1 \otimes_\k \mathrm{Id}_{\Bar \A} \otimes_\k \mathrm{Id}_{\A}  + \mathrm{Id}_{\A}
    \otimes_\k b \otimes_\k \mathrm{Id}_\A +  \mathrm{Id}_\A \otimes_\k \mathrm{Id}_{\Bar \A}
    \otimes_\k \m_1 \\\notag
    &+ \left(\sum_{d \geq 2 }
(\m_d \otimes_{\k} \mathrm{Id}_{\Bar \A}) \circ (\mathrm{Id}_\A \otimes_\k \t^{\otimes_\k {d-1}} \otimes_\k \mathrm{Id}_{\Bar
\A})  \circ
(\mathrm{Id}_{\A} \otimes_\k
    \Delta_2^{(d)}) \right) \otimes_{k} \mathrm{Id}_\A \\\notag 
    &+ \mathrm{Id}_\A \otimes_k \left(\sum_{d \geq 2 }
    (\mathrm{Id}_{\Bar \A} \otimes_\k \m_d) \circ (\mathrm{Id}_{\Bar \A} \otimes_\k \t^{\otimes_\k {d-1}} \otimes_\k \mathrm{Id}_{\A})  \circ
    (\Delta_2^{(d)} \otimes_\k \mathrm{Id}_{\A} ) \right).
\end{align}

This can be used to compute Hochschild homology and cohomology of $\A$ with coefficients in
an $\A$-bimodule $\mathscr{M}$. 

Consider instead a strictly counital $A_\infty$-coalgebra $\C$ with operations $\Delta_i$ and
with a strict coaugmentation $\eta\colon \k \to \C$. Let $\Cbar = \mathrm{coker}(\eta)$ be the coaugmentation ideal. We next recall the \emph{cobar construction} which associates a DG-algebra $\Omega \C$ to $\C$.  As an augmented algebra $\Omega \C$ is:
    \begin{align} \label{cobarconstr} \Omega \C = \k \oplus \Cbar [-1] \oplus \Cbar [-1]^{\otimes_\k 2} \oplus \ldots .\end{align}
As before, we write a typical monomial as \[ [c_d|c_{d-1}|\ldots |c_1] = s^{-1} c_d \otimes_\k
s^{-1} c_{d-1} \otimes_\k \ldots
\otimes_\k s^{-1} c_1, \]
where for $c \in \Cbar$, $s^{-1} c \in \Cbar[-1]$ denotes the corresponding element in $\Cbar[-1]$ with degree shifted up by 1. The differential $\m_{1}$ on $\Omega \C$ vanishes on $\k$, $\m_{1}|_{\k}=0$, and acts as follows on monomials:
\[ \m_1 ([c_m|\ldots |c_1]) = \sum_{i,j} (-1)^{|c_{1}|+\ldots+|c_i|-i}
    [c_m|\ldots|c_{i+2}|\Delta_j(c_{i+1})|c_{i}|\ldots |c_1]. \]
Here, by abuse of notation, we write $\Delta_j$ for the induced coproduct $\Cbar[-1] \to \Cbar[-1]^{\otimes j}$.

    The product $\m_{2} \colon \Omega \C \otimes \Omega \C \to \Omega \C$ is given by:
    \[ \m_2 ([c_m|\ldots |c_{i+1}], [c_{i}|\ldots|c_1]) =
    (-1)^{|c_{1}|+\ldots |c_i|-i} [c_m|\ldots|c_{i+1}|c_{i}|\ldots|c_1] \]
The slightly unusual sign $(-1)^{|c_{1}|+\ldots+|c_i|-i}$ appears as a consequence of the following two facts:
\begin{itemize}
\item[$(\mathrm{i})$]	
The equation $\m_1^2=0$ is equivalent to co-$A_\infty$-relations \eqref{coainf} for
$(\Delta_j)_{j\geq 1}$. 
\item[$(\mathrm{ii})$]
The pair $(\m_1,\m_2)$ satisfies the $A_\infty$-relations \eqref{ainf}. 
\end{itemize} 
Redefining $(\m_1,\m_2)$ to $(d, \cdot) $
using \eqref{left} removes the sign in $\m_2$, and $(d, \cdot)$ becomes a (usual) associative
DG-algebra, where Leibniz rule \eqref{Leib} holds. The augmentation $\epsilon\colon \Omega \C \to \k$ is given by letting $\epsilon_1$ to be the projection
to $\k$ and $\epsilon_i =0$ for $i>0$. 

There is a decreasing, exhaustive, bounded above filtration on the complex     $\Omega \C$:
    \[ \Omega \C = \F^{0}\Omega \C \supset \F^1 \Omega \C \supset \ldots,\]
given by
\[ \F^p \Omega \C := \Cbar[-1]^{\otimes_\k p} \oplus \Cbar[-1]^{\otimes_\k (p+1) }\oplus \ldots. 
\]
This gives the \emph{word length spectral sequence} with
\[ E_1^{p,q} = H^{p+q} (\F^{p} \Omega \C / \F^{p+1} \Omega \C). \]
Unlike the case of the word length filtration on the bar construction, for the cobar construction, in general, convergence may fail. Thus, we introduce
completions. We define the \emph{completed cobar construction} to be: 
\[ \widehat{\Omega} \C = \varprojlim_s (\Omega \C)/ (\F^{s} \Omega \C) \]
The length filtration on $\Omega \C$ induces a filtration $\widehat{\F}$ on $\widehat{\Omega}
\C$ defined by:
\[ \widehat{\F}^p \widehat{\Omega} \C  = \varprojlim_s (\F^{p} \Omega \C) / (\F^{s} \Omega \C) \]
which is decreasing, exhaustive, bounded above and complete Hausdorff. The spectral sequence associated to the filtration $\widehat{\F}$ on $\widehat{\Omega}
\C$
is isomorphic to the length spectral sequence associated with the filtration $\widehat{\F}$ on
$\Omega \C$ and converges conditionally to $H^*(\widehat{\Omega} \C)$ (see \cite[Theorem
9.2]{boardman}). It converges strongly to $H^*(\widehat{\Omega} \C)$ if the spectral sequence is
regular, i.e., only finitely many of the differentials $d_{r}^{p,q}$ are nonzero for each $p$ and $q$
(see \cite[Theorem 7.1]{boardman}). This holds, for example, if $\Omega \C$ is locally finite.

We say that $\Omega\C$ is complete if the natural map $\Omega \C \to \widehat{\Omega}
\C$ is a quasi-isomorphism. For example, it is easy to see that this is the case if $\C$ is locally
finite and simply-connected. 

If $\mathfrak{f}\colon \C \to \D$ is an $A_\infty$-comap which is a quasi-isomorphism of $A_\infty$-coalgebras and if $\Omega\C$ and $\D$ are complete then $\Omega \mathfrak{f}$ is a quasi-isomorphism. (This follows from \cite[Theorem 7.4]{EM}, see also \cite[Theorem 5.5.11]{weibel}).
The completeness assumptions are necessary and are related to the completeness of the
word length filtration. A counter-example when the completeness assumptions are dropped can be found in \cite[Section 2.4.1]{LV}. 

There is a universal twisting cochain $\t^\C\colon \C \to \Omega \C$ given by the composition of
canonical projection $\C \to \Cbar[-1]$ and the canonical inclusion $\Cbar[-1] \to \Omega \C$.

\subsubsection{Bar-cobar adjunction}\label{ssec:barcobaradj}
Suppose that $\C$ is a coaugmented $A_\infty$-coalgebra and $\A$ is an augmented DG-algebra, then we have a canonical
bijection 
\begin{equation}\label{bij1}  
\hom_{\mathrm{DG}} (\Omega \C, \A) \to \mathrm{Tw}(\C,\A),
\end{equation} 
given by $\phi \mapsto \phi \circ \t^\C$.
Similarly, if $\C$ is a coaugmented conilpotent DG-coalgebra and $\A$ is an augmented $A_\infty$-algebra, then we have a canonical
bijection 
\begin{equation}\label{bij2} 
\hom_{\mathrm{coDG}} (\C, \Bar \A) \to \mathrm{Tw}(\C,\A), 
\end{equation}
given by $\phi \mapsto \t_\A \circ \phi$ (see \cite[Lemme 3.17]{proute}).

Therefore, when $\C$ is a coaugmented conilpotent DG-coalgebra, and $\A$ is an augmented DG-algebra, we have the bar-cobar adjunction: 
\[ \hom_{\mathrm{DG}}( \Omega \C, \A) \cong \hom_{\mathrm{coDG}}(\C, \Bar \A). \]

Moreover, the natural DG-maps
 \begin{equation} \label{barcobar} \Omega \Bar \A \to \A \ \ , \ \ \C \to \Bar \Omega \C \end{equation}
     are quasi-isomorphisms for any DG-algebra $\A$ and conilpotent DG-coalgebra $\C$
     (see \cite[Section 6.10]{Posit}). It is also
 true that for any $A_\infty$-algebra $\A$, the $A_\infty$-algebra map \[ \A \to \Omega \Bar \A\] given
 by the adjunction map $\Bar \A \to \Bar \Omega \Bar \A$ is an
 $A_\infty$-quasi-isomorphism (see
 \cite[Lemme 2.3.4.3]{L-H}). Note that any $A_\infty$-quasi-isomorphism is invertible up to
 homotopy (\cite[Corollary 1.4]{seidelbook}).

Similarly, for any $A_\infty$-coalgebra $\C$, the $A_\infty$-comap 
 \[ \Bar \Cobar \C \to \C \]
given by the adjunction map $\Omega \Bar \Omega \C \to \Omega \C$ is an $A_\infty$-quasi isomorphism. 

However, an $A_\infty$-quasi isomorphism for a general $A_\infty$-coalgebra is not usually a convenient notion,
since as we remarked above a quasi-isomorphism of $A_\infty$-coalgebras between $\C$ and $\C'$ does not
necessarily induce a quasi-isomorphism of DG-algebras $\Cobar \C$ and $\Cobar \C'$.

For this reason, one considers the category of \emph{conilpotent} $A_\infty$-coalgebras. Let $\C$ be
a coaugmented $A_\infty$-coalgebra generated over $\k$ by variables $(c_i)_{i\in I}$ with $I$ some countable index set, such that there exists a total ordering 
    \[ c_{\sigma(1)} < c_{\sigma(2)} < \ldots \]
    where $\sigma\colon I \to I$ is a bijection. This produces an increasing  filtration 
    \[ \mathcal{F}^0 =\k \subset \mathcal{F}^1 \subset \cdots \subset \Omega \C \] 
    by setting $\mathcal{F}^p = \k \langle c_{\sigma(1)},\ldots, c_{\sigma(p)} \rangle$.
    Suppose that the structure maps $(\Delta_i)_{i\geq 1}$ are compatible with this filtration, in
    the sense that, $\Delta_i(c_{\sigma(p)}) \subset \mathcal{F}^{p-1}$ for all $i$ and $p$. Then we
    call $\C$ a conilpotent $A_\infty$-coalgebra. (More generally, a homotopy retract of such
    $A_\infty$-coalgebras are called conilpotent, cf. \cite[Sections 6.10 and 9]{Posit}. This notion is called \emph{finite type} in \cite{Kontsevich}.). 
    Given two such $A_\infty$-coalgebras $\C$
    and $\C'$, one considers filtered $A_\infty$-comaps between them. 
In the case of a conilpotent DG-coalgebra $\C$ there exists an increasing filtration on $\Omega \C$
given by the subalgebras $\mathrm{Ker}(\Delta^{(n)})$ that plays the same role (see \cite[Lemme
1.3.2.3]{L-H}). 

We next state the following elementary lemma for later convenience.

\begin{lem} \label{dualbar} Let $\A$ be an augmented $A_\infty$-algebra, such that the $\k$-bimodule
    structure on $\A$ and $\Bar \A$ are locally finite, then there are quasi-isomorphisms of
    augmented DG-algebras: \begin{align*} \Omega(\A^\#) \to (\Bar \A)^\#  \text{\ \ \ and\ \ \ }
        \Omega(^\#\A) \to \!^\#(\Bar \A). 
    \end{align*}  \end{lem}

Note that the assumption is satisfied when $\A$ is locally finite and simply-connected. We shall briefly consider the
case when $\A$ is only assumed to be locally finite and connected, then we have:
\begin{lem} \label{dualbar0} Let $\A = \bigoplus_{i} \A^{i}$ be a connected, locally finite
    $\k$-bimodule equipped with an augmented $A_\infty$-algebra structure, then there are maps of
    DG-algebras
    \[ \Omega(\A^\#) \to (\Bar \A)^\# \text{\ \ \ and \ \ \ }\Omega(^\#\!\A) \to \!^\#(\Bar \A),
    \]  
which become quasi-isomorphisms  after completion:
    \[ \widehat{\Omega}(\A^\#) \to (\Bar \A)^\# \text{\ \ \ and \ \ \ } \widehat{\Omega}(^\#\!\A) \to \!^\#(\Bar \A).  \] 
\end{lem}

\subsection{Koszul duality} 
    \label{algkoszul} 

Suppose $\C$ is a coaugmented conilpotent $A_\infty$-coalgebra and $\A$ is an augmented DG-algebra, via the bijection \eqref{bij1}, any twisting cochain $\t \in \mathrm{Tw}(\C,\A)$ is of the form $\t = \phi \circ
    \t^\C$ for some unique $\phi \in \hom_{\mathrm{DG}}(\Omega\C,\A)$. 
    Similarly, if $\C$ is a coaugmented conilpotent DG-coalgebra and $\A$ is an augmented $A_\infty$-algebra, any twisting cochain $\t \in \mathrm{Tw}(\C,\A)$ is of the form $\t = \t_\A \circ \phi$
    for some $\phi \in \hom_{\mathrm{coDG}}(\C, \Bar \A)$. 

\begin{defi}\label{ktc}    
In either case above we call $\t$ a \emph{Koszul twisting cochain} if $\phi$ is a quasi-isomorphism, and we denote the set of Koszul twisting cochains by $\mathrm{Kos}(\C,\A)$. 
\end{defi}

The terminology of Koszul twisting cochains is taken from \cite{LV}. They are also called $\emph{acyclic twisting cochains}$ in other sources (cf. \cite{L-H}, \cite{Posit}). This terminology is due to the well known fact that, under various locally finiteness assumptions, a twisting cochain $\t$ is Koszul if and only if the Koszul complex (\ref{koszulcomplex}) associated to $\t$ is acyclic (cf. \cite[Appendix A]{Posit}).

Informally, if $\t \in \mathrm{Kos}(\C, \A)$ then, depending on whether we write
    $\t=\phi\circ\t^{\C}$ or $\t=\t_{\A}\circ\phi$, either $\A$ can be used in place of $\Omega \C$
    or $\C$ can be used in place of $\Bar \A$
in various resolutions. This, in turn, may lead to smaller complexes to compute with. For example,
    one can compute Hochschild homology and cohomology of $\A$ and $\Omega \C$ using the
    $\A$-bimodule resolution of $\A$ given by the complex:
    \[ \A \otimes_{\k}^{\t} \C \otimes_{\k}^{\t} \A \] 
    with the differential as in \eqref{bimodulediff}, see \cite{herscovich}.

Suppose that $\A$ is an $A_\infty$-algebra with an augmentation
$\epsilon\colon \A\to \k$. The augmentation $\epsilon$ makes $\k$ into a left $\A$-module or equivalently a right $\A^{op}$-module. 

\begin{defi} The \emph{Koszul dual} of an augmented $A_\infty$-algebra $\A$ is the DG-algebra of left $\A$-module maps     from $\k$ to
    itself:
        \[ E(\A) := \rhom_{\A} (\k,\k). \]
    \end{defi}

Recall that for a unital $A_\infty$-algebra $\A$ over a field $\mathbb{K}$ (or a semisimple ring such
    as $\k$) any $A_\infty$-module is both $h$-projective and $h$-injective, that is, if $M$ is an
    $A_\infty$-module over $\A$ and $N$ is an acyclic $A_\infty$-module over $\A$, then the
    complexes $\rhom_{\A}(M,N)$ and $\rhom_{\A}(N,M)$ are acyclic \cite[Lemma 1.16]{seidelbook}. Hence, the DG-algebra
    $\rhom_{\A}(\k,\k)$ can be computed as the $A_\infty$-module homomorphisms from     $\k$ to
    itself. (More generally, this holds if $\A$ is $h$-projective as a complex of $\k$-
    modules, which implies that $\k$ is $h$-projective as an $A_\infty$-module over $\A$.) Therefore, we have the following.

\begin{prop} If $\A$ (resp. $\A^{op}$) is an augmented unital $A_\infty$-algebra then
    \[ \rhom_{\A}(\k,\k) \cong
    (\Bar\A)^\# \ \ \ (\text{resp.} \ \  \!^\#(\Bar \A) \  ). \]    
\end{prop}
\begin{proof} Recall that $\A \otimes_\k \Bar\A$ is quasi-isomorphic to $\k$ as an $\A$-module. Hence,
        by the hom-tensor adjunction, we have $\rhom_{\A} (\A \otimes_\k \Bar\A , \k) \cong
        \rhom_{\k} (\Bar\A,
            \k )$. Since $\A$ is $h$-projective as a complex of $\k$-modules, so is $\Bar\A$, hence
            the latter
                is computed by $(\Bar\A)^\#$.
            \end{proof}
In this model of $E(\A)$, the $\k$-bimodule structure on $E(\A)$ can be seen as by \eqref{leftkbimodule} since $\k$ is viewed as a left $\k$-module induced from its structure as a left $\A$-module. If instead, we have an augmentation of $\A^{op}$, then we view $\k$ as a right $\A$-module, and the $\k$-bimodule structure on $\rhom_{\A}(\k,\k)$ would be given by \eqref{rightkbimodule}. 

The cohomology of $E(\A)$ is a graded algebra:     \[ \mathrm{Ext}_{\A}(\k,\k) := H^*          (\rhom_{\A}(\k,\k)) \cong
H^*((\Bar\A)^\#). \]
Dually, we also have the derived tensor product $\k \widehat{\otimes}_\A \k$, which can be computed
by the complex $\Bar\A$. The cohomology is a graded coalgebra:
\[ \mathrm{Tor}_{\A}(\k,\k) := H^*(\k \widehat{\otimes}_\A \k) \cong H^*(\Bar\A). \]
In particular, note that if $\k$ is a field, we have that $\mathrm{Ext}_{\A}(\k,\k) \cong
(\mathrm{Tor}_{\A}(\k,\k))^\#$ by the universal coefficient theorem.

\begin{rem} If $\A$ is a {\bf commutative} algebra (or more generally an $E_2$-algebra), then $\mathrm{Tor}_{\A}(\k,   \k)$ also has a graded
algebra structure defined via:
\[ \mathrm{Tor}_{\A}(\k,\k) \otimes \mathrm{Tor}_{\A}(\k, \k) \to
\mathrm{Tor}_{\A \otimes \A }(\k \otimes \k, \k \otimes \k) \to                     \mathrm{Tor}_{\A}(\k,\k) \]
induced by the algebra map $\A \otimes \A \to \A$ (which exists since $\A$ is       commutative). This should not be confused with the natural coalgebra structure      above.
\end{rem}

Note that $\A$ itself can be viewed as a left $\A$-module and the map
$\epsilon \colon \A \to \k$ is a map of left $\A$-modules, hence it induces a map of left $E(\A)^{op}$-modules:
\[ \tilde{\epsilon}\colon \rhom_{\A}(\k,\k)^{op} \to \rhom_{\A}(\A,\k), \]
which can in turn be viewed as an augmentation of $E(\A)^{op} = \rhom_{\A}(\k, \k)^{op}$ since
$\rhom_{\A}(\A,\k)$ can again be
identified with $\k$ as it is the Yoneda image of $\k$ as an $\A$-module. Hence, $\k$ can be viewed
as a right $E(\A)$-module.
\begin{defi}
The double-dual of
$\A$ is defined to be $E(E(\A)):=\rhom_{E(\A)}(\k,\k)$.
\end{defi}

There is a natural map from $\A$ to its double-dual:
\[ \Phi\colon \A \to \rhom_{E(\A)}(\k,\k) \]
defined via viewing the right $E(\A)$-module $\k$ as $\rhom_{\A}(\A,\k)$ and acting on the left by $\A
\cong\rhom_{\A}(\A,\A)$.

\begin{defi} We say that $\A$ and $E(\A)$ are Koszul dual if $\Phi\colon \A \to \rhom_{E(\A)}(\k,\k)$ is a quasi-isomorphism.
\end{defi}

One standard situation when Koszul duality holds is the following:

\begin{thm} \label{doubledual} Suppose $\C = \bigoplus_{i \leq 0} \C^i$ is a locally finite, simply-connected $\k$-bimodule
    equipped with an $A_\infty$-coalgebra structure and the coaugmentation $\k \cong \C^0 \to \C$.
    Let $\A = \Omega \C$, which is an augmented connected DG-algebra. Then $E(\A) \cong \C^\#$ and $\A$ and $\C^\#$ are Koszul dual. In other words, the
    natural morphism 
    \[ \Omega \C \to  \mathrm{RHom}_{E(\A)}(\k,\k) \]
is a quasi-isomorphism.
\end{thm}
\begin{proof} First, observe that indeed $E(\A) \cong (\Bar \A)^\# \cong (\Bar \Omega \C)^\# \cong \C^\#$ by \eqref{barcobar} and because $\mathrm{Hom}_{\k-}(-,\k)$ preserves quasi-isomorphisms. Next, we have that 
    \[ \mathrm{RHom}_{E(\A)}(\k,\k) \cong \!^\# (\Bar (\C^\#)) \cong \Omega \C \]
    where we applied Lemma \ref{dualbar} to $\C^\#$ and used the fact that $^\#( \C^\#) \cong \C$ since $\C$ is locally finite. 
\end{proof}

Rather than making the grading assumptions on $\C$ as in Thm. \ref{doubledual}, which guarantee that $B\C^\#$ is locally finite. One can directly assume that the grading on the cohomology $H^*(\Omega \C)$ is locally finite. This assumption is harder to check in practice but Koszul duality still holds under this assumption which one can prove by combining the above argument with the homological perturbation lemma (see for example Thm. 2.8 of \cite{KY}). 

In the case that $\C = \bigoplus_{i \leq 0} \C^i$ is a locally finite, connected (but not
simply-connected) $\k$-bimodule,
Lemma \ref{dualbar} no longer applies. We instead use Lemma $\ref{dualbar0}$ to 
deduce the following weaker duality result. 
\begin{prop}\label{almostdualbar} 
	Let $\C = \bigoplus_{i \leq 0} \C^i$ be a connected, locally finite $\k$-bimodule,
    equipped with an $A_\infty$-coalgebra structure and coaugmentation $\k \cong \C^0 \to \C$, and $\A = \Omega \C = \k \oplus
    \bigoplus_{j \geq1} (\overline{\C}[-1])^{\otimes_{\k} j}$, which is an augmented
    DG-algebra where augmentation is given by projection to $\k$.
    Then $E(\A) \cong \C^\#$ and there is a quasi-isomorphism 
\[\widehat{\Omega}{\C} \to  \mathrm{RHom}_{E(\A)}(\k,\k). \]
\end{prop}

Note that in Proposition \ref{almostdualbar}, $\A = \Omega \C$ is not connected, and may admit other
augmentations $\epsilon \colon \A \to \k$ than that induced by the cobar construction.
Such augmentations will be considered below. For example, suppose that $\C \cong \k \oplus
\overline{\C}$ is a coaugmented $A_\infty$-coalgebra such that
$\overline{\C} =  \mathbb{K} \langle c
|  c \in \mathcal{R} \rangle$ is generated by elements $c$ from an indexing set $\mathcal{R}$ and that $\epsilon\colon \Omega \C \to \k$ is an
augmentation, which is induced by a map $\C \to \k$ since $\Omega \C$ is free. Now we can consider
the coaugmented $A_\infty$-coalgebra $\C^\epsilon = \k \oplus \overline{\C}^\epsilon$
such that 
\[ \overline{\C}^{\epsilon} = \mathbb{K} \langle c - \epsilon (c)1_\k | c \in \mathcal{R} \rangle \] 
Then $\Omega \C$ and $\Omega \C^{\epsilon}$ are quasi-isomorphic as non-augmented DG-algebras, and the augmentation
on $\Omega \C^{\epsilon}$ induced by the cobar construction coincides with the given
augmentation $\epsilon$ on $\Omega$. 

\begin{rem} \label{locfinrem} When $\C$ is not simply-connected, the proof of duality fails precisely because
    $\Bar\C^\#$ is not locally finite. Nevertheless, the duality result can still be proved in certain cases where an extra
    \emph{weight} grading (internal degree, or Adams degree) is available (see \cite{LPWZ}, \cite[Appendix A.2]{Posit},
    \cite{herscovich}). We will not study this situation systematically in this paper, but it is  important as it extends the range of applicability of Koszul duality theory.
    In the setting of Chekanov-Eliashberg DG-algebras, such a situation was considered in \cite{EtLe}. 
\end{rem}

\section{Legendrian (co)algebra}
\label{maininvariants}

In this section we introduce our Legendrian invariants. We start with discussing a model for loop space coefficients in Section \ref{Coefficients}. In Section \ref{ssec:Leginv} we define the Chekanov-Eliashberg algebra with loop space coefficients using moduli spaces of disks of all dimensions, and in Section \ref{ssec:parallelcopies} we give a more computable version which uses only rigid disks and which carries the same information if the Legendrian submanifold is simply connected.

\subsection{Coefficients} 
\label{Coefficients} 

Before defining our Legendrian invariants, we describe
chain models for their coefficients $C_{-*}(\Omega_{p_v} \Lambda_v)$ for $v \in \Gamma^+$. (Notation is as above, $\Lambda_{v}$ is a $+$ decorated connected component of the Legendrian $\Lambda$.)  We work over a field $\mathbb{K}$. 

Let $\Omega_{p_v} \Lambda_v$ denote the topological monoid of Moore loops based at
$p_v$, where the monoid structure comes from concatenation of loops (see \cite{AH}).
Write $C_{-*}(\Omega_{p_v} \Lambda_v)$ for the cubical chain complex (graded
cohomologically). Since $\Omega_{p_v} \Lambda_v$ is a topological monoid, the complex
$C_{-*}(\Omega_{p_v} \Lambda_v)$ becomes a DG-algebra using the natural product map $\times$ on cubical chains, where the DG-algebra product is given as follows:
\[ 
C_{-*}(\Omega_{p_v} \Lambda_v) \otimes C_{-*}(\Omega_{p_v} \Lambda_v)
\xrightarrow{\times} C_{-*}(\Omega_{p_v} \Lambda_v \times \Omega_{p_v} \Lambda_v)
\xrightarrow{\circ} C_{-*}(\Omega_{p_v} \Lambda_v). 
\]

We point out that the $\times$-map
\[ \mathrm{\times}\colon C_{-*}(\Omega_{p_v} \Lambda_v) \otimes C_{-*}(\Omega_{p_v} \Lambda_v)
\to C_{-*}(\Omega_{p_v} \Lambda_v \times \Omega_{p_v} \Lambda_v), \]
when both sides are equipped with the Pontryagin product, is a DG-algebra map.

In what follows, we shall also make use of an inverse to the $\times$, known as the
Serre diagonal \cite{serre} and the cubical analogue of the Alexander-Whitney map,
\begin{equation} \label{serrediag} \eta\colon  C_{-*}(\Omega_{p_v} \Lambda_v \times \Omega_{p_w} \Lambda_w)
\to C_{-*}(\Omega_{p_v} \Lambda_v) \otimes C_{-*}(\Omega_{p_w} \Lambda_w).\end{equation} 
To define this map consider the $n$-cube $I^{n}$ with coordinates $(x_{1},\dots,x_{n})$. For an
ordered $j$-element subset $J\subset\{1,2,\dots,n\}$, $J=(i_{1},\dots,i_{j})$, $i_{1}<\dots<i_{j}$
and for $\epsilon\in \{0,1 \}$ let $\iota^{\epsilon}_{J}\colon I^{j}\to I^{n}$ be the map given in
coordinates $y= (y_1,\ldots, y_j)$ by
\[ 
x_{i_{r}}(\iota_{J}(y))=y_{r},\quad x_{m}(\iota_{J}(y))=\epsilon \text{ if } m\notin J.
\] 
Consider a cubical chain $(\sigma,\tau) \colon I^n \to  \Omega_{p_v} \Lambda_v \times \Omega_{p_w} \Lambda_w$.  If $J$ is an ordered subset of $\{1,\dots,n\}$ let $J'$ denote its complement ordered in the natural way. Define $\eta$ by
\[ 
\eta(\sigma,\tau) = \sum_{J} (-1)^{JJ'}(\sigma\circ\iota_{J}^{0})  \otimes
(\tau\circ \iota_{J'}^{1}),
\]
where the sum ranges over all ordered subsets $J$ and $(-1)^{JJ'}$ is the sign of the permutation $JJ'$. 
This is a strictly associative chain map inducing a quasi-isomorphism. Note also that there are obvious extensions of $\eta$ to several products of loop spaces.

As the cubical chain complex $C_{-*}(\Omega_{p_v}\Lambda_v)$ is very large it is not the most effective complex for computation. We next discuss smaller models. Starting with a 0-reduced simplicial set $X$ with geometric realization
$|X|=\Lambda_v$, an explicit economical model for $C_{-*}(\Omega_{p_v} \Lambda_v)$ is obtained
by taking normalized chains on the Kan loop group $GX$ \cite{Kan}. We will not say much about this, but point out that $GX$ is a free simplicial group such that its geometric realization $|GX|$ is
homotopy equivalent to $\Omega |X|$ (cf. \cite[Corollary 5.11]{GJ}). Hence, by the monoidal Dold-Kan
correspondence (cf. \cite{SS}) the normalized
chains on $GX$ gives a (weakly) equivalent model of $C_{-*}(\Omega_{p_v} \Lambda_v)$. (Another similar construction is sketched in \cite{Kontsevich} and leads to a free model.)

Alternatively, one can work with CW complexes. We start with the simply connected case: for a
$1$-reduced (unique 0-cell and no 1-cells) CW-structure on $\Lambda_v$ the Adams-Hilton construction
\cite{AH} gives a free DG-algebra model for $C_{-*}(\Omega_{p_v} \Lambda_v)$ as follows. Denote the
$k$-cells of $\Lambda_{v}$ by $e_k^i$, $k\geq 2$ and $i=1,\ldots,m_{k}$. The Adams-Hilton
construction gives a CW monoid with a single 0-cell and generating cells $\overline{e}_k^i$ in
dimension $k-1$ that is quasi-isomorphic to $\Omega_{p_v}(\Lambda_v)$ as a monoid (cf. \cite{CM}). This gives a DG-algebra structure on the free algebra:
\[ 
A(\Lambda_v) := \mathbb{K} \langle \overline{e}_2^1,\ldots,\overline{e}_{2}^{m_{2}},
\overline{e}_3^1,\ldots,\overline{e}_3^{m_{3}}, \ldots,
\ldots  \rangle , \ \ |\overline{e}_k^i| = 1-k. 
\] 
and a DG-algebra map 
\[ A(\Lambda_v) \xrightarrow{\Psi} C_{-*}(\Omega_{p_v}\Lambda_v) \]  
which is a quasi-isomorphism. The differential $d$ on $A(\Lambda_{v})$ is generally not explicit. It  is defined recursively as follows. For every 2-cell $e_{2}^{i}$, we have $d(\overline{e}_{2}^{i}) = 0$. In general, assume
that $d_{k-1}$ and $\Psi_{k-1}$ have been defined on
the $k$-skeleton $\Lambda_v^{(k)}$ of $\Lambda_{v}$, then for each $(k+1)$-cell $e$, with attaching map $f\colon S^{k} \to
\Lambda_v^{(k)}$, define $d_{k} \overline{e} = c$ so that $(\Psi_{k-1})(c) = (\Omega f)_* (\xi)$
where $\xi$ a generator of $H_{k-1}(\Omega S^{k})$, and $\Psi_{k}(\overline{e})$ to be the $k$-chain
of loops in $e$ (which then depends on earlier choices along the boundary of $\overline{e}$).
We remark that $A(\Lambda_v)$ can be identified isomorphically with $\Omega C_*^{\text{\tiny
CW}}(\Lambda_v)$ for a suitable
$A_\infty$-coalgebra structure on the cellular chain complex $C_*^{\text{\tiny CW}}(\Lambda)$.

This construction can be generalized to the
non-simply-connected case as follows. (See \cite{holstein, holstein2}, a generalization was given earlier in \cite{FT}, however that paper contains an error.) Begin with a $0$-reduced
CW-structure on $\Lambda_v$. Denote the $k$-cells by $e_{k}^{i}$ for
$i=1,\ldots, m_{k}$. For each $k$-cell $e^k_i$ with $k\geq 2$, we have a free
variable in degree $1-k$ which we again denote by $\overline{e}_k^i$. For each $1$-cell $e_1^j$, $j=1,\dots,m_{1}$ we have two variables $t_j$ and $t_j^{-1}$ in degree 0 such that $t_j t^{-1}_j = 1 = t^{-1}_j t_j$. Thus, the underlying
algebra is the ``almost free'' algebra of the form
\[ A(\Lambda_v) := \mathbb{K} \langle t_1^{\pm 1},\ldots, t_{m_{1}}^{\pm 1},
\overline{e}_2^1,\ldots,\overline{e}_2^{m_{2}},\overline{e}_{3}^{1},\ldots,\overline{e}_{3}^{m_{3}}, \ldots\rangle.
\] 

This presentation is often more efficient than the presentation one gets from the Kan loop group
construction using a simplicial set presentation of $\Lambda_v$. However, the
differential in the Adams-Hilton model is not easy to describe explicitly. 
Note that we have 
\[
d(t_j) = d(t_j^{-1}) =0, 
\] 
for degree reasons. For every 2-cell $e_2^{i}$, we have 
\[
d(\overline{e}_2^{i}) = 1 -c_i 
\] 
where $c_i \in \mathbb{K} \langle t^{\pm 1}_j |\;j=1,\dots,m_{1}\rangle$
represents the class of the attaching map of $e_2^{i}$. The differential on higher dimensional cells is generally harder to compute and is exactly as in the simply-connected case discussed above.  

Augmentations $\epsilon\colon A(\Lambda_v) \to \mathbb{K}$ correspond to solutions of the equations
\[
\begin{cases} 
\epsilon(t_j) \epsilon(t^{-1}_j) = 1, & j=1,\dots,m_{1},\\   
\epsilon(d \overline{e}_{2}^i)= 0, & i=1,\dots,m_{2}.
\end{cases}
\]
Since $\K \langle t_j^{\pm 1} ,  j=1,\dots, m_1 |\;  d\overline{e}_{2}^{i},  i=1,\ldots,m_{2} \rangle $ is a presentation of 
the fundamental group algebra $\K[\pi_1(\Lambda_v,p_v)]$, augmentations correspond exactly to local systems $\pi_1(\Lambda_v, p_v) \to \K$.

We will use the cubical chain complex $C_{-*}(\Omega_{p_v} \Lambda_v)$ to define Legendrian invariants below. Cubical chains work uniformly for all spaces $\Lambda_v$ and
are convenient for showing that the fundamental classes of moduli spaces of pseudoholomorphic disks $\mathcal{M}^{\sy}$, via evaluation maps, take values in the chain complex. 
The Legendrian invariants can also be studied using any of the smaller models discussed above. It is however important to note that in the non-simply connected case, we only have either weak equivalence in the homotopy category of
DG-algebras or Morita equivalence \cite{holstein, holstein2} of these models and the cubical chain complex $C_{-*}(\Omega_{p_v} \Lambda_v)$.

In the case that $\Lambda_v$ is simply-connected we can use a DG-algebra map 
\[ 
\Phi\colon C_{-*}(\Omega_{p_v} \Lambda_v) \to A(\Lambda_v) 
\] 
that goes in the opposite direction to the
Adams-Hilton map to pass to a more economical quasi-isomorphic model.
Such a homotopy equivalence $\Phi$ is constructed in two steps: first, in
\cite{Neisendorfer} using Eilenberg-Moore methods, a DG-algebra quasi-isomorphism:

\begin{equation}\label{eq:adamswithref} 
C_{-*}(\Omega_{p_v} \Lambda_v) \to \Omega C_*(\Lambda_v)
\end{equation} 
is constructed where in both instances $C_*$ refers to the normalised singular chains. Second, using
the standard $A_\infty$-coalgebra quasi-isomorphism between the DG-coalgebra of singular chains $C_*(\Lambda_v)$ and
the $A_\infty$-coalgebra $C_{*}^{\text{\tiny CW}}(\Lambda_v)$ of normalised cellular chains, one obtains a DG-algebra quasi-isomorphism
\[ 
\Omega C_*(\Lambda_v) \to \Omega C_*^{\text{\tiny CW}}(\Lambda_v) = A(\Lambda_v),
\] 
since we assumed that the complexes $C_*$ and $C_*^{\text{\tiny CW}}$ are simply-connected. (In Section \ref{CEsimplyconnected},
we also give a more geometric construction of a DG-algebra quasi-isomorphism $\Phi$ corresponding to \eqref{eq:adamswithref} landing in Morse chains, using Morse flow
trees.)

Similarly, if $\Lambda_v$ is homotopy equivalent to an Eilenberg-Maclane space $\mathrm{K}(\pi_1,1)$, then the singular chains can be replaced with the group algebra $\K[\pi_1]$: there exists a
quasi-isomorphism of DG-algebras 
\[ 
C_{-*}(\Omega_{p_v} \Lambda_v) \to \K [\pi_1] 
\] 
given by
sending a $0$-chain to its homology class, and sending all higher dimensional chains to 0. Note that this DG-algebra
map exists for any space $\Lambda_v$, but is a quasi-isomorphism only in the case that $\Lambda_v$ is
homotopy equivalent to $\mathrm{K}(\pi_1,1)$.

It is often convenient to use a cofibrant (or free) replacement for $\K[\pi_1]$. For example, 
if $\Lambda_v = S^1$, then $\K[\pi_1] \cong \K[t,t^{-1}]$ and a cofibrant replacement is given by
the free graded algebra
\[ \mathbb{K} \langle s_1, t_1, k_1, l_1, u_1 \rangle, \ \ |s_1|=|t_1|=0, |k_1|=|l_1|=-1,
|u_1|= -2 \] 
with the differential 
\begin{align*}
    dk_1 &= 1 - s_1 t_1 \\
    dl_1 &= 1 - t_1 s_1 \\
    du_1 &= k_1 s_1 - s_1 l_1 
\end{align*}
A DG-algebra defined over $\K[t,t^{-1}]$ can be pulled back to a weakly equivalent DG-algebra over this cofibrant
replacement. (See \cite{tabuada} for background in model categories on DG-algebras that we are using in a very simple case here.)

\subsection{Construction of Legendrian invariants}\label{ssec:Leginv} 

As above, let $X$ be a Liouville domain with $c_1(X)=0$ (for $\Z$-grading) and $\partial X = Y$ its contact boundary. Let     $\Lambda =
\bigsqcup_{v \in \pi_0(\Lambda)} \Lambda_v$ be a Legendrian submanifold in $Y$ where $\Lambda_v$ is a connected
component of $\Lambda$. Assume that $\Lambda$ is relatively spin and that its Maslov class vanishes. Let each connected component $\Lambda_{v}$ be decorated with a sign and
write $\Lambda^{+}$ and $\Lambda^{-}$ for the union of the components decorated accordingly. (Our
different treatments of $\Lambda^{+}$ and $\Lambda^{-}$ is natural from the point of view of handle
attachments; recall from the introduction that when $\Lambda^{-}$ is a union of
spheres we attach usual Lagrangian disk-handles to $\Lambda^{-}$ and handles with cotangent ends to
$\Lambda^{+}$.)
When we have an exact Lagrangian filling $L$ of $\Lambda$ (relatively spin and with
vanishing Maslov class), $L$
can also be decomposed into embedded components $L = \bigcup_{v \in \Gamma} L_v$. These embedded components are
not disjoint, they are allowed to intersect transversely at finitely many points. There is a
bijection between $\Gamma$ and the embedded components of $L$.

We require that if two components $\Lambda_{w_{1}}$ and $\Lambda_{w_{2}}$ are boundary components of the same embedded component $L_{v}$ then either both belong to $\Lambda^{-}$ or both to $\Lambda^{+}$. Using this property, we get a decomposition $\Gamma =
\Gamma^+ \sqcup \Gamma^-$, corresponding to the decomposition $\Lambda = \Lambda^+ \sqcup \Lambda^-$. 

Let $\k$ be the semi-simple ring generated by mutually orthogonal idempotents
$\{ e_v \}_{v \in \Gamma}$. If we are not given a filling of $\Lambda$, then
the index set $\Gamma$ is taken to be the connected components, $\pi_0(\Lambda)$, instead. If we need
to distinguish between the two choices, we will denote them as $\k_{\Lambda}$ and $\k_{L}$. Note that there is an injective ring map $\k_{L} \to \k_{\Lambda}$ which takes the idempotent $e_{v}$ corresponding to an embedded component $L_{v}$ to the sum $e_{w_{1}}+\dots+e_{w_{r}}$ of idempotents of its boundary components $\Lambda_{w_{j}}$. In particular, this map turns any $\k_{\Lambda}$-bimodule into a $\k_{L}$-bimodule.

Let $\mathcal{R}$ denote the set of non-empty Reeb chords of $\Lambda$. This is a graded set: the grading of a
chord $c \in \mathcal{R}$ is given by $|c| = -\mathrm{CZ}(c)$ where $\mathrm{CZ}(c)$ is the Conley-Zehnder grading, see Section \ref{sec:mdlispaces}.
(With this convention, the unique chord $c$ of the standard Legendrian unknot in
$\R^3$ has $|c|=-2$ and for the corresponding Legendrian unknot in $\R^{2n-1}$ with one Reeb chord
$c$, $|c|=-n$. See also Remark \ref{spherecon}.) 

Note that the vector space generated by $\mathcal{R}$ is a $\k$-bimodule, where $e_v \mathcal{R} e_w$ corresponds to the set of
Reeb chords from $\Lambda_v$ to $\Lambda_w$. The underlying algebra of the standard
Chekanov-Eliashberg DG-algebra is generated freely by $\mathcal{R}$ over $\k$. We need to
modify this in the case $\Lambda^+$ is non-empty to incorporate chains in the based loop space of
$\Lambda_v$ for $v\in \Gamma^+$. Let us first do this using cubical chains.  

For each $v \in \Gamma^+$, consider the cubical chains
$C_{-*}(\Omega_{p_v}\Lambda_v)$ as a $\k$-algebra by requiring that the left or right
action of $e_w$ is trivial except if $w=v$ when it acts as identity. Let $CE^*$ be the algebra over $\k$ given by adjoining elements of $\mathcal{R}$ to the union
of $C_{-*}(\Omega_{p_v} \Lambda_v)$ for $v \in \Gamma^+$. Thus an element of $CE^*$ is a a sum of alternating words in Reeb chords, $\sigma_{1}c_{1}\sigma_{2}c_{2}\dots \sigma_{m}c_{m}\sigma_{m+1}$, where $c_{j}$ are Reeb chords and $\sigma_{j}$ chains of based loops in the component of the Legendrian where the adjacent Reeb chord lies.

Now the differential on $CE^*$ is defined by extending the differential on the cubical complexes
$C_{-*}(\Omega_{p_v}\Lambda_v)$ for $v \in \Gamma_+$. We describe the differential on a single Reeb chord and extend it by the graded
Leibniz rule. The differential $d$ on a Reeb chord decomposes to a sum 
\[ d = \sum_{i \geq 0} \Delta_i, \]
where for any Reeb chord $c_0$ only finitely many $\Delta_i(c_0)$ are non-zero. The operations $\Delta_{i}(c_{0})$ are defined as follows.

Consider moduli spaces of holomorphic disks with positive puncture at $c_{0}$, see Section
\ref{sec:mdlispaces} for definitions and notation. More precisely, consider Reeb chords $c_i, \ldots, c_1$ such that $c_{0}c_{i}\dots c_{1}$ is a composable word and let $\mathbf{c}=c_{0}^{+}c_{i}^{-}\dots c_{1}^{-}$. 
Consider the space of disks $D_{i+1}$ with one distinguished positive puncture and $i$ negative punctures (across which the boundary numbering is constant in the terminology of Section \ref{sec:mdlispaces}). Consider the moduli space $\mathcal{M}^{\sy}(\mathbf{c})$. Theorems
\ref{thm:mdlitv} and \ref{thm:mdlicmpct} imply that this is a smooth orientable manifold with a natural compactification as a stratified space that carries a fundamental chain. It follows that, via the evaluation map at a point in the
boundary arcs of $D_{i+1}$, $\mathcal{M}^{\sy}(\mathbf{c})$ parameterizes a chain of paths in the
$(i+1)$-fold product $\Lambda^{\times(i+1)}$. We transform these chains of paths to chains of based loops as follows. Pick on each component $\Lambda_{v}$ reference arcs connecting all Reeb chord enpoints to the base point. Let $U_\nu$ be a disk which we can take as a neighborhood of these arcs. Then $(\Lambda_{v},\ast_{v})$, where $\ast_{v}$ is the base point is homotopy equivalent to $\Lambda_{v}/U_{v}$ and the moduli space naturally define chains of loops in the latter. 
In this way, the chains of paths become chains of based loops in
$(\Omega_{p}\Lambda)^{\times (i+1)}$. We treat two cases separately. First, if all boundary components of $D_{i+1}$ map to components in $\Lambda^{-}$ then we let
\begin{equation}\label{eq:only-} 
[\mathcal{M}^{\sy}(\mathbf{c})]=
\begin{cases}
n c_{i}\dots c_{1} &\text{if } \dim(\mathcal{M}^{\sy}(\mathbf{c}))=1,\\
0 &\text{if } \dim(\mathcal{M}^{\sy}(\mathbf{c}))\ne 1,
\end{cases}
\end{equation}
where $n$ is the algebraic number of $\R$-components in the moduli space. 
Second, if some boundary component maps to a component in $\Lambda^{+}$ then
we write $[\mathcal{M}^{\sy}(\mathbf{c})]$ for the chain of paths in $(\Omega_{p}\Lambda)^{i+1}$, where we separate the components in the product by the Reeb chords $\mathbf{c}'=c_{i}\dots c_{1}$:
\[ 
[\mathcal{M}^{\sy}(\mathbf{c})]=\sigma_{i+1}c_{i}\sigma_{i}\dots \sigma_{2}c_{1}\sigma_{1},
\]
where $\sigma_{j}$ are the components of the fundamental chain $\sigma\colon
\mathcal{M}^{\sy}(\mathbf{c})\to(\Omega\Lambda)^{\times (i+1)}$. Further, we write $e_{v}$ for each boundary component that maps to a
component in $\Lambda^{-}$ in between the Reeb chords $c_{i}\dots c_{1}$ as above. 

A subtle point here is that the moduli space $\mathcal{M}^{\sy}(\mathbf{c})$ naturally gives rise to
a chain $\sigma$ in $C_{-*}( \Omega \Lambda^{\times(i+1)})$ rather than in $C_{-*}(\Omega
\Lambda)^{\otimes (i+1)}$. Note that $\sigma_i$
are simply components of $\sigma$, they are not considered as chains. To separate these out we apply the
cubical Alexander-Whitney map, recalled in Section \ref{Coefficients}:
\[ \eta : C_{-*}((\Omega \Lambda)^{\times (i+1)}) \to C_{-*}(\Omega
\Lambda)^{\otimes
(i+1) } \]

With these conventions we then define for $i \geq 0$,  
\[ 
\Delta_i(c_0) := \sum_{\mathbf{c}=c_0^{+}c_i^{-}\dots c_1^{-}} \eta[\mathcal{M}^{\sy}(\mathbf{c})],
\] 
where we separate the components of the tensor product by the Reeb chords in $\mathbf{c}'$ in analogy with the notation for the product chain and where $\eta$ is the Serre diagonal from Equation \ref{serrediag}. The output of $\Delta_{i}(c_{0})$ is thus a sum of alternating words of chains of loops in $C_{-\ast}(\Omega\Lambda)$ and Reeb chords, and   
$\Delta_{i}$ is an operation of degree $2-i$ on $LC_*(\Lambda)$. We point out that if there are $\Lambda^+$ components then higher dimensional moduli spaces contribute to
the differential (unlike the case when $\Lambda = \Lambda^-$). 
Note also that it is possible to have holomorphic disks contributing to $\Delta_0$, which means that the chord $c_0$ is the positive puncture of a disk without negative punctures.

Our next result shows that the operations $\Delta_{i}$ give a differential on $CE^{\ast}$. The proof of that result uses boundaries of moduli spaces of holomorphic disks. By SFT compactness \cite{BEHWZ} and standard gluing results, see e.g., \cite[Appendix A]{ES} and \cite[Appendix B]{Ersft} the boundary of a moduli space $\mathcal{M}^{\sy}(\mathbf{c})$ consists of several level holomorphic buildings of curves with top level in $\mathcal{M}^{\sy}(\mathbf{c}')$ and lower levels in $\mathcal{M}^{\sy}(\mathbf{c}'')$ where the positive puncture of a curve in a lower level is attached at a negative puncture of a curve above it. We will use the compact notation $\star$ to denote all such broken configurations and write simply
\[ 
\partial\mathcal{M}(\mathbf{c})= \mathcal{M}(\mathbf{c}')\star\mathcal{M}(\mathbf{c}'').
\] 
We next need to consider the fundamental chain of loops $[\mathcal{M}(\mathbf{c}')\star\mathcal{M}(\mathbf{c}'')]$ carried by 
$\mathcal{M}(\mathbf{c}')\star\mathcal{M}(\mathbf{c}'')$, or in other words the codimension $1$ boundary of $[\mathcal{M}(\mathbf{c})]$. If the dimension of $\mathcal{M}(\mathbf{c})$ is $d$ then its boundary components give $d-1$ dimensional chains of loops in $\Lambda$. Consider a several level building with moduli space components $\mathcal{M}_{j}$ of dimension $d_{j}>1$, $j=1,\dots,m$. Then, by SFT compactness, $d=\sum_{j=1}^{m} d_{j}$. A boundary component of a several level disk that consists of boundary segments from $k$ disks in $\mathcal{M}_{j_{1}},\dots,\mathcal{M}_{j_{k}}$ will then carry a chain of loops in $\Lambda$ of dimension $\sum_{r=1}^{k}(d_{j_{r}}-1)\le d-2$, with equality only if the broken configuration consists of only two levels. It follows that only two level curves contribute to $[\mathcal{M}(\mathbf{c}')\star\mathcal{M}(\mathbf{c}'')]$ and that the chains of loops along boundary two level boundary segments are given by the Pontryagin product of the two chains at one level segments that form the two level segment.   

\begin{prop} 
    Let $d\colon CE^{\ast} \to CE^{\ast}$ be the map extended to $CE^*$ by the graded
Leibniz rule. 
Then $d$ is a differential, $d^{2}=0$. We call $CE^{\ast}$ with the differential $d$
the \emph{Chekanov-Eliashberg DG-algebra}. 
\end{prop}

\begin{rem}
    When $\Lambda=\Lambda^{-}$, $CE^*$ was called $LCA^*$ in \cite{EtLe} - this is cohomologically graded version of the usual Legendrian homology algebra $LHA_*$ in \cite{BEE}. By definition, we have $LHA_*= CE^{-*}$. 
\end{rem}

\begin{proof} 
In case there are only components in $\Lambda^{-}$ involved the result follows from standard arguments involving the boundary of 1-dimensional moduli spaces, see e.g., \cite{EESPxR,Ersft,BEE}. Consider therefore the case when there are chains in the loop space involved. 

Let $\mathbf{c}=c_{0}^+ c_{m}^-\dots c_{1}^-$. The $d$-dimensional moduli space
$\mathcal{M}^{\sy}(\mathbf{c})$ contributes to $d c_{0}$. Its boundary consists
of broken curves with one level of dimension $d-k$ and one of dimension $k$ for
$0< k <d$, we find, with $\partial$ denoting the natural tensor extension of
the boundary operator in singular homology over boundary components involved in
    $\Lambda^{+}$, that \[ \partial [\mathcal{M}^{\sy}(\mathbf{c})] =
[\mathcal{M}^{\sy}(\mathbf{c}')\star \mathcal{M}^{\sy}(\mathbf{c}'')] \] where
$\star$ is as explained above,
where $\mathbf{c}'=c_{0}^+ c_{m}^-\dots
c_{j}^{-} b^{-} c_{j-k}^{-}\dots c_{1}^{-}$, and  $\mathbf{c}''=
b^{+}c_{j-1}^-\dots c_{j-k+1}^{-}$.  We next apply the cubical
Alexander-Whitney map $\eta$ to this formula to deduce \begin{align*}
\partial\circ\eta[\mathcal{M}^{\sy}(\mathbf{c})]&=\eta\circ\partial[\mathcal{M}^{\sy}(\mathbf{c})]=\eta[\mathcal{M}^{\sy}(\mathbf{c}')\star
\mathcal{M}^{\sy}(\mathbf{c}'')]\\ &= \eta[\mathcal{M}^{\sy}(\mathbf{c}')]\cdot
\eta[\mathcal{M}^{\sy}(\mathbf{c}'')], 
\end{align*} 
where $\cdot$ is the Pontryagin product and we used that the
$\eta$ is a chain map and is compatible with the product. The fact that $\eta$
is a chain map is well known. We verify that it is compatible with the product
below. It follows that the terms contributing to $d^{2}$ which arises from the
differential acting on chains and acting on Reeb chords cancel.

It remains to check that $\eta$ is compatible with the product,
we need to check that the following compositions agree:
\begin{align*} 
C_{-*}(\Omega_{p_u} \Lambda_u &\times \Omega_{p_v} \Lambda_v) \otimes 
C_{-*}(\Omega_{p_v} \Lambda_v \times \Omega_{p_w} \Lambda_w) \\
    &\xrightarrow{\times}   C_{-*}(\Omega_{p_u} \Lambda_u \times \Omega_{p_v} \Lambda_v \times \Omega_{p_v} \Lambda_v \times
    \Omega_{p_w} \Lambda_w)\\ & \xrightarrow{\id \times \cdot \times \id}  
C_{-*}(\Omega_{p_u} \Lambda_u \times \Omega_{p_v} \Lambda_v \times
    \Omega_{p_w} \Lambda_w)\\ & \xrightarrow{\eta} 
    C_{-*}(\Omega_{p_u} \Lambda_u)  \otimes C_{-*}(\Omega_{p_v} \Lambda_v \times
    \Omega_{p_w} \Lambda_w) \\ & \xrightarrow{\id \otimes \eta} 
    C_{-*}(\Omega_{p_u} \Lambda_u)  \otimes C_{-*}(\Omega_{p_v} \Lambda_v) \otimes
    C_{-*}(\Omega_{p_w} \Lambda_w)  
\end{align*}
and 
\begin{align*} C_{-*}(\Omega_{p_u} \Lambda_u &\times \Omega_{p_v} \Lambda_v) \otimes 
    C_{-*}(\Omega_{p_v} \Lambda_v \times \Omega_{p_w} \Lambda_w) \\
    &\xrightarrow{\eta \otimes \eta}  
    C_{-*}(\Omega_{p_u} \Lambda_u) \otimes C_{-*}(\Omega_{p_v} \Lambda_v) \otimes
    C_{-*}(\Omega_{p_v} \Lambda_v) \otimes C_{-*}(\Omega_{p_w} \Lambda_w)\\ & \xrightarrow{\id
    \otimes \times \otimes \id}  
    C_{-*}(\Omega_{p_u} \Lambda_u) \otimes C_{-*}(\Omega_{p_v} \Lambda_v \times \Omega_{p_v}  \Lambda_v)
    \otimes C_{-*}(\Omega_{p_w} \Lambda_w)\\ & \xrightarrow{\id \otimes \cdot \otimes \id} 
    C_{-*}(\Omega_{p_u} \Lambda_u)  \otimes C_{-*}(\Omega_{p_v} \Lambda_v) \otimes
    C_{-*}(\Omega_{p_w} \Lambda_w)  
\end{align*}
This is easily checked by evaluating it on a test chain $(\sigma,\tau)$. Note that this uses the fact that the cubical chain complex is a quotient. Namely, degenerate cubical chains are divided out.
\end{proof} 

\begin{rem}
	As discussed in Section \ref{sec:mdlispaces}, the moduli spaces in the definition of the differential on $CE^{\ast}$ above are defined in terms of anchored moduli spaces $\mathcal{M}^{\sy}$, i.e., moduli spaces of disks with additional interior punctures where holomorphic planes in the filling with asymptotic markers are attached. We point out that in order to calculate the differential one need only take into account rigid such holomorphic planes of dimension zero. For higher dimensional moduli spaces of planes of dimension $d_{0}>0$ the dimension of the curves in the symplectization is $d+1-d_{0}$ and does not contribute to the $d$-dimensional chain $[\mathcal{M}^{\sy}]$.   
\end{rem}

\begin{rem}\label{r:topologicalversion}
As mentioned in the introduction we relate $CE^{\ast}$ as defined above to a parallel copies version
    of the same algebra which is defined solely in terms of rigid moduli spaces. In order to do so
    it is convenient to use a topologically simpler but algebraically more complicated model of
    $CE^{\ast}$, defined as follows. The generating set of our algebra is extended to chains in the
    product $(\Omega \Lambda)^{\times (i+1)}$, where we separate the coordinate functions by Reeb chords. We define the product of two such chains by taking the Pontryagin product of the chains at adjacent factors giving an operation 
\[ 
    C_{-\ast}((\Omega \Lambda)^{\times (i+1)})\otimes C_{-\ast}((\Omega
    \Lambda)^{\times(j+1)})\to
    C_{-\ast}((\Omega \Lambda)^{\times(i+j+1)}).
\]
The differential on this version of $CE^{\ast}$ is then defined by the singular differential on the chain and as
\[ 
\Delta_i(c_0) := \sum_{\mathbf{c}=c_0^{+}c_i^{-}\dots c_1^{-}} [\mathcal{M}^{\sy}(\mathbf{c})],
\]
on Reeb chord generators. (In other words we define it as above but disregard the diagonal
    approximation.) It follows from the K\"unneth formula that the two versions of $CE^{\ast}$ are quasi-isomorphic.    
\end{rem}

Although, the definitions of $CE^*$ given above works generally, from a computational
perspective, it is hard to get our hands on, as the cubical chain complexes $C_{-*}(\Omega_{p_v} \Lambda_v)$ have uncountably many elements. 

Next, we provide a modification of the definition that gives a quasi-isomorphic DG-algebra, under the assumption that for each $v\in
\Gamma^+$, there exists a
DG-algebra quasi-isomorphism 
\[ \Phi \colon C_{-*}(\Omega_{p_v} \Lambda_v) \to \K \langle \mathcal{E}_v \rangle \] 
where $\K \langle \mathcal{E}_v \rangle$ is a DG-algebra structure on a free algebra generated by a
graded finite set $\mathcal{E}_v$. For example, as
discussed in Section \ref{Coefficients}, such a DG-algebra map exists when $\Lambda_v$ is
simply-connected (or if $\Lambda_v$ is a $K(\pi_1,1)$ space, one can first work with the group ring
$\K[\pi_1]$ and base-change to a cofibrant replacement of it). 

We define a graded quiver $\mathcal{Q}_\Lambda$ with vertex set $\mathcal{Q}_0 = \Gamma$ and arrows in correspondence with
\[ \mathcal{Q} := \mathcal{R} \cup \bigcup_{v \in \Gamma^+} \mathcal{E}_v. \] 
More precisely, there are arrows from vertex $v$ to $w$ corresponding to the set of Reeb chords from
$\Lambda_v$ to $\Lambda_w$. In addition, for each $v \in \Gamma^{+}$, there are arrows from $v$ to $v$ corresponding to the elements in $\mathcal{E}_v$. 
Let $LC_*(\Lambda)$ be the graded $\k$-bimodule generated by $\mathcal{Q}$. Thus, there is one generator for each arrow in $\mathcal{Q}$ and an idempotent $e_v$ for each vertex $v \in \mathcal{Q}_0$ and write $\overline{LC}_{\ast}(\Lambda)$ for the submodule without the idempotents.

Let $CE^*(\Lambda)$ be $\k$-algebra given by the tensor algebra 
\[ 
CE^*(\Lambda) = \k \oplus \bigoplus_{i=1}^\infty \overline{LC}_{\ast}(\Lambda)[-1]^{\otimes_{\k} i}
\]
Recall that the path algebra of a quiver is defined as a vector space having all paths in the quiver as basis (including, for each vertex $v$ an idempotent $e_v$), and multiplication given by concatenation of paths. Thus, the $\k$-bimodule $CE^*(\Lambda)$ is the path algebra of the quiver $\mathcal{Q}_\Lambda$ where the grading of each arrow is shifted up by $1$. Just like in the cobar construction we write elements in $CE^{\ast}(\Lambda)$ as
\[ 
[x_{m}|\dots|x_{1}] = s^{-1}x_{m}\otimes_\k \dots \otimes_\k s^{-1}x_{1}\in CE^{\ast}(\Lambda),
\]
where $x_{j}\in \overline{LC}_{\ast}(\Lambda)$.

Next, we equip the $\k$-algebra $CE^*(\Lambda)$ with a differential using the moduli spaces of holomorphic disks defined in Section \ref{sec:mdlispaces}. This differential is induced by operations 
\[ 
\Delta_i \colon \overline{LC}_{\ast}(\Lambda) \to \overline{LC}_*(\Lambda)^{\otimes_\k i}, \quad i=0,1,\ldots, 
\]
where $\overline{LC}_{\ast}(\Lambda)^{\otimes_\k 0}=\k$,
that give $LC_*(\Lambda)$ the structure of an $A_\infty$-coalgebra if $\Delta_0=0$.

Consider a generator of $\overline{LC}_{\ast}(\Lambda)$. If it is a generator $\sigma \in
\mathcal{E}_v$ of the free model of $C_{-\ast}(\Omega_{p_v} \Lambda_{v})$ for some component
$\Lambda_{v}\subset\Lambda^{+}$ then we define
\begin{equation}\label{eq:coalgchain}
\Delta_{i}\sigma = d_{i}\sigma,
\end{equation}
where $d_{i}$ is the co-product that corresponds to $i^{th}$ homogeneous piece of the differential
in the free model $\K\langle \mathcal{E}_v \rangle$ of $C_{-*}(\Omega_{p_v}
\Lambda_v)$. If it is a Reeb chord $c_{0}$ the define
$\Delta_i(c_0)$ as before using moduli spaces $\mathcal{M}^{sy}(\mathbf{c})$ but now take the image of all
the singular chains in $C_{-*}(\Omega_{p_v} \Lambda_v)$ under the map $\Phi\colon
C_{-*}(\Omega_{p_v} \Lambda_v) \to \K \langle \mathcal{E}_v \rangle$. Since the map $\Phi$ is a DG-algebra map, the proof that $d$
is differential on $CE^*$ is the same. Furthermore, since $\Phi$ is a quasi-isomorphism, we
get a quasi-isomorphic chain complex $CE^*$ if we use $\K\langle \mathcal{E}_v \rangle$ coefficients instead of
$C_{-*}(\Omega_{p_v} \Lambda_v)$.

From now on, unless otherwise specified, we will always assume that we work with a free (over $\k$) model of $CE^*$. 

If there exists an augmentation $\epsilon\colon CE^*(\Lambda) \to \k$, then there is a change of coordinates which turn $LC_{\ast}(\Lambda)$ into a $A_{\infty}$-coalgebra. More precisely, consider the restriction of $\epsilon$, $\epsilon_{1}\colon LC_{\ast}(\Lambda)\to\k$, where we think of $LC_{\ast}(\Lambda)$ as the degree $1$ polynomials in $CE^{\ast}(\Lambda)$. Define
\[ 
LC_{\ast}^{\epsilon}=\k\oplus \ker(\epsilon_{1}).
\]   
Note that $\ker(\epsilon_{1})$ is generated idempotents $e_{v}$ and by $c-\epsilon(c)$, where $c$ ranges over the generators of $\overline{LC}_{\ast}(\Lambda)$. Let
\[ 
\phi_{\epsilon}\colon \bigoplus_{i\geq 0} LC_*^{\otimes_{\k}i}\to \bigoplus_{i\geq 0} LC_*^{\otimes_{\k}i},
\]
be the $\k$-algebra automorphism defined on generators as $\phi_{\epsilon}(c)=c+\epsilon(c)$. Define the operations $\Delta_i^\epsilon\colon LC_{\ast}^{\epsilon}(\Lambda)\to LC_{\ast}^{\epsilon}(\Lambda)^{\otimes_{\k} i}$ by
\[ \Delta_i^\epsilon  = \phi_\epsilon \circ \Delta_i \circ \phi_{\epsilon}^{-1}. \]
 
\begin{thm} \label{noncurved} 
The operations $(\Delta_{i})_{i \geq 1}$ satisfy the $A_\infty$-coalgebra relations and with these operations $LC_{\ast}^{\epsilon}(\Lambda)$ is a coaugmented conilpotent coalgebra.
\end{thm}

\begin{proof}
Let $d$ denote the differential on $CE^{\ast}(\Lambda)$ and let $c$ be a generator of $\overline{LC}_{\ast}(\Lambda)$. Since $\epsilon$ is a augmentation, $\epsilon(dc)=0$ and it follows that $\Delta_{0}^{\epsilon}=0$. The $A_{\infty}$-coalgebra relations then follows by combining the equation $d^{2}=0$ from Theorem \ref{noncurved} with the automorphism $\phi_{\epsilon}$. 

The coaugmentation is simply the inclusion of $\k$. The fact that $LC^{\epsilon}_{\ast}$ is
    conilpotent follows from Stokes' theorem: the sum of the actions of the Reeb chords at the negative end of a disk contributing to the differential is bounded above by the action of the Reeb chord at the positive end. This gives the desired finiteness.    
\end{proof}
 
\begin{rem}
Note that if the original operation $\Delta_{0}$ on $LC_{\ast}(\Lambda)$ equals $0$ then the map $\epsilon$ which takes all generators of $\overline{LC}_{\ast}(\Lambda)$ to $0$ is an augmentation and in this case $LC_{\ast}(\Lambda)^{\epsilon}=LC_{\ast}(\Lambda)$ by construction.
\end{rem} 

\begin{rem}
If there is an augmentation $\epsilon\colon CE^{\ast}(\Lambda)\to \k$ then $CE^{\ast}(\Lambda)$ can be expressed as the cobar construction of a coalgebra: by construction	
\[ 
CE^*(\Lambda) = \Omega(LC_*^\epsilon(\Lambda)). 
\] 
\end{rem}

We next consider the $\k$-linear dual $LA^*_\epsilon(\Lambda) :=
(LC_*^\epsilon(\Lambda))^\#$ of $LC_{\ast}^{\epsilon}(\Lambda)$. 
It follows from Section \ref{ssec:gradeddual} that this is an augmented $A_\infty$-algebra. We call it the \emph{Legendrian $A_\infty$-algebra}.

\begin{rem}
In case $\Lambda=\Lambda^{-}$, it can be shown that this $A_\infty$-algebra can be obtained
from the endomorphism algebra of the augmentation $\epsilon$ in $Aug_{-}$ category  of \cite{BC} by
adjoining a unit to it, but is, in general, different from the endomorphism algebra in the $Aug_{+}$
category of \cite{NRSSZ}.	
\end{rem}

\begin{defi} Given an augmentation $\epsilon\colon CE^*(\Lambda) \to \k$, we define the \emph{ completed
		Chekanov-Eliashberg DG-algebra} to be $\widehat{CE}_\epsilon^*  := \Bar (LA_\epsilon^*)^{\#}$. The underlying $\k$-algebra is the completed tensor algebra:
	\[ \widehat{CE}^*(\Lambda) = \varprojlim_i CE^*(\Lambda)/I^{i} = \k \langle \langle \overline{LC}_{\ast}^{\epsilon}(\Lambda)[-1] \rangle \rangle \]
	where $\overline{LC}_{\ast}^{\epsilon}(\Lambda)$ is the ideal determined by the natural augmentation.  
\end{defi} 

Note that there is a natural chain map:
\[ \Phi \colon CE^*(\Lambda) \to \widehat{CE}^*_\epsilon(\Lambda). \]

%

%
%
%
%
%

\begin{rem} \label{spherecon}  
To illustrate the various gradings, the unique Reeb chord for the standard Legendrian unknot $\Lambda$ in $\R^{2n-1}$ has degree is $-n$ in $LC_*$, $n$ in
    $LA^*$ and $-(n-1)$ in $CE^*$ (while it is $(n-1)$ in $LHA_*$).
    Therefore, we have the graded isomorphisms: $H_*(LC) \cong H_{-*}(S^{n})$, $H^*(LA)
    \cong H^*(S^{n})$ and $H^*(CE) \cong H_{-*} (\Omega S^{n})$.
\end{rem}

\subsection{Parallel copies}\label{sec:parallel}
In this section we will describe the perturbation scheme we will use to define various version of Lagrangian Floer cohomology. For exact Lagrangian submanifolds with Legendrian boundary we get an induced perturbation scheme for the Legendrian boundary that will allow us to define a simpler version of the Chekanov-Eliashberg algebra which is isomorphic to it in case the Legendrian is simply connected. 

Let $X$ be a Weinstein manifold and let $L\subset X$ be an exact Lagrangian submanifold with
Legendrian boundary $\Lambda$. We assume that $\Lambda$ is embedded but allow $L$ to be a several
component Lagrangian with components that intersect transversely. Assume that the components of $L$
are decorated with signs and write $L^{+}$ and $L^{-}$ for the union of the components decorated
$+$ and $-$, respectively. We will use specific families of Morse functions to shift Lagrangian
and Legendrian submanifolds off of themselves in order to relate holomorphic curve theory to Morse
theory, and to perform Floer cohomology calculations without Hamiltonian perturbations. Before we discuss the details of this we recall some general results for Morse flow trees.

\subsubsection{General results for flow trees}\label{ssec:flowtreebasics}
In this section we recall several basic results for Morse flow trees from \cite{E}. Morse flow trees live in a neighborhood of a given Lagrangian or Legendrian and are thus defined in the corresponding cotangent bundle or the 1-jet space. In this paper, we will consider only the case graphical Lagrangians and Legendrians in the cotangent bundle and 1-jet space, respectively. That corresponds to a simple special case of the more general situation considered in \cite{E}, where the nearby Lagrangians and Legendrians are allowed to have singularities when projected to the zero-section.  

Let $M$ be a smooth manifold with cylindrical ends of the form $\partial M\times[0,\infty)$. Consider the cotangent bundle $T^{\ast}M$ and the $1$-jet space $J^{1}M$. We consider graphical Lagrangians and associated Legendrians $\Gamma_{dF}\subset T^{\ast} M$ and $\Gamma_{j^{1}F}\subset J^{1}M$. In the ends our functions will have the form $F=e^{t}f(q)+c$, where $t\in[0,\infty)$, $f\colon\partial M\to\R$ and $c$ is a constant. Let $L_{1},\dots, L_{m}$ be a collection of graphical Legendrians in $T^{\ast}M$ and $\tilde L_{j}$ be a Legendrian lift of $L_{j}$. As in \cite[Section 2.2.2]{E}, $\tilde L_{j}$ defines local gradients as well as cotangent and 1-jet lifts of paths in $M$. Furthermore, \cite[Lemma 2.8]{E} shows that there are maximal flow lines and as in \cite[Definition 2.9]{E} we define their flow orientation. We define flow trees of $\tilde L=\bigcup_{j}\tilde L_{j}$ as in \cite[Definition 2.10]{E} and we will also use partial flow trees, which are flow trees with 'free' 1-valent vertices, not necessarily at a critical point. 

We next discuss transversality for flow trees following \cite[Section 3]{E}. There are two dimension concepts of a flow tree $\Gamma$ involved here: the formal dimension $\dim(\Gamma)$, see \cite[Definition 3.4]{E} and the geometric dimension $\mathrm{gdim}(\Gamma)$, see \cite[Definition 3.5]{E}. In the graphical case considered here these can be described as follows. The formal dimension $\dim(\Gamma)$ is the dimension of the space of flow trees around a tree $\Gamma$ without degeneracies (i.e., only tri-valent internal vertices and non-zero length flow lines at positive punctures not at a minimum and negative punctures not at a maximum) assuming transverse intersections of flow manifolds at each vertex. The geometric dimension on the other hand, which is the dimension of a flow trees near $\Gamma$ with fixed degeneracies (higher valence vertices etc). It is then clear that $\mathrm{gdim}(\Gamma)\le \dim(\Gamma)$. 

We will use a transversality result that says that for generic geometric data the following holds:
\begin{itemize}
\item[$(\mathrm{FT})$] Every flow tree $\Gamma$ comes in a smooth family of dimension $\mathrm{gdim}(\Gamma)$. If $\Gamma$ is degenerate then there is a natural Whitney stratification of the $\dim(\Gamma)$-dimensional space of flow trees around $\Gamma$ with strata of dimension $\mathrm{gdim}(\Gamma)$.	
\end{itemize}
This result follows from \cite[Proposition 3.14]{E}. We next discuss the adaption (simplification actually) in the current set up of the results from \cite[Section 3]{E} that leads to \cite[Proposition 3.14]{E}. First, since all Lagrangians considered here are graphical their front projections are smooth with empty singular locus of the front and the preliminary transversality conditions of \cite[Section 3.1.1]{E} hold trivially. This absence of sigularities also means that all the results \cite[Lemmas 3.9 -- 12]{E} guaranteeing finitely many vertices for trees in the presence of front singularities hold automatically. The \cite[Proposition 3.14]{E} then follows readily and shows that for an open dense set of graphical Lagrangians or Legendrians $(\mathrm{FT})$ holds.    

We say that a finite collection of functions $F_{1},\dots,F_{k}$ on $M$ is flow tree generic provided $(\mathrm{FT})$ holds. It is a consequence of \cite[Proposition 3.14]{E} that any collection of functions can be made flow tree generic by an arbitrarily small perturbation and furthermore that if $F_{1},\dots,F_{k-1}$ is already flow tree generic then $F_{1},\dots, F_{k-1}, F_{k}$ can be made flow tree generic by arbitrarily small perturbation of $F_{k}$.

\subsubsection{Systems of parallel copies}\label{ssec:paralleldetails} 
In this section we give an account for how to choose systems of parallel copies for Lagrangians and Legendrians in such a way that higher product and co-product operations on the Morse complexes can be directly defined (without mapping telescopes of continuation maps, typically used in Hamiltonian Floer cohomology). 

Let $L$ be a Lagrangian with Legendrian boundary $\Lambda$, then a neighborhood of $L$ in $X$ looks like $T^{\ast} L$ and in the cylindrical end $[0,\infty)\times \Lambda$, the vector tangent to $T^{\ast}L$ in the direction of the dual of the $[0,\infty)$-direction corresponds to the Reeb direction. We consider a collection of parallel copies $L_{j}$, $j=0,1,2,\dots$, $L=L_0$. Here $L_{j}$ is the graph in $T^{\ast}L$ of the differentials $dF_{j}$ of a Morse function $F_{j}\colon L\to\R$. The Morse functions $F_{j}$ will have critical points in the compact part of $L$ and in the cylindrical ends they will look like Morsifications of the Reeb push-off, see below for details.

We next discuss the main strategy without all technicalities: the first Morse function $H_{1}=F_{1}$ gives the first parallel copy at small distance $\epsilon>0$ from $L_{0}=L$. We define $L_{1}$ as the graph of the differential of $\epsilon F_{1}$. We want all other copies to be good approximations of $L_{1}$ as seen from $L_{0}$, so that flow lines between $L_{j}$ and $L_{0}$ and between $L_{1}$ and $L_{0}$ are sufficiently close that the corresponding spaces of flow lines can be canonically identified. Let  $L_{j}^{1}=L_{1}$ for $j>1$. 

We next construct $L_{2}=L_{2}^{2}$ as the graph of the differential of a function $\epsilon^{2}H_{2}$ over $L_{2}^{1}=L_{1}$. For small $\epsilon>0$, $L_{2}$ is then well-approximated by $L_{1}$ as seen from $L_{0}$, and flow lines between $L_{0}$ and $L_{1}$ can be identified with flow lines between $L_{0}$ and $L_{2}$. We also want spaces of flow lines between $L_{1}$ and $L_{2}$ to be identified with flow lines between $L_{0}$ and $L_{1}$. This holds provided $H_{2}$ is a sufficiently good approximation of $H_{1}$. Thus we take
\[ 
H_{2}=F_{1} + \epsilon H_{2}=H_{1}+\epsilon H_{2},
\] 
where $H_{2}$ is sufficientkly close to $H_{1}$ and such that the following further condition holds. The Lagrangians $L_{0}$, $L_{1}$, and $L_{2}$ together also define flow trees with three punctures. We take $H_{2}$ so that $(\mathrm{FT})$ holds for $L_{0}$, $L_{1}$, and $L_{2}$. It follows from \cite[Proposition 3.14]{E} that this can be achieved by arbitrarily small perturbation of $H_{2}$.

The construction now proceeds in the same manner. First, preliminarily, set $L_{j}^{2}=L_{2}$, $j>2$. Then let $L_{3}=L_{3}^{3}$ be the graph of the function $\epsilon^{3}H_{3}$ over $L_{3}^{2}=L_{2}$. In order for $L_{3}$ to look like $L_{1}$ from the point of view of $L_{2}$ and like $L_{2}$ from the point of view of $L_{1}$ we take
\[ 
H_{3}= F_{2} +\epsilon^{2} H_{3} = H_{1} +\epsilon H_{2} + \epsilon^{2} H_{3}.
\]
For $\epsilon>0$ sufficiently small we may then identify flow lines and flow trees of any three of the functions and after arbitrarily small perturbation of $H_{3}$, condition $(\mathrm{FT})$ holds for $L_{j}$, $j=0,1,2,3$. Continuing like this we get 
\begin{equation}\label{eq:fctnHoutline} 
H_{k}= F_{k-1} +\epsilon^{k-1} H_{k}=\sum_{k=1}^{k}\epsilon^{k-1}H_{k}. 
\end{equation} 
The corresponding collection $L_{0},\dots,L_{k}$ of parallel copies then has the following properties: flow trees with boundary on any increasing collection $L_{i_{1}},\dots, L_{i_{k}}$ are arbitrarily close to flow trees of $L_{0},\dots, L_{k-1}$ and condition $(\mathrm{FT})$ holds for $L_{0},\dots,L_{k}$.

In order to get the system of parallel copies we need first need to set conventions for the description of the ends. In the ends our Lagrangians $L$ look like cylinders over Legendrians $\Lambda$. A small neighborhood of $\Lambda$ in the contact boundary can be identified with a small neighborhood of the zero section in the $1$-jet space of $\Lambda$. We think of this as the intersection of $(-\delta,\delta)\times T^{\ast}\Lambda$ with a small neighborhood of the zero-section in the cotangent bundle factor and the contact form is $ds-pdq$, where $s$ is a coordinate on $(-\epsilon,\epsilon)$. In this end the Lagrangian is $[0,\infty)\times\Lambda$ and the corresponding neighborhood is $(-\epsilon,\epsilon)\times[0,\infty)\times T^{\ast}\Lambda\subset T^{\ast}[0,\infty)\times\Lambda$. We observe then that the result of moving $\Lambda$, $\epsilon$ units along the Reeb flow is the graph of the differential of the function $B(t,q)=\epsilon t$, $(t,q)\in[0,\infty)\times \Lambda$. We will Morsify this Bott-situation by considering graphs of $F(t,q)=\epsilon (t + f(q))$, where $f(q)$ is a Morse function. Then the Reeb chords inside the neighborhood of $L$ at infinity  between the graph of $F$ and $[0,\infty)\times\Lambda$ are in natural 1-1 correspondence with critical points of $f$. 
In this set up, with infinite ends, there are also flow trees with positive punctures asymptotic to Reeb chords at infinity. In the compactification of the space of flow trees there are flow trees entirely in the $\R$-invariant end, $T^{\ast}\R\times\Lambda$. In the end we have $dF = \epsilon(dt + \frac{\partial f}{\partial q}dq)$ and the results about Morse flow trees from Section \ref{ssec:flowtreebasics} follow readily from the corresponding results of flow trees of $f$ on $\Lambda$.

We now turn to a more detailed description of the construction of parallel copies such that flow trees of ordered subcollections of parallel copies can be identified as discussed above.
Write $[0,\infty)\times Y$ and $[0,\infty)\times\Lambda$ for the ends of $X$ and $L$, respectively.
In the end, with coordinates $(t,q)\in [0,\infty)\times \Lambda$. Consider a collection of pairs of Morse functions $(F_j,f_j)$, where $F_j\colon L\to\R$ and
$f_{j}\colon\Lambda\to\R$, $j=1,2,\dots$ are related as follows in the ends:  
\[ 
F_j(t,q)=\epsilon(t + f_{j}(q))+b,\quad \text{ for }1\gg \epsilon > 0, \text{ and }b>0, 
\] 
and that $F_j$ does not have any local maxima. We next discuss further restrictions related to critical poin 

Let $(F_1,f_1)$ be any pair of positive Morse functions as above. Let $z_1^{1},\dots, z_{1}^{m}$ be
the critical points of $F_1$ and let $x_1^{1},\dots,x_{1}^{l}$ be the critical points of $f_1$. Fix
disjoint coordinate balls $B_j^{1}\subset L$ around $z_j^{1}$ and $D_j^{1}\subset\Lambda$ of
$x^{1}_{j}$ such that $F_1$ and $f_1$ are given by quadratic polynomials in these coordinates. Fix small $\sigma>0$ such that 
\begin{align*} 
|dF_1|&>\sigma_1=\sigma \quad\text{ on }\quad L-\bigcup_{j=1}^{m}B^{j}_{1},\\
|df_1|&> \sigma_1=\sigma \quad\text{ on }\quad \Lambda-\bigcup_{j=1}^{l}D^{j}_{1},
\end{align*}
Let $(F_2,f_2)$ be another pair of positive Morse functions with $m$ respectively $l$ critical points $z_{2}^{1},\dots, z_{2}^{m}$ and $x_{1}^{1},\dots, x_{2}^{l}$, where 
\begin{align*}
z_{2}^{j}&\in B_{1}^{j}\quad\text{ and }\quad \ind(z_{2}^{j})=\ind(z_{1}^{j}),\\
x_{2}^{j}&\in D_{1}^{j}\quad\text{ and }\quad \ind(x_{2}^{j})=\ind(x_{1}^{j})
\end{align*}
Let $\sigma_{2}<\sigma\sigma_{1}$ and fix coordinate balls $B^{j}_{2}\subset B_{1}^{j}$ and $D^{j}_{2}\subset D^{j}_{2}$  such that 
\begin{align*} 
|dF_2|&>\sigma_{2}\quad\text{ on }\quad L-\bigcup_{j=1}^{m}B^{j}_{2},\\
|df_2|&>\sigma_{2}\quad\text{ on }\quad \Lambda-\bigcup_{j=1}^{l}D^{j}_{2}.
\end{align*} 
Finally, we make sure that $F_2<\sigma F_1$, and $f_{2}< \sigma f_{1}$ which we obtain by over all scaling. Note that we might have to shrink $\sigma_2$ after scaling.

We continue inductively and construct a family of pairs $(F_k,f_k)$, $k=1,2,\dots$ of positive Morse
functions with the following properties. Each $F_k$ has $m$ critical points $z_{k}^{1},\dots, z_k^{m}$, each $f_k$ has $l$ critical points $x_{k}^{1},\dots,x_{k}^{l}$. There is $0<\sigma_k$ and disjoint coordinate balls $B_{k}^{j}$ around $z_{k}^{j}$ and $D_{k}^{j}$ around $x_{k}^{j}$ such that 
\begin{align*}
|dF_k|&>\sigma_{k}\quad\text{ on }\quad L-\bigcup_{j=1}^{m}B^{j}_{k},\\
|df_k|&>\sigma_{k}\quad\text{ on }\quad L-\bigcup_{j=1}^{l}D^{j}_{k}.
\end{align*}
Furthermore, the following hold
\begin{itemize}
	\item $B^{j}_{k}\subset B^{j}_{k-1}$ and $D^{j}_{k}\subset D^{j}_{k-1}$,
	\item $\sigma_k\le \sigma\sigma_{k-1}$,
	\item $F_{k}< \sigma F_{k-1}$ and $f_k\le \sigma f_{k-1}$.
\end{itemize}

We next take into account the sign decoration. Assume that $L^{-}$ is non-empty. In this case we
first construct functions $\{(G_{j},g_{j})\}_{j=1}^{\infty}$ exactly as above on all components of
$L$. The actual functions $\{(F_{j},f_{j})\}$ on $(L,\Lambda)$ are then
$(F_{j},f_{j})=(G_{j},g_{j})$ on $L^{+}$ and $(F_{j},f_{j})=(-G_{j},-g_{j})$ on $L^{-}$.

Consider now the Morse functions
\[
\left(H_k= \sum_{j=1}^{k} F_j, h_k= \sum_{j=1}^{k} f_j\right).
\]
(Compare $(F_{j},f_{j})$ to the function $\epsilon^{k}F_{k}$ in the discussion preceding \eqref{eq:fctnHoutline}.) Define the system of parallel copies $L_{0},\dots,L_{k},\dots$ by letting $L_{k}$ be the graph $\Gamma_{d H_{k}}$ of the differential of $H_{k}$. Then we have the following.

\begin{lem}\label{l:flowtree1}
For generic choice of functions $(F_{j},f_{j})$ and all sufficiently small $\sigma>0$ in the construction above, the resulting system of parallel copies $\{L_{j}\}_{j=0}^{\infty}$ has the following properties:
\begin{itemize}
\item Intersection points $L_{k_{0}}\cap L_{k_{1}}$ are transverse and are in natural 1-1 correspondence with intersection points of $L_{0}\cap L_{1}$ (or, in terms of $L$ only: critical points of $F_{1}$ and self-intersection points of $L=L_{0}$).
\item On $L_{+}$, if $k_{0}< k_{1}$ then Reeb chords form $\Lambda_{k_{0}}$ to $\Lambda_{k_{1}}$ are in natural 1-1 correspondence with Reeb chords from $\Lambda_{0}$ to $\Lambda_{1}$ (or, in terms of $\Lambda$ only, critical points of $f_{1}$ and Reeb chords of $\Lambda=\Lambda_{0}$).
\item Flow tree transversality $(\mathrm{FT})$ holds for all ordered finite subcollections $L_{k_{0}},L_{k_{1}},\dots L_{k_{m}}$, $k_{0}<k_{1}<\dots< k_{m}$ of parallel copies. Furthermore, the space of flow trees for any two such ordered collections $L_{k_{0}},L_{k_{1}},\dots L_{k_{m}}$ and $L_{j_{0}},L_{j_{1}},\dots L_{j_{m}}$ are canonically isomorphic.  
\end{itemize}
\end{lem}

\begin{proof}
Intersection points  $L_{k_0}\cap L_{k_1}$ corresponds to intersections between the components of $L$ and critical points of $H_{k_1}-H_{k_0}$. Since
\[ 
H_{k_1}-H_{k_0}=F_{k_{0}+1}+\mathcal{O}(\sigma F_{k_0+1}), 
\]
we find that from the point of view of $L_{k_0}$, $L_{k_1}$ can be viewed as a small perturbation of $L_{k_0+1}$. In particular, if $k_0\ne k_1$ then $L_{k_0}\cap L_{k_1}$ is transverse and there is a unique intersection point near each intersection point in $L\cap L_{1}$ which corresponds to critical points of $F_1$ and self-intersections of $L$. The statement on Reeb chords follows similarly. The last statement follows from the special case of \cite[Proposition 3.14]{E} as described above.
\end{proof}


\begin{rem}
	In Sections \ref{ssec:CWBEE} and \ref{ssec:CWpwrap} we will also apply this construction to Lagrangian submanifolds $C\subset W$, where $W$ is a Weinstein cobordism with both positive and negative ends. Our Lagrangian submanifolds $C$ that will be equipped with systems of parallel copies will however have only positive ends in that case and the above discussion applies without change. See Remark \ref{r:inducednegparallel} for a version when both positive and negative ends are perturbed. 
\end{rem}

\subsubsection{Holomorphic disks and flow trees}
In this section we discuss results relating holomorphic disks and Morse flow trees that are use in computations with parallel copies. 

Let $L\subset X$ be a Lagrangian with cylindrical end $\R\times\Lambda\subset \R\times Y$ and let $\overline{L}(\sigma)=\{L_{j}(\sigma)\}_{j=0}^{\infty}$ be a system of parallel copies for $L$ constructed as in Section \ref{sec:parallel}, where $\sigma>0$ is the scaling parameter. (Roughly, $L_{0}=L$ and $L_{k}(\sigma)$ is at distance $\sigma^{k}$ from $L_{k-1}(\sigma)$, $k=1,2,3,\dots$.) 

We first consider the relation between local holomorphic disks and Morse flow trees. We have the following result for words $\mathbf{a}$ and $\mathbf{c}$ of Reeb chords and intersection points corresponding to critical points of the shifting functions $(F_{1},f_{1})$, see Section \ref{sec:parallel}. 

\begin{lem}\label{l:localdisktreecorr}
	If $\overline{L}$ is flow tree generic and $\kappa=(\kappa_{0},\kappa_{1},\dots,\kappa_{m})$ is an increasing (decreasing) boundary numbering then for all $\sigma>0$ sufficiently small there is a natural 1-1 correspondence between rigid holomorphic disks in $\mathcal{M}^{\fl}(\mathbf{a},\kappa)$ and rigid flow trees of $L_{\kappa_{0}},\dots,L_{\kappa_{m}}$ with asymptotics according to $\mathbf{a}$: there is a neighborhood of the cotangent lift of each rigid tree that contains the boundary of a unique rigid holomorphic disk that is transversely cut out and each rigid disk has boundary in the neighborhood of some rigid tree. Similarly, there are natural 1-1 correspondences between rigid disks in $\mathcal{M}^{\co}(\mathbf{c};\kappa)$ and $\mathcal{M}^{\sy}_{\parallel}(\mathbf{c},\kappa)$ and corresponding rigid flow trees determined by $L_{\kappa_{0}},\dots, L_{\kappa_{m}}$.   
\end{lem} 

\begin{proof}
	This is a consequence of the main results in \cite{E}, \cite[Theorem 1.2 and 1.3]{E}, that show, for compact $L$, as $\sigma\to 0$, any sequence of rigid disks converges to a rigid flow tree (`compactness') and also that near any rigid flow tree in the limit there is a unique rigid holomorphic disk for all sufficiently small $\sigma>0$ (`gluing'). It is essential for this 1-1 corresponednce to hold that there are no multiply covered disks. In the present case the increasing (decreasing) condition guarantees no disk is multiply covered. The modifications necessary for the case of cylindrical ends and corresponding Langrangians of the form $\Lambda\times\R\subset T^{\ast}\Lambda\times\R$ are straightforward (see e.g. \cite{EKH} for flow tree results in a related setting). 
\end{proof}

The second result concern a mixed picture where the disks do not lie entirely in the cotangent bundle. In this case holomorphic disks on a system of parallel copies admit a description with holomorphic disks on the underlying Lagrangian with flow trees attached along their boundaries. Such configurations were considered and the main correspondence was worked out in the setting of knot contact homology in \cite{EENS}. Consider a Lagrangian $L$ and a system of parallel copies $\overline{L}=\{L_{j}\}_{j=0}^{\infty}$ as above and let $\kappa$ be an increasing (decreasing) boundary decoration. A quantum flow tree of $\overline{L}$ is a finite collection of holomorphic disks $D_{1},\dots, D_{m}$ with boundaries on sub-decorations $\kappa^{1}\subset\kappa^{m}$ of $\kappa$ with flow trees emanating from their boundaries with boundary sub-decorations $\theta^{1},\dots,\theta^{n}$, such that inserting flow tree domains in the disks gives a disk and such that inserting the cotangent lifts of of the flow tree at the insertion points we get a boundary condition respecting the decoration $\kappa$. 

In order to establish the desired correspondence between rigid quantum flow trees and rigid holomorphic disks we need additional transversality properties of the shifting Morse function that controls the interface of holomorphic disks and flow trees. The argument is the following. Start with a system of parallel copies that satisfies flow tree transversality and perturb the almost complex structure so that moduli spaces of holomorphic disks with decreasing (increasing) boundary decoration (that cannot be multiply covered) are transversely cut out. Then perturb the shifting Morse functions slightly so that partial flow trees are transverse to the boundary evaluation maps of the transversely cut out holomorphic curves. We say that parallel copies and almost complex structures with this transversality properties are quantum flow tree transverse. Arguing as for flow tree transversality it is straightforward to show that flow tree transversality holds after arbitrarily small perturbation of the shifting Morse functions.

\begin{rem}
	We note one feature of using parallel copies that all approximate a single push off: for a rigid configuration of quantum disks, disk components can be connected only by Morse flow lines (not trees). To see this, consider a rigid configuration with three disks connected by a tree with a trivalent vertex. As $\sigma\to 0$ the tree converges to a flow line and all three disks have to intersect it. This will generically not happen for a rigid configuration, i.e., such configurations can appear only when the formal dimension is at least one. 
\end{rem}      

\begin{lem}\label{l:diskquantumtreecorr}
	If $\overline{L}$ is quantum flow tree generic and $\kappa=(\kappa_{0},\kappa_{1},\dots,\kappa_{m})$ is an increasing (decreasing) boundary numbering then for all $\sigma>0$ sufficiently small there is a natural 1-1 correspondence between rigid holomorphic disks in $\mathcal{M}^{\fl}(\mathbf{a},\kappa)$ and rigid quantum flow trees of $L_{\kappa_{0}},\dots,L_{\kappa_{m}}$ with asymptotics according to $\mathbf{a}$: there is a neighborhood of the cotangent lift of each rigid quantum tree that contains the boundary of a unique rigid holomorphic disk that is transversely cut out and each rigid disk has boundary in the neighborhood of some rigid quantum tree. Similarly, there are natural 1-1 correspondences between rigid disks in $\mathcal{M}^{\co}(\mathbf{c};\kappa)$ and $\mathcal{M}^{\sy}_{\parallel}(\mathbf{c},\kappa)$ and corresponding rigid quantum flow trees determined by $L_{\kappa_{0}},\dots, L_{\kappa_{m}}$.   
\end{lem}

\begin{proof}
	As for flow trees there are two main ingredients: `compactness': as $\sigma\to 0$, any sequence of rigid disks converges to a rigid quantum flow tree and `gluing': near any rigid flow tree in the limit there is a unique rigid holomorphic disk for all sufficiently small $\sigma>0$. The technically most difficult point is gluing. It is present already in the flow tree case in \cite{E} discussed above. The full correspondence were worked out with all details in the case of knot conormals in \cite[Section 5.3]{EENS} and \cite[Section 5.4]{EENS}, respectively. The case considered here can be established in a similar way.   
\end{proof}

\begin{rem}\label{r:smallpossmallneg}
In the setting of Lemma \ref{l:diskquantumtreecorr}, the action-integral $\int d\alpha$ of a holmorphic disk in $\mathcal{M}^{\sy}_{\parallel}(\mathbf{c},\kappa)$ is positive since the almost complex structure $J$ is compatible with the contact form. By Stokes theorem the action-integral equals the difference between the action of the Reeb chord at the positive puncture and the sum of the actions at the negative punctures. It follows in particular that if the positive puncture is a Morse chord then all negative punctures are Morse chords as well and the moduli spaces are controlled already by Lemma \ref{l:localdisktreecorr}.   
\end{rem}

\subsection{Chekanov-Eliashberg algebra with parallel copies}\label{ssec:parallelcopies}
We will relate the Legendrian invariants $LA^{\ast}$ and $LC_{\ast}$ to Floer (co)homology of exact
Lagrangian submanifolds. When studying Lagrangian Floer (co)homology we employ the technique of
parallel copies to achieve transversality and it will be convenient to express the Legendrian
invariants in the same language. As it turns out, in the case that the Legendrians are simply
connected this technique leads to a simpler formulation of the theory which incorporates a model of
chains on the based loop space automatically. Recall that Lagrangian fillings of Legendrian
submanifolds induce augmentations, which after change of variables lead to non-curved Legendrian
$A_\infty$-coalgebras. Geometrically, this means that one uses anchored holomorphic disks. We will
assume that an augmentation of $\Lambda$ has been fixed in this section and all disks considered
will be anchored with respect to this augmentation. We now turn to the description of this theory in
the Legendrian setting. 
   
Let $\Lambda$ be as above with decomposition $\Lambda=\Lambda^{+}\sqcup\Lambda^{-}$. Fix a Morse
function $f\colon\Lambda\to\R$ which is positive on $\Lambda^{+}$ and negative on $\Lambda^{-}$. Use
it as described in Section \ref{sec:parallel} to construct a system $\bar\Lambda=\{\Lambda_{j}\}_{j=1}^{\infty}$ of parallel copies of $\Lambda=\Lambda_{0}$.    

Let $\mathcal{Q}_{0}$, $\k$, and $\mathcal{R}$ be as above. Let $\mathcal{R}^{+}$ denote the Reeb
chords connecting $\Lambda_{0}$ to $\Lambda_{1}$ that lie in a small neighborhood of $\Lambda_{0}$.
By construction there is then a natural 1-1 correspondence between $\mathcal{R}^{+}$ and the set of
critical points of $f$ on $\Lambda^{+}$. (Since $f$ shifts $\Lambda^{-}$ in the negative Reeb
direction there are no such short chords near $\Lambda^{-}$.) Write
\[ \mathcal{R}_{\parallel}=\mathcal{R}\cup\mathcal{R}^{+}, \]
and think of chords in $\mathcal{R}^{+}$ as connecting a component $\Lambda_{v}$ to itself.
Again the set $\mathcal{R}_{\parallel}$ is a graded set: Reeb chords $c\in \mathcal{R}$ are graded as above $|c|=-\mathrm{CZ}(c)$ and the grading of a short chord $c \in \mathcal{R}^{+}$ equals the negative of the Morse index of the critical point of $f$ corresponding to $c$.

We define a graded quiver $\mathcal{Q}_{\parallel,\Lambda}$ with vertex set $\mathcal{Q}_0 = \Gamma$ and arrows in correspondence with
\[ \mathcal{Q}_{\parallel} := \mathcal{R}_{\parallel}. \] 
More precisely, there are arrows from vertex $v$ to $w$ corresponding to the set of Reeb chords from
$\Lambda_v$ to $\Lambda_w$ if $v\ne w$ and corresponding to short Reeb chords from $\Lambda_{v}={\Lambda_{v}}_{0}$ to ${\Lambda_{v}}_{1}$ if $v=w$. 
 
Let $LC_*^{\parallel}(\Lambda)$ be the graded $\k$-bimodule generated by $\mathcal{Q}_{\parallel}$. 
We define an $A_\infty$-coalgebra structure on $LC^{\parallel}_{\ast}(\Lambda)$ given by operations $\Delta_i$ as follows. Given a chord $c_0$ (input) and a chords $c_i, \ldots, c_1$ (outputs), we consider the disk $D_{i+1}$ with distinguished puncture at $c_{0}$ and a \emph{strictly decreasing} boundary decoration $\kappa$. 
Let $\mathbf{c}=c_{0}^{+}c_{i}^{-}\dots c_{1}^{-}$. Consider the moduli space
$\mathcal{M}^{\sy}_{\parallel}(\mathbf{c};\kappa)$. We write $|\mathcal{M}^{\sy}_{\parallel}(\mathbf{c};\kappa)|$ for the
algebraic number of $\R$-components in this moduli space provided
$\dim(\mathcal{M}^{\sy}(\mathbf{c};\kappa))=1$ and $|\mathcal{M}^{\sy}(\mathbf{c};\kappa)|=0$
otherwise. Define for $i > 0$,  
\[ 
\Delta_i(c_0) := \sum_{\mathbf{c}=c_0^{+}c_i^{-}\dots c_1^{-}} |\mathcal{M}^{\sy}_{\parallel}(\mathbf{c};\kappa)|\mathbf{c}',
\] 
where $\mathbf{c}'=c_{i}\dots c_{1}$. This gives an operation of degree $2-i$ on
$LC^{\parallel}_{*}(\Lambda)$. Note that $\Delta_0 =0$ trivially, since the decoration
$\kappa$ is strictly increasing. 

It is not a priori clear that the maps $\Delta_i$ are well-defined as the sum may involve infinitely
many terms. This can be avoided easily if $\Lambda_v$ are simply-connected for $v \in \Gamma^+$.
Namely, the simply-connectedness of $\Lambda_v$ guarantees that the short chords cost index, whereas
the long chords cost energy. In the non-simply connected case, we will only consider a completed version of $LC_{\ast}^{\parallel}$ where infinite expressions are allowed. We define the parallel copies algebra first in the simply connected case and turn to the non-simply connected case in Section \ref{ssec:parallelwithpi1}.

\subsubsection{The parallel copies dg-algebra in the simply connected case}\label{ssec:ssec:parallelwithoutpi1}
When the operations $\Delta_{i}$ defined above has suitable finiteness properties they define a co-algebra. More precisely, we have the following. 

\begin{lem}\label{l:d2=0parallelcoalg}
    Suppose $\Lambda$ is simply-connected or more generally, suppose that $\Delta_i$
    factorizes through the natural inclusion 
    \[ \bigoplus_{i=1}^\infty LC_{\ast}^{\parallel}[-1]^{\otimes_{\k}i} \to
    \prod_{i=1}^{\infty} LC_{\ast}^{\parallel}[-1]^{\otimes_{\k}} \]
    for all $i \geq 1$,  then $LC_{\ast}^{\parallel}$ equipped with the operations
    $(\Delta_{i})_{i\geq 1}$ is an $A_\infty$-coalgebra. 
\end{lem}

\begin{proof}
Recall that the moduli spaces involved in the definition of the operations $\Delta_{i}$ are moduli spaces of anchored disks, where we use the augmentation induced by the filling $L$. Consider a 1-dimensional moduli space $\mathcal{M}^{\sy}(\mathbf{c};\kappa)$ and note that its boundary consists of two level rigid disks by Theorems \ref{thm:mdlicmpct} and \ref{thm:mdlitv}. There are two cases, either the levels are joined at a chord connecting the same copy of $\Lambda$ to itself or the chord connects distinct copies. Since we count anchored disks and since the augmentation satisfies the chain map equation, the first type of breaking cancels algebraically. 
Theorem \ref{thm:mdlicopies} then shows that there are canonical identifications between moduli spaces for different increasing numberings and the proof follows since the second type of breakings  are exactly what contributes to the coalgebra relations and are then in algebraic one to one correspondence with the endpoints of an oriented compact 1-manifold.
\end{proof}
%
%

We will define the parallel copies version of the Chekanov-Eliashberg algebra as the reduced bar construction of a the co-algebra above. To simplify matters we show that for for generic system of parellel copies the coalgebra has a strict counit. (For more general systems of parallel copies one can instead use the definition in Section \ref{ssec:parallelwithpi1}.) Consider a
generator of $LC_{\ast}^{\parallel}$ which is a small Reeb chord $z$ that corresponds to a critical
point of the Morse function $f$. The coalgebra operations $\Delta_{i}(z)$ then counts holomorphic
disks with positive puncture at $z$ which, by an action argument, must lie in a small neighborhood of $\R\times\Lambda$ and rigid such holomorphic disks are in natural correspondence with
rigid Morse flow trees, see Remark \ref{r:smallpossmallneg}. In the simply connected case, i.e., for Morse functions without critical points of index $1$ and $n-1$ we have the following result.

\begin{lem}\label{l:DanR}
Suppose $\Lambda$ is simply connected. If the Morse functions for parallel copies described above are sufficiently close to the first function (i.e., if $\epsilon>0$ in the construction of shifting copies is sufficiently small) then if $x_{v}$ is the minimum of the Morse function on the component $\Lambda_{v}$, the following holds: $\Delta_{1}(x_{v})=0$, $\Delta_{2}(x_{v})=x_{v}\otimes x_{v}$, and
$\Delta_{i}(x_{v})=0$ for $i>2$. Furthermore, if $c$ is any other generator corresponding to
an arrow from $v$ to $w$ then $\Delta_{2}(c)=c\otimes x_{v}+(-1)^{|c|}x_{w}\otimes c+D_{2}(c)$, where $D_{2}(c)$ does not contain any $x_{v}$ factor, and
$\Delta_{i}(c)$ does not contain any factor $x_{v}$ for $i\ne 2$.
\end{lem}

\begin{proof}
Consider a system $\overline{\Lambda}(\sigma)=\{\Lambda_{j}(\sigma)\}$ of parallel copies as in Section \ref{ssec:paralleldetails}.	Note that $x_{v}$ is a Morse chord of action $\mathcal{O}(\sigma)$. Hence, disks with one positive positive puncture at $x_{v}$ can have only other Morse chords as negative punctures. For sufficiently small $\sigma>0$, Lemma \ref{l:localdisktreecorr} then shows we can compute $\Delta_{i}(x_{v})$ by counting all flow trees with a positive puncture at $x_{v}$ and Lemma \ref{l:flowtree1} shows that the flow trees are independent of increasing boundary decoration. The equation $\Delta_{1}(x_{v})$ follows since there is no (negative) gradient flow line emanating from a minimum. For the equation $\Delta(x_{v})=x_{v}\otimes x_{v}$ we consider three copies $L_{0},L_{1},L_{2}$ and observe that there is a unique flow tree with positive puncture at $x_{v}$ and two negative punctures at $x_{v}$, this flow tree consists simply of two flow lines starting at the minumum chord connecting $L_{0}$ to $L_{2}$ and ending at the minimum chords connecting $L_{0}$ to $L_{1}$ and $L_{1}$ to $L_{2}$, respectively.

To see the equations $\Delta_{i}(x_{v})=0$, $i>2$, we start from a general limiting argument for flow trees of parallel copies. Consider a flow tree with positive puncture at a Morse chord $a$ and negative punctures at Morse chords $b_{1},\dots,b_{m}$. As we take the limit $\sigma\to 0$, all shifting functions approach to multiples of the same Morse function and the flow tree limits to a broken flow line starting at $a$ connecting to $b_{i_{1}}$, then from $b_{i_{1}}$ to $b_{i_{2}}$, continuing in this way until all negative punctures have been met. 

Consider now a tree with positive puncture at $x_{v}$. In the limit this converges to a flow line emanating from $x_{v}$, which must then be constant. This shows that all negative punctures must be $x_{v}$ as well. Since the dimension of a tree positive puncture at $x_{v}$ and $i$ negative punctures at $x_{v}$ is $i-2$, $\Delta_{i}(x_{v})=0$ if $i>2$. 

We next consider the properties of $\Delta_{i}(c)$ for $c\ne x_{v}$. The flow trees contributing $c\otimes x_{v}+(-1)^{|c|}x_{w}\otimes c$ to $\Delta_{2}(c)$ are easily found. Consider three parallel copies $L_{0},L_{1},L_{2}$ the flow trees consists of a single flow line from either one of the endpoints of the chord $c$ to the minimum $x_{v}$ or $x_{w}$ of the corresponding component.

We next show that these are the only contributions with negative punctures at minima. We start in the case when $c$ is a Morse chord. Consider a tree with positive puncture at some Morse chord $c$ which does not limit to a single flow line to the minimum as $\sigma\to 0$ and assume that one of the negative punctures is $x_{v}$. Consider first the case when $x_{v}$ is at the last puncture corresponding to the smallest function difference. Assume that the negative punctures are $x_{v}$ and a word of punctures $b$ then $\|c\|-\|b\|-\|x_{v}\|-1=\|c\|-\|b\|=0$, where $\|y\|=\ind(y)-1$. Consider the limit when this function difference goes to zero. Then the flow tree goes to a flow tree with a flow line to $x_{v}$ attached. The remaining flow tree has dimension $\|c\|-\|b\|-1=-1$ and hence does not exist by the flow tree transversality condition $(\mathrm{FT})$ in Section \ref{ssec:flowtreebasics} for the subset of parallel copies obtained by forgetting the last copy. 

Consider next the case that $x_{v}$ is at some other function difference. Then we have negative punctures $b$ before $x_{v}$ and $a$ after $x_{v}$ and thus $a$ are Morse chords of smaller function differences. Consider the limit when all these smaller function differences shrink. In the limit we find a flow tree with negative punctures at $(b,x_{v})$ with a partial flow tree with negative punctures at $a$ attached. The evaluation dimension of the latter tree (i.e. the dimension of the partial tree with a free positive puncture) is
\[ 
(n-1)-\|a\|< n-1,
\]  
where we use the simple connectivity to get strict inequality. Applying the degeneration above to the remaining tree we get a tree with a flow line to $x_{v}$ attached and its evaluation dimension is
\[ 
\|c\|-\|b\|.
\]
Now, $\|c\|-\|b\|-\|a\|-1=-1$ so these two trees do not meet by condition $(\mathrm{FT})$ in Section \ref{ssec:flowtreebasics}. 

The remaining possibility is that the small tree with negative punctures $a$ intersects the flow line towards $x_{v}$. However,  such a tree can be viewed as the original partial tree merging with a flow line from the minimum and then continuing. For $\sigma>0$, at the scale of the tree with punctures $c$ and $b$ (i.e. $\sigma^{k}$ for some $k$) the flow line from the minimum and the flow at the positive puncture of the partial tree attached are very close to parallel (non-parallel only at order $\sigma^{k+l}$ for $l>0$), therefore the evaluation map at the positive puncture of the partial tree with negative punctures $(x_{v},a)$ is arbitrarily close to the evaluation map of the original partial tree with negative punctures $a$ and taking the limit $\sigma\to 0$, the above dimension count: $\|c\|-\|b\|-\|a\|-1=-1$ shows that these does not intersect if $(\mathrm{FT})$ in Section \ref{ssec:flowtreebasics} holds and $\sigma>0$ is  sufficiently small.

We finally consider the case when $c$ is not a Morse chord and a disk which in the limit $\sigma\to 0$ does not converge to a constant disk with a flow line attached and which has a puncture at $x_{v}$. Such a disk must have a non-constant disk component in the limit and by Lemma \ref{l:diskquantumtreecorr} it converges to this disk with flow trees attached in the limit. Let $b$ denote the negative punctures of the disk in the limit and $a$ all Morse chord negative punctures except $x_{v}$. If a flow line to $x_{v}$ is directly attached to the disk then the dimension of the quantum flow tree obtained by removing this flow line is $\|c\|-\|b\|-\|a\|-1=-1$ and hence it does not exist by quantum flow tree transversality, see Lemma \ref{l:diskquantumtreecorr}. If this is not the case then $x_{v}$ is one of the negative punctures in a flow tree attached to the disk. Now that partial flow tree with positive puncture constrained to the evaluation map of the disk must be rigid and arguing exactly as for the trees above we see that quantum flow tree transversality shows that no such configuration exists for sufficiently small $\sigma>0$.   
\end{proof}

\begin{rem}
The simple connectivity is used in the above proof to ensure that cutting with a small tree really reduces dimension. Here cutting means, intersect and start a flow from the intersection locus. In the case that there are index $1$ critical points one could have $|a|=0$ in the above and indeed there are trees with arbitrarily many punctures at index $1$ critical points and then a puncture at $x_{v}$.
\end{rem}

Lemma \ref{l:DanR} shows that, in the simply connected case, there is a
strict coaugmentation  
\begin{equation}\label{eq:coaugparallel}
        \eta\colon\k\to \k_{-} \oplus LC_{\ast}^{\parallel},\quad \eta(e_{v})=x_{v}.
\end{equation}
where $\eta$ is defined by 
\begin{equation} 
\eta(e_{v})=
\begin{cases}
        x_{v} &\text{if $\Lambda_v \subset \Lambda^+$},\\
        e_{v} &\text{if $\Lambda_v \subset \Lambda^-$},
\end{cases}
\end{equation}

\begin{defi}
If $\Lambda$ is simply-connected, the parallel copies Chekanov-Eliashberg DG-algebra is 
\[ 
        CE^{\ast}_{\parallel} = \Omega (\k_- \oplus LC_{\ast}^{\parallel}).
\]
\end{defi}

\subsubsection{The parallel copies dg-algebra in the non-simply connected case}\label{ssec:parallelwithpi1}
In the non-simply connected case, the operations $\Delta_{i}$ defined counting holomorphic curves are not necessarily finite. To get a workable definition we will instead start from an algebra structure on the dual $LA^{\ast}$ of $LC_{\ast}$. More precisely, we proceed as follows. 

Let $LA^*_{\parallel}(\Lambda)$ be the graded $\k$-bimodule generated by $\mathcal{Q}_{\parallel}$. 
We define an $A_\infty$-algebra structure on $LA_{\parallel}^{\ast}(\Lambda)$ given by operations $\Delta_i'$ as follows. Given chords $c_i, \ldots, c_1$ (inputs) and a chord $c_0$ (output), we consider the disk $D_{i+1}$ with distinguished puncture at $c_{0}$ and a \emph{strictly increasing} boundary decoration $\kappa$. As above, let $\mathbf{c}=c_{0}^{+}c_{i}^{-}\dots c_{1}^{-}$ and consider 
$\mathcal{M}^{\sy}(\mathbf{c};\kappa)$. Define for $i > 0$,  
\[ 
\Delta_i'(\mathbf{c}') := \sum_{\mathbf{c}=c_0^{+}c_i^{-}\dots c_1^{-}} |\mathcal{M}^{\sy}_{\parallel}(\mathbf{c};\kappa)|c_{0},
\] 
where $\mathbf{c}'=c_{i}\dots c_{1}$. This gives an operation of degree $2-i$ on
$LA^{\ast}_{\parallel}(\Lambda)$. Note that $\Delta_0' =0$ trivially, since the decoration
$\kappa$ is strictly increasing. 

\begin{lem}\label{l:d2=0parallelalg}
	$LA^{\ast}_{\parallel}$ equipped with the operations
	$(\Delta_{i})_{i\geq 1}'$ is an $A_\infty$-algebra. 
\end{lem}
 
\begin{proof}
Identical to the proof of Lemma \ref{l:d2=0parallelcoalg}.
\end{proof} 

In order to define the parallel copies dg-algebra we first add idempotents $e_{v}$ to $LA^{\ast}_{\parallel}$, one for each component $\Lambda_{v} \subset \Lambda^{-}$. We then get the algebra  $\k_{-} \oplus LA^{\ast}_{\parallel}$. Equip it with the trivial augmentation $\epsilon'$ which is the projection to $\k$. Define
\[ 
{{\widetilde{CE}}^{\ast}}_{\parallel}=(B(\k_{-} \oplus LA^{\ast}_{\parallel}))^{\#},
\] 
and let $\mathscr{I}$ denote the subalgebra of ${{\widetilde{CE}}^{\ast}}_{\parallel}$ as the space of functionals which vanish on monomials not containing $u_{x}$ for some minimum chord $x\in\Lambda$. 

\begin{lem}\label{l:quotientalalgebra}
	The subalgebra $\mathscr{I}$ is closed under the differential.   
\end{lem}

\begin{proof}
	To see that $\mathscr{I}$ is closed under the differential we note that there is no negative gradient flow line that starts at the minimum, therefore any negative gradient tree that starts at $u_x$ must also have a negative puncture at $u_{x}$. 
\end{proof}

\begin{rem}\label{r:gendefparallelsimplyconnected}
In the simply connected case, $LA^{\ast}_{\parallel}=(LC_{\ast}^{\parallel})^{\#}$ and there is a natural restriction map $\rho\colon {{\widetilde{CE}}^{\ast}}_{\parallel}\to \Omega LC_{\ast}^{\parallel}$. Since $u_{x}$ are strict idempotents by Lemma \ref{l:DanR}, this is a chain map. The kernel of $\rho$ is $\mathscr{I}$ and consequently, $\Omega LC_{\ast}^{\parallel}={{\widetilde{CE}}^{\ast}}_{\parallel}/\mathscr{I}$.
\end{rem}

Guided by Remark \ref{r:gendefparallelsimplyconnected} we define the parallel copies DG-algebra as follows in the non-simply connected case.
 
\begin{defi}     
If $\Lambda$ is not simply-connected, then we define with notation as above, the completed DG-algebra
    \[ \widehat{CE}^{\ast}_{\parallel}(\Lambda) =  {{\widetilde{CE}}^{\ast}}_{\parallel}/\I.\] 
\end{defi}


\subsection{Isomorphism between Chekanov-Eliashberg algebras in the simply connected case} 
\label{CEsimplyconnected}

We next show that if $\Lambda$ is simply connected then $CE^{\ast}_{\parallel}(\Lambda)$ is in fact isomorphic to $CE^{\ast}(\Lambda)$. To this end we first establish a Morse theoretic version of Adams result mentioned in the introduction, which here corresponds to the purely local situation of the zero-section in a 1-jet space.

Let $Q$ be a \emph{simply-connected} smooth manifold with a base point $q\in Q$. Fix a system of positive Morse functions
$\bar f=\{f_{j}\}_{j=1}^{\infty}$ as in Section \ref{sec:parallel} and assume that the functions have only one minimum and no index
one critical points. (This can always be arranged by handle cancellation if $\dim(Q)\ge 5$, see Remarks \ref{r:dim4} and \ref{r:dim4iso} for the lower dimensional case.) We will first discuss a Morse flow tree model for chains on $Q$ that we denote
$CM_{-\ast}(Q)$. Our treatment of Morse flow trees follows \cite{E}, see Sections \ref{ssec:flowtreebasics} and \ref{ssec:paralleldetails}. We first recall the details of the flow tree definitions from \cite[Section 2]{E} in the special case needed here.  

Consider a strip $\R\times[0,m]$ or half-strip $[T,\infty)\times[0,m]$ with coordinates $s+i\tau$ and with $m-1$ slits
along $[a_{j},\infty)\times j$, $j=1,\dots, m-1$, and $T\le a_{j}$ in the half-strip case. In the
half strip case the vertical segment $T\times[0,m]$ is a finite end that will be used as an input,
and we do not consider it as a part of the boundary of the strip with slits. In the strip case, the
input is at the puncture $-\infty\times[0,m]$, and in both cases we call punctures at $+\infty$
output. Order the boundary components according to the positive boundary orientation of the disk
with punctures starting from the input and decorate its boundary components by a strictly increasing
sequence of positive integers $\kappa_1<\kappa_2<\ldots<\kappa_{m}$. Let $\kappa = \{\kappa_i
\}_{i=1}^{m}$ denote this decoration. Cutting the strip by line segments $a_j
\times [0,m]$, $j=1,\dots,m-1$, subdivides it into \emph{strip regions} of the form
$[s_{0},s_{1}]\times[\tau_{0},\tau_{1}]$, where $s_{0}\in\{-\infty, a_{1},\dots,a_{m}\}$,
$s_{1}\in\{a_{1},\dots,a_{m},\infty\}$, and $\tau_{0},\tau_{1}\in\{0,1,\dots,m\}$ and with a
numbering $\kappa_{j}$ on each boundary component $[s_{0},s_{1}]\times\{\tau_{0}\}$ and
$[s_{0},s_{1}]\times\{\tau_{1}\}$. 

\begin{defi} \emph{(\cite[Definition 2.10]{E})}
A flow tree is a continuous map from a strip with slits into $Q$ which in each strip region
    $[s_{0},s_{1}]\times [\tau_{0},\tau_{1}]$ depends only on the first coordinate $s\in[s_{0},s_{1}]$ and there satisfies the gradient equation 
\[
\dot x(s) =-\nabla (f_{\kappa_{i}}-f_{\kappa_{j}})(x(s)),
\]
where $\kappa_{i}$ is the numbering of the upper horizontal boundary of the strip region and $\kappa_{j}$ that of the lower.

A partial flow tree is defined analogously except that the domain is a half-strip with slits $[T,\infty)\times [0,m]$. 	
\end{defi}

If $y$ is a critical point of $f_{1}$ then we let $|y|=-\mathrm{index}(y)$ denote the negative Morse index of $y$.
If $\mathbf{y}=y_{0}y_{1}\dots y_{m}$ is a word of critical points of $f_{1}$ then the space of flow trees $\mathcal{T}(\mathbf{y})$ with input puncture at $y_{0}$ and output punctures $\mathbf{y}'=y_{1}\dots y_{m}$, in the order induced by the boundary orientation, has dimension (formal dimension in the language of Section \ref{ssec:flowtreebasics})
\begin{equation}\label{eq:dimtree} 
\dim(\mathcal{T}(\mathbf{y}))= |\mathbf{y}'| - |y_{0}| + (m-2).
\end{equation}
For sufficiently small perturbing system of Morse functions $\bar f$, the space of flow trees is
independent of the increasing boundary decoration $\kappa$, see Lemma \ref{l:flowtree1}.

Let $CM_{-\ast}(Q)$ denote the
$\mathbb{K}$-module generated by critical points of $f_{1}$ and equip $CM_{-\ast}(Q)$ with the structure of a coalgebra with operations $\Delta_{i}$ given by
\[ 
\Delta_{i}(y_{0})=\sum_{|\mathbf{y}'|=|y_{0}|-(i-2)}|\mathcal{T}(\mathbf{y})|\mathbf{y}',
\] 
where the sum ranges over all $\mathbf{y}'$ of word length $i$. It is not hard to see that the boundary of a 1-dimensional space of flow trees consists of broken rigid flow trees from which it follows that the operations $\Delta_{i}$ satisfy the coalgebra relations, compare Lemma \ref{l:d2=0parallelcoalg}. Furthermore, by Lemma \ref{l:DanR}, the coalgebra has a natural co-unit, the critical point which is the minimum of $f_{1}$. We will call this critical point the \emph{co-unit critical point}. We say that a flow tree with no puncture mapping to the co-unit critical point is \emph{co-unit free}.  

The coalgebra $CM_{-\ast}(Q)$ agrees with the Floer coalgebra $CF_{\ast}(Q)$,
where we view $Q$ as the zero section in its own cotangent bundle $T^{\ast}Q$ as follows. Let $\bar Q(\eta)=\{Q_{j}(\eta)\}_{j=0}^{\infty}$ be the system of parallel copies of the
$0$-section $Q \subset T^{\ast}Q$ corresponding to the system of functions $\bar f$, where $\eta>0$ gives the size of the perturbation: $|f_{k+1}-f_{k}|_{C^{s}}=\mathcal{O}(\eta^{k+1})$, see Section \ref{ssec:paralleldetails}.  
Then, by Theorem \ref{l:localdisktreecorr}, 
there is, for all sufficiently small shifts, a natural one to one correspondence between rigid
holomorphic disks with boundary on $\bar Q$ and Morse flow trees in $Q$ determined by $\bar f$. This
gives a chain isomorphism $CM_{-\ast}(Q)\to CF_{\ast}(Q)$. 

We now return to the Morse theoretic approach to Adams' result. We will define a map
\[
\phi\colon C_{-\ast}(\Omega Q)\to \Omega CM_{-\ast}(Q),
\]  
in terms of the operation of attaching co-unit free partial flow trees to Moore loops 
$\sigma_{v}\colon [0,r_v]\to Q$, $r_{v}\ge 0$, based at $q\in Q$, parameterized by a simplex, $v\in\Delta$. To define this operation we will use the following notion: we say that a partial flow tree parameterized by a half-strip $\gamma\colon [T,\infty)\times[0,m]\to Q$ starts at a point $p\in Q$ if its input puncture maps to $p$, $\gamma(T\times[0,m])=p$.

\begin{figure}[h!]
    \centering
    \begin{tikzpicture}[scale=0.6]
    \tikzset{->-/.style={decoration={ markings,
                mark=at position #1 with {\arrow{>}}},postaction={decorate}}}
    \draw [black, thick=1.5] (-5,-1) to (0,-1);
    \draw [black, thick=1.5] (-5,2) to (0,2);
    \draw [black, thick=1.5] (0,0) to (-1,0);
     \draw [black, thick=1.5] (0,1) to (-2,1);

    \draw [black, thick=1.5] (-5,3) to (0,3);
    \draw [black, thick=1.5] (-5,5) to (0,5);
    \draw [black, thick=1.5] (0,4) to (-1.5,4);
 
    \draw [black, thick=1.5] (-5,6) to (0,6);
    \draw [black, thick=1.5] (-5,8) to (0,8);
     \draw [black, thick=1.5] (0,7) to (-2,7);

    \draw[black,dashed] (-5,2) to (-5,3);
    \draw[black,dashed] (-5,5) to (-5,6);

  \draw [black, dashed] (-5, -1) to[in=135,out=225] (-5,8);
    \node at (-5.5,0.5) {\footnotesize{$\sigma_v(t_2)$}}  ;
    \node at (-5.5,4) {\footnotesize{$\sigma_v(t_1)$}}  ;
    \node at (-5.5,7) {\footnotesize{$\sigma_v(t_0)$}}  ;

        \node at (0.4,-1) {$\kappa_8$};
        \node at (0.4,2) {$\kappa_5$};
        \node at (0.4,0) {$\kappa_7$}; 
        \node at (0.4,1) {$\kappa_6$};

        \node at (0.4,3) {$\kappa_5$};
        \node at (0.4,5) {$\kappa_3$};
        \node at (0.4,4) {$\kappa_4$};
        \node at (0.4,6) {$\kappa_3$};
        \node at  (0.4,8) {$\kappa_1$};
        \node at  (0.4,7) {$\kappa_2$};
    \end{tikzpicture}
    \caption{A configuration with 3 partial flow trees attached to $\sigma_v$ at the points $\sigma_v(t_0)$ ,
    $\sigma_v(t_1)$, $\sigma_v(t_2)$. The numbers on the right determines the gradient equation at
    that end. The dashed part represents the loop $\sigma_v$}   \label{adamsmap}
\end{figure}

Attaching partial flow trees to $\sigma_{v}\colon[0,r_{v}]\to Q$ then means fixing points $0\le t_{0}\le t_{1}\dots\le t_{m}\le r_{v}$ and partial flow trees $\Gamma_{j}$, $j=1,\dots,m$ that start at $\sigma_{v}(t_{j})$. Our map $\phi$ takes
values in $\Omega CM_{-\ast}(Q)$ and will accordingly be defined by attaching flow
trees which have no output at the co-unit. 

Take the system of parallel copies $\bar Q(\eta)$ to be flow tree generic (to satisfy $(\mathrm{FT})$ of Section \ref{ssec:flowtreebasics}). Then the set of positive punctures of minimum free partial flow trees for a fixed
increasing boundary numbering constitutes a codimension two subset in $Q$ and that the corresponding
subset for any numbering lies in an $\mathcal{O}(\eta^{2})$-neighborhood of it. We note that, flow tree genericity, in particular, means that the base point $q$ does not lie on any minimum free partial flow tree.

If $\sigma_{v}\colon [0,r_{v}]\to Q$ is a loop with flow trees attached at $0\le t_{0}\le t_{1}\dots\le t_{m}\le r_{v}$, we also introduce a numbering of the components of $[0,r_{v}]-\{t_{1},\dots,t_{m}\}$ induced by the flow trees attached as follows. The right most interval $(t_{m},r_{v}]$ is numbered by $\kappa_{0}$. The right boundary segment of the strip with slits attached at $\sigma_{v}(t_{m})$ is numbered by $\kappa_{0}$ as well, whereas the left boundary segment of its domain is numbered by $\kappa_{k_{m}}$. 
We number the boundary segment $(t_{m-1},t_{m})$ and the right boundary segment of the domain of the flow tree attached at $\sigma_{v}(t_{m-1})$ by $\kappa_{k_{m}}$ as well. 
The left boundary segment of the flow tree attached at $\sigma_{v}(t_{m-1})$ is then numbered by $\kappa_{k_{m-1}}$, 
which determines the numbering of the segment $(t_{m-2},t_{m-1})$ as well as the right boundary
segment in the flow tree at $\sigma_{v}(t_{m-2})$, etc. We view the end result of this process as
the domain for a loop with flow trees attached with numbering $\kappa$ that decreases, see Figure
\ref{adamsmap}. 

Note next that if $\sigma\colon I^{d}\to\Omega Q$ is a $d$-dimensional cube in general
position with respect to $\bar Q$ (i.e., transverse to the stratified space of the partial puncture of all minimum free partial flow trees) the set of $\sigma_{v}$, $v\in I^{d}$ for which a single partial
flow tree can be attached is at most $(d-1)$-dimensional. Attaching more partial flow trees, the
dimension decreases further, by at least one for each flow tree. We say that the loops in $\sigma$
with flow trees attached which form a $0$-dimensional family are the rigid loops with flow trees in
$\sigma$. If $\sigma$ is a cubical simplex in $\Omega Q$ and if $\mathbf{y}$ is a word of critical points then we let 
\[
\mathcal{T}(\sigma;\mathbf{y})
\]
denote the space of loops with flow trees in $\sigma$, where the critical points at punctures of the flow trees read in order give the word $\mathbf{y}$. The formal dimension of $\mathcal{T}(\sigma;\mathbf{y})$ is then
\[
\dim(\mathcal{T}(\sigma;\mathbf{y}))=|\mathbf{y}|+\dim(\sigma)+(\ell(\mathbf{y})-1),
\]
where $\ell$ is the word length, and for chains transverse to the system of parallel copies the formal dimension equals the actual dimension.

Note that if the set of loops in $\sigma$ with flow trees is transversely cut out then 
by definition construction of the system of parallel copies, see Theorem \ref{thm:mdlicopies},
loops with flow trees corresponding to different
decreasing numberings are canonically diffeomorphic. We
define the map $\phi$ by counting rigid loops with flow trees in cubical simplices $\sigma$:
\begin{equation}\label{eq:deflooptotree}
\phi(\sigma)=\sum_{\dim(\mathcal{T}(\sigma;\mathbf{y}))=0}|\mathcal{T}(\sigma;\mathbf{y})|\mathbf{y}.
\end{equation}

\begin{rem}
	The map $\phi$ is defined only for chains that are suitably transverse to the system of parallel copies $\bar Q(\eta)$. Since the shifting Morse functions do not have any index one critical points a partial flow tree has at most $\dim(M)$ punctures. Consider the natural evaluation map on partial flow trees that takes a flow tree to the location of its positive puncture discussed above. The image of this map for partial flow trees not involving the minimum is a stratified space of codimension two and by construction of parallel copies the corresponding set for partial flow trees defined by distinct boundary numberings lie $\mathcal{O}(\eta^{2})$-close to each other. The map $\phi$ above is defined for chains of loops with evaluation maps that are smooth and transverse to this stratified subset. It is straightforward to see that any chain of loops is homotopic to such a transverse chain and that any two transverse chains that are homologous are homologous through a transverse chain of loops. This means that we can replace all chains by chains that represent these transversality conditions. 
\end{rem}

In order to connect this to the path loop fibration we consider a similar map
\[
    \hat{\phi} \colon C_{-\ast}(\mathrm{P} Q)\to \Omega CM_{-\ast}(Q)\otimes^{\t} CM_{-\ast}(Q),
\]
where $\t$ denotes the canonical twisting cochain of the cobar construction and $\mathrm{P}Q$ the based path space. This map can be described geometrically as follows. The chain complex $C_{-\ast}(\mathrm{P} Q)\to \Omega CM_{-\ast}(Q)\otimes^{\t} CM_{-\ast}(Q)$ can be thought of as generated by words of critical points in which the last letter is distinguished and may be the minimum $x$, in other words, the words are either minimum free, or the last letter (and only the last) is the minimum. The differential counts rigid flow trees as usual and also here only the last letter may be the minimum. To define the map $\hat{\phi}$ we consider chains of paths. As above we attach co-unit free partial flow trees to paths in such a chain at interior points and also attach a partial flow tree with last puncture distinguished at the endpoint of the path. Also here, only the distinguished (i.e., the last puncture in the tree at the end point of the path) may be the minimum.
The map $\hat\phi$ then counts rigid paths with flow trees attached as described.

\begin{lem}
    The maps $\phi$ and $\hat{\phi}$ are chain maps.
\end{lem}

\begin{proof}
	To see this note that the contributions to the chain map equation in both cases correspond to the boundary of an oriented compact 1-dimensional manifold. 
\end{proof}

\begin{rem}
	The codimension one boundary of $\mathcal{T}(\sigma;\mathbf{y})$ corresponds either to the loop or path moving to the boundary of $\sigma$ or to the breaking of a flow tree at a critical point. 
	Instances when two trees are attached at the same point are naturally interior points of the
    moduli space where the disks with slits join to a new disk with a slit of width equal to the
    sum of the widths. See Figure \ref{interiorpoints}.
\end{rem}

\begin{figure}[h!]
    \centering
    \begin{tikzpicture}[scale=1]
    \tikzset{->-/.style={decoration={ markings,
                mark=at position #1 with {\arrow{>}}},postaction={decorate}}}

        \begin{scope}[scale=0.4]
    \draw [black, thick=1.5] (-5,3) to (0,3);
    \draw [black, thick=1.5] (-5,5) to (0,5);
    \draw [black, thick=1.5] (0,4) to (-1.5,4);
 
    \draw [black, thick=1.5] (-5,6) to (0,6);
    \draw [black, thick=1.5] (-5,8) to (0,8);
     \draw [black, thick=1.5] (0,7) to (-2,7);

     \draw[black,dashed] (-5,2) to (-5,3);
    \draw[black,dashed] (-5,5) to (-5,6);
   \draw[black,dashed] (-5,8) to (-5,9);
   \end{scope} 

    \draw[-stealth,decorate,decoration={snake,amplitude=3pt,pre length=2pt,post length=3pt}]
    (0.4,2.2) -- ++(0.5,0);

        \begin{scope}[scale=0.4, xshift=8cm]
    \draw [black, thick=1.5] (-5,3.5) to (0,3.5);
    \draw [black, thick=1.5] (0,4.5) to (-1.5,4.5);
 
    \draw [black, thick=1.5] (-5,5.5) to (0,5.5);
    \draw [black, thick=1.5] (-5,7.5) to (0,7.5);
     \draw [black, thick=1.5] (0,6.5) to (-2,6.5);

     \draw[black,dashed] (-5,2) to (-5,3.5);
   \draw[black,dashed] (-5,7.5) to (-5,9);
   \end{scope} 

    \draw[-stealth,decorate,decoration={snake,amplitude=3pt,pre length=2pt,post length=3pt}]
    (3.6,2.2) -- ++(0.5,0);

        \begin{scope}[scale=0.4, xshift=16cm]
    \draw [black, thick=1.5] (-5,3.5) to (0,3.5);
    \draw [black, thick=1.5] (0,4.5) to (-1.5,4.5);
 
    \draw [black, thick=1.5] (-3,5.5) to (0,5.5);
    \draw [black, thick=1.5] (-5,7.5) to (0,7.5);
     \draw [black, thick=1.5] (0,6.5) to (-2,6.5);

     \draw[black,dashed] (-5,2) to (-5,3.5);
   \draw[black,dashed] (-5,7.5) to (-5,9);
   \end{scope} 

        \begin{scope} [yshift=-1cm, xshift=-1.4cm]  

   \draw[black, dashed] (0,0) to[in=90, out=90] (1,0);  
   
   \draw[black, thick=1] (0.3,0.27) to (0.2, -0.3); 
   \draw[black, thick=1] (0.2, -0.3) to (0, -0.5); 
   \draw[black, thick=1] (0.2, -0.3) to (0.4, -0.5); 

   \draw[black, thick=1] (0.7,0.27) to (0.8, -0.3); 
   \draw[black, thick=1] (0.8, -0.3) to (1, -0.5); 
   \draw[black, thick=1] (0.8, -0.3) to (0.6, -0.5); 

    \end{scope} 

 \draw[-stealth,decorate,decoration={snake,amplitude=3pt,pre length=2pt,post length=3pt}]
    (0.4,-1.2) -- ++(0.5,0);

 \begin{scope} [yshift=-1cm, xshift=1.8cm]  

   \draw[black, dashed] (0,0) to[in=90, out=90] (1,0);  
   
   \draw[black, thick=1] (0.3,0.27) to (0.2, -0.3); 
   \draw[black, thick=1] (0.2, -0.3) to (0, -0.5); 
   \draw[black, thick=1] (0.2, -0.3) to (0.4, -0.5); 

   \draw[black, thick=1] (0.3,0.27) to (0.8, -0.3); 
   \draw[black, thick=1] (0.8, -0.3) to (1.1, -0.4); 
   \draw[black, thick=1] (0.8, -0.3) to (0.7, -0.5); 

    \end{scope} 

 \draw[-stealth,decorate,decoration={snake,amplitude=3pt,pre length=2pt,post length=3pt}]
    (3.6,-1.2) -- ++(0.5,0);

 \begin{scope} [yshift=-1cm, xshift=5cm]  

   \draw[black, dashed] (0,0) to[in=90, out=90] (1,0);  
 
   \draw[black, thick=1] (0.3, 0.27) to (0.3, -0.03);
   \draw[black, thick=1] (0.3,-0.03) to (0.2, -0.7); 
   \draw[black, thick=1] (0.2, -0.7) to (0, -0.9); 
   \draw[black, thick=1] (0.2, -0.7) to (0.4, -0.9); 

   \draw[black, thick=1] (0.3,-0.03) to (0.8, -0.7); 
   \draw[black, thick=1] (0.8, -0.7) to (1.1, -0.8); 
   \draw[black, thick=1] (0.8, -0.7) to (0.7, -0.9);

 \end{scope}

    \end{tikzpicture}
    \caption{Flow trees attached at the same point are interior points: in the source (top),
in the target (bottom). }\label{interiorpoints} \end{figure}

With this established we can now prove Adams' result:
\begin{thm}\label{thm:Adams}
	The flow tree map $\phi\colon C_{-\ast}(\Omega Q)\to \Omega CM_{-\ast}(Q)$ is a quasi-isomorphism.
\end{thm}

\begin{proof} The first observation is that the chain complex $\Omega CM_{-\ast}(Q)\otimes^{\t}
    CM_{-\ast}(Q)$ is acyclic. To see that, note that for each critical point $y$ there is a unique
    flow tree with positive puncture $y$ and two negative punctures, one at the co-unit $x$ and one at $y$, see Lemma \ref{l:DanR}.
    
    Add a constant to the Morse function $f$ used to build the parallel copies so that the minimum $x$ lies at level $0$ and all other critical points at positive levels. We then filter by action, more precisely we associate to a word $y_{1}\dots y_{m}$ of critical points the action $\sum_{j}f(y_{j})$. Then by definition of a flow tree the differential does not increase action. Since all flow trees except those involving the co-unit decrease action we find that the differential on the associated graded complex acts only on the last (distinguished) letter in the word and it acts there as $y\mapsto yx$ if $y\ne x$ and $x\mapsto 0$. Since this is an isomorphism from words not ending with the co-unit $x$ to those which does end with
    $x$, the desired acyclicity follows. Clearly, $C_{-\ast}(\mathrm{P}Q)$ is also acyclic. 
	
    Consider next the stratification of $Q$ induced by the stable manifolds of the Morse function $f$ and the corresponding filtration on $C_{-\ast}(\mathrm{P}Q)$ induced by evaluation at the endpoint. The corresponding filtration on $\Omega CM_{-\ast}(Q)\otimes^{\t} CM_{-\ast}(Q)$ is filtration by degree of the distinguished (the last) critical point and the map $\hat{\phi}$ respects these filtrations. The corresponding $E_{1}$-terms with induced maps are
	\[
        C_{-\ast}(Q;H_{-\ast}(\Omega Q))  \to H_{\ast}(\Omega CM_{-\ast}(Q))\otimes CM_{-\ast}(Q),	\] 
	and $E_{2}$-terms
	\[
	H_{-\ast}(Q;H_{-\ast}(\Omega Q))  \to H_{-\ast}(Q;H_{\ast}(\Omega CM_{-\ast}(Q))).
	\] 
    Zeeman's comparison theorem \cite[Sec. 3.3]{McCleary} then establishes the result.
\end{proof}

\begin{rem}\label{r:dim4}
	If $\dim(Q)=4$ then the Morse function may have critical points of index one. In this case we
    use stabilization as follows. Multiply $Q$ by $\R^{N}$ for any $N \geq 2$, and consider the function $F(q,x)=f(q)+x^{2}$, then $F$ has the same critical points as $f$ and $-\nabla F$ is inward point at infinity.  In $Q\times \R^{N}$ there is room to cancel 1-handles and the above applies. In this case we define $CM_{-\ast}(Q)$ to be $CM_{-\ast}(Q\times \R^{N})$  (which is a 1-reduced version of the original complex). Noting that $C_{-\ast}(Q\times \R^{N})$ and $C_{-\ast}(Q)$ are canonically isomorphic the result follows also in this case.   
\end{rem}

We next show that $CE^{\ast}_{\parallel}(\Lambda)$ and $CE^{\ast}(\Lambda)$ are isomorphic if $\Lambda$ is simply connected.
This is a more or less direct consequence of the description of rigid disks on a Legendrian with parallel copies in Lemma \ref{l:diskquantumtreecorr} and the isomorphism in Theorem \ref{thm:Adams}. Since components in $\Lambda_{-}$ are not affected by this choice of $CE^{\ast}_{\parallel}$ versus $CE^{\ast}$, we disregard them and assume that $\Lambda=\Lambda_{+}$ in what follows.

Recall the definition of $CE^{\ast}(\Lambda)$ given in Remark \ref{r:topologicalversion}, generated by
chains $C_{-\ast}((\Omega_{p}\Lambda)^{\times (i+1)})$ in the product of the based loop space of $\Lambda$ with factors separated by Reeb chords. 
Here the differential of a chain
is just the usual differential of the chain whereas the differential of a Reeb chord is a sum over
all moduli spaces of disks with one positive puncture at the chord and any number of negative
punctures. Such a moduli space carries a fundamental chain and the contribution to the differential is alternating word of chains in the based loop space corresponding to the boundary arcs of the disk carried by the fundamental chain and Reeb chords at the negative puncture:
\[
d c_{0} = \sum_{\mathbf{c}'}[\mathcal{M}^{\sy}(\mathbf{c})],
\]  
where $\mathbf{c}=c_{0}\mathbf{c}'$ and $\mathbf{c}'$ is a word of Reeb chords, here we use the the diagonal in the product of loop spaces, see Remark \ref{r:topologicalversion}.

We next consider a system of parallel copies $\bar\Lambda(\eta)=\{\Lambda_{j}(\eta)\}_{j=0}^{\infty}$ defined by a system of positive Morse functions, see Section \ref{ssec:paralleldetails}, where $\Lambda_{0}=\Lambda$. Recall that the generators of the algebra $CE^{\ast}_{\parallel}(\Lambda)$ are the Reeb chords connecting $\Lambda_{0}$ to $\Lambda_{1}$, and that these are long, corresponding to Reeb chords of $\Lambda$ and short, corresponding to critical points of $f_{1}$, except for the minimum. The differential counts rigid disks with one positive puncture in $\mathcal{M}^{\sy}_{\parallel}(\mathbf{b};\kappa)$ where $\kappa$ is a decreasing boundary numbering, $\mathbf{b}=b_{0}\mathbf{b}'$. 

%
%

We next consider the map 
\[
\phi\colon CE^{\ast}(\Lambda)\to CE^{\ast}_{\parallel}(\Lambda),
\]
which takes every Reeb chord to itself and takes a chain $\sigma$ in the based loop space to $\phi(\sigma)$, where $\phi$ is as in \eqref{eq:deflooptotree} and where we identify the critical points of $f_{1}$ with the corresponding Reeb chords connecting $\Lambda_{0}$ to $\Lambda_{1}$.

\begin{thm}\label{t:parallel=loops}
	The map $\phi$ is a DG-algebra map and	
	if $\Lambda^{+}$ is simply connected then $\phi$ is a quasi-isomorphism. 		
\end{thm}

\begin{proof}
	The fact that $\phi$ is a chain map follows from Lemma \ref{l:diskquantumtreecorr}.
	Filter the algebras by action of Reeb chords in the left hand side and actions of long Reeb chords in the right hand side. The map respects this filtration. The $E_{2}$-pages are obtained by acting by the differential on the chains on the based loop space only in the left hand side and on Morse chords only in the right hand side. The result is words of Reeb chords separated by homology classes in the based loop space and by homology classes in the (reduced) bar construction on the Morse coalgebra on the left and right hand sides, respectively. On these $E_{2}$-pages the map $\phi$ induces an isomorphism by Theorem \ref{thm:Adams}. Since the sum of actions of the Reeb chords at the negative end is bounded by that at the positive end, the spectral sequences converges. The theorem follows.  
\end{proof}

\begin{rem}\label{r:dim4iso}
	Note that the isomorphism in Theorem \ref{t:parallel=loops} is compatible with the stabilization of Remark \ref{r:dim4}. To see this we multiply the ambient contact manifold $Y$ with contact form $\alpha$ by $T^{\ast}\R^{N}$ and consider $\Lambda\times\R^{N}\subset Y\times T^{\ast}\R^{N}$ with contact form $\theta=(\alpha - y\cdot dx)$. The Reeb chord of $\Lambda\times\R^{N}$ then come in $\R^{N}$-families, one for each Reeb chord of $\Lambda$. Consider the contact form $e^{x^{2}}\theta$ and note that with respect to this contact form the Reeb chords of $\Lambda\times\R^{N}$ are in natural one to one correspondence with those of $\Lambda$ and there is a canonical isomorphism between $CE^{\ast}(\Lambda)$ and $CE^{\ast}(\Lambda\times\R^{N})$. In fact the disks in the differential are canonically identified. It follows in particular, that Theorem \ref{t:parallel=loops} holds also if $\dim(Q)\le 4$.   
\end{rem}

\section{Lagrangian (co)algebra}
As before, $X$ is a Liouville manifold with $c_1(X)=0$ and $L$ is an
exact relatively spin Lagrangian in $X$ with vanishing Maslov class and ideal boundary given by the Legendrian $\Lambda$. 

We will associate several chain level structures to $L$. To begin with, let us first assume that $L$
is an embedded Lagrangian. Since $L$ has boundary, in classical topology, one can consider either $C^*(L)$ or $C^*(L,
\partial L)$. In our case, these two choices are reflected in the choices of $+$ or $-$ decorations
on $L$, respectively. More generally, let $L^{v}$, $v\in \Gamma$ be the (irreducible) components of $L$. As
with the Legendrian submanifolds in Section \ref{ssec:Leginv}, we assume these components of
$L$ are decorated by signs and we write $L=L^{+}\cup L^{-}$ for the corresponding decomposition. 
Let $F\colon L\to\R$ be a Morse function with prescribed behavior at infinity (depending on
the $+$ or $-$ decoration) as explained in
Section \ref{sec:parallel}. We use this to construct a system of parallel copies $\bar
L=\{L_{j}\}_{j=1}^{\infty}$, as in Section \ref{sec:parallel}, shifted at
infinity along the Reeb flow either in the positive or negative direction on $L^{+}$ and $L^{-}$, respectively. 

Now, using the parallel copies, $\{L_j\}_{j=1}^\infty$, we define a graded quiver $\mathcal{Q}_L$ as
follows. The parallel copies $\{L_j \}_{j=1}^\infty$ give rise to following sets, fix $i_1< i_2$
positive integers, and $v,w \in \Gamma$ with $v \neq w$: \begin{itemize} 
    \item Intersection points $L^v_{i_1} \cap L^v_{i_2}$ in bijection with the union of critical points of
        $F|_{L_v}$. (These critical points may depend on the $+$ or $-$ decoration on $L_v$, for example, one can turn a $-$ decorated component $L^v$ into $+$ decorated, by introducing critical points corresponding to the topology of $\partial L_v$, see Figure
\ref{intersectiongen}).   
    \item Intersection points $L^v_{i_1} \cap L^w_{i_2}$ in bijection with $L^v \cap L^w$.  
    \end{itemize}

\begin{figure}[h!]
    \centering
    \begin{tikzpicture}[scale=1]
    \tikzset{->-/.style={decoration={ markings,
                mark=at position #1 with {\arrow{>}}},postaction={decorate}}}
    \draw [black, thick=1.5] (1.5,3) to[in=330,out=270] (0,0);
    \draw [black, thick=1.5] (0,0) to[in=270,out=150] (-1,3);
    \draw [black, thick=1.5] (0,0) to[in=270,out=210] (-1.5,3);
    \draw [black, thick=1.5] (0,0) to[in=270,out=30] (1,3);
   
    \draw [black, thick=1.5] (-3.5,2) to[in=330,out=270] (-5,0);
    \draw [black, thick=1.5] (-3.5,3) to[in=90,out=270] (-4,2);
    \draw [black, thick=1.5] (-4,3) to[in=90,out=270] (-3.5,2);
     \draw [black, thick=1.5] (-5,0) to[in=270,out=30] (-4,2);
   
    \draw [black, thick=1.5] (-6.5,3) to[in=90,out=270] (-6,2);
    \draw [black, thick=1.5] (-6,3) to[in=90,out=270] (-6.5,2);
    \draw [black, thick=1.5] (-5,0) to[in=270,out=150] (-6,2);
    \draw [black, thick=1.5] (-5,0) to[in=270,out=210] (-6.5,2);

    \node at (1.5,3.2) {\footnotesize{$i_2$}}  ;
    \node at (1,3.2) {\footnotesize{$i_1$}}  ;
    \node at (-1.5,3.2) {\footnotesize{$i_1$}}  ;
     \node at (-1,3.2) {\footnotesize{$i_2$}}  ;

    \node at (-3.5,3.2) {\footnotesize{$i_1$}}  ;
    \node at (-4,3.2) {\footnotesize{$i_2$}}  ;
    \node at (-6.5,3.2) {\footnotesize{$i_2$}}  ;
     \node at (-6,3.2) {\footnotesize{$i_1$}}  ;

    \node at (0, 2) {\footnotesize{$L^{-}$}}; 
    \node at (-5, 2) {\footnotesize{$L^{+}$}};

    \end{tikzpicture}
    \caption{Difference between $+$ and $-$ generators for $i_1 < i_2$. Both the left and the right hand side depicts shifts corresponding to Morse functions with a maximum. One of the intersection points in $L_{+}$ is the minimum and corresponds to the unit for the Floer cohomology product.}
    \label{intersectiongen}
\end{figure}

Furthermore, by the construction in Section \ref{sec:parallel}, there are canonical bijections between the above sets associated
to any pairs $(i_1, i_2)$ and $(i'_1, i'_2)$ with $i_1 < i_2$ and $i'_1 < i'_2$. So, fix a pair $(i_1,i_2)$ such that $i_1 < i_2$
and define a graded quiver $\mathcal{Q}_L$ with vertex set $\Gamma$ and with an arrow connecting $v$
to $w$ (possibly equal to $v$) for each element of the above sets. Let $\mathcal{I}$ denote
the set of arrows. 

Alternatively, one can describe the generators as the set of intersection points in $L_{0}\cap L_{1}$, between the original $L$ and the first shifted copy.

Let $CF^*(L)$ be the graded $\k$-bimodule generated by $\mathcal{I}$. Thus, there is one generator $x_{vw}$ in degree $|x_{vw}|$ for each arrow in $\mathcal{Q}_L$ from $v$ to $w$.  We endow $CF^{\ast}(L)$ with the structure of an $A_{\infty}$-algebra. Let $x_{0}$ be an intersection point generator and let $\mathbf{x}'=x_{i}\dots x_{1}$ be a word of intersection points. Consider the disk $D_{i+1}$ with $i+1$ boundary punctures and with a decreasing numbering of its boundary arcs $\kappa$. Let $\mathbf{x}=x_{0}x_{i}\dots x_{1}$.  Consider the moduli space $\mathcal{M}^{\fl}(\mathbf{x};\kappa)$, see Appendix \ref{sec:mdlispaces} for notation.  Define the operations $\m_{i}$ by
\[ 
\m_{i}(\mathbf{x}') = \sum_{\ell(\mathbf{x}')=i, |x_{0}|=|\mathbf{x}'|+(2-i)} |\mathcal{M}^{\fl}(\mathbf{x};\kappa)|x_{0}
\] 
where $\ell(\mathbf{y})$ denotes the word length of $\mathbf{y}$ and where $|\mathcal{M}^{\fl}(\mathbf{x};\kappa)|$ denotes the algebraic number of points in the oriented 0-dimensional manifold.

\begin{figure}[h!]
    \centering
    \begin{tikzpicture}[scale=1]
    \tikzset{->-/.style={decoration={ markings,
                mark=at position #1 with {\arrow{>}}},postaction={decorate}}}

        \draw[thick, red] (-1,0) arc (0:90:1);
        \draw[thick, blue] (-2,1) arc (90:180:1);
        \draw[thick, blue] (-3,0) arc (180:240:1);
        \draw[thick, red] (-2.5,{-sin(60)} ) arc (240:300:1);
        \draw[thick, blue] (-1.5,{-sin(60)} ) arc (300:360:1);

           \draw[thick, fill=black] (-2,1) circle(.05);
        \draw[thick, fill=black] (-1.5, {-sin(60)} ) circle(.05); 
        \draw[thick, fill=black] (-2.5, {-sin(60)} ) circle(.05);

        \draw[thick, fill=black] (-1,0) circle(.05);
        \draw[thick, fill=black] (-3,0) circle(.05);

            \node at (-0.7,0.6) {$L_{\kappa_1}$}; 
            \node at (-3.2,0.6) {$L_{\kappa_0}$}; 
            \node at (-3.3,0.0) {$x_{0}$};
            \node at (-2.6,-1.1) {$x_{1}$};
            \node at (-1.2,-1.0) {$x_{2}$};
            \node at (-0.7,0.0) {$x_{3}$};
            \node at (-2.1,1.25) {$x_{4}$};
             \node at (-0.8,-0.7) {$L_{\kappa_2}$}; 
    \node at (-3.2,-0.6) {$L_{\kappa_4}$}; 
    \node at (-2,-1.3) {$L_{\kappa_3}$}; 

\end{tikzpicture}
\caption{An example of a disk with labellings. The blue labelled Lagrangians are perturbed with $+$ perturbations and red labelled Lagrangrians are perturbed with $-$ perturbation.}

\end{figure}

\begin{lem}
The operations $\m_{i}$ satisfy the $A_{\infty}$-algebra relations and are independent of the decreasing boundary labelling $\kappa$.
\end{lem}

\begin{proof}
This follows by the usual argument: after noting that the decreasing boundary numbering ensures that there is no boundary bubbling, one observes that the terms in the $A_{\infty}$-algebra relations count the ends of a 1-dimensional oriented compact manifold by Theorems \ref{thm:mdlitv} and \ref{thm:mdlicmpct}. Note also that the operations compose because of Theorem \ref{thm:mdlicopies}.
\end{proof}

We call $CF^{\ast}(L)$ the Lagrangian Floer cohomology algebra of $L$. Let $u_{v}$ denote the generator corresponding to the minimum on the component $L_{v}\subset L^{+}$. If $L_{v}$ is simply connected then by Lemma \ref{l:DanR} $u_{v}$ is a strict idempotent. 
We write $\k_{-} \oplus CF^*(L)$ for the augmented algebra where we adjoined an idempotent $e_{w}$
for each component $L_w\subset L^{-}$. (On these components the shifting function is decreasing at
infinity and has a maximum rather than a minimum in the compact part.) Note that this is a connected
algebra over $\k$. 

\begin{rem} The two different choices of perturbations at infinity corresponding to $+$ and $-$ are
    the two extremal constructions where one pushes the copies either always in positive direction or
    always in the negative direction. One can also choose perturbations at infinity to depend on the
    topology of the manifold at infinity (see, for example, Section 4 of
    \cite{abouzaid}). All our constructions should extend meaningfully to this more
    general setting but we have not pursued this direction in this paper.
\end{rem}

We next consider various linear duals of $CF^{\ast}(L)$ and associated algebraic objects. The
simplest case occurs when $CF^{\ast}(L)$ is simply connected. In this case the linear dual
$CF_{\ast}(L)$ is a coalgebra with operations $\Delta_{i}$ dual to $\m_{i}$, and as before we can
adjoin $\k_-$ so that $\k_- \oplus CF_*(L)$ is co-augmented over $\k$. Then, we define the  
\emph{Adams-Floer DG-algebra}
\[ 
\Omega (\k_{-} \oplus CF_*(L)),
\]
by applying the cobar construction.

In the non-simply connected case, we replace this object by the \emph{completed Adams-Floer DG-algebra}: 
\[
(\Bar (\k_{-} \oplus CF^*(L)))^\#.
\]

%


\begin{ex} Let $L$ be the standard Lagrangian $D^n$ filling the standard Legendrian unknot $\Lambda
    \subset S^{2n-1}$. The Floer cohomology can be computed as follows:
\[ 
CF^*(L) =
\begin{cases}
\mathbb{K}x,\, |x|=0, &\text{if $L$ is decorated with $+$,}\\
\mathbb{K}c,\, |c|=n.   & \text{if $L$ is decorated with $-$.}
\end{cases}
\]
Alternatively, if we want compatibility with the inclusion $C^{\ast}(D^{n},\partial D^{n})\to C^{\ast}(D^{n})$ as
\[ 
CF^*(L) =
\begin{cases}
\mathbb{K}c \oplus \mathbb{K} y \oplus \mathbb{K} x,\\ 
|c|=n, |y|=n-1, |x|=0,\,
dy =c &\text{if $L$ is decorated with $+$,}\\
\mathbb{K}c,\, |c|=n.   & \text{if $L$ is decorated with $-$.}
\end{cases}
\]
\end{ex}

In Section \ref{sec:CWnoHam}, we introduce a model for wrapped Floer cohomology without Hamiltonian term and
prove it is quasi-isomorphic to the usual wrapped Floer cohomology. We refer to there for details
and give only a short description here. The chain complex underlying $CW^{\ast}(L)$ is the
following. Let $L=L_0$ and shift $L$ off itself to $L_{1}$ by a Morse function that is positive at
infinity (as in the definition of parallel copies when $L$ is decorated $+$). The generators of $CW^{\ast}(L)$ are then Reeb chords connecting $L_{1}$ to $L_{0}$ and intersection points in $L_{0}\cap L_{1}$.

There is an $A_\infty$-functor, often called the \emph{acceleration functor}, 
\[ CF^*(L) \to CW^*(L). \]
If $L$ is decorated $+$, it can be shown that this functor is unital.

%


\section{Maps relating Legendrian and Lagrangian (co)algebras}\label{Sec:LegLag}
We continue with our usual set-up, where $X$ is a Liouville manifold with $c_1(X)=0$ and $L$ is an
exact Lagrangian in $X$ with vanishing Maslov class and ideal boundary given by the Legendrian
$\Lambda$. Let $\Gamma$ be the set of embedded components of $L$ subdivided into
$\Gamma^{+}\cup\Gamma^{-}$. Let $\Theta$ be the set of components of $\Lambda$ with induced
subdivision $\Theta^{+}\cup\Theta^{-}$. 

In this section we will define twisting cochains and associated dg-algebra maps relating the parallel copies version $CE^{\ast}_{\parallel}(\Lambda)$ and the Floer cohomology $CF^{\ast}(L)$. Since $L$ is an exact filling, we have an augmentation $\epsilon_L \colon CE^{\ast}_{\parallel}(\Lambda) \to \k$. As in Section \ref{ssec:Leginv} we use this augmentation throughout to change coordinates in such a way that $\Delta_{0}=0$. 

As explained in Section \ref{ssec:parallelcopies}, the definition of $CE^{\ast}_{\parallel}$ differs depending on whether or not the components of $\Lambda$ in $\Theta^{+}$ are simply connected or not. We will start in the simply connected case and turn to the non-simply connected case, using the definitions in Section \ref{ssec:parallelwithpi1} later. 

Assume thus that all components of $\Lambda$ in $\Theta^{+}$ are simply connected. 
As usual let $\mathbf{k}_{-}\oplus LC_{\ast}^{\parallel}(\Lambda)$ denote the coalgebra corresponding to $CE^{\ast}_{\parallel}(\Lambda)$ augmented by the Lagrangian filling, with co-units $e_{v}$ adjoined to all components $\Lambda_{v}$ in $\Theta^{-}$. As $\Lambda^{+}$ has simply connected components, by Lemma \ref{l:DanR}, this is a co-unital coalgebra with co-unit 
\[ 
\sum_{v\in\Theta^{+}} x_{v} + \sum_{v\in\Theta^{-}} e_{v}.
\]
Let $\eta\colon \k\to \k_{-}\oplus LC_{\ast}^{\parallel}(\Lambda)$ denote the co-augmentation
\begin{equation}\label{eq:augfortwist} 
\eta(e_{v})=
\begin{cases}
x_{v} &\text{if $v\in\Theta^{+}$},\\
e_{v} &\text{if $v\in\Theta^{-}$},
\end{cases}
\end{equation}
see \eqref{eq:coaugparallel},
so that $CE^{\ast}_{\parallel} = \Omega (\k_{-}\oplus LC_{\ast}^{\parallel}) $. Note that this
means that $x_v$ are traded for $e_v$ for $v \in \Theta^+$. 


Consider the Floer cohomology $A_{\infty}$-algebra $CF^{\ast}(L)$. In case all components of $L^{+}$ are simply connected there exists a strict
idempotent $u_{v}\in CF^{\ast}(L)$ for each $v\in\Gamma^{+}$ corresponding to the minimum of the
shifting Morse function and that we make $CF^{\ast}(L)$ unital by adding an idempotent $e_{w}$ for each $w\in\Gamma^{-}$. We write the strictly unital algebra $\k_{-}\oplus CF^{\ast}(L)$. Let $\epsilon\colon \k_{-}\oplus CF^{\ast}(L)\to \k$ be the augmentation that maps $u_{v}$ to $e_{v}$ for $v\in\Gamma^{+}$ and $e_{w}$ to $e_{w}$ for $w\in\Gamma^{-}$. Consider the dual of the bar construction:
\begin{equation}\label{eq:simplyconnectedualbar} 
\A=(B(\k_{-}\oplus CF^{\ast}(L)))^{\#},
\end{equation}
or in other words the completed Adams-Floer DG-algebra. In what follows we will represent $\A$ as a quotient in way that generalizes to the non-simply connected case in analogy with the construction in Section \ref{ssec:parallelwithpi1}. In the non-simply connected case we introduce strict idempotents by hand as follows.

Consider adding extra idempotents $e_{v}$, $v\in\Gamma^{+}$ to $\k_{-}\oplus CF^{\ast}(L)$. This gives $\k\oplus CF^{\ast}(L)$ and we equip it with the trivial augmentation $\epsilon'$ which is the projection to $\k$. Let
\[ 
\A'=(B(\k\oplus CF^{\ast}(L)))^{\#},
\] 
and let $\mathscr{I}$ denote the subalgebra of $\A'$ given by the
space of functionals which vanish on monomials not containing $u_v$ for some $v\in\Gamma^{+}$. Let $\rho\colon \A'\to\A$ denote the restriction map induced by the inclusion $\overline{CF}^{\ast}(L)\to CF^{\ast}(L)$. 

\begin{lem}\label{l:quotientalalgebra2}
The subalgebra $\mathscr{I}$ is closed under the differential. In the simply connected case,	
the map $\rho$ is a chain map with kernel $\mathscr{I}$ and consequently, $\A$ is quasi-isomorphic to $\A'/\mathscr{I}$.  
\end{lem}

\begin{proof}
To see that $\mathscr{I}$ is closed under the differential we note that there is no negative gradient flowline that starts at the minimum, therefore any negative gradient tree that starts at $u_v$ must also have a negative puncture at $u_{v}$. Lemma \ref{l:localdisktreecorr} then implies that the differential takes $\mathscr{I}$ to $\mathscr{I}$.

In the simply connected case, it is clear that $\mathscr{I}$ is the kernel of $\rho$. 
\end{proof}


In the general case we define $\A=\A'/\I$. Lemma \ref{l:quotientalalgebra2} shows that in the simply connected case this definition agrees with the alternate definition of $\A$ given in \eqref{eq:simplyconnectedualbar}.  

\begin{rem}
The somewhat artificial construction of $\A$ as $\A=\A'/\I$ is used to adapt the bar construction to a not necessarily strictly unital algebra. 
\end{rem}

We next define a map $\t'$ on generators of $CE^{\ast}_{\parallel}(\Lambda)$ which then gives a map
\[ 
\t'\colon LC_*^{\parallel}(\Lambda) \to \A', 
\]
in the simply connected case and in that case it will induce a twisting co-chain
\[ 
\t\colon \k_{-}\oplus LC_{\ast}^{\parallel}(\Lambda)\to\A.
\]
The map $\t'$ is defined by the following curve count for generators of $LC_{\ast}^{\parallel}(\Lambda)$.
Fix systems of parallel copies $\bar L$ of $L$. Recall that the components labelled with a $+$ are shifted by a positive Morse function and the components labelled with a $-$ sign are shifted by a negative Morse function.

Let $c$ be a Reeb chord of $\bar{\Lambda}$ and let $\mathbf{x}_{0}=x_{0;1}\dots x_{0;j}$, $j>0$ be a non-empty word of intersection points of $\bar{L}$. Let 
\[ 
\mathbf{c}=c\mathbf{x}_{0}
\]
and define
\begin{equation}\label{eq:deft'}
\t'(c)=\sum_{|\mathbf{x}_{0}|=|c|+(1-j)}
|\mathcal{M}^{\fl}(\mathbf{c})|\mathbf{x}_{0},
\end{equation}
where we interpret $\mathbf{x}_{0}$ as en element in $\A'$.

\begin{figure}[h!]
    \centering
    \begin{tikzpicture}[scale=1]
    \tikzset{->-/.style={decoration={ markings,
                mark=at position #1 with {\arrow{>}}},postaction={decorate}}}

        \draw[thick, red] (-1,0) arc (0:90:1);
        \draw[thick, blue] (-2,1) arc (90:180:1);
        \draw[thick, blue] (-3,0) arc (180:240:1);
        \draw[thick, red] (-2.5,{-sin(60)} ) arc (240:300:1);
        \draw[thick, blue] (-1.5,{-sin(60)} ) arc (300:360:1);

           \draw[thick, fill=black] (-2,1) circle(.05);
        \draw[thick, fill=black] (-1.5, {-sin(60)} ) circle(.05); 
        \draw[thick, fill=black] (-2.5, {-sin(60)} ) circle(.05);

        \draw[thick, fill=black] (-1,0) circle(.05);
        \draw[thick, fill=black] (-3,0) circle(.05);

            \node at (-0.7,0.6) {$L_{\kappa_0}$}; 
            \node at (-3.2,0.6) {$L_{\kappa_4}$}; 
             \node at (-0.8,-0.7) {$L_{\kappa_1}$}; 
    \node at (-3.2,-0.6) {$L_{\kappa_3}$}; 
    \node at (-2,-1.3) {$L_{\kappa_2}$}; 

            \node at (-2,1.3) {$c$};
            \node at (-0.5,0) {$x_{0;4}$};
            \node at (-3.4,0) {$x_{0,1}$};
            \node at (-2.7,-1.2) {$x_{0,2}$};
            \node at (-1.3,-1.2) {$x_{0,3}$};

\end{tikzpicture}
\caption{An example of a disk contributing to $\t'$. The blue labelled Lagrangians are perturbed with $+$ perturbations and red labelled Lagrangrians are perturbed with $-$ perturbation.}

\end{figure}

\begin{rem}
In the non-simply connected case we use the same formula to define $\t'(c)$ and note that the the sum in the definition may be infinite.
\end{rem}

We have the following.
\begin{prop}\label{twch}
If $v\in\Theta^{+}$ is such that $\Lambda_{v}$ is a boundary component of $L_{w}$ for $w\in\Gamma^{+}$ then
\[ 
\t'(x_{v})= u_{w}',
\] 
where $u_{w}'$ is the dual of the minimum $u_{w}$. Furthermore,
$\t'$ satisfies the equation of a twisting cochain. 
\end{prop} 

\begin{figure}[h!]
    \centering
    \begin{tikzpicture}[scale=1]
    \tikzset{->-/.style={decoration={ markings,
                mark=at position #1 with {\arrow{>}}},postaction={decorate}}}

 \begin{scope} 
     \clip (-6.25,2.75) to[in=270,out=110] (-6.5, 3.5) to (-6, 3.5) to[in=70,out=270] (-6.25,2.75);

\clip[preaction={draw,fill=gray!30}] (-7,4) rectangle (-3,2);
                             \end{scope}

    \draw [black, thick=1.5] (-3.5,2) to[in=330,out=270] (-5,0);
    \draw [black, thick=1.5] (-3.5,3.5) to[in=90,out=270] (-4,2);
    \draw [black, thick=1.5] (-4,3.5) to[in=90,out=270] (-3.5,2);
     \draw [black, thick=1.5] (-5,0) to[in=270,out=30] (-4,2);
   
    \draw [black, thick=1.5] (-6.5,3.5) to[in=90,out=270] (-6,2);
    \draw [black, thick=1.5] (-6,3.5) to[in=90,out=270] (-6.5,2);
    \draw [black, thick=1.5] (-5,0) to[in=270,out=150] (-6,2);
    \draw [black, thick=1.5] (-5,0) to[in=270,out=210] (-6.5,2);

    \node at (-6.25,3.7) {\footnotesize{$x_v$}}  ;
     \node at (-6.45,2.75) {\footnotesize{$u_w$}}  ;

    \end{tikzpicture}
    \caption{Minimum $x_v$ is sent to the minimum $u_w$} 
    \label{mintomin}
\end{figure}

\begin{proof} 
For the first property we need to understand rigid holomorphic disks with positive puncture at the
    Reeb chord $x_{v}$. By small action such a holomorphic disk must lie in a neighborhood of
    $L_{w}$ and is hence given by Morse flow trees. There is only one flow line emanating from the
    minimum in $\Lambda_{v}$ and the flow line generically ends at the minimum of the shifting of
    $L_{w}$, see Figure \ref{mintomin}. The first equation follows. 	
	
To see the twisting co-chain equation,	
we need to check that 
\[ 
\m_1 \circ \mathfrak{t}' - \mathfrak{t}' \circ \Delta_1 + \sum_{d \geq 2} (-1)^d \m_2^{(d)}
\circ {\mathfrak{t}'}^{\otimes_\k d} \circ \Delta_d = 0 
\]
where $\m_2^{(2)} := \m_2$, and $\m_2^{(i)} := \m_2\circ (\mathrm{Id}_\A \otimes_\k \m_2^{(d-1)})$. 
To this end, we consider the boundary of the $1$-dimensional moduli space $\mathcal{M}^{\fl}(c_{0}\mathbf{x})$. By Theorem \ref{thm:mdlicmpct} this corresponds to two level curves which by Theorems \ref{thm:mdlitv} and \ref{thm:mdlicopies} form the boundary of an oriented compact 1-manifold. 
\end{proof}

\begin{figure}[h!]
    \centering
    \begin{tikzpicture}[scale=1]
    \tikzset{->-/.style={decoration={ markings,
                mark=at position #1 with {\arrow{>}}},postaction={decorate}}}

        \draw[thick, red] (-1,0) arc (0:30:1);
        \draw[thick, blue] (-1.15,{sin(30)}) arc (30:150:1);
        \draw[thick, blue] (-2.87,{sin(30}) arc (150:270:1);
        \draw[thick, red] (-2,-1 ) arc (270:360:1);
  
        \draw[thick, red] (0.64,{sin(30)}) arc (-30:90:1);
        \draw[thick, blue] (-0.23,2) arc (90:210:1);
        \draw[thick, red] (-1.1,{sin(30}) arc (210:330:1);

        \draw[thick, blue] (2.55,0) arc (0:30:1);
        \draw[thick, red] (2.42,{sin(30)}) arc (30:150:1);
        \draw[thick, red] (0.68,{sin(30}) arc (150:270:1);
        \draw[thick, blue] (1.55,-1 ) arc (270:360:1);

        \draw[thick, fill=black] (-1.15, {sin(30)}) circle(.05);
        \draw[thick, fill=black] (-2.87, {sin(30)} ) circle(.05); 
        \draw[thick, fill=black] (-2, -1 ) circle(.05);

        \draw[thick, fill=black] (-0.23,2 ) circle(.05);

        \draw[thick, fill=black] (2.4, {sin(30)}) circle(.05);
        \draw[thick, fill=black] (0.68, {sin(30)} ) circle(.05); 
        \draw[thick, fill=black] (1.55, -1 ) circle(.05); 
  
            \node at (-1.5, {sin(30)}) {$c_1$};
            \node at (-3.2, {sin(30)}) {$x_{0;1}$};
            \node at (-2.1, -1.3) {$x_{0;2}$};

        \node at (-0.23,2.2 ) {$c$};
            \node at (2.8, {sin(30)}) {$x_{0;4}$};
            \node at (0.98, {sin(30)} ) {$c_2$}; 
            \node at (1.55, -1.3 ) {$x_{0;3}$};

\end{tikzpicture}
\caption{
        a two-story disk which contributes to the term $\m_2\circ \t^{\otimes2} \circ \Delta_2$ applied to $c$.There are similar disks in the compactification of the 1-dimensional moduli space with $2$  replaced by $n$.  
        }

\end{figure}

\begin{figure}[h!]
    \centering
    \begin{tikzpicture}[scale=1]
    \tikzset{->-/.style={decoration={ markings,
                mark=at position #1 with {\arrow{>}}},postaction={decorate}}}

        \draw[thick, blue] (0,0) arc (0:150:1);
            \draw[thick, blue] (-1.87,{sin(30)}) arc (150:210:1);
            \draw[thick, red] (-1.87,{-sin(30)} ) arc (210:360:1);
       
       \draw[thick, red] (0,0) arc (180:330:1);
        \draw[thick, red] (1.87,{sin(30}) arc (30:90:1);
         \draw[thick, blue] (1,1) arc (90:180:1);
            \draw[thick, blue] (1.87,{-sin(30)} ) arc (-30:30:1);

        \draw[thick, fill=black] (0,0) circle(.05);
            \draw[thick, fill=black] (1.87,{sin(30)}) circle(.05);
            \draw[thick, fill=black] (1.87,{-sin(30)}) circle(.05);
            \draw[thick, fill=black] (-1.87,{sin(30)}) circle(.05);
            \draw[thick, fill=black] (-1.87,{-sin(30)}) circle(.05);
       \draw[thick, fill=black] (1,1) circle(.05);

        \node at (0.25,0) {$x'$}; 
        \node at (1,1.2) {$c$};
        \node at (2.3,{sin(30)}) {$x_{0;4}$};
        \node at (2.3,{-sin(30)}) {$x_{0;3}$};
        \node at (-2.3,{sin(30)}) {$x_{0;1}$};
        \node at (-2.3,{-sin(30)}) {$x_{0;2}$};

\end{tikzpicture}
\caption{
        a two-story disk which contributes to the term $\m_1\circ \t$ applied to $c$.
}

\end{figure}

Proposition \ref{twch} shows that $\t'$ maps the submodule generated by $x_{v}$, $v\in\Theta^{+}$ into $\mathscr{I}\subset\A'$. Hence, by letting $\t(e_v)=0$, $\t'$ induces a map
\[ 
\t\colon \k_{-}\oplus LC_{\ast}^{\parallel}(\Lambda)\to \A'/\mathscr{I}=\A.
\]
Note that if $\eta \colon \k \to \k_{-}\oplus LC_{\ast}^{\parallel}(\Lambda) $ is the co-augmentation in \eqref{eq:augfortwist} and $\epsilon\colon\A\to\k$ is the trivial augmentation then $\epsilon\circ\t=\t\circ\eta=0$.
\begin{cor}
The map $\t$ is a twisting co-chain.
\end{cor}

\begin{proof}
Since $\t'$ satisfies the twisting co-chain equation so does $\t$. 	
\end{proof}

This twisting cochain is always defined, and determines a map:
\begin{equation} \label{ingen} 
CE^*_{\parallel}(\Lambda) \to \A. 
\end{equation}
We next consider the question whether $\t \in \mathrm{Kos}(LC_{\ast}^{\parallel}(\Lambda), \A)$.
The following theorem gives a sufficient condition for $\t$ to be Koszul: 

\begin{thm} \label{mainthm} Suppose that $LC_*^{\parallel}(\Lambda)$ is locally finite, simply-connected $\k$-bimodule.
    Suppose, in addition, that $HW^*(L)=0$, then $\A$ is quasi-isomorphic to $\Omega(\k_{-}
    \oplus CF_*(L))$ and $\t \colon LC_{\ast}^{\parallel}(\Lambda) \to \A$ is a
    Koszul twisted cochain. In
    other words, the induced DG-algebra map:
    \[ CE^*_{\parallel}(\Lambda) \to \A\approx \Omega (\k_{-} \oplus CF_*(L)) \]
is a quasi-isomorphism.
\end{thm}

\begin{cor}\label{cor:CE=loopspace} 
	In the situation of Theorem \ref{mainthm}, suppose that $L$ is connected and decorated by $-$ and that
    $\partial L = \Lambda$ is diffeomorphic to a sphere $S^{n-1}$. Writing $\overline{L} = L \cup_{\partial} D^n$, there exists a quasi-isomorphism of DG-algebras
    \[ CE^*(\Lambda) \to C_{-*}(\Omega \overline{L}), \] 
where $\Omega \overline{L}$ is the based loop space of $\overline{L}$. 
\end{cor} 
\begin{proof} We first note that there exists a quasi-isomorphism $\k \oplus CF^*(L) \to
    C^*(\overline{L})$ since $L$ is an exact Lagrangian, this is well-known and can be deduced e.g.~from Lemma \ref{l:localdisktreecorr}. We next use Theorem \ref{mainthm}
    and Adams' cobar equivalence (\cite{adams}): \[ \Omega (C_*(\overline{L})) \simeq C_{-*}(\Omega
    {\overline{L}}) \]
    which holds for the simply-connected topological space $\overline{L}$.     
\end{proof}

Theorem \ref{mainthm} is obtained as a  corollary of Theorem \ref{doubledual} and the following
result:

\begin{thm} \label{ainfiso} Suppose that $HW^*(L) =0$. Then, there exists a quasi-isomorphism of augmented $A_\infty$-algebras:
        \[\e\colon CF^*(L) \to LA_{\parallel}^*(\Lambda), \] 
such that 
\[ \e(u_v) = \sum_{\Lambda_{w}\subset\partial L_{v}}x_{w} \text{\ for $v\in\Gamma^{+}$}.
\]
   
\end{thm}

Note that $HW^*(L)=0$ if $X$ is subcritical, or more generally a flexible Weinstein
manifold. 

Note also that if there is no bijection between connected components of $L$ and connected components of $\Lambda$, then we have to work over the semisimplering of idempotents determined by the Lagrangian. 

\begin{proof} 
We construct an $A_\infty$-map $\e = (\e_i)_{i \geq 1}$ where 
\[ 
\e_i \colon (CF^*(L))^{\otimes_{\k} i} \to LA_{\parallel}^*(\Lambda) 
\]
$\e_i$ are constructed by dualizing the components of the map $\t'$. More explicitly,
given $c_{0}$ and $\mathbf{x}_{0}=x_{0;n},\ldots, x_{0;1}$, write $\mathbf{c}=c_{0}\mathbf{x}_{0}$ as in \eqref{eq:deft'}, and define
\[ 
\e_i(\mathbf{x}_{0}) = 
\sum_{c_{0}\in\mathcal{R}} |\mathcal{M}^{\fl}(\mathbf{c})| c_{0} 
\]
The proof that $(\e_i)_{i \geq i}$ is an $A_\infty$-map follows as in the proof of Proposition \ref{twch}. 

Now, we need to check that $\e_1$ is a quasi-isomorphism. We point out that $\e_{1}$ concern only strips and are defined using only two parallel copies. In the case that $L=L^{-}$, this is a consequence of the exact sequence for wrapped Floer cohomology induced by the subdivision of the complex into high and low energy generators:
\[ 
\begin{CD}
0 @>>> CW^{\ast}_{0}(L) @>>> CW^{\ast}(L) @>>> CW^{\ast}_{+}(L) @>>> 0,
\end{CD}
\]
as follows. In terms of the version of wrapped Floer cohomology presented in Section
    \ref{sec:wrappediso}, the low energy sub-complex $CW^{\ast}_{0}(L)$ is generated by Lagrangian
    intersection points in $L_{0}\cap L_{1}$, where $L=L_{0}$ and $L_{1}$ is the first parallel copy
    of $L$, shifted by a positive Morse function $f$ that increases in the end, see Section
    \ref{sec:parallel}. The differential on $CW^{\ast}_{0}(L)$ counts holomorphic strips which is
    incoming along $L_{1}$ and outgoing along $L_{0}$ at the output puncture. Similarly, the high
    energy quotient $CW^{\ast}_{+}(L)$ is generated by Reeb chords connecting $\Lambda_{1}$ to
    $\Lambda_{0}$, the differential counts holomorphic strips interpolating between such Reeb
    chords.  Since $CW^{\ast}(L)$ is acyclic it follows that the connecting homomorphism $HW^{\ast}_{+}(L)\to HW^{\ast+1}_{0}(L)$   is an isomorphism. In order to connect this to $CF^{\ast}(L)$ and $LA_{\parallel}^{\ast}(\Lambda)$, renumber the parallel copies so that $L_{1}$ now lies in the negative Reeb direction of $L_{0}$ at infinity and the shifting function $f$ is replaced by $-f$. Then note that since $L$ is labeled by $-$ the following holds. 
\begin{itemize}
\item The linear dual of $CW^{\ast-1}_{+}(L)$ is canonically identified with
    $LA^{\ast}_{\parallel}(\Lambda)$ as a chain complex (Note that $CW^{\ast-1}_+(L)$ also has an
        $A_\infty$-coalgebra structure as defined in \cite{EO} and this should dualize to the $A_\infty$-algebra structure on $LA^{\ast}$ but we do not need that here).
\item The linear dual of $CW^{\ast}_{0}(L)$ is canonically identified with $CF^{\ast}(L)$
    as a chain complex. 
\item The linear dual of the connecting homomorphism can be canonically identified with the map
    $\e_{1}\colon \k_{-} \oplus CF^{\ast}(L)\to \k_{-}\oplus LA^{\ast}_{\parallel}(\Lambda)$ on critical points that counts
        strips with an input puncture at $L_0 \cap L_1$ and output puncture at a Reeb chord and is the
        canonical map on $\k_{-}$.
\end{itemize}
Since the connecting homomorphism is an isomorphism so is its linear dual. (The argument here is originally due to Seidel, compare \cite[Theorem 7.2]{EKH}.) 

In the case that $L^{+}\ne \varnothing$ the argument just given applies after certain deformation that we describe next. For components in $L^{+}$, $CF^{\ast}(L)$ and $LA^{\ast}_{\parallel}(\Lambda)$ are defined via parallel copies shifted in the positive Reeb direction at infinity. To connect to the previous case we consider a Lagrangian cobordism of two cylinders $\R\times\Lambda_{0}$ which is constant and $\R\times \Lambda_{1}$ which is the trace of an isotopy pushing $\Lambda_{1}$ across $\Lambda_{0}$ in the negative Reeb direction. This can be arranged so that the intersection points of the two cylinders are in natural 1-1 correspondence between the short Reeb chords between $\Lambda_{0}$ and $\Lambda_{1}$. Furthermore, it is straightforward to show that there is exactly one transverse holomorphic strip connecting each intersection point to the corresponding Reeb chord at the negative end of the cobordism. 

Adding these cylinders to $L^{+}$ we get a 1-parameter family of pairs of Lagrangian submanifolds $\hat L^{\rho}_{0}$ and $\hat L_{1}^{\rho}$, where $\rho>0$ is a gluing parameter that measures the length of the trivial cobordism between $L^{+}$ and the added cylinders. The wrapped Floer cohomology $CW^{\ast}(\hat L_{0}^{\rho},\hat L^{\rho}_{1})$ between Lagrangians $\hat L_{0}^{\rho}$ and $\hat L_{1}^{\rho}$ vanishes: it is isomorphic to the wrapped Floer cohomology $CW^{\ast}(L)$ by Hamiltionian deformation invariance. Write $CW^{\ast}(\hat L^{\rho}):=CW^{\ast}(L_{0}^{\rho},L^{\rho}_{1})$. This complex is then acyclic and is generated by the set of long Reeb chords $C^{+}$ from $L_{0}$ to $L_{1}$, the set of intersection points $I$ between the cylinders, and intersection points $P$ in $L$. Let $C^{-}$ denote the short Reeb chords connecting $L_{0}$ to $L_{1}$ and recall the natural 1-1 correspondence $C^{-}\approx I$ above. Let $\rho>0$. We claim that the following sets are in natural 1-1 correspondence  for all sufficiently large $\rho$: 
\begin{itemize}
\item[$(i)$] Rigid strips of $\hat L^{\rho}$ with input puncture at  $c\in C^{+}$ and output puncture at $p\in P$ and rigid strips of $L$ with input puncture at $c \in C^+$ and output puncture at $p \in P$.
\item[$(ii)$] Rigid strips of $\hat L^{\rho}$ with input puncture at  $c\in C^{+}$ and output puncture at $q\in I$ and rigid-up-to-translation strips of $\R\times\Lambda$ with input puncture in $c\in C^{+}$ and output puncture at $q\in C^{-}$.
\item[$(iii)$] Rigid strips of $\hat L^{\rho}$ with input puncture at  $p\in P$ and output puncture at $q\in I$ and rigid strips of $L$ with input puncture at $p\in P$ and positive puncture at $q\in C^{-}$.
\end{itemize}   
To see this note first that the strips in $(i)$ are unaffected by adding the almost trivial cobordism: the strips are transversely cut out and therefore solutions for small variations of the boundary data are canonically identified. Taking $\rho$ sufficiently large the boundary data of the disks can be made arbitrarily close. 

For $(ii)$, in the limit $\rho\to\infty$ the disks limits to an anchored disk with a positive puncture and a puncture at the intersection point. Gluing the basic strip connecting the intersection point to the short Reeb chord to it gives a 1-dimensional moduli space. The other boundary component of this moduli space consists of a rigid strip in the cobordism and a disk or plane of dimension one in either symplectization end. Since the only rigid strips in the cobordism with negative ends at Reeb chords are either close to trivial strips or a basic strip it follows that the other boundary component of the moduli space must be a trivial strip followed by a strip connecting $c$ to $q$ in the negative symplectization end. For $(iii)$ we note that every rigid strip must break under stretching into two rigid strips. Since the only rigid strips in the upper part are the basic strips the claim follows.  

Observe that the strips in $(i)$ and $(iii)$ contribute to $\t'$ and the strips in $(ii)$ to the differential on $LC_{\ast}^{\parallel}(\Lambda)$. The vanishing of the wrapped Floer cohomology of $\hat L^{\rho}$ then implies that $\e_{1}$ is a quasi-isomorphism. The last statement follows from Proposition \ref{twch}. 
\end{proof}

\begin{proof}{(\it of Theorem \ref{mainthm})} Note that the $A_\infty$-quasi isomorphism given in Theorem \ref{ainfiso} induces a
    quasi-isomorphism of DG-algebras
    \[ \Phi\colon \Bar(\k_{-} \oplus CF^*(L)) \to \Bar(LA_{\parallel}^*(\Lambda)) \]
by an application of \cite[Thm.~7.4]{EM} with respect to length filtrations on the bar construction.

By the local finiteness and simply-connectedness assumptions, each of these bar constructions are
    locally finite. So, we can apply the linear dual operation to get a quasi-isomorphism of
    DG-algebras:
    \begin{align} \label{Phi} \Phi^\# \colon \Bar(LA_{\parallel}^*(\Lambda))^\# \to \Bar(\k_{-} \oplus CF^*(L))^\#
    \end{align}
    The result then follows as in Theorem \ref{doubledual} where locally finiteness of the grading
    enabled us to appeal to Lemma \ref{dualbar}. Therefore, the quasi-isomorphism given in \eqref{Phi}
    induces the quasi-isomorphism:
  \[ \Omega(LC^{\parallel}_*(\Lambda)) \to \Omega(\k_{-} \oplus CF_*(L))  \]
as required. 
\end{proof}

We next turn to the non-simply connected case where we use $CE^{\ast}_{\parallel}(\Lambda)$ as defined in Section \ref{ssec:parallelwithpi1} directly without using the corresponding co-algebra. 
Note that the $A_\infty$-algebras $\k_{-} \oplus CF^*(L)$ and
$LA^*_{\parallel}(\Lambda)$ are finite rank $\k$-bimodules (in particular, they are locally
finite), thus we can consider their $\k$-duals which are, by definition, the $A_\infty$-coalgebras
$\k_{-} \oplus CF_*(L)$ and $LC_*^{\parallel}(\Lambda)$. However, unless we have the simply-connectedness assumption, the $A_\infty$-quasi isomorphism 
\[ \e \colon \k_{-} \oplus CF^*(L) \to LA_{\parallel}^* (\Lambda) \] 
does not necessarily dualize to an $A_\infty$-comap 
\[ \mathfrak{f}\colon  LC_*^{\parallel}(\Lambda) \to \k_{-} \oplus CF_*(L) \] 
because $A_\infty$-comaps are required to factorize through inclusion of the corresponding direct
sum into the direct product as in \eqref{factor}. This is to ensure that
$A_\infty$-coalgebra maps $\mathfrak{f}$ induces a DG-algebra map $\Omega \mathfrak{f}$ on the cobar
construction. 

If we drop this condition, we can observe that the $A_\infty$-quasi isomorphism
dualizes to DG-algebra map 
\[ \widetilde{CE}^*_{\parallel}(\Lambda) \to \A \] and this is just the map \eqref{ingen} induced by the twisting cochain $\mathfrak{t}$. Furthermore by Proposition \ref{twch} this gives a map a DG-algebra map
\[ \widehat{\Omega}(\mathfrak{f})\colon \widehat{CE^*_{\parallel}}(\Lambda) \to \widehat{\Omega} (\k_{-} \oplus CF_*(L))\]  

Now, since $\mathfrak{f}$ is a quasi-isomorphism, by using the length filtration on
$\widehat{\Omega}$, and appealing to \cite[Thm. 7.4]{EM}, we can conclude the following. 

\begin{thm} \label{comple} 
	Suppose that $HW^*(L)=0$, then there exists a quasi-isomorphism of DG-algebras
    \[ \widehat{CE^*_{\parallel}}(\Lambda) \to \widehat{\Omega}(\k_{-} \oplus CF_*(L)). \]
\end{thm}

Note that the completion $\widehat{CE_{\parallel}^{*}}$ is in general a cruder invariant than
both $CE^{\ast}_{\parallel}(\Lambda)$ and $CE^*(\Lambda)$. Though, we always have a map 
\[ CE^*(\Lambda) \to \widehat{CE^*_{\parallel}}(\Lambda). \] 
Theorem \ref{comple} can be used to compute $\widehat{CE^*_{\parallel}}$ in a variety of
cases. For example, if $L$ is connected Lagrangian filling decorated by $-$ of a Legendrian $\Lambda$ which is
diffeomorphic to a sphere $S^{n-1}$, then
writing $\overline{L} = L \cup_{\partial} D^n$, then we have a quasi-isomorphism $\k \oplus
CF_*(L) \simeq C_*(\overline{L})$ since $\overline{L}$ is an exact Lagrangian. Hence, we have a map
\[ CE^*(\Lambda) \to \widehat{\Omega}(C_*(\overline{L})) \]
Here the right hand side can often be computed, in particular
$H^0(\widehat{\Omega}(C^*(L))$ is the group ring of the unipotent completion of the fundamental group
$\pi_1(L)$ (cf. \cite{chen}). 

In particular, any information on the completion map $CE^*(\Lambda) \to
\widehat{CE^*_{\parallel}}(\Lambda)$ can help to obtain information about $CE^*(\Lambda)$. We
will see an application of this idea in the next section.

We end this section with a discussion of the twisting cochains constructed above from an after surgery perspective. Assume that all components of $\Lambda^{-}$ are spheres and recall the surgery isomorphism
\[ 
\Theta\colon CW^{\ast}(C)\to CE^{\ast}(\Lambda)
\]
of Conjecture \ref{appconj}. Let $\widetilde{\Theta}=\phi\circ\Theta$, where $\phi$ is the identity map on components in $\Lambda_{-}$ and the map $\phi$ of Theorem \ref{t:parallel=loops}. We next note that there is a natural $A_{\infty}$-algebra map 
\[ 
\Psi\colon CW^{\ast}(C)\to \Bar(\k\oplus CF^*(L))^\#=\A'\to\A=\A'/\I,
\]
where we identify $\k_{-}\oplus CF^{\ast}(L)$ with the Floer cohomology $CF^{\ast}(L')$ of the manifold after surgery obtained by capping off all boundary spheres in $\Lambda^{-}$ by disks. (Note that in the simply connected case, the shifting Morse function then extends with a unique minimum in each disk which gives an idempotent corresponding to $e_{v}$.)

The map $\Psi$ is defined by a curve count. Fix systems of parallel copies $\bar C$ of $C$ and $\bar L'$ of $L'$. Let $\mathbf{c}'=c_{1}\dots c_{i}$ be a word of Reeb chords of $C$ and let $\mathbf{x}_{0}=x_{0;1}\dots x_{0;j}$ be a word of intersection points of $L'$. Let 
\[ 
\mathbf{c}=\mathbf{c}' z_{v}\mathbf{x}_{0} z_{w}
\]
and define
\[ 
\Psi(\mathbf{c}')=\sum_{|\mathbf{x}_{0}|=|\mathbf{c}'|+(1-i)}
|\mathcal{M}^{\overline{\co}}(\mathbf{c})|\mathbf{x}_{0}.
\]

\begin{rem}\label{r:inducednegparallel}
We require here that the parallel copies $\bar L'$ give a system of parallel copies of $\Lambda$ near the surgery region such that for the components of $\Lambda_{v}$ labeled by $-$ (resp. $+$) the induced parallel copies $\bar\Lambda_{v}$ lies in the negative (resp. positive) Reeb direction. Compare Figure \ref{intersectiongen}.
\end{rem}

\begin{thm}
	The map $\Psi$ is a map of $A_{\infty}$-algebras and $\Psi(z_{v})=u_{v}'$, where $z_{v}$ is the strict unit in $CW^{\ast}(C_{v})$ and $u_{v}'$ is the dual of the 
	unit in $\mathbb{K}e_{v}\oplus CF^{\ast}(L_{v})$ for each $v$.
\end{thm}

\begin{proof}
	This follows as usual by identifying terms contributing to the $A_{\infty}$-relations with the oriented boundary of an oriented 1-manifold and Theorems \ref{thm:mdlitv}, \ref{thm:mdlicmpct}, \ref{thm:mdlicopies}.   
	
	To compute $\Psi(z_{v})$ note that we can represent $z_{v}$ as the minimum of the shifting Morse function of $C$ and note that there is a unique flow line from this minimum to the intersection point $C_{v}\cap L_{v}'$ and a unique flow line in $L_{v}'$ from the intersection point to the minimum in $L_{v}$. The corresponding holomorphic disk starts at the intersection point between $C_{0}$ and $C_{1}$, has two corners at $C_{0}\cap L_{0}'$ and $C_{1}\cap L_{1}'$, and ends at the intersection point in $L_{0}'\cap L_{1}'$ corresponding to the minimum of the shifting Morse function. 
\end{proof}

The pre-twisting cochain $\t'\colon LC_{\ast}^{\parallel}(\Lambda)\to\A$ can now be seen to arise via SFT-stretching as follows. Consider the first component $\Psi_{1}$ of the $A_{\infty}$-map above and a holomorphic disk contributing to it. Now stretch the lower end of the cobordism. Then by SFT-compactness the curves contributing to $\Psi_{1}$ breaks up into a curve contributing to the map $\phi\circ \Theta$ followed by the twisting cochain at each negative puncture. Noting that $z_{v}$ are the only low energy generators, it follows that the map induces a map of the high energy quotient into $\mathrm{coker}(\eta)$ and we can write the induced map $\Psi_{1}^{+}$ as $\Psi_{1}^{+}=(\Omega\t)\circ\phi\circ\Theta^{+}$.

\section{Examples and Applications}
\label{ssec:applications} 
\subsection{Concrete calculations}
In this section we compute Legendrian and Lagrangian invariants in a number of concrete examples.

\subsubsection{The unknot}
Consider the Legendrian unknot $\Lambda \subset S^{2n-1}$ for $n>1$ with its
standard filling $L=D^{n} \subset D^{2n}$. Then $\Lambda$ can be represented as a standard unknot in a small Darboux chart which has effectively only one Reeb chord with respect to the standard Reeb flow on $S^{2n-1}$, see \cite[Section 7.1]{BEE}. 

We work over a field $\mathbb{K}$. Consider first the case when $L$ is decorated by $-$. We then have
\[ 
LC_*(\Lambda) = \mathbb{K}\cdot 1 \oplus \mathbb{K} \cdot c, \ \ |c|=-n 
\]
with all coalgebra maps $(\Delta_i)_{i\geq 1}$ trivial, except $\Delta_2$ for which we have:
\[ 
\Delta_2(1) = 1 \otimes 1, \ \ \Delta_2(c) = 1 \otimes c + c \otimes 1 
\]
by the co-unitality. 
Using a Morse function on $D^{2n}$ with a unique local maximum $a$ and which decreases in the end corresponding to shift in the negative Reeb direction, we have 
\[ 
\mathbb{K} \oplus CF_*(L) = \mathbb{K} \cdot 1  \oplus \mathbb{K} \cdot a, \ \ |a|=-n 
\]
Let $\A = \Omega( \mathbb{K} \oplus CF_*(L))$ be the Adams-Floer algebra, where degree of $a$ is now shifted up by 1. Then, we have the twisting cochain 
\[ 
\t \colon LC_{\ast}(\Lambda) \to \A, \ \ \t(c) =a. 
\]
Here the disk with input $c$ and output $a$ corresponds to the family of disks with positive puncture at $c$ that sweeps $L$ once. 

(The case of $n=2$ is drawn in Figure (\ref{unknot})).
\begin{figure}[htb!]
    \centering
    \begin{tikzpicture}[scale=1]
              \tikzset{->-/.style={decoration={ markings,
                      mark=at position #1 with {\arrow[scale=2,>=stealth]{>}}},
                      postaction={decorate}}}
 \begin{scope} 
     \clip (0,0) to (0.3,0.3) to[in=170,out=45] (1.5,1) to[in=90,out=350] (2.3,0) to[in=10,out=270] (1.5,-1)to[in=315,out=190] (0.3,-0.3) to
     (0,0) ;
\clip[preaction={draw,fill=gray!30}] (-2,-2) rectangle (8,8);
                             \end{scope}
                             
    \draw [black] (-1.5,1) to[in=90,out=190] (-2.3,0);
    \draw [black] (-1.5,1) to[in=135,out=10] (-0.3,0.3);         
    \draw [black]  (-1.5,-1)to[in=270,out=170] (-2.3,0);
    \draw [black] (-1.5,-1)to[in=225,out=350] (-0.3,-0.3)  ;
    
    \draw [black] (1.5,1) to[in=90,out=350] (2.3,0);
    \draw [black] (1.5,1) to[in=45,out=170] (0.3,0.3);         
    \draw [black]  (1.5,-1)to[in=270,out=10] (2.3,0);
    \draw [black] (1.5,-1)to[in=315,out=190] (0.3,-0.3)  ;

    \draw [black] (-0.3,0.3) to (0.3,-0.3);
    \draw [black] (-0.3,-0.3) to (0.3,0.3);
\end{tikzpicture}
    \caption{Computation of $\t$ for the Legendrian unknot, the disk drawn lies in a $(n-1)$-dimensional family that sweeps the filling once.}
    \label{unknot}
\end{figure}

The Koszul complex is generated over $\mathbb{K}$ by $a^{k}$, $a^{k}c$ for $k \geq 0$.
We can compute the non-trivial part of the differential to be:
\[ d^\t(a^{k}c ) = a^{k+1}, \ \ \text{for all } k \geq 0 \]
hence $\t$ is acyclic. This is consistent with the classical Koszul duality between the algebras $C^*(S^n)$ and $C_{-*}(\Omega S^n)$
for $n>1$. 

Consider next the case when $L$ is decorated by $+$. Since $S^{n-1}$ is simply connected $CE^{\ast}(\Lambda)\approx CE^{\ast}_{\parallel}(\Lambda)$ and we will use the parallel copies version in our calculation. Choose a Morse function on $\Lambda$ with a single minimum and a single maximum. Denote the corresponding Reeb chords $x$ (the co-unit chord) and $y$. Then
\[ 
LC_{\ast}^{\parallel}(\Lambda)= 
\mathbb{K}\cdot x \oplus \mathbb{K}\cdot y\oplus \mathbb{K} \cdot c, \ \ |x|=0,\;|y|=-(n-1),\;|c|=-n. 
\]
Here 
\[ 
\Delta_{1}(c)=y,\quad \Delta_{2}(x)=x\otimes x + +(-1)^{n-1}x\otimes y + y\otimes x + (-1)^{n}x\otimes c + c\otimes x,
\]
and all other operations are trivial. It follows that
\[ 
CE^{\ast}_{\parallel}(\Lambda)=\Omega(LC_{\ast}^{\parallel}(\Lambda)) \simeq \mathbb{K}.
\]
This is in line with Conjecture \ref{conjintro} which says that
$CE^{\ast}(\Lambda)\simeq CE^{\ast}_{\parallel}(\Lambda)$ is isomorphic to $CW^{\ast}(C)$, where $C$ is the cotangent fiber in the manifold obtained by attaching a cotangent end $T^{\ast}(S^{n-1}\times[0,\infty))$ to the ball along $\Lambda$. The manifold that results from this attachment is simply $T^{\ast}\R^{n}$ and the wrapped Floer cohomology of the cotangent fiber $C$ has rank 1 and is generated by the minimum in the disk $C$, in accordance with the above calculations.

Finally, the twisting co-chain $\t$ in the $+$ case is the canonical map $\K\to\K$, and again the
Koszul complex is acyclic. As explained in Section \ref{Sec:LegLag} this map is induced by a map
$\t'\colon LC_{\ast}^{\parallel}(\Lambda)\to (\Bar{CF^{\ast}(L)})^{\#}$. To define $CF^{\ast}(L)$ we use a Morse function on $L$ with a single local minimum and which is increasing in the end corresponding to a shift in the positive Reeb direction. Denote the generator of $CF^{\ast}(L)$ corresponding to the minimum $u$, $|u|=0$. Then  
\[ 
\t'(x)=u',
\] 
where $u'$ is the dual of $u$ and the holomorphic strip is the thin strip corresponding to a rigid Morse flow line from the minimum $u$ to the minimum $y$ in the boundary.  
\subsubsection{Geometric twisting cochain for the Hopf link}
\begin{figure}[h!]
	\centering
	\begin{tikzpicture}[scale=1]

	\tikzset{->-/.style={decoration={ markings,
				mark=at position #1 with {\arrow{>}}},postaction={decorate}}}
	
	\draw [blue, thick=1.5, ->-=.7] (-1.6,0.93) to[in=90,out=190] (-2.3,0);
	\draw [blue, thick=1.5] (-1.4,1) to[in=135,out=10] (-0.3,0.3);         
	\draw [blue, thick=1.5]  (-1.6,-0.93)to[in=270,out=170] (-2.3,0);
	\draw [blue, thick=1.5] (-1.4,-1)to[in=225,out=350] (-0.25,-0.25)  ;
	
	\draw [blue, thick=1.5] (1.5,1) to[in=45,out=170] (0.3,0.3);         
	
	\draw [blue, thick=1.5]  (1.5,-1)to[in=270,out=10] (2.3,0);
	\draw [blue, thick=1.5] (1.5,1) to[in=140,out=350] (1.9,0.8);
	\draw [blue, thick=1.5] (2.05,0.65) to[in=90,out=320] (2.3,0);

	\draw [blue, thick=1.5] (1.5,-1)to[in=315,out=190] (0.3,-0.3)  ;
	
	\draw [blue, thick=1.5] (-0.3,0.3) to (0.3,-0.3);
	\draw [blue, thick=1.5] (-0.15,-0.15) to (-0.05,-0.05);
	\draw [blue, thick=1.5] (0.05,0.05) to (0.3,0.3);

	\begin{scope}[xshift=-11]
	\draw [cyan, thick=1.5, ->-=.7] (-1.5,1) to[in=90,out=190] (-2.3,0);
	\draw [cyan, thick=1.5] (-1.5,1) to[in=135,out=10] (-0.3,0.3);         
	\draw [cyan, thick=1.5]  (-1.5,-1)to[in=270,out=170] (-2.3,0);
	\draw [cyan, thick=1.5] (-1.5,-1)to[in=225,out=350] (-0.3,-0.3)  ;
	
	\draw [cyan, thick=1.5] (2,1.25) to[in=45,out=170] (0.25,0.25);         
	
	\draw [cyan, thick=1.5]  (2,-1.25)to[in=270,out=10] (2.95,0);
	\draw [cyan, thick=1.5] (2,1.25) to[in=135,out=350] (2.5,0.95);
	\draw [cyan, thick=1.5] (2.65,0.8) to[in=90,out=315] (2.95,0);

	\draw [cyan, thick=1.5] (2,-1.25)to[in=315,out=190] (0.3,-0.3)  ;
	
	\draw [cyan, thick=1.5] (-0.3,0.3) to (0.3,-0.3);
	\draw [cyan, thick=1.5] (-0.3,-0.3) to (-0.05,-0.05);
	\draw [cyan, thick=1.5] (0.05,0.05) to (0.15,0.15);
	
	\end{scope}

	\draw [red, thick=1.5] (2.5,1) to[in=90,out=190] (1.7,0);
	
	\draw [red, thick=1.5] (2.5,-1) to[in=330,out=170] (2.32,-0.95);
	\draw [red, thick=1.5] (2.2,-0.87) to[in=310,out=150] (2.07,-0.77);
	
	\draw [red, thick=1.5] (1.95,-0.65) to[in=270,out=130] (1.7,0);

	\draw [red, thick=1.5] (2.5,1) to[in=135,out=10] (3.7,0.3);         
	\draw [red, thick=1.5] (2.5,-1)to[in=225,out=350] (3.7,-0.3)  ;
	
	\draw [red, thick=1.5] (5.5,1) to[in=90,out=350] (6.3,0);
	\draw [red, thick=1.5] (5.5,1) to[in=45,out=170] (4.3,0.3);         
	\draw [red, thick=1.5, ->-=.7]  (6.3,0)to[in=10,out=270] (5.5,-1);
	\draw [red, thick=1.5] (5.5,-1)to[in=315,out=190] (4.3,-0.3)  ;
	
	\draw [red, thick=1.5] (3.7,0.3) to (4.3,-0.3);
	\draw [red, thick=1.5] (3.7,-0.3) to (3.9,-0.1);
	\draw [red, thick=1.5] (4.1,0.1) to (4.3,0.3);
	
	\node at (0.3,0) {\footnotesize{$c_1$}}; 
	\node at (1.7,0.7) {\footnotesize{$c_{21}$}}; 
	\node at (1.7,-0.7) {\footnotesize{$c_{12}$}}; 
	\node at (4,0.3) {\footnotesize{$c_2$}}; 
	\node at (-1.5,1.2) {\footnotesize{$y$}}; 
	\node at (-1.5,-1.2) {\footnotesize{$x$}}; 
	
	\end{tikzpicture}
	\caption{Hopf link with one (blue) marked $+$ and one (red) marked $-$ components.}
	\label{hopfparallel}
\end{figure}

In this section we carry out the geometric calculation of the twisting cochain for the Hopf link. As
explained we cannot directly calculate the twisting cochain into the Legendrian coalgebra with
coefficients in chains of the based loop space. We can however calculate the corresponding twisting
cochain when we replace chains on the based loop space with the Morse theory of parallel copies for
the components decorated by a positive sign. To carry out the calculation we pick a Morse function
on the component $\Lambda^{+}$ with on minimum $x$ and a maximum $y$. We place them on the circle
and choose parallel copies as shown in Figure \ref{hopfparallel}. The parallel copies algebra $CE_{\parallel}^{\ast}(\Lambda)$ is
\[ 
\k\langle x,y,c_{1},c_{12},c_{21},c_{2}\rangle,
\]
where we use notation for Reeb chords and Floer cohomology generators as in Section \ref{ssec:Hopflink}. The differential is
\begin{align*}
d c_{1} &= xc_{1} + c_{1}x + y + c_{12}c_{21},\\
d x &= xx,\\
dy &= xy - yx,\\
dc_{12} &= -xc_{12},\\
dc_{21} &= c_{21}x,\\
dc_{2} &= c_{21}c_{12}.
\end{align*}
Passing to $CE^{\ast}_{\parallel}(\Lambda)$ means dividing out by the cokernel of the coaugmentation that takes $e_{1}$ to $x$. This gives the algebra
\[ 
\k\langle y,c_{1},c_{12},c_{21},c_{2}\rangle,
\]
and the differential becomes
\begin{align*}
d c_{1} &= y + c_{12}c_{21},\\
dc_{2} &= c_{21}c_{12}.
\end{align*}

The twisting cochain $\t$ is induced from a map $\t'\colon LC_{\ast}^{\parallel}(\Lambda)\to (\Bar(\k_{-}\oplus CF^{\ast}(L)))^{\#}$ that counts holomorphic disks with one positive puncture and boundary on $L$, and with several punctures at Lagrangian intersection points in the compact part, see \eqref{eq:deft'}. In the current example it is straightforward to find these disk. Note first that by general properties, see Proposition \ref{twch},
\[ 
\t'(x)=a_{1}^{\vee},
\]
where $a_{1}$ is the idempotent corresponding to the minimum of the shifting Morse function on $L_{1}$. The holomorphic disk corresponds to a Morse flow line connecting $x$ to $u_{1}$.  We next consider $\t'(c_{1})$ and $\t'(c_{2})$. Consider first the moduli spaces $\mathcal{M}^{\fl}(c_{j})$ of holomorphic disks with a positive puncture at $c_{j}$ and boundary on $L_{j}$. As in the case of the unknot this moduli space sweeps $L_{j}$. On both $L_{1}$ and $L_{2}$ the shifting functions have one critical point, on $L_{1}$
it is a minimum and on $L_{2}$ a maximum. The maximum is constraining for the map into the linear dual of $CF^{\ast}(L_{j})$, whereas the minimum is not. We find that
\[ 
\t'(c_{1})=0,\quad \t'(c_{2})= a_{2}^{\vee}.
\] 

The space $\mathcal{M}^{\fl}(c_{j})$ give further information of the twisting co-chain as follows. Note that as the evaluation map hits the double point one can glue on a constant disk. These broken disks are also the boundary of the 1-dimensional moduli spaces $\mathcal{M}^{\fl}(c_{1}a_{12}a_{21})$ and $\mathcal{M}^{\fl}(c_{2}a_{21}a_{12})$. The other end of these moduli spaces correspond to broken disks with one level in the symplectization, a disk in $\mathcal{M}^{\sy}(c_{1}c_{12}c_{21})$ in the former case and in $\mathcal{M}^{\sy}(c_{2}c_{21}c_{12})$ in the latter, and two disks one in $\mathcal{M}^{\fl}(c_{12}a_{12})$ and one in $\mathcal{M}^{\fl}(c_{21}a_{21})$ attached at its negative end. The last disks contribute to $\t'$ and we conclude that
\[ 
\t'(c_{12})=a_{12}^{\vee},\quad \t'(c_{21})=a_{21}^{\vee}.
\]
Finally, we compute $\t'(y)$. Since $y$ is a small chord corresponding to a critical point at infinity of the shifting Morse function the only contributions to $\t'(y)$ comes from small holomorphic disks that are controlled by the Morse theory. It is straightforward to check that the only rigid disk corresponds to a flow line in $L_{1}$ connecting $y$ to the intersection point and that this flow line corresponds to a holomorphic triangle in $\mathcal{M}^{\fl}(ya_{12}a_{21})$. It follows that
\[ 
\t'(y)=a_{21}^{\vee}a_{12}^{\vee}.
\] 

The actual twisting co-chain takes the cokernel $\overline{LC}_{\ast}(\Lambda)$ of the co-augmentation into the cokernel of the co-units. Concretely, this means disregarding $x$ and $a_{1}^{\vee}$, and we get
\[ 
\t(c_{1})=0,\;\t(c_{2})=a_{2}^{\vee},\;\t(c_{12})=a_{12}^{\vee},\;\t(c_{21})=a_{21}^{\vee},
\;\t(y)=a_{12}^{\vee}a_{21}^{\vee}.
\]

\begin{rem}
	We point out that the parallel copies algebra $CE_{\parallel}^{\ast}(\Lambda)$ is defined using a fixed augmentation (in the current example the zero augmentation) on the one copy version of $CE^{\ast}(\Lambda)$. Here this is reflected in the change of variables $t-e_{1}=y$.
\end{rem}

\subsubsection{Products of spheres}
We will consider a Legendrian embedding $\Lambda\subset\R^{2(m+k)-1}$, where the ambient space is
standard contact $(2(m-k)-1)$-space with coordinates $(x,y,z)\in\R^{m+k-1}\times\R^{m+k-1}\times\R$
and contact form $dz-y\cdot dx$. We will define it by describing its front in $\R^{m+k-1}\times\R$.
To this end consider first the following construction of the front of the Legendrian unknot in
$\R^{n}\times\R$. Take a disk $D^{n}$ lying in $\R^n$. Think of it as having multiplicity two and
lift one of the sheets up in the auxiliary $\R$-direction (with coordinate $z$) keeping it fixed
along the boundary. In this way we construct the front of the standard unknot with Reeb chord at the maximum distance between the two sheets and a cusp edge along the boundary of $D^{n}$. Consider now instead the base $\R^{m+k-1}$ and the standard embedding of the $k$-dimensional sphere $S^{k}$ into this space. A tubular neighborhood of this embedding has the form $S^{k}\times D^{m-1}$ with fibers $D^{m-1}$. Now take two copies of this tubular neighborhood and repeat the above construction for the $D^{m-1}$ in each fiber. The result is the front of a Legendrian $S^{k}\times S^{m-1}$ with an $S^k$ Bott-family of Reeb chords corresponding to the maxima in fibers. Figure \ref{prodsphere} shows this front after Morse perturbation. The resulting Legendrian $\Lambda$ has only two Reeb chords. We denote them $a$ of grading $|a|=-(m+k)$ and $b$ of grading $|b|=-m$. Note also that $\Lambda$ has an exact Lagrangian filling $L\approx D^{m}\times S^{k}$.
 
\begin{figure}[h!]
    \centering
    \begin{tikzpicture}[scale=1]

    \tikzset{->-/.style={decoration={ markings,
                mark=at position #1 with {\arrow{>}}},postaction={decorate}}}


        \begin{scope} 

            \draw[black, thick=1] (0,0) to[in=180,out=0] (1,1); 
            \draw[black, thick=1] (2,0) to[in=0,out=180] (1,1); 
 \draw[black, thick=1] (0,0) to[in=180,out=0] (1,-1); 
 \draw[black, thick=1] (2,0) to[in=0,out=180] (1,-1); 

       \draw[black, dashed] (1,-1) to (1,1); 
       \node at (1.2, 0) {\footnotesize{$a$}}; 
        \end{scope}

        \begin{scope} [xshift=3cm] 

            \draw[black, thick=1] (0,0) to[in=180,out=0] (1,1); 
            \draw[black, thick=1] (2,0) to[in=0,out=180] (1,1); 
 \draw[black, thick=1] (0,0) to[in=180,out=0] (1,-1); 
 \draw[black, thick=1] (2,0) to[in=0,out=180] (1,-1); 
   \draw[black, dashed] (1,-1) to (1,1); 
       \node at (1.2, 0) {\footnotesize{$b$}}; 
    
       \draw[black, ->-=1 ] (2.7,-1.5) to[in=315, out=180] (1.7,-0.9); 
        \node at (3.2, -1.5) {\footnotesize $S^{m-1}$};
        \end{scope}        

        \begin{scope} [xshift=1.5cm, yshift=-2cm] 

      \draw[thick, fill=black, scale=5] (0,0) circle(.01);
      \draw[thick, fill=black, scale=5] (0.4,0) circle(.01);

         \draw(0,0) to[in=90, out=90] (2,0);
         \draw(0,0) to[in=270, out=270] (2,0);
        \node at (1,0) {\footnotesize{$S^k$}}; 

            \draw[dotted] (0,0) to (-0.4,0.8); 
            \draw[dotted] (2,0) to (2.4,0.8);

        \end{scope} 

\end{tikzpicture}
    \caption{Front of products of spheres}
    \label{prodsphere}
\end{figure}

Consider first the case when $L$ is decorated by $-$. If $d$ is the differential in 
$CE^{\ast}(\Lambda)$ we have
\[ 
da = 0,\; db=0.
\]
The Floer cohomology of $L$ is defined by choosing a shifting function which is decreasing at infinity and we find that $CF^{\ast}$ has two generators $M$, $|M|=m+k$ and $S$, $|S|=m$. As in the case of the unknot $\mathcal{M}^{\fl}(a)$ sweeps $L$ and $\mathcal{M}^{\fl}(b)$ sweeps $D^{m}\times \mathrm{pt}$. It follows that the twisting co-chain is
\[ 
\t(a)=M^{\vee},\; \t(b)=S^{\vee}
\] 
and duality holds.

Consider second the case when $L$ is decorated by $+$. In this case there are additional generators of $LC_{\ast}^{\parallel}(\Lambda)$ corresponding to the Morse theory of $\Lambda$. We have in addition to $a$ and $b$ above also
\[ 
x,\;|x|=0,\quad s_{1},\;|s_{1}|=-k,\quad  s_{2},\;|s_{2}|=-(m-1),\quad y,|y|=-(m+k-1).
\]
It follows that $CE^{\ast}_{\parallel}(\Lambda)$ is generated by $s_{1}$, $s_{2}$, $y$, $a$, and $b$. Using the flow tree description of moduli spaces $\mathcal{M}^{\sy}$ one verifies that if $d$ is the differential on $CE^{\ast}_{\parallel}(\Lambda)$ then
\begin{align*} 
d a &=y - ((-1)^{km}bs_{1}+s_{1}b),\; dy = s_{1}s_{2} + (-1)^{k(m-1)}s_{2}s_{1},\\
db &=s_{2},\; ds_{1}=0,\; ds_{2}=0.
\end{align*}

The Floer cohomology of $L$ is defined by choosing a shifting function which is increasing at
infinity and we find that $CF^{\ast}$ has two generators $M$, $|M|=0$ and $S$, $|S|=k$, where
$M$
is the unit. It follows that $(\Bar CF^{\ast}(L))^\# \simeq \Omega CF_{\ast}(L))$ is generated by
$S^{\vee}$ with $|S^{\vee}|=-k$ and the twisting co-chain is
\[ 
\t(a)=0,\; \t(y) = 0,\; \t(b)=0,\; \t(s_{2})=0,\;\t(s_{1})=S^{\vee},
\] 
and duality holds.

\subsubsection{Plumbings of simply-connected cotangent bundles} 

Let $\mathcal{T}$ be a tree with vertex set $\Gamma$. For each $v \in \Gamma$, let $M_v$ be a compact
simply-connected manifold of dimension $n \geq 3$. We will see that duality holds between
the wrapped and the compact Fukaya categories of the symplectic manifold $X_\mathcal{T}$
obtained by plumbing the collection of $T^*M_v$ according to the tree $\mathcal{T}$. As usual, we take the pre-surgery
perspective. Hence, consider a handle decomposition of each $M_v$ with a unique top-dimensional $n$-handle. Removing this handle, we get manifolds $L_v$ with spherical boundary $\Lambda_v$. Let
$W_{\mathcal{T}}$ be the subcritical
Weinstein manifold obtained by plumbing the cotangent bundles $T^*L_v$ according to the tree
$\mathcal{T}$. We take the plumbing region to be away from the boundary of $L_v$. Write $\Lambda =
\bigsqcup_v \Lambda_v$ for the Legendrian in the boundary of the subcritical Weinstein manifold
$W_\mathcal{T}$ which is filled by the Lagrangian $L=\bigcup_{v} L_{v}$.
Equip the components of $\Lambda$ with either $+$ or $-$ labeling. 

\begin{thm} \label{plumbs}  
	If $n \geq 3$ and $L_v$ is simply-connected for each $v\in\Gamma$, then $CE^*(\Lambda)$ and $CF^*(L)$ are Koszul dual. 
\end{thm}
\begin{proof} 
	Consider first the case $n \geq 5$. Pick a handle decomposition of $L_{v}$ without
    $1$- and $(n-1)$-handles. The existence of such a handle decomposition is equivalent to
    simply-connectivity in high dimensions by the work of Smale. Consider the corresponding
    Weinstein handle decomposition of $T^{\ast}L_{v}$.  Attaching a $k$-handle alters the Legendrian
    boundary by surgery and adds Reeb chords in the co-core sphere of the handle of grading $\le
    -(n-k)$ and it follows that all Reeb chords of $\Lambda_{v}$ have grading $\le -2$. To construct
    $\Lambda\subset \partial W_{\mathcal{T}}$, we perform a version of boundary connected sum as
    follows. For each edge in $\mathcal{T}$ we pick a $2n$-ball $B$ with two transversely
    intersecting Lagrangian disks $D\subset B$ that intersect the boundary sphere $\partial B$ in a
    standard Legendrian Hopf link $\Delta$. We then make boundary connected sum adding $(B,\Delta)$
    to join the $T^{\ast}L_{v}$ according to the tree. This adds Reeb chords of index $\le -(n-1)$
    in the boundary connected sum handles and further Reeb chords corresponding to the Reeb
    chords of each Hopf link, which effectively has four Reeb chords: two Reeb chords connecting the unknot components to themselves of grading $-n$ and two mixed Reeb chords connecting distinct components. We can pick gradings so that the gradings of these two mixed chords are $-d$ and
    $-(n-d)$ for any $d$, where the first Reeb chord goes from the component closest to priori fixed
    root of the tree $\mathcal{T}$ to the one further from it and the other in the opposite
    direction. Taking $d$ between $2$ and $n-2$ we find that $LC_*(\Lambda)$ is simply connected as is $\Lambda$. The result then follows from Theorem \ref{mainthm}.
    
For $n=4$, we can stabilize by multiplying by $\R^{N}$ as described in Remarks \ref{r:dim4} and \ref{r:dim4iso} to get 1-reduced versions of the Legendrian and Lagrangian algebras and then use the above argument.
         
Finally for $n=3$, the assumptions of Theorem \ref{mainthm} does not hold but we recall from
Remark \ref{locfinrem} that to apply the duality result from Theorem \ref{mainthm} all we really need is that $\Bar (LA^*)$ is locally finite. This is easily seen to be the case in our case, since the
    plumbing is according to a tree (which by definition has no cycles) and for any word of Reeb chords
    that do not consist of connecting Reeb chords going away from the root of the tree the
    grading has to increase with the size of the word. Alternatively, in this case we know by
    Perelman's theorem that we can take $\Lambda$ as a link of standard Legendrian spheres linked
    according to the tree as Hopf links and $\Bar (LA^*)$ can be seen to be locally finite directly. \end{proof}
\begin{rem}
	The case $n=2$ corresponds to plumbing of copies of $T^*S^2$. This case was studied in \cite{EtLe} and a
    version of the duality result still holds, at least when $\mathrm{char}(\mathbb{K}) =0$.
    However, the above argument fails in that case and a more complicated argument using an additional
    grading is used in \cite{EtLe}. Also a set of examples for $n=1$ are studied 
    in \cite{LePol}, where the plumbing tree is a star and the corresponding
    symplectic manifolds is a punctured torus. The duality still holds in this case, even though
    this is a plumbing of $T^*S^1$'s (not simply-connected). The proof of duality given in \cite{LePol} uses homological mirror symmetry. 
\end{rem}

\subsubsection{The trefoil} 

\begin{figure}[h!]
    \centering
    \begin{tikzpicture}[scale=1]

    \tikzset{->-/.style={decoration={ markings,
                mark=at position #1 with {\arrow{>}}},postaction={decorate}}}
            
\draw [black, thick=1, ->-=.7] (-1.5,1) to[in=90,out=190] (-2.3,0);
        \draw [black, thick=1] (-1.5,1) to[in=135,out=10] (-0.3,0.3);         
    \draw [black, thick=1]  (-1.5,-1)to[in=270,out=170] (-2.3,0);
    \draw [black, thick=1] (-1.5,-1)to[in=225,out=350] (-0.3,-0.3)  ;
    
    \draw [black, thick=1] (1.4,1.7) to[in=45,out=170] (0.3,0.3);         
       
    \draw [black, thick=1] (1.4,1.7) to[in=140,out=350] (1.8,1.1);
        \draw [black, thick=1] (2.05,0.85) to[in=40,out=320] (1.8,-0.3);
    \draw [black, thick=1] (2.05,-0.35) to[in=40,out=320] (1.8,-1.5);

     \draw [black, thick=1] (1.8,-1.5)to[in=10,out=220] (1.5,-1.7)  ;

    \draw [black, thick=1] (1.5,-1.7)to[in=315,out=190] (0.3,-0.3)  ;

    \draw [black, thick=1] (-0.3,0.3) to (0.3,-0.3);
    \draw [black, thick=1] (-0.3,-0.3) to (-0.1,-0.1);
    \draw [black, thick=1] (0.1,0.1) to (0.3,0.3);    
        
        \draw [black, thick=1] (2.5,1.7) to[in=40,out=190] (1.8,0.9);
        \draw [black, thick=1]  (1.8,0.9)to[in=140,out=220] (1.8,-0.1);
        \draw [black, thick=1]  (1.8,-0.3)to[in=140,out=220] (1.8,-1.3);

 \draw [black, thick=1]  (2.05,-1.45)to[in=170,out=320] (2.5,-1.7);

    \draw [black, thick=1] (2.5,1.7) to[in=135,out=10] (3.7,0.3);         
    \draw [black, thick=1] (2.5,-1.7)to[in=225,out=350] (3.7,-0.3)  ;
    
    \draw [black, thick=1] (5.5,1) to[in=90,out=350] (6.3,0);
    \draw [black, thick=1] (5.5,1) to[in=45,out=170] (4.3,0.3);         
    \draw [black, thick=1]  (6.3,0)to[in=10,out=270] (5.5,-1);
    \draw [black, thick=1] (5.5,-1)to[in=315,out=190] (4.3,-0.3)  ;

    \draw [black, thick=1] (3.7,0.3) to (4.3,-0.3);
    \draw [black, thick=1] (3.7,-0.3) to (3.9,-0.1);
    \draw [black, thick=1] (4.1,0.1) to (4.3,0.3);

    \node at (0,0.3) {\footnotesize{$c_1$}}; 
    
    \node at (1.5,1) {\footnotesize{$b_{1}$}}; 
    \node at (1.5,-0.25) {\footnotesize{$b_{2}$}}; 
    \node at (1.5,-1.4) {\footnotesize{$b_{3}$}}; 
    
    \node at (4,0.3) {\footnotesize{$c_2$}}; 

    \end{tikzpicture}
    \caption{Trefoil }
    \label{trefoil}
\end{figure}

Consider the standard Legendrian trefoil $\Lambda \subset S^{3}$ described in Figure
\ref{trefoil}. Let us first consider the case when $\Lambda$ is marked $-$. With respect to the standard choice of orientation datum, the
Chekanov-Eliashberg DG-algebra $CE^{\ast}$ is then given by the free algebra
\[ \mathbb{K} \langle c_1, c_2, b_1,b_2,b_3 \rangle, \ \ |c_1|=|c_2|=-1, |b_{1}|=|b_{2}|=|b_{3}|=0 \]
and the non-trivial part of the differential can be read from Figure \ref{trefoil} as follows: 
\begin{align*}
    dc_1 &= 1 + b_{1} + b_{3} + b_{3}b_{2}b_{1} \\
    dc_2 &=-1 - b_{1} - b_{3} - b_{1}b_{2}b_{3}. 
\end{align*}

It is well known that $\Lambda$ has an exact Lagrangian torus filling. (In fact, there are at least
five of them, see \cite{EKH}).  Any of these can be obtained by doing surgery (pinch move) at
$b_1$, $b_2$ and $b_3$ in some order. Corresponding augmentations $\epsilon_L\colon CE^* \to \K$ are given by their values
$\epsilon(b_{1}), \epsilon(b_{2}), \epsilon(b_{3}) \in \K$ subject to the condition:
\[ 
1+ \epsilon(b_{1}) + \epsilon(b_{3}) + \epsilon(b_{1}) \epsilon(b_{2}) \epsilon(b_{3})  =0.   
\]

Note that since $CE^*$ is not simply-connected, Theorem \ref{mainthm} does not apply 
here. In fact, duality does not hold in this example. However, Theorem
\ref{comple} shows that there is a quasi-isomorphisms of completions  
\[
\widehat{CE}^* \cong \widehat{\Omega} C_*(T^2) \cong \K [[u, v]] 
\] 
where the latter is a power series ring in two commuting variables concentrated in degree 0. 

Thus, the completion map composed with the twisting cochain gives an algebra map: 
\[ 
H^0(CE^*) = \K \langle b_{1},b_{2},b_{3} \rangle / \langle 1+b_{1}+b_{3}+b_{3}b_{2}b_{1} ,
1+b_{1}+b_{3}+b_{1}b_{2}b_{3} \rangle   \to \K[[u,v]]. \] 
We claim that this map is injective. Indeed, observe that $H^0(CE^*)$ is a commutative
algebra: 
\[ 
d( c_1 + c_2 + b_{3}b_{2} c_1 + c_1 b_{2}b_{3}) = b_{3}b_{2} - b_{2}b_{3}, 
\]
and thus, $b_{2}$ commutes with $b_{3}$ in homology. Using this, one shows similarly that $b_{1}$ commutes with $b_{2}$ and $b_{3}$, see \cite{CaMu}. 
Hence, we have a completion map from a commutative ring to its completion
\[ 
H^0(CE^*) = \K [b_{1},b_{2},b_{3}] / (1+b_{1}+b_{3}+b_{1}b_{2}b_{3}) \to \K[[u,v]]. 
\]
It is a well known theorem, known as Krull intersection theorem, 
in commutative algebra that for any commutative Noetherian ring which is an
integral domain the completion map is injective. Thus, even though duality fails in this example, partial information can still be obtained by considering completions.

We next describe the twisting cochain $\t\colon LC_{\ast}(\Lambda)\to \Bar CF^{\ast}(L)^{\#}$ for one of the Lagrangian fillings $L$ of $\Lambda$. More precisely, we choose the filling which is obtained by pinching first at $b_{1}$, then at $b_{2}$, then at $b_{3}$, and finally filling the resulting unknots with Lagrangian disks, see \cite[Section 8.1]{EKH}. We think of the Lagrangian filling in as two disks connected by three twisted bands corresponding to the three pinchings. We next consider moduli spaces of holomorphic disks with boundary in $L$. Here \cite[Sections 4 and 5]{EKH} gives a description in terms of Morse flow trees which gives the following:
\begin{itemize}
\item $\mathcal{M}^{\fl}(b_{1})$ consists of one disk, $\delta_{1}^{1}$. 
\item $\mathcal{M}^{\fl}(b_{2})$ consists of two disks, $\delta_{2}^{1}$ and $\delta_{2}^{2}$.
\item $\mathcal{M}^{\fl}(b_{3})$ consists of two disks, $\delta_{3}^{1}$ and $\delta_{3}^{2}$
\end{itemize}
The boundaries of these disks are as follows.
\begin{itemize}
	\item $\partial \delta_{1}^{1}$ is fiber in the first twisted band.
	\item $\partial \delta_{2}^{1}$ is a fiber in the second twisted band, and $\partial \delta_{2}^{2}$ runs across the first and the second twisted band.
	\item $\partial\delta_{3}^{1}$ is a fiber in the third twisted band, and $\partial \delta_{3}^{2}$ runs across the second and the third twisted bands.
\end{itemize}
We next describe the moduli spaces $\mathcal{M}^{\fl}(c_{2})$, $\mathcal{M}^{\fl}(c_{1})$
in a completely analogous manner. The space has four components, all diffeomorphic to intervals, as follows:
\begin{itemize}
\item $\theta_{1}$ with one boundary point the disk in
    $\mathcal{M}^{\sy}(c_{2}b_{1})$ with $\delta_{1}^{1}$ attached, and the other the disk in
        $\mathcal{M}^{\sy}(c_{2}b_{1}b_{2}b_{3})$ with $\delta_{1}^{1}$, $\delta_{2}^{1}$, and $\delta_{3}^{2}$ attached.
\item $\theta_{2}$ with one boundary point the disk in
    $\mathcal{M}^{\sy}(c_{2}b_{3})$ with $\delta_{3}^{1}$ attached, and the other the disk in
        $\mathcal{M}^{\sy}(c_{2}b_{1}b_{2}b_{3})$ with $\delta_{1}^{1}$, $\delta_{2}^{2}$, and $\delta_{3}^{1}$ attached.
\item $\theta_{3}$ with one boundary point the disk in
    $\mathcal{M}^{\sy}(c_{2}b_{3})$ with $\delta_{3}^{2}$ attached, and the other the disk in
        $\mathcal{M}^{\sy}(c_{2}b_{1}b_{2}b_{3})$ with $\delta_{1}^{1}$, $\delta_{2}^{2}$, and $\delta_{3}^{2}$ attached.
\item $\theta_{4}$ with one boundary point the disk in $\mathcal{M}^{\sy}(c_{2})$ and the other the
    disk in $\mathcal{M}^{\sy}(c_{2}b_{1}b_{2}b_{3})$ with $\delta_{1}^{1}$, $\delta_{2}^{1}$, and $\delta_{3}^{1}$ attached.
\end{itemize}
Here the disks in $\theta_{4}$ sweeps the right hand disk of $L$, whereas the evaluation maps of
$\theta_{j}$, $1\le j\le 3$ does not map surjectively onto either disk. The moduli space
$\mathcal{M}^{\fl}(c_{1})$ also has four components only one of which sweeps the left hand disk of $L$.

In order to compute the Floer cohomology $CF^{\ast}(L)$ we pick a Morse function with one maximum $M$ and two saddle points $S_{1}$ and $S_{2}$. Letting the maximum lie in the right hand disk we find:
\[ 
\t(c_{1})=0,
\]
since the only way to rigidify a disk of dimension one is that its boundary passes the maximum $M$.

The twisting cochain can now in principle be computed from the moduli spaces $\mathcal{M}^{\fl}$ described above by attaching flow trees. To get an algebraically feasible twisting co-chain we first arrange the perturbation scheme so that 
\[ 
\t(b_{1})= S_{1}^{\vee}+S_{1}^{\vee}S_{2}^{\vee},
\] 
where the first term is $\delta_{1}^{1}$ with a flow line and the second with a flow tree, and next so that
\[ 
\t(b_{3})= -S_{1}^{\vee},
\] 
the disk $\delta_{2}^{2}$ with a flow line contributes, other contributions cancel (the disk $\delta_{2}^{2}$ with two flow lines and the same disk with a flow tree, the two distinct disks with one flow lines). Remaining parts of the twisting cochain are now determined from the co-product in the Floer homology:
\[ 
dM^{\vee}= S_{1}^{\vee}S_{2}^{\vee}-S_{2}^{\vee}S_{1}^{\vee},
\]
the twisting cochain equation, and the sweeping property of $\theta_{4}$, as follows: 
\[ 
\t(c_{2})=M^{\vee}\left(\frac{1+S_{1}^{\vee}}{1+S_{1}^{\vee}+S_{1}^{\vee}S_{2}^{\vee}}\right) .
\] 
and
\[ 
\t(b_{2})= \frac{S_{2}^{\vee}}{(1+S_1^{\vee}+S_{1}^{\vee}S_{2}^{\vee})}.
\]
It is possible to check that these power series agree with the geometric count. We leave out the
details but describe the mechanism. In order to arrange that only one flow line can be attached to
$\delta_{1}^{1}$ we order the stable manifolds of the flow line of the parallel copies so that if a
flow line (or flow tree) between copy $j$ and $j-l$ is attached then following $\delta_{1}^{1}$ we already passed all intersections with stable manifolds between copy $j-l$ and $k$ for $k<j-l$. Now, if the disk intersects the collection of stable manifolds in the opposite direction this means that we can jump down in all ways, which then gives the desired power series.

We finish this section by a brief discussion of the case $L$ decorated by $+$ and parallel copies. As usual this introduces two extra generators in addition to the Reeb chords above in $LC_{\ast}$, $x$, $|x|=0$ and $y$, $|y|=-1$, where $x$ is the counit corresponding to the minimum of the shifting function and $y$ is the maximum. In the reduced coalgebra (disregarding $x$) we get the new differential (using the augmentation $\epsilon_{L}$ above to change coordinates): 
\begin{align*}
dc_1 &= b_{1} + b_{3} - b_{3}b_{2}+b_{3}b_{2}b_{1} \\
dc_2 &=-y - b_{1} - b_{3} +b_{2}b_{3} - b_{1}b_{2}b_{3}. 
\end{align*}
In this case $CF^{\ast}$ is defined instead by choosing a shifting function that increases at infinity, and $CF^{\ast}$ is generated by the unit $u$, $|u|=0$ and $s_{1}$, $s_{2}$ with $|s_{j}|=1$. The new twisting cochain is
\[ 
\t(c_{1})=\t(c_{2})=0,\; \t(y)= \left(s_{1}^{\vee}s_{2}^{\vee}-s_{2}^{\vee}s_{1}^{\vee}\right)\left(\frac{1+s_{1}^{\vee}}{1+s_{1}^{\vee}+s_{1}^{\vee}s_{2}^{\vee}}\right),
\]
and $\t(b_{j})$ is exactly as above, after the substitutions $s_{j}\to S_{j}$, $j=1,2$.

\subsubsection{Mirror of $7_2$}\label{lenny}

We next discuss an example where Koszulity fails and also the completion map fails to be injective,
even though $CE^*$ is supported in non-positive degrees. This
example was shown to us by Lenhard Ng. 

\begin{figure}[h!]
    \centering
    \begin{tikzpicture}[scale=1]

    \tikzset{->-/.style={decoration={ markings,
                mark=at position #1 with {\arrow{>}}},postaction={decorate}}}

 \begin{scope}[scale=0.5, yshift=2.5cm]
    \draw [black, thick=1] (5.5,1) to[in=90,out=350] (6.3,0);
    \draw [black, thick=1, ->-=.3] (5.5,1) to[in=45,out=170] (4.3,0.3);         
    \draw [black, thick=1]  (6.3,0)to[in=10,out=270] (5.5,-1);
    \draw [black, thick=1] (5.5,-1)to[in=315,out=190] (4.3,-0.3)  ;

    \draw [black, thick=1] (3.7,0.3) to (4.3,-0.3);
    \draw [black, thick=1] (3.7,-0.3) to (3.9,-0.1);
    \draw [black, thick=1] (4.1,0.1) to (4.3,0.3);
  
    \node at (4,0.5) {\footnotesize{$a_1$}}; 
    
    \draw[black, thick=1] (3.7, 0.3) to[in=0 ,out=135] (1, 1);
   
    \draw[black, thick=1] (3.7, -0.3) to[in=45 ,out=225] (2.85, -1.15);
     \draw[black, thick=1] (2.65, -1.3) to[in=0 ,out=225] (1, -2);
     \node at (2.75, -1.8) {\footnotesize $b_1$}; 
     
    \draw[black, thick=1] (1, -2) to[in=0 ,out=180] (-1.7, -1);

    \end{scope}

    \begin{scope}[scale=0.5]
    \draw [black, thick=1] (5.5,1) to[in=90,out=350] (6.3,0);
    \draw [black, thick=1] (5.5,1) to[in=45,out=170] (4.3,0.3);         
    \draw [black, thick=1]  (6.3,0)to[in=10,out=270] (5.5,-1);
    \draw [black, thick=1] (5.5,-1)to[in=315,out=190] (4.3,-0.3)  ;

    \draw [black, thick=1] (3.7,0.3) to (4.3,-0.3);
    \draw [black, thick=1] (3.7,-0.3) to (3.9,-0.1);
    \draw [black, thick=1] (4.1,0.1) to (4.3,0.3);

    \node at (4,0.5) {\footnotesize{$a_2$}};

    \draw[black, thick=1] (3.7, 0.3) to[in=0 ,out=135] (1, 2);
    \draw[black, thick=1] (3.7, -0.3) to[in=0,out=225] (-1.8,-2.4);

    \draw[black, thick=1] (1, 2) to[in=80 ,out=180] (0.2, 0.9);
    \draw[black, thick=1] (0.1, 0.5) to[in=0 ,out=260] (-1.7, -1.5);
    \node at (0.4, 0.2) {\footnotesize $b_2$};  
        
    \draw[black, thick=1] (-1.7, -1.5) to[in=0 ,out=180] (-3.7, -1);
    
    \node at (-3.3,-0.65) {\footnotesize $b_3$}; 
    \draw[black, thick=1] (-3.7, -1) to[in=0 ,out=180] (-5.4, -1);
     \node at (-5.9,-1.35) {\footnotesize $b_4$}; 
    \draw[black, thick=1] (-5.8, -1) to[in=0 ,out=180] (-10.3, -1);
      \node at (-8.1,-0.65) {\footnotesize $b_5$}; 
    
    \node at (-10.8,-1.35) {\footnotesize $b_6$}; 

    \draw[black, thick=1] (-10.65, -1) to[in=225 ,out=180] (-13.3, 1.5);
    
    \draw[black, thick=1] (-11.4, -2.4) to[in=225 ,out=180] (-12.3, -0.7);
  
    \draw[black, thick=1] (-13.3, 1.5) to[in=180 ,out=45] (1, 3.5);
   
     \draw[black, thick=1] (-12.1, -0.4) to[in=180 ,out=45] (-1.7, 1.5);
  
     \node at (-12.8,-0.6) {\footnotesize $b_7$}; 

    \draw[black, dashed] (-5, 1.8) to (-5,3.7);
    \draw[black, dashed] (1, 2.2) to (1,3.5);

    \end{scope}

    \begin{scope}[scale=0.3, xshift=-8cm ]
    \draw [black, thick=1] (5.5,1) to[in=90,out=350] (6.3,0);
    \draw [black, thick=1] (5.5,1) to[in=45,out=170] (4.3,0.3);         
    \draw [black, thick=1]  (6.3,0)to[in=10,out=270] (5.5,-1);
    \draw [black, thick=1] (5.5,-1)to[in=315,out=190] (4.3,-0.3)  ;

     \node at (4,0.7) {\footnotesize{$a_3$}};

    \draw [black, thick=1] (3.7,0.3) to (4.3,-0.3);
    \draw [black, thick=1] (3.7,-0.3) to (3.9,-0.1);
    \draw [black, thick=1] (4.1,0.1) to (4.3,0.3);
    \draw [black, thick=1] (3.7,-0.3) to[in=90, out=225] (3, -1.5);
    \draw [black, thick=1] (3,-2.2) to[in=180, out=270] (5, -4);
    \end{scope} 

    \begin{scope}[scale=0.3, xshift=-16cm ]
    \draw [black, thick=1] (5.5,1) to[in=90,out=350] (6.3,0);
    \draw [black, thick=1] (5.5,1) to[in=45,out=170] (4.3,0.3);         
    \draw [black, thick=1]  (6.3,0)to[in=10,out=270] (5.5,-1);
    \draw [black, thick=1] (5.5,-1)to[in=315,out=190] (4.3,-0.3)  ;
 
  \node at (4,0.7) {\footnotesize{$a_5$}}; 
    
    \draw [black, thick=1] (3.7,0.3) to (4.3,-0.3);
    \draw [black, thick=1] (3.7,-0.3) to (3.9,-0.1);
    \draw [black, thick=1] (4.1,0.1) to (4.3,0.3);
    \draw [black, thick=1] (3.7,-0.3) to[in=90, out=225] (3, -1.5);
    \draw [black, thick=1] (3,-1.8) to[in=180, out=270] (5, -4);

    \end{scope}

     \begin{scope}[scale=0.3, xshift=-12.5cm, yshift=-3.3cm]
    \draw [black, thick=1] (5.5,1) to[in=90,out=350] (6.3,0);
    \draw [black, thick=1] (5.5,1) to[in=45,out=170] (4.3,0.3);         
    \draw [black, thick=1]  (6.3,0)to[in=10,out=270] (5.5,-1);
    \draw [black, thick=1] (5.5,-1)to[in=315,out=190] (4.3,-0.3)  ;

  \node at (4,-0.7) {\footnotesize{$a_4$}}; 
    
    \draw [black, thick=1] (3.7,0.3) to (4.3,-0.3);
    \draw [black, thick=1] (3.7,-0.3) to (3.9,-0.1);
    \draw [black, thick=1] (4.1,0.1) to (4.3,0.3);
    \draw [black, thick=1] (3.7,-0.3) to[in=0, out=225] (1.5, -0.7);
    
    \draw[black,thick=1] (3.7,0.3) to[in=225, out=135](3.3, 2);  
    \draw[black,thick=1] (3.3,2) to[in=135, out=45](8.2, 3.6);  
     
    \end{scope} 

 \begin{scope}[scale=0.3, xshift=-20.5cm, yshift=-3.3cm]
    \draw [black, thick=1] (5.5,1) to[in=90,out=350] (6.3,0);
    \draw [black, thick=1] (5.5,1) to[in=45,out=170] (4.3,0.3);         
    \draw [black, thick=1]  (6.3,0)to[in=10,out=270] (5.5,-1);
    \draw [black, thick=1] (5.5,-1)to[in=315,out=190] (4.3,-0.3)  ;

  \node at (4,-0.7) {\footnotesize{$a_6$}}; 
    
    \draw [black, thick=1] (3.7,0.3) to (4.3,-0.3);
    \draw [black, thick=1] (3.7,-0.3) to (3.9,-0.1);
    \draw [black, thick=1] (4.1,0.1) to (4.3,0.3);
    \draw [black, thick=1] (3.7,-0.3) to[in=0, out=225] (1.5, -0.7);
    
    \draw[black,thick=1] (3.7,0.3) to[in=225, out=135](3.3, 2);  
    \draw[black,thick=1] (3.3,2) to[in=135, out=45](8.2, 3.6);  
     
    \end{scope}

    \end{tikzpicture}
    \caption{Mirror of $7_2$}
    \label{ngknot}
\end{figure}

Consider the Legendrian knot $\Lambda$ drawn in Figure \ref{ngknot} decorated $-$. It is easy to see
that two pinch moves indicated by dashed lines in the Figure \ref{ngknot} gives a Legendrian
unknot, hence $\Lambda$ has a Lagrangian torus filling, call it $L$. Thus, we again have a
completion map:
\[ H^0(CE^*) \to \K [[u, v]] \]
where $K[[u,v]]$ is the commutative powerseries algebra in two variables.

$CE^*$ is given by the free algebra:
\[ \K \langle a_1, a_2, a_3, a_4, a_5, a_6, b_1 ,b_2, b_3, b_4, b_5,b_6, b_7 \rangle ,
|a_i|=-1, |b_i| =0 \]
The differential is given by 
\begin{align*}
    da_1 &= -1 + (1+b_1 b_2)b_7 + b_1(1+b_4b_3) (1+b_6 b_5) \\
    da_2 &= 1- b_3(1+ b_2b_1) \\
    da_3 &= 1+ b_3b_4 \\
    da_4 &= 1+ b_5 b_4 \\
    da_5 &= 1+ b_5 b_6 \\
    da_6 &= 1+ b_7 b_6 
\end{align*}

Taking the quotient of $H^0(CE^*)$ by letting $b_4=b_6$, $b_3=b_5=b_7$, $b_1=1$ and $b_2=-1-b_4$
gives
\[ \langle  b_3, b_4 \rangle / \langle 1+ b_3b_4 \rangle \]
which is a non-commutative algebra.

Thus, the completion map cannot be injective in this case. Otherwise, $H^0(CE^*)$ and thus any
quotient of it would have been commutative.

\subsection{Simply connected Legendrian submanifolds} 
Let $\Lambda\subset Y$ be a Legendrian
$(n-1)$-submanifold with $\pi_1(\Lambda)=1$ in the boundary $Y$ of a Weinstein $2n$-manifold $X$
that bounds an exact Lagrangian $L\subset X$. Assume that $c_{1}(X)=0$ and that the Maslov class of $L$ vanishes and that $L$ is relatively spin. Decorate $L$ by $-$. Our next result shows that if the symplectic homology of $X$ vanishes and if all Reeb
chords of $\Lambda$ have negative grading as generators of $CE^{\ast}(\Lambda)$ then $CE^{\ast}(\Lambda)$ is determined by the topology of $L$ and conversely. If $\Lambda$ is a sphere then we write $\overline{L}=L\cup_{\partial} D^{n}$ for
the closed manifold obtained by adding a disk to $L$ along $\Lambda$. 

\begin{thm} Suppose that $\Lambda =\Lambda^{-}$ is simply-connected. Assume that $SH^*(X)=0$ and that $CE^{\ast}(\Lambda)$ is supported in degrees $\le -1$. Then $L$ is simply connected. Moreover, if $\Lambda$ is a sphere then $CE^*(\Lambda)$ is isomorphic to $C_{-*}(\Omega \overline{L})$.
\end{thm}  

\begin{proof}
	Consider the wrapped Floer cohomology $HW_{\pi_1}^{\ast}(L)$ of $L$ with coefficients in $\Z[\pi_1(L)]$. Using our model for wrapped Floer cohomology in Section \ref{sec:CWnoHam}, a chain complex $CW_{\pi_{1}}^{\ast}(L)$ which calculates $HW_{\pi_{1}}^{\ast}(L)$ can be described as follows. Let $L=L_{0}$ and let $L_{1}$ be a parallel copy of $L$ shifted in the negative Reeb direction at infinity. The complex $CW_{\pi_{1}}^{\ast}(L)$ is then generated over $\Z[\pi_{1}]$ by the intersection points in $L_{0}\cap L_{1}$ which we call the Morse generators and the Reeb chords starting on $\Lambda_{0}$ and ending on $\Lambda_{1}$. The differential of a generator counts the usual rigid holomorphic strips, keeping track of the homotopy class of the loop obtained from the boundary component of the disk in $L_{1}$ completed by the reference paths connecting Reeb chord endpoints and intersection points to the base point. We point out that since there are no Reeb chords of degree $0$ the augmentation induced by $L$ is trivial and the high energy part of the differential on $CW^{\ast}_{\pi_{1}}$ counts honest holomorphic strips (without extra negative punctures at augmented Reeb chords).
	As for usual wrapped Floer homology, $HW_{\pi_{1}}^{\ast}(L)$ is naturally a module over symplectic cohomology $SH^*(X)$ and hence vanishes. 
    
    We next describe a geometric version of the complex $CW^{\ast}_{\pi_{1}}(L)$ that we call $CW^{\ast}_{\tilde p}$ and that also computes $HW^{\pi_1}(L)$. Let $\tilde p\colon\tilde L\to L$ denote the universal covering of $L$ and let $\tilde{\Lambda}=\tilde p^{-1}(\Lambda)$. Pick a Morse function $f\colon \tilde L\to\R$ with the following properties
    \begin{itemize}
    	\item $f$ has exactly one local maximum $M$. 
    	\item $f$ has no index $n-1$ critical points.
    	\item $f$ has no local minima.
    	\item $f$ is constant on $\tilde \Lambda$ where it attains its global minimum and if $\nu$ is the unit normal vector field along $\tilde{\Lambda}$ then $df(\nu)=-1$. 
    \end{itemize}
    The generators of $CW^{\ast}_{\tilde p}$ are of two types
    \begin{itemize}
    \item[(i)] The preimages of end points of Reeb chords $L_{0}\to L_{1}$ in $L\approx L_{1}$ under $\tilde p$ graded as the corresponding Reeb chord in $CE^{\ast}(\Lambda)$.
    \item[(ii)] The critical points of the Morse function $f\colon \tilde L\to \R$ graded by the negative of the Morse index.
    \end{itemize}  
    
    Let $M_{-\ast}$ denote the Morse chain complex of $f$ with cohomological grading, with generators as in $(ii)$ and  differential $\delta$ which counts negative gradient flow lines. Then $M_{-\ast}$ is supported in degrees $d$ with $-n\le d\le -1$ and $M_{-(n-1)}=0$. Let $C^{\ast}$ denote the complex generated by the generators of type $(i)$ and equip it with the differential $\partial$ that counts lifts of the boundary of holomorphic strips in the symplectization interpolating between Reeb chords. (This corresponds naturally to the high-energy part of the differential on $CW^{\ast}_{\pi_{1}}$.) 
    By our assumption on Reeb chord grading the grading of $C^{\ast}$ supported in degrees $d$ where $d\le -1$.
	
	We now define the complex $CW^{\ast}_{\tilde p}=C^{\ast}\oplus M_{-\ast}$ with differential
	\[ 
	d=\left(
	\begin{matrix}
	\partial & \phi \\
	0 & \delta
	\end{matrix}
	\right),
	\]
	where $\delta$ and $\partial$ are the differentials on $M^{\ast}$ and $C^{\ast}$ and where $\phi$ counts rigid lifts of disks with flow lines of $f$ attached. (This is the linear part of the map $\phi$ in \eqref{eq:deflooptotree}.) 
	
	The homology of $d$ is then isomorphic to $HW_{\pi_{1}}^{\ast}(L)$. To see this note that we can describe $CW^{\ast}_{\pi_{1}}(L)$ exactly as $CW^{\ast}_{\tilde p}(L)$ just replacing the Morse function $f$ above with a Morse function $h\circ \tilde p$, where $\tilde p$ is a Morse function on $L$ without minimum and with the required boundary behavior. Thus passing from $CW^{\ast}_{\pi_{1}}(L)$ to $CW^{\ast}_{\tilde{p}}(L)$ corresponds to deforming the Morse function on $\tilde L$ and it is well known that this induces a homotopy of complexes.  
    In particular $CW^{\ast}_{\tilde p}(L)$ is acyclic.
	
	We next want to show that $\pi_{1}(L)\approx 1$ or equivalently that the map $\tilde L\to L$ has
    degree one. To show this we first observe that since there are no Reeb chords of grading $0$ the
    augmentation of $CE^*(\Lambda)$ is trivial and the differential on $C^{\ast}$ counts honest holomorphic strips in the symplectization. This in turn means that the whole boundary of any holomorphic strip contributing to $\partial$ actually lies in $\Lambda\times\R$ and therefore cannot pick up any non-trivial $\Z[\pi_1]$-coefficient.
	
	Consider the following part of the chain complex $C^{\ast}\oplus M_{-\ast}$:
	\[ 
	\begin{CD}
	\dots @>>> C^{-(n+1)} @>>> C^{-n}\oplus M_{-n} @>>> C^{-(n-1)} @>>> \dots,
	\end{CD}
	\]
	where we use that $M_{-k}=0$ for $k=n-1$ and $k>n$. It follows from the above discussion and the vanishing of the wrapped homology $HW^{\ast}(L)$ with trivial coefficients that that the cohomology $HW^{-n}_{\pi_{1}}(L)$ in degree $-n$ has one generator for each non-trivial element in $\pi_1$. On the other hand $HW^{-n}_{\pi_{1}}(L)=0$ and we conclude that $\pi_1=1$.
	
	The statement about the isomorphism class of $CE^{\ast}(\Lambda)$ then follows from Corollary \ref{cor:CE=loopspace}.
\end{proof}

\appendix

\section{Basic results for moduli spaces}\label{sec:mdlispaces}
Consider as above a Weinstein manifold with an exact Lagrangian submanifold $(X,L)$, which outside a
compact set agrees with the positive part of the symplectization of the contact manifold with Legendrian
submanifold $(Y,\Lambda)$. We assume that the Maslov class of $L$ vanishes and that $L$ is relatively spin. We will consider several versions of punctured holomorphic spheres and disks
with boundary on $L$. The most basic disks we consider will lie either in $X$ or in the
symplectization $\R\times Y$. We call the former \emph{filling curves} and the latter
\emph{symplectization curves}. We will also consider a more general cobordism setting where, like in the
symplectization, disks may have both positive and negative punctures. Here we assume that $(W,K)$ is
a Weinstein cobordism with negative end $(-\infty,0]\times (Y,\Lambda)$ and positive end
$[0,\infty)\times(Z,\Gamma)$, where $Z$ is a contact manifold and $\Gamma$ is a Legendrian
submanifold.   We call disks in $W$ \emph{cobordism disks}.

When we will want to consider the relation of our `pre-surgery' invariants to `post- surgery' invariants, we will consider the case when $K$ decomposes
as a Lagrangian $C\subset W$ with positive end $\Gamma$
and empty negative end and a Lagrangian $L$ with negative end $\Lambda$ and empty positive end. We
further assume that there is a natural one-to-one correspondence between the components of $L^{v}$ of $L$ and $C^{v}$ of $C$, $v\in Q_{0}$, and that corresponding components $L^{v}$ and $C^{v}$ intersect transversely in one point $z^{v}$ and that $L^{v}\cap C^{w}=\varnothing$ if $v\ne w$.

We will describe a geometric setup that covers the cases considered below. Let $Y$ be the contact boundary of the Weinstein manifold $X$, where $c_{1}(X)=0$. Let $\Lambda\subset Y$ be a Legendrian with connected components $\Lambda_{1},\dots,\Lambda_{m}$. Let $D$ denote the unit disk in $\C$ and let $z_{1},\dots,z_{r}$ be boundary punctures and $\zeta_{1},\dots,\zeta_{k}$ be interior punctures. Let each component of $\partial D\setminus\{z_{1},\dots,z_{r}\}$ be decorated with a component $\Lambda_{j}$. The boundary punctures come in two types, positive and negative, all interior punctures are negative. Following \cite{Ersft} we make the further requirement that the disk be \emph{admissible}:

Any arc in $D$ that connects two boundary arcs in $\partial D\setminus\{z_{1},\dots, z_{r}\}$ subdivides the boundary punctures into two subsets. If both these subsets contain positive punctures then the labels of the two boundary segments at the endpoints of arc are different.

\begin{rem}
When we consider parallel copies of Lagrangians and Legendrians, boundary arcs labeled by different numbers in the numbering of parallel copies lie on copies shifted by different Morse functions and correspond to distinctly labeled boundary conditions in the current discussion.
\end{rem}

We also assume that the disk has one distinguished boundary puncture. Note that using a conformal model where the distinguished positive puncture lies at $1\in\partial D$ and an interior puncture $\zeta_{j}$ at the origin the positive real axis determines an asymptotic marker at $\zeta_{j}$ for each $j$. In the conformal model of the upper half plane with the distinguished puncture at infinity, this marker at any interior puncture is that determind by the vertical axis.  


\subsection{Moduli spaces of spheres for anchoring and compactifications of moduli spaces of disks}      

Following \cite{BEE}, all symplectization and cobordism disks we consider will be
anchored. This means that the actual disks we consider have, aside from their boundary punctures, also additional interior punctures where the maps are asymptotic to Reeb orbits at
the negative end. An anchored disk is such a disk completed by holomorphic planes in $X$
at all its negative interior punctures, see Figure \ref{anchored}.

When defining our version of the Chekanov-Eliashberg algebra $CE^{\ast}$,
use moduli spaces of anchored disks to parameterize chains of boundary paths.
When defining our versions of the Legendrian coalgebra $LC_{\ast}$ and the wrapped Floer $A_{\infty}$-algebra $CW^{\ast}$, we will need to consider disks in the symplectization with additional interior and boundary punctures, completed by rigid planes in $X$ and disks in $(X,L)$, respectively. We will also call such disks `anchored disks'.

Although, standard arguments using classical methods allow us to prove transversality for the disks with boundary punctures that we consider, anchored disks require transversality and gluing also for holomorphic planes in $(X,L)$ and that requires a more general perturbation scheme. Necessary perturbations for such curves were constructed in \cite{Esurgerycurves}. 

To state the relevant result let $Y$ denote the contact boundary of $X$ and consider a Reeb orbit $\gamma$ in $Y$ with a marker (i.e., a point $p\in\gamma$) on it. We write $\gamma'$ for the orbit $\gamma$ with a marker.
Let $\mathcal{M}(\gamma')$ denote the moduli space of holomorphic planes in $X$ with positive puncture with an asymptotic marker where the curve is asymptotic to $\gamma$ with the asymptotic maker mapping to the marker on $\gamma$. As in \cite[Theorem 1.1]{Esurgerycurves}, we define perturbation data $\lambda$ so that $\mathcal{M}^{\lambda}(\gamma')$ is a transversely cut out space of solutions to a perturbed Cauchy-Riemann equation $\bar\partial_{J_{\lambda}}u=0$, where $J_{\lambda}$ is a domain dependent almost complex structure that is allowed to depend also on the map $u$ in the neighborhood of $\mathcal{M}(\gamma')$. The moduli space $\mathcal{M}^{\lambda}(\gamma')$ furthermore has a natural compactification $\overline{\mathcal{M}}(\gamma')$ as a manifold with boundary with corners, where boundary strata correspond to several level spheres. Here levels not in moduli spaces of the form $\mathcal{M}^{\lambda}(\beta')$ just discussed lie in a moduli space $\mathcal{M}^{\sy,\lambda}(\beta',\boldsymbol{\eta}')$, where $\beta'$ is a Reeb orbit with marker and $\boldsymbol{\eta}'=\eta_{1}'\dots\eta_{k}'$ is a word of Reeb orbits with markers. Elements in $\mathcal{M}^{\sy,\lambda}(\beta',\boldsymbol{\eta}')$ are maps $u\colon S\to\R\times Y$ of a punctured spheres $S$ into the symplectization. There are fixed cylindrical ends $S^{1}\times[0,\infty)$ near the punctures in $S$ that are compatible with breaking in the sense of \cite[Section 2.1]{EO}. The map $u$ takes $(1,\infty)$ at a puncture to the marker of the corresponding Reeb orbit and $u$ again solves a perturbed Cauchy-Riemann equation $\bar\partial_{J_{\lambda}} u=0$, where $J_{\lambda}$ is domain dependent and only depends on the angular coordinate in the ends near the punctures, has a positive puncture at $\beta'$ and negative punctures at the orbits in $\boldsymbol{\eta}'$. 

We refer to \cite[Section 2.4]{Esurgerycurves} for more details on $\overline{\mathcal{M}}(\gamma')$. Here we only point out that the asymptotic marker at the positive puncture of a curve in the compactification determines asymptotic markers at all negative punctures and that the level structure is compatible with this in the sense that the asymptotic marker at the positive puncture in a lower level curve agrees with the asymptotic marker at the negative puncture where it is attached.

We next consider holomorphic disks in a cobordism $(Z,K)$ with positive and negative ends $(\partial_{\pm} Z,\partial_{\pm}K) $. We include also the case when the cobordism $(Z,K)$ is trivial, i.e. the symplectization $(\R\times Y,\R\times\Lambda)$, with $(\partial_{\pm}Z,\partial_{\pm} K)=(Y,\Lambda)$. Let $\mathbf{c}$ be a word of Reeb chords of $\partial_{\pm} K$. Let $\boldsymbol{\gamma}=\gamma_{1}\dots\gamma_{k}$ be a word of Reeb orbits in $\partial_{-}Z$. We define  $\mathcal{M}^{\mathrm{neg}}(\mathbf{c},\boldsymbol{\gamma}')$ to be the moduli space of punctured holomorphic disks in $Z$ with boundary on $K$, with boundary punctures mapping to Reeb chords in the word $\mathbf{c}$ and one distinguished boundary puncture, and with additional negative interior punctures at $\zeta_{1},\dots,\zeta_{k}$ mapping to Reeb orbits in the word $\boldsymbol{\gamma}$. Note that the distinguished boundary puncture determines an asymptotic marker at each interior puncture $\zeta_{j}$ that determines a marker on the corresponding $\gamma_{j}$. Let $\boldsymbol{\gamma}'$ denote the corresponding word or Reeb orbits with markers.  Below we will show that such moduli spaces of disks with interior negative punctures with markers that are relevant to our study cannot contain multiple covers and are transversely cut out for generic almost complex structure. 

Recall the symplectic filling $X$ of the negative end $Y$ of the cobordism above. We will use punctured spheres curves in $X$ in the compactification of the moduli spaces $\overline{\mathcal{M}}^{\lambda}(\gamma')$ to fill the interior punctures of the disks and treat them as disks with only boundary punctures. For this purpose, we define the moduli space of anchored disks $\mathcal{M}^{\mathrm{anc}}(\mathbf{c})$ as
\[ 
\mathcal{M}^{\mathrm{anc}}(\mathbf{c})=\bigcup_{\boldsymbol{\gamma}'}\bigl(\mathcal{M}^{\mathrm{neg}}(\mathbf{c},\boldsymbol{\gamma}')\times\prod_{\gamma_{j}'\in\boldsymbol{\gamma}'}\mathcal{M}^{\lambda}(\gamma_{j}')\bigr),
\]
where markers on Reeb orbits are induced from the distinguished boundary puncture. Note here that the topology on the moduli space of anchored curves is the product topology. This means in particular that the dimension of the boundary evaluation map equals the dimension $\dim(\mathcal{M}^{\mathrm{anc}}(\mathbf{c}))$ only on components $\mathcal{M}^{\mathrm{neg}}(\mathbf{c},\boldsymbol{\gamma}')\times\prod_{\gamma_{j}'\in\boldsymbol{\gamma}'}\mathcal{M}^{\lambda}(\gamma_{j}')$, where $\dim(\mathcal{M}^{\lambda}(\gamma_{j}'))=0$, for all $j$.

We consider next the case when the cobordism is trivial $(\R\times Y,\R\times \Lambda)$ and when all punctures mapping to chords in $\mathbf{c}$ are positive. The above construction then gives a stratification of the moduli space $\mathcal{M}(\mathbf{c})$ of holomorphic disks in $(X,L)$ as follows. First, since all moduli spaces $\mathcal{M}^{\mathrm{neg}}(\mathbf{c},\boldsymbol{\gamma}')$ are transversely cut out, the corresponding moduli spaces $\mathcal{M}^{\mathrm{neg},\lambda}(\mathbf{c},\boldsymbol{\gamma}')$, where a small perturbation near the negative ends corresponding to the perturbation $\lambda$ of holomorphic planes with asymptotic marker has been turned on, is canonically diffeomorphic to $\mathcal{M}^{\mathrm{neg}}(\mathbf{c},\boldsymbol{\gamma}')$. Gluing the curves in $\mathcal{M}^{\mathrm{neg},\lambda}(\mathbf{c},\boldsymbol{\gamma}')$ to the curves in $\mathcal{M}^{\lambda}(\gamma_{j}')$, and extending the perturbation, we get a compactification of the moduli space $\mathcal{M}(\mathbf{c})$ with boundary given by tree configurations of anchored disks. Near a broken configuration the moduli space is a manifold with boundary with corners with corner structure induced by the gluing parameters, compare \cite[Sections 6.4 -- 6.6]{ES}.

For example, if $\mathbf{c}=c$ is a single Reeb chord then $\mathcal{M}(c)$ has a natural compactification with boundary of the form 
\[ 
\mathcal{M}^{\mathrm{anc}}(c,\mathbf{b})\times\prod_{b_{j}\in\mathbf{b}}\mathcal{M}^{\mathrm{anc}}(b_{j}).
\]

All moduli spaces of disks considered in this paper will be anchored and we will drop the superscript $\mathrm{anc}$ from the notation.

\begin{figure}[h!]
    \centering
    \begin{tikzpicture}[scale=1.5]
    \tikzset{->-/.style={decoration={ markings,
                mark=at position #1 with {\arrow{>}}},postaction={decorate}}}
    \draw [black, thick=1.5] (2,3) to[in=75,out=270] (1.6,0.95);
     \draw [black, thick=1.5] (1.53,0.55) to[in=90,out=260] (1.5,0); 
        
    \draw [black, thick=1.5] (2,0) to[in=90,out=90] (2.5,0);         
    \draw [black, thick=1.5] (3,0) to[in=90,out=90] (3.5,0);         
        \draw [black, thick=1.5] (2.5,3) to[in=90,out=270] (4,0);
    \draw [black, thick=1 ] (1.5,0) to (2,0); 
    \draw [black, thick=1 ] (2.5,0) to (3,0); 
    \draw [black, thick=1 ] (3.5,0) to (4,0); 
    \draw [black, thick=1] (2,3) to (2.5,3); 

    \draw [black, thick=1.5] (2.2,1.5) to[in=90,out=270] (-0.5,0);
    \draw [black, thick=1.5] (2.7,1.5) to[in=90,out=270] (1,0);
     \draw [black, thick=1.5] (0.5,0) to[in=60,out=120] (0,0); 

    \draw [black, thick=1.5, dashed] (-0.5,0) to[in=90,out=90] (0,0); 
    \draw [black, thick=1.5] (-0.5,0) to[in=270,out=270] (0,0); 
    \draw [black, thick=1.5, dashed] (0.5,0) to[in=90,out=90] (1,0); 
        \draw [black, thick=1.5] (0.5,0) to[in=270,out=270] (1,0); 

    \draw [black, thick=1.5, dashed] (-0.5,-0.35) to[in=90,out=90] (0,-0.35); 
    \draw [black, thick=1.5] (-0.5,-0.35) to[in=270,out=270] (0,-0.35); 
    \draw [black, thick=1.5, dashed] (0.5,-0.35) to[in=90,out=90] (1,-0.35); 
        \draw [black, thick=1.5] (0.5,-0.35) to[in=270,out=270] (1,-0.35); 

    \draw [black, thick=1.5] (-0.5,-0.35) to[in=180,out=270] (-0.25,-0.7); 
    \draw [black, thick=1.5] (0,-0.35) to[in=0,out=270] (-0.25,-0.7); 
  \draw [black, thick=1.5] (0.5,-0.35) to[in=180,out=270] (0.75,-0.7); 
    \draw [black, thick=1.5] (1,-0.35) to[in=0,out=270] (0.75,-0.7); 

\end{tikzpicture}
    \caption{Anchored disks}
    \label{anchored}
\end{figure}

\subsection{Moduli spaces of anchored disks}\label{ssec:perturbationsanchor}
Consider a system of parallel copies $\bar L=\{L_j\}_{j=0}^{\infty}$, where $L_0=L$ is an embedded Lagrangian, as in Section \ref{sec:parallel}. This induces a system $\bar\Lambda=\{\Lambda_j\}_{j=0}^{\infty}$ of parallel copies of $\Lambda$ in $Y$.
We first discuss numberings that determine the boundary conditions of the holomorphic disks that we use in the various theories considered.
Let $D_{m}$ denote the unit disk in the complex plane with $m$ boundary punctures
$\zeta_{1},\dots,\zeta_{m}$. One of the boundary punctures is distinguished. We choose notation so
that $\zeta_{1}$ is distinguished. The $m$ punctures subdivides the boundary of $D_{m}$ into $m$
boundary arcs. We will consider disks with numbered boundary arcs where the numbers correspond to
the parallel copies $L_{j}$. We will consider two types of numberings \emph{increasing} and
\emph{decreasing}. Traversing the boundary of the disk across a boundary puncture in the positive direction, the numbering increases, remain constant, or decreases as we pass the puncture. We call punctures increasing, constant and decreasing accordingly. We call a disk increasing (decreasing) if all its non-distinguished punctures are either increasing (decreasing) or constant. Then its distinguished puncture is decreasing (increasing) or constant. 

\begin{figure}[h!]
    \centering
    \begin{tikzpicture}[scale=1]
    \tikzset{->-/.style={decoration={ markings,
                mark=at position #1 with {\arrow{>}}},postaction={decorate}}}

        \draw[thick, black] (-1,0) arc (0:90:1);
        \draw[thick, black] (-2,1) arc (90:180:1);
        \draw[thick, black] (-3,0) arc (180:240:1);
        \draw[thick, black] (-2.5,{-sin(60)} ) arc (240:300:1);
        \draw[thick, black] (-1.5,{-sin(60)} ) arc (300:360:1);

           \draw[thick, fill=black] (-2,1) circle(.05);
        \draw[thick, fill=black] (-1.5, {-sin(60)} ) circle(.05); 
        \draw[thick, fill=black] (-2.5, {-sin(60)} ) circle(.05);

        \draw[thick, fill=black] (-1,0) circle(.05);
        \draw[thick, fill=black] (-3,0) circle(.05);

            \node at (-0.7,0.6) {$L_{\kappa_1}$}; 
            \node at (-3.2,0.6) {$L_{\kappa_0}$}; 
             \node at (-0.8,-0.7) {$L_{\kappa_2}$}; 
    \node at (-3.2,-0.6) {$L_{\kappa_4}$}; 
    \node at (-2,-1.3) {$L_{\kappa_3}$};

            \node at (-3.3,0) {$c$};

         \draw[thick, black] (3,0) arc (0:90:1);
        \draw[thick, black] (2,1) arc (90:180:1);
        \draw[thick, black] (1,0) arc (180:240:1);
        \draw[thick, black] (1.5,{-sin(60)} ) arc (240:300:1);
        \draw[thick, black] (2.5,{-sin(60)} ) arc (300:360:1);

           \draw[thick, fill=black] (2,1) circle(.05);
        \draw[thick, fill=black] (2.5, {-sin(60)} ) circle(.05); 
        \draw[thick, fill=black] (1.5, {-sin(60)} ) circle(.05);

        \draw[thick, fill=black] (3,0) circle(.05);
        \draw[thick, fill=black] (1,0) circle(.05);

            \node at (3.3,0.6) {$L_{\kappa_0}$}; 
            \node at (0.8,0.6) {$L_{\kappa_1}$}; 
             \node at (3.2,-0.7) {$L_{\kappa_4}$}; 
    \node at (0.8,-0.6) {$L_{\kappa_2}$}; 
    \node at (2,-1.3) {$L_{\kappa_3}$};

    \node at (3.3,0) {$x$};

    \end{tikzpicture}
        \caption{strictly decreasing(left) and strictly increasing(right) }

\end{figure}

When defining operations $\Delta_{k}$ for the the Legendrian $A_{\infty}$-coalgebra $LC_{\ast}$, we count anchored increasing disks in the symplectization asymptotic to Reeb chords at all punctures. When defining the operations for the Lagrangian $A_{\infty}$-algebras $CF^{\ast}$ and $CW^{\ast}$, we count decreasing disks in $X$. When defining the twisting cochain 
$\mathfrak{t}\colon LC_{\ast}^{\parallel} \to \Omega(\mathbf{k}_{-}\oplus CF_{\ast})$, 
we count increasing disks in X which are asymptotic to a Reeb chord at the distinguished puncture, and to Lagrangian intersection points at the other punctures.

We next consider asymptotic conditions at the boundary punctures. There are two basic forms of asymptotics, a puncture is either asymptotic to a Lagrangian intersection point or to a Reeb chord. The former case is the standard form of asymptotics in Lagrangian Floer theory and the latter in Legendrian DG-algebras. More precisely, we choose an almost complex structure on $X$ which in the cylindrical end $\R\times Y$ is invariant under $\R$-translation, leaves the contact planes invariant and is compatible with the symplectic form induced on the contact planes by the contact form. Furthermore, it pairs the $\R$-direction with the Reeb direction. This means in particular that the Reeb chord strip which is the product of a Reeb chord and $\R$ is holomorphic. Reeb chord asymptotics means, and we study holomorphic disks with boundary punctures that are asymptotic to these Reeb chord solutions,  while Lagrangian intersection point asymptotics means we study holomorphic disks that are asymptotic to constant strips at intersection points. 

First we consider disks in the filling $(X,L)$. Consider a disk $D_{m}$ as above with strictly increasing or decreasing numbering $\kappa=(\kappa_1,\dots,\kappa_{m})$ and let $\mathbf{a}=a_{1}\dots a_{m}$ be a word of Reeb chords and Lagrangian intersection points in $L_{0}\cap L_1$. We let $\mathcal{M}^{\fl}(\mathbf{a};\kappa)$ denote the moduli space of holomorphic disks $u\colon (D_{m},\partial D_{m})\to (X,\bar L)$ that satisfies the following conditions:
\begin{itemize}
\item $u$ takes the boundary component labeled by $\kappa_{j}$ to the Lagrangian $L_{\kappa_j}$.
\item $u$ is asymptotic to the unique Reeb chord or Lagrangian intersection point of $L_{\kappa_j}$ and $L_{\kappa_{j+1}}$ near $a_{j}$ at $\zeta_{j}$, where we let $\kappa_{m+1}=\kappa_{1}$.  	
\end{itemize}

We next consider disks in the symplectization.  
Consider the disk $D_{m+k}$ with increasing or decreasing boundary numbering $\kappa'$ and punctures $\zeta_{1},\dots,\zeta_{m+k}$.  We next note that in the symplectization there are two possible Reeb chord asymptotics, positive or negative according to the sign of the $t$-coordinate near the puncture. Let $\mathbf{c}'=c_{1}^{\sigma_{1}}\dots c_{m+k}^{\sigma_{m+k}}$ be a word of signed Reeb chords of $\Lambda_{0}\cup \Lambda_1$, where $\sigma\in\{+,-\}$ is a sign. 
If $m>1$ then we require that the Reeb chords $c_{r}$ at all constant punctures connect $\Lambda_{0}$ to $\Lambda_{0}$ and that their signs are all negative, $\sigma_{r}=-1$. (These constant punctures will be capped by augmentation disks.) We let $\mathcal{M}^{\sy,\circ}_{\parallel}(\mathbf{c}';\kappa')$ denote the moduli space of anchored holomorphic disks $v\colon (D_{m},\partial D_{m})\to (\R\times Y,\R\times \Lambda)$ that satisfy the following conditions:
\begin{itemize}
	\item $v$ takes the boundary components labeled by $\kappa_{j}$ to the Lagrangian $\Lambda_{\kappa_j}$.
	\item $v$ is asymptotic at positive or negative infinity according to the sign of $\sigma_{j}$ to the unique Reeb chord between $\Lambda_{\kappa_j}$ and $\Lambda_{\kappa_{j+1}}$ near $c_{j}$ at a puncture $\zeta_{j}$.  		 
\end{itemize}
If $\mathbf{c}$ is a word of strictly increasing (decreasing) Reeb chords then we define the moduli space $\mathcal{M}^{\sy}_{\parallel}(\mathbf{c};\kappa)$ by anchoring also at constant boundary punctures:
\[ 
\mathcal{M}^{\sy}_{\parallel}(\mathbf{c};\kappa)=
\bigcup_{\mathbf{c}\subset\mathbf{c}'}\left(\mathcal{M}^{\sy,\circ}_{\parallel}(\mathbf{c}';\kappa')\times\prod_{c_{r}\in\mathbf{c}'\setminus\mathbf{c}}\mathcal{M}(c_{r})\right),
\] 
where the union runs over all words $\mathbf{c}'$ extending $\mathbf{c}$ by constant punctures.
 
We will also consider a simpler version with only one copy of the Legendrian $\Lambda=\Lambda_{0}$ in the case that exactly one puncture in $\mathbf{c}$ is positive and all other are negative. In the language above all punctures of such a disk are constant and we write $\mathcal{M}^{\sy}(\mathbf{c})$ for the moduli space of such disks. The (constant) negative punctures of disks in this moduli space are typically not filled by augmentation disks.

Finally, we consider disks in the cobordism and in the filled cobordism. We start with the cobordism
disks without filling. 
with the cobordism disks without filling. Recall that we assume a decomposition
$K = C \cup L$ where $C$ has no negative end, and $L$ has no positive end. Consider a system of parallel copies $\bar C=\{C_{j}\}_{j=0}^{\infty}$.
Consider the disk $D_{i+j+2}$ where we fix two punctures that subdivide the boundary of the disk
into two arcs, upper and lower. Let $\kappa$ be a decreasing boundary numbering of the boundary
components in the upper arc and extend it to a constant numbering in the lower arc. Let
$\mathbf{c}_{0}=c_{0;1}\dots c_{0;j}$ be a composable word of Reeb chords connecting $\Lambda_{v}$
to $\Lambda_{w}$, and let $\mathbf{c}'=c_{i}\dots c_{1}$ be a word of Reeb chords of $\Gamma$. Consider the word of Reeb chords and intersection points 
\[ 
\mathbf{c}=c_{0;1}\dots c_{0;j}z^{w}c_{i}\dots c_{1}z^{v}
\]
and let 
\[ 
\mathcal{M}^{\co}(\mathbf{c};\kappa)
\]
denote the moduli space of holomorphic disks $u\colon (D_{i+j+2},\partial D_{i+j+2})\to (W,\bar C\cup L)$ such that the following holds.
\begin{itemize}
	\item $u$ is asymptotic to the Reeb chord $c_{0;r}$ at its $r^{\rm th}$ constant puncture, takes adjacent boundary arcs to $L$, and neighboring punctures to the unique intersection point near $z^{w}$ in $L\cap C^{w}_{\kappa_{i}}$ and near $z^{v}$ in $L\cap C^{v}_{\kappa_{i}}$, respectively. 
	\item On remaining boundary arcs and punctures the boundary maps as described by $\mathbf{c}$ and the numbering $\kappa$, exactly as above.
\end{itemize}

\begin{figure}[h!]
    \centering
    \begin{tikzpicture}[scale=1]
    \tikzset{->-/.style={decoration={ markings,
                mark=at position #1 with {\arrow{>}}},postaction={decorate}}}

        \draw[thick, black] (-1,0) arc (0:60:1);
            \draw[thick, black] (-1.5,{sin(60)}) arc (60:120:1);
            \draw[thick, black] (-2.5,{sin(60)}) arc (120:180:1);
        \draw[thick, black] (-3,0) arc (180:240:1);
        \draw[thick, black] (-2.5,{-sin(60)} ) arc (240:300:1);
        \draw[thick, black] (-1.5,{-sin(60)} ) arc (300:360:1);

            \draw[thick, fill=black] (-1.5,{sin(60)}) circle(.05);
            \draw[thick, fill=black] (-2.5,{sin(60)}) circle(.05);
            \draw[thick, fill=black] (-1.5, {-sin(60)} ) circle(.05); 
        \draw[thick, fill=black] (-2.5, {-sin(60)} ) circle(.05);

        \draw[thick, fill=black] (-1,0) circle(.05);
        \draw[thick, fill=black] (-3,0) circle(.05);

            \node at (-2, 1.3) {$C_{\kappa_1}$};
            \node at (-0.7,0.6) {$C_{\kappa_2}$}; 
            \node at (-3.2,0.6) {$C_{\kappa_0}$}; 
             \node at (-0.8,-0.7) {$L$}; 
    \node at (-3.2,-0.6) {$L$}; 
    \node at (-2,-1.3) {$L$}; 
    \node at (-3.5, 0) {$z^v$};
     \node at (-0.5, 0) {$z^w$};
     \node at (-2.5, {-sin(60)-0.3}) {$c_1$};
     \node at (-1.5, {-sin(60)-0.3}) {$c_2$};
            \node at (-2.6, {sin(60)+0.2}) {$c_{0;1}$};
            \node at (-1.3, {sin(60)+0.2}) {$c_{0;2}$};

\end{tikzpicture}
\caption{Disk contributing to $\mathcal{M}^{\co}(\mathbf{c};\kappa)$
  }

\end{figure}

The disks in the filled cobordism are entirely analogous. Here we assume that $L$ is a Lagrangian submanifold in $X=X_{0}\cup W$. Consider again a system of parallel copies $\bar C=\{C_{j}\}_{j=0}^{\infty}$ and also a system of parallel copies $\bar L=\{L_{j}\}_{j=0}^{\infty}$ of $L$. Consider the disk $D_{i+j+2}$ where we fix two punctures that subdivide the boundary of the disk into two arcs, upper and lower. Let $\kappa$ be a decreasing boundary numbering of the boundary components in the upper and lower arcs. Let $\mathbf{x}_{0}=x_{0;1}\dots x_{0;j}$ be a word of intersection points of $L$ and let $\mathbf{c}'=c_{i}\dots c_{1}$ be a word of Reeb chords of $\Gamma$. Consider the word of Reeb chords and intersection points 
\[ 
\mathbf{c}=x_{0;1}\dots x_{0;j}z^{w}c_{i}\dots c_{1}z^{v}
\]
and let 
\[ 
\mathcal{M}^{\overline{\co}}(\mathbf{c};\kappa)
\]
denote the moduli space of holomorphic disks $u\colon (D_{i+j+2},\partial D_{i+j+2})\to (X,\bar C\cup \bar L)$ such that the following holds.
\begin{itemize}
	\item $u$ is asymptotic to the intersection point $x_{0;r}$ at its $r^{\rm th}$ puncture in the lower arc, takes adjacent boundary arcs to $L$, and neighboring punctures to the unique intersection point near $z^{w}$ in $L_{\kappa_{j}}^{w}\cap C^{w}_{\kappa_{j+1}}$ and the unique intersection point near $z^{v}$ in $K_{\kappa_{1}}\cap C^{v}_{\kappa_{i+j}}$. 
	\item On remaining boundary arcs and punctures the boundary maps as described by $\mathbf{c}$ and the numbering $\kappa$, exactly as above.
\end{itemize}

The disks in the cobordism without filling will be used to map into the dg-algebra of the negative end whereas the disks in the filled cobordism will be used to map into the Floer cohomology of a Lagrangian. This is why we use parallel copies in one case but not the other.   

The formal dimension of the moduli spaces above is computed in terms of the negative of a Conley-Zehnder index $\mathrm{CZ}$ of the Reeb chords. Recall $\mathrm{CZ}(a)$ of a Reeb chord $a$ of $\Lambda$ as defined for example in \cite[Section 2.1]{BEE}: we pick paths connecting base points in the boundary of the various components of $L$, and paths connecting Reeb chord endpoints to the base points. We define $\mathrm{CZ}(a)$ to be the Maslov index of this path closed up by a positive rotation in the contact plane, and the grading $|a|=-\mathrm{CZ}(a)$. For a Lagrangian intersection $x$ between $L^{1}$ and $L^{2}$ we similarly pick paths connecting to the base points and use these to form a loop $\gamma$ starting in $L^{2}$ and ending in $L^{1}$ and define $\mathrm{CZ}(x)$ to be the Maslov index of the loop of Lagrangian planes that results from closing up the path of Lagrangian planes along $\gamma$ by a positive rotation, and $|x|=-\mathrm{CZ}(x)$. The Conley-Zehnder index is independent of the base point paths since the Maslov class vanishes.

\begin{rem}
The above gradings are related to the grading $|\cdot|_{\mathrm{Leg}}$ in the Legendrian contact homology algebra, see e.g. \cite{EESPxR,BEE} as follows: 
\[ 
|c|_{\mathrm{Leg}}=-|c|-1.
\]
Note that gradings generally depend on the choice of paths connecting endpoints to the base point: two such choices differ by a loop and the grading is shifted by the Maslov index of that loop. In particular, if the Maslov class vanishes the grading is well defined. Also, the paths connecting tangent planes at base points in different components are defined only up to choice. Changing the homotopy class shifts the Maslov potential between components and indices of mixed Reeb chords accordingly.  
\end{rem} 
 
\begin{lem}
The formal dimension of the moduli space $\mathcal{M}^{\fl}(\mathbf{a};\kappa)$ equals
\[ 
\dim\left(\mathcal{M}^{\fl}(\mathbf{a};\kappa)\right)=
(n-3) - \sum_{j=1}^{m}(|a_{j}|-(n-2)).
\]
The formal dimension of the moduli spaces $\mathcal{M}^{\sy}_{\parallel}(\mathbf{c};\kappa)$ equals
\[ 
\dim\left(\mathcal{M}^{\sy}_{\parallel}(\mathbf{c};\kappa)\right)=
(n-3) + \sum_{\sigma_{j}=-1} (|c_{j}|+1) - \sum_{\sigma_{j}=+1}(|c_{j}|-(n-2)).
\]
The formal dimension of the moduli spaces $\mathcal{M}^{\sy}(\mathbf{c})$ equals
\[ 
\dim\left(\mathcal{M}^{\sy}(\mathbf{c})\right)=
(n-3) + \sum_{\sigma_{j}=-1} (|c_{j}|+1) - \sum_{\sigma_{j}=+1}(|c_{j}|-(n-2)).
\]
The formal dimension of the moduli space $\mathcal{M}^{\co}(\mathbf{c};\kappa)$ equals
\[ 
\dim\left(\mathcal{M}^{\co}(\mathbf{c};\kappa)\right)=
1 - \sum_{r=1^{i}}(|c_{r}|-(n-2)) + \sum_{s=1}^{j} (|c_{0;s}|+1).
\]
The formal dimension of the moduli space $\mathcal{M}^{\overline{\co}}(\mathbf{c};\kappa)$ equals
\[ 
\dim\left(\mathcal{M}^{\overline{\co}}(\mathbf{c};\kappa)\right)=
1 - \sum_{r=1^{i}}(|c_{r}|-(n-2)) - \sum_{s=1}^{j} (|x_{0;s}|-(n-2)).
\]
\end{lem}

\begin{proof}
See \cite[Theorem A.1]{CEL}.
\end{proof}

We next study topological properties of the moduli spaces just defined. It turns out to be comparatively simple because of two key features. First, since we require our disks to switch copies at punctures ``in the same direction'' they cannot be multiply covered, and second, for the same reason there can be no boundary splitting. As in \cite{Ersft}, the first property allows us to prove transversality by perturbing the almost complex structure and the second shows that the moduli spaces admit compactifications consisting only of punctured curves joined at Reeb chords or Lagrangian intersection points. Precise formulations of these results are as follows. 

\begin{thm}\label{thm:mdlitv}
For generic almost complex structure $J$ the moduli spaces $\mathcal{M}^{\fl}(\mathbf{a};\kappa)$, $\mathcal{M}^{\sy}_{\parallel}(\mathbf{c};\kappa)$, $\mathcal{M}^{\sy}(\mathbf{c})$, $\mathcal{M}^{\co}(\mathbf{c};\kappa)$, and $\mathcal{M}^{\overline{\co}}(\mathbf{c};\kappa)$ are transversely cut out manifolds of dimensions $\dim(\mathcal{M}^{\fl}(\mathbf{a};\kappa))$ and $\dim(\mathcal{M}^{\sy}_{\parallel}(\mathrm{c});\kappa)$, $\dim(\mathcal{M}^{\sy}(\mathrm{c}))$,  $\dim(\mathcal{M}^{\co}(\mathrm{c});\kappa)$, and $\dim(\mathcal{M}^{\overline{\co}}(\mathrm{c});\kappa)$, respectively. 
\end{thm}  

\begin{proof} 
	A well-known argument gives transversality for disks that are somewhere injective on the boundary by perturbing the almost complex structure: any element in the cokernel of the linearized operator must be zero on the set of injectivity and then identically zero by unique continuation. Our disks are not necessarily somewhere injective but the disks cannot be multiple covers and there is a region with the property of the region of injectivity above. We briefly recall the argument for this from \cite[Lemma 4.5]{EESPxR}. 
	Fix a puncture of $u$. 
	
	Consider first the more difficult case when this puncture maps to a Lagrangian intersection.
    Pick coordinates so that the intersection point lies at the origin in $\Cc^{n}$ and so that the
    two Lagrangians correspond to $\R^{n}$ and $i\R^{n}$. Let $(x_{1}+iy_{1},\dots,x_{n}+iy_{n})$ be
    standard coordinates on $\Cc^{n}$. Consider the complex hyperplanes
    $H_{\pm\epsilon}=\{x_{1}+iy_{1}=\pm\epsilon(1+i)\}$. Looking
    at the Fourier expansion of $u$ near the puncture it is clear that for suitable coordinates
    (such that the leading Fourier coefficient of $u$ lies in direction of the first coordinate) the
    number of intersection points of the image of $u$ and $H_{\pm\epsilon}$ near the puncture have
    different parities depending on the sign of $\epsilon$. By analytic continuation, other disks
    and half disks mapping there either have images agreeing completely in the ball or there are
    injective points of the disk near the puncture. In the case where they agree completely we note
    that the parity of the number of intersection points in each local sheet not containing the
    puncture with $H_{\pm\epsilon}$ is independent of the sign of $\pm\epsilon$. We then find that we can achieve
    transversality by perturbing the complex structure near $H_{\pm\epsilon}$: because the sheet
    with the puncture intersects only one of $H_{\pm\epsilon}$, if the contributions of the sheets
    mapping to $H_{-\epsilon}$ cancel then those mapping to $H_{+\epsilon}$ cannot cancel and vice
    versa, by unique continuation.  
    
    Once transversality is achieved the statement that solutions form manifolds follows from a well-known argument, see e.g. \cite[Proposition 2.3]{EESPxR}. 
\end{proof}

\begin{thm}\label{thm:mdlicmpct}
The moduli space $\mathcal{M}^{\fl}(\mathbf{a};\kappa)$ admits a compactification consisting of several level disks joined at Reeb chords and intersection points, where some levels may lie in the symplectization.

The moduli spaces $\mathcal{M}^{\sy}_{\parallel}(\mathbf{c};\kappa)$ and $\mathcal{M}^{\sy}(\mathbf{c})$ admit compactifications consisting of several level disks joined at Reeb chords. 

The moduli space $\mathcal{M}^{\co}(\mathbf{a};\kappa)$ admits a compactification consisting of several level disks joined at Reeb chords. There is one level of disks in the cobordisms and remaining levels in the symplectization ends.

The moduli space $\mathcal{M}^{\overline{\co}}(\mathbf{a};\kappa)$ admits a compactification consisting of several level disks joined at Reeb chords and intersection points. 

\end{thm}

\begin{proof}
Note that the boundary condition on our punctured disks have the following property: any arc in a disks with more than one positive puncture that subdivides the source into two components with a positive puncture in each must connect boundary components numbered with distinct numbers. This shows that there can be no boundary splitting. The theorem then follows from SFT compactness \cite{BEHWZ}.
\end{proof}

We next discuss orientations of moduli spaces following \cite{EESori}. We fix capping operators at
all Reeb chords and Lagrangian intersection points so that the two capping operators there glue to a
disk with the Fukaya orientation, see \cite{FOOO}. Recall the relative spin structure on the Lagrangian submanifold induces an orientation on the determinant bundle over the space of disks with boundary condition in the Lagrangian, see. As in \cite{EESori} we see that these choices then induces a system of coherent orientations on the moduli spaces.

We will use one more property of the moduli spaces which says that they are effectively independent of the increasing or decreasing boundary labeling $\kappa$.

\begin{thm}\label{thm:mdlicopies}
Let $\kappa$ and $\kappa'$ be two increasing (decreasing) boundary numberings. Then there are canonical orientation preserving diffeomorphisms 
\begin{align*}
\mathcal{M}^{\fl}(\mathbf{a};\kappa)&\approx \mathcal{M}^{\fl}(\mathbf{a};\kappa'),\\
\mathcal{M}^{\sy}_{\parallel}(\mathbf{c};\kappa)&\approx \mathcal{M}^{\sy}_{\parallel}(\mathbf{c};\kappa'),\\
\mathcal{M}^{\co}(\mathbf{c};\kappa)&\approx \mathcal{M}^{\co}(\mathbf{c};\kappa')
\end{align*}
\end{thm}

\begin{proof}
Let $\mathcal{M}(\kappa)$ denote either one of the above moduli spaces. This moduli space is the transverse zero set of a Fredholm section in a Banach bundle. Changing the numbering from $\kappa$ to $\kappa'$ corresponds to an arbitrarily small isotopy which induces an arbitrarily small deformation of the section. The theorem follows. 
\end{proof}

\section{Wrapped Floer cohomology and Legendrian surgery}
In this section we present the argument that establishes the isomorphism between
$CE^{\ast}(\Lambda)$ and $CW^*(C)$ where $C$ is the co-core disk of the surgery. Our proof is a
generalization of the corresponding result under Lagrangian handle attachment exposed in \cite{BEE} and uses the technical results on relevant moduli spaces in \cite{Esurgerycurves}.

We first define a version of wrapped Floer cohomology using only purely holomorphic disks and show that the
resulting theory agrees with the usual version defined in terms of holomorphic disks with a
Hamiltonian term. Second we discuss the surgery isomorphism in \cite{BEE}, and third we discuss how
to generalize that argument to partially wrapped Floer cohomology calculations.

\subsection{Wrapped Floer cohomology  without Hamiltonian}\label{sec:CWnoHam}
Let $X$ be a Weinstein manifold and $L$ be an exact Lagrangian. Fix a system of shifting Morse functions that are positive at infinity and let $\bar L=\{L_{j}\}_{j=0}^{\infty}$ be the corresponding family of parallel
Lagrangian submanifolds. Define $CW^*(L)$ to be the chain complex generated by Reeb chords of $L$
and intersection points $L_{0}\cap L_{1}$. We define operations $\m_{i}$ on $CW^*(L)$ using what we call \emph{partial holomorphic buildings}.

We start in the simplest case when the output of $\m_{i}$ is an intersection point $c_{0}$. Consider
$i$ generators $c_{i}\dots c_{1}$ and consider a disk $D_{i+1}$ with a decreasing boundary numbering
$\kappa$, distinguished negative (output) puncture and remaining punctures positive (inputs). Let $\mathbf{c}'=c_{i}\dots c_{1}$ and $\mathbf{c}=c_{0}c_{i}\dots c_{1}$. Define
\[ 
\m_{i}'(\mathbf{c}') = \sum_{|c_{0}|=|\mathbf{c}'|+(2-i)}|\mathcal{M}^{\fl}(\mathbf{c}; \kappa)|c_{0}.
\]
Here we use the temporary notion $\m_{i}'$ to denote the summand of the full operation $\m_{i}$ that takes values in intersection points. We next turn to the more complicated definition of the part $\m_{i}''$ of the operation that takes values in Reeb chord generators and to this end we introduce the notion of a partial holomorphic building.

The domain of a partial holomorphic building is a possibly broken disk $D_{i+1}$ with decreasing
boundary numbering $\kappa$. The partial holomorphic buildings we consider always have exactly one
disk in the symplectization. We call it the primary disk of the building. We require that the
distinguished puncture is increasing and is a negative puncture of this primary disk. If the
distinguished puncture is the only negative puncture of the primary disk  then the partial building
consists only of its primary component. If on the other hand the primary disk has additional negative
punctures then we require that at each additional negative puncture (which is decreasing or constant) there is a
disk in the filling with decreasing boundary condition is attached at its distinguished increasing
or constant puncture to the additional negative puncture. We call these disks the secondary disks of
the partial building. The resulting partial holomorphic building is then a disk with domain a broken
$D_{i+1}$, with distinguished puncture a negative puncture at a Reeb chord and with remaining $i$
punctures either Reeb chords or intersection points. See Figure \ref{partial}.

\begin{figure}[h!]
    \centering
    \begin{tikzpicture}[scale=1.1]
    \tikzset{->-/.style={decoration={ markings,
                mark=at position #1 with {\arrow{>}}},postaction={decorate}}}
    
    \draw [black, thick=1.5] (2,3) to[in=90,out=270] (0.5,0); 
    \draw [black, thick=1.5] (2.5,3) to[in=90,out=270] (4,0);
 
    \draw [black, thick=1.5] (2,0) to[in=90,out=90] (2.5,0);         
    \draw [black, thick=1.5] (3,0) to[in=90,out=90] (3.5,0);         
    \draw [black, thick=1.5] (1,0) to[in=90,out=90] (1.5,0);         
    
    \draw [black, thick=1.5] (0.5,-0.2) to[in=270,out=270] (0,-0.2);         
  \draw [black, thick=1.5] (-0.5,-0.2) to[in=270,out=270] (-1,-0.2);         
  \draw [black, thick=1.5] (-1.5,-0.2) to[in=270,out=270] (-2,-0.2);         
 
  \draw [black, thick=1.5] (5,-0.2) to[in=270,out=270] (5.5,-0.2);         
  \draw [black, thick=1.5] (4,-0.2) to[in=270,out=270] (4.5,-0.2);

    \draw [black, thick=1.5] (1,-0.2) to[in=270,out=270] (-2.5,-0.2);         
 
    \draw [black, thick=1.5] (3.5,-0.2) to[in=270,out=270] (6,-0.2);         
     \draw [black, thick=1.5] (3,-0.2) to[in=270,out=270] (6.5,-0.2);

    \draw [black, thick=1.5] (2.5,-0.2) to[in=270,out=270] (7,-0.2);

    \draw [black, thick=1,->-=.5 ] (-2,-0.2) to (-2.5,-0.2); 
    \draw [black, thick=1,->-=.5 ] (-1,-0.2) to (-1.5,-0.2); 
    \draw [black, thick=1,->-=.5 ] (0,-0.2) to (-0.5,-0.2); 
    \draw [black, thick=1,->-=.5 ] (1,-0.2) to (0.5,-0.2); 
    \draw [black, thick=1, ->-=.5] (1,0) to (0.5,0); 
    \draw [black, thick=1, ->-=.5 ] (2,0) to (1.5,0); 
    \draw [black, thick=1, ->-=.5 ] (3,0) to (2.5,0); 
    \draw [black, thick=1, ->-=.5 ] (4,0) to (3.5,0); 
    \draw [black, thick=1, ->-=.5 ] (3,-0.2) to (2.5,-0.2); 
    \draw [black, thick=1, ->-=.5 ] (4,-0.2) to (3.5,-0.2); 
     \draw [black, thick=1, ->-=.5 ] (5,-0.2) to (4.5,-0.2); 
    \draw [black, thick=1,->-=.5 ] (6,-0.2) to (5.5,-0.2); 
    \draw [black, thick=1,->-=.5 ] (7,-0.2) to (6.5,-0.2); 
    \draw [black, thick=1,->-=.5] (2.5,3) to (2,3);

 \node at (1.5,2.5) {\small $L_3$}; 
    \node at (3,2.5) {\small $L_4$}; 
    \node at (1.3, 0.4) {\small $L_0$}; 
    \node at (2.3, 0.4) {\small $L_7$}; 
    \node at (3.3, 0.4) {\small $L_6$}; 
     \node at (4.3, -0.6) {\small $L_4$}; 
     \node at (5.3, -0.6) {\small $L_5$};         
    \node at (6.1, -0.6) {\small $L_6$};         
    \node at (6.6, -0.6) {\small $L_6$};         
    \node at (7.2, -0.6) {\small $L_7$};         
    \node at (0.3, -0.6) {\small $L_3$}; 
  \node at (-0.7, -0.6) {\small $L_2$}; 
  \node at (-1.7, -0.6) {\small $L_1$}; 
  \node at (-2.7, -0.6) {\small $L_0$};

\end{tikzpicture}
    \caption{The domain of a partial building contributing to the operation $\m_7$ of $CW^*(L)$
    with a possible decoration. The map sends the tensor product of chords $L_0 \leftarrow L_1,
    L_1 \leftarrow L_2, \ldots, L_6 \leftarrow L_7$ to a chord $L_0 \leftarrow
    L_7$.} \label{partial}
\end{figure}

\begin{rem}
At additional negative punctures there may be holomorphic disks with one positive puncture and boundary on $L$ attached. These are the usual augmentation disks, or disks on $L$ used as anchoring disks in the definition of $\mathcal{M}^{\sy}_{\parallel}(\mathbf{c},\kappa)$.   
\end{rem}
We write the punctures of the partial holomorphic disk building as $\mathbf{c}=c_{0}c_{i}\dots c_{1}$, where $c_{0}$ is the distinguished puncture. Write 
\[ 
\mathcal{M}^{\pb}(\mathbf{c};\kappa)
\]
for the moduli space of partial holomorphic disk buildings with boundary condition on $\bar L$ according to $\kappa$. Using this we define for generators $\mathbf{c}'=c_{i}\dots c_{1}$ the operation
\[ 
\m_{i}''(\mathbf{c}')=\sum_{|c_{0}|=|\mathbf{c}'|+(2-i)}|\mathcal{M}^{\pb}(\mathbf{c};\kappa)|c_{0},
\]
where the sum ranges over Reeb chords $c_{0}$ with grading as indicated. Finally we define the total operation $\m_{i}$ as the sum
\[ 
\m_{i}(\mathbf{c}')=\m_{i}'(\mathbf{c}') + \m_{i}''(\mathbf{c}').
\]

\begin{lem}
The $A_{\infty}$-relations hold for the operations $\m_i$.
\end{lem}
\begin{proof}
First, Theorem \ref{thm:mdlicopies} shows that the operations compose and that they are independent of the choice of decreasing boundary numbering. To see that the relations hold we will as usual identify the terms contributing to them with the boundary of an oriented 1-dimensional compact manifold. 

To this end we first consider 1-dimensional moduli spaces $\mathcal{M}'$ of the form    	
$\mathcal{M}'=\mathcal{M}^{\fl}(\mathbf{c};\kappa)$, where the distinguished puncture $c_{0}$ is an intersection point. As usual, the boundary numbering precludes boundary bubbling and we find that the boundary consists of broken disk that either break at an intersection point in which case the holomorphic parts both have dimension zero, or break into a partial holomorphic building with a rigid disk attached at its negative puncture, in which case the primary component of the partial building has dimension one. We find the boundary points of $\mathcal{M}'$ are in 1-1 correspondence with disks contributing to compositions of $\m_{i}'$ and $\m_{j}'$ (disks breaking at intersection points) and
disks contributing to $\m_{i}''$ and $\m_{j}'$. 

Remaining contributions to the $A_{\infty}$-relations correspond to compositions of $\m_{i}''$ and $\m_{j}''$. We show that all contributions to this composition constitute the boundary of an oriented 1-manifold. The contributions are of two forms: either the output puncture of the first operation (which lies in the primary disk of the corresponding partially broken configuration) is glued to an input puncture of the primary disk in the partially broken configuration of the second operation or it is glued to an input puncture in a secondary disk. 

The configurations of the former type correspond to a part of the boundary of the moduli space $\mathcal{M}^{\sy}_{\parallel}(\mathbf{b};\kappa)$ of dimension two with distinguished negative puncture and decreasing boundary numbering (after we divide out the natural $\R$-action this is a 1-dimensional space) capped off at by rigid disks in $\mathcal{M}^{\co}(\mathbf{a})$ at all non-distinguished negative punctures. This is the part of the boundary where the distinguished puncture belongs to the lower level disk.  

The configurations of the second type correspond to the part of the boundary of the $1$-dimensional moduli space $\mathcal{M}^{\co}(\mathbf{a};\kappa)$, with a distinguished increasing positive puncture where a negative puncture in the primary disk of the second operation is attached (other negative punctures in the primary disk of the second operation is capped off as usual). The part of the boundary where the distinguished positive puncture lies in the rigid disk in the cobordism.  

Finally, the remaining part of the boundary in the first case is two level buildings in $\mathcal{M}^{\sy}_{\parallel}(\mathbf{b};\kappa)$ where the distinguished negative puncture belongs to the top level curves. This is exactly the configurations that we get from the remaining parts of the boundary in the second case (i.e., configurations where the distinguished puncture belong to the component in the symplectization) when we glue to it the primary disk of the second operation.  
 
We conclude that also the composition of $\m_{i}''$ and $\m_{j}''$ cancels. The lemma follows.
\end{proof}

\subsubsection{Isomorphism with the Hamiltonian version}\label{sec:wrappediso}

In this section we show that the above definition of wrapped Floer cohomology agrees with the standard theory. Similar results can be found in, e.g., \cite{EO, EKH}. Here we will give a sketch.
We keep the geometric setting as above and write $CW^{\ast}_{\mathrm{Ham}}(L)$ for the usual version of Hamiltonian wrapped Floer cohomology. We give a brief recollection of the definition.

We define the wrapped Floer cohomology complex $CW^{\ast}_{\mathrm{Ham}}(L)$ of $L$ as follows. Write $X=\overline{X}\cup [0,\infty)\times Y$, where $\overline{X}$ is a compact domain and $[0,\infty)\times Y$ the positive end of the Weinstein manifold $X$. Consider time dependent Hamiltonians $H_{a}\colon X\times [0,1]\to\R$ which are perturbations of functions that equal 0 on $\overline{X}$ and are linear of the form
\[ 
(t,y)\mapsto a e^{t} + b,
\] 
on $[0,\infty)\times Y$, where $a$ is not in the chord and orbit spectrum or the contact form on $\Lambda$. We choose these Hamiltonians in such a way that if $a_{0}< a_{1}$ then $H_{a_{0}}< H_{a_{1}}$ on $X$. After small perturbation, Hamiltonian time 1 chords and Hamiltonian time 1 orbits are non-degenerate.

Define the chain complex $CW^{\ast}_{\mathrm{Ham}}(L,H_{a})$ to be generated by Hamiltonian chords $\gamma\colon[0,1]\to X$ of $C$ of action 
\[ 
\mathfrak{a}(\gamma)=\int_{0}^{1} (\lambda(\dot\gamma(t)) - H_{a}(\gamma(t)))\,dt< a
\]
The differential on $CW^{\ast}(L,H_{a})$ is defined by counting solutions of the perturbed Cauchy-Riemann equation over the strip with coordinates $s+it\in\R\times[0,1]$:
\[ 
(du + X_{H_{a}}\otimes dt)^{0,1}=0.
\] 

Choosing an increasing interpolation between $H_{a_{0}}$ and $H_{a_{1}}$ we get continuation maps 
\[ 
CW^{\ast}_{\mathrm{Ham}}(L,H_{a_{0}})\mapsto CW^{\ast}_{\mathrm{Ham}}(L,H_{a_{1}})
\]  
and we define the wrapped Floer cohomology complex as the direct limit
\[ 
CW^{\ast}_{\mathrm{Ham}}(L)=\underrightarrow{\lim}_{a} \,CW^{\ast}_{\mathrm{Ham}}(C,H_{a}).
\]

The wrapped Floer cohomology $HW^{\ast}_{\mathrm{Ham}}(C)$ is the homology of this complex. Writing $HW^{\ast}_{\mathrm{Ham}}(L,H_{a})$ for the homology of $CW^{\ast}_{\mathrm{Ham}}(L,H_{a})$ we then have
\[ 
HW^{\ast}_{\mathrm{Ham}}(L)=\underrightarrow{\lim}_{a}\, HW^{\ast}_{\mathrm{Ham}}(L,H_{a}),
\] 
by exactness of direct limits.

A well known argument \cite{EKH} shows that that $CW^{\ast}(L)$ with differential $\m_{1}$ is quasi-isomorphic to the wrapped Floer cohomology by a geometrically defined chain map. We recall the argument:

The filling $L$ of $\Lambda$ gives an augmentation of $CE^{\ast}(\Lambda)$ and we define $CW^{\ast}(L)$ (as a chain complex disregarding higher product operations) without Hamiltonian as the 'Morse extended linearized Chekanov-Eliashberg complex' with respect to this augmentation as generated by Reeb chords and the critical point of a Morse function on $L$ with a unique minimum, and take the differential to count unperturbed augmented and anchored holomorphic strips. We also introduce the subcomplexes $CW^{\ast}(L,a)$ generated by chords of action $<a$. Then by definition
\[ 
CW^{\ast}(L)=\underrightarrow{\lim}_{a}\, CW^{\ast}(L,a).
\] 

The isomorphism $CW^{\ast}(L)\to CW^{\ast}_{\mathrm{Ham}}(L)$ is now constructed by interpolating exactly as in the continuation maps above from the zero Hamiltonian (ordinary Cauchy-Riemann equation) to the Hamiltonians $H_{a}$ above. Choosing the interpolations compatibly, we get the following commutative diagram
\[ 
\begin{CD}
CW^{\ast}(L,a_{0}) @>>> CW^{\ast}(L,a_{1}) @>>> \dots @>>> CW^{\ast}(L,a_{j}) @>>> \dots\\
@VVV @VVV  \dots @. @VVV\\
CW^{\ast}_{\mathrm{Ham}}(L,H_{a_{0}}) @>>> CW^{\ast}_{\mathrm{Ham}}(L,H_{a_{0}})  @>>> \dots @>>> CW^{\ast}_{\mathrm{Ham}}(L,H_{a_{j}}) @>>> \dots
\end{CD}
\]
Here each vertical arrow are chain isomorphisms by the standard argument, see e.g., \cite[Section 6]{EO}, and taking limits we find a chain isomorphism
\[ 
CW^{\ast}(L)\to CW^{\ast}_{\mathrm{Ham}}(C).
\]

We extend this chain map to an $A_{\infty}$-map, then the standard spectral sequence argument establishes the desired $A_{\infty}$ quasi-isomorphism. 

We follow the approach in \cite{EO} where similar isomorphisms between contact and
symplectic differential graded algebras were constructed. More precisely, as there we
construct a splitting compatible non-negative field of $1$-forms with values in Hamiltonian vector
fields, and further a $1$-parameter family of such forms interpolating between the zero Hamiltonian
at the positive end and the Hamiltonian used to define wrapped Floer cohomology at the negative end, see \cite[Section 2]{EO}. We then define the corresponding moduli spaces over the deformation interval. Keeping the notation from \cite{EO} we denote them
\[ 
\mathcal{F}_{\R}(\mathbf{a},b).
\]

In order for the asymptotics at infinity of these maps to make sense we need to include the
parallel copies of the Lagrangians according to boundary numbering, and in particular also to
incorporate this in the description of wrapped Floer cohomology. More precisely, as in the case above we will have moduli spaces of Floer holomorphic disks with boundary in distinct Lagrangians that are arbitrarily close. The analogue of Theorem \ref{thm:mdlicopies} holds by the same argument and the corresponding moduli spaces are canonically isomorphic for sufficiently small perturbations. Using these observations we then define $\Phi\colon CW^{\ast}(L)\to CW^{\ast}_{\mathrm{Ham}}(L)$ by
\[ 
\Phi(\mathbf{a})=\sum_{\dim(\mathcal{F}_{\R}(\mathbf{a},b))=0}|\mathcal{F}_{\R}(\mathbf{a},b)|b.
\]
\begin{lem}
The map $\Phi$ is an $A_{\infty}$ homomorphism.
\end{lem}

\begin{proof}
To see this we note again that the disks which contributes to the $A_{\infty}$ relations correspond exactly to the ends of 1-dimensional moduli space.
\end{proof}

\begin{lem}\label{l:CW=CWHam}
The map $\Phi$ is a quasi-isomorphism.
\end{lem}

\begin{proof}
The map respects the word length filtration and is the standard isomorphism from the linearized Legendrian cohomology to the wrapped Floer cohomology, discussed above, on the $E_2$-page.
\end{proof}

\subsection{Wrapped Floer cohomology and Lagrangian handle attachment}\label{ssec:CWBEE}
In this subsection we prove the results in \cite{BEE} which gives a Legendrian surgery description of the wrapped Floer cohomology of a co-core disk in a Weinstein manifold obtained by Lagrangian handle attachment along a Legendrian sphere, referring to \cite{Esurgerycurves} for the results on holomorphic curves missing in \cite{BEE}. To state this result we first introduce notation. 

Let $X_0$ be a Weinstein $2n$-manifold with ideal boundary the contact $(2n-1)$-manifold $Y_0$. Let $\Lambda=\Lambda_1\cup\dots\cup\Lambda_m$ be a Legendrian submanifold such that all of its components $\Lambda_j$ are parameterized $(n-1)$-spheres. Let $X$ be the Weinstein manifold that results from attaching Lagrangian handles $H$ to $\Lambda$. Here $H=H_{1}\cup\dots\cup H_{m}$, where each component $H_j$ is a disk sub-bundle of the cotangent bundle $T^{\ast}D$ of the $n$-disk $D$, and where $H_j$ is attached to $\Lambda_j$. Then $X$ contains $m$ co-core disks corresponding to the cotangent fibers at the center of the disk in each $H_j$. We let $C_j\subset X$ denote co-core disk in $H_j$, $\Gamma_j\subset Y$ denote its Legendrian boundary inside the contact boundary $Y$ of $X$, and write $\Gamma=\Gamma_1\cup\dots\cup\Gamma_m$. 

As a first step in the calculation of the wrapped Floer cohomology of $C$ we describe the generators of the underlying chain complex. By definition, see Section \ref{sec:CWnoHam}, generators of $CW^{\ast}(C)$ are of two kinds Lagrangian intersection points and Reeb chords. Here the Lagrangian intersection points are easily understood: pick the shifting Morse function so that it has one minimum on each component of $C$ and no other critical points, then there is exactly one intersection point for each component of $C$. We denote the intersection point of $C_{j}$ by $m_{j}$ and we denote the subcomplex generated by the $m_{j}$ by $CW^{\ast}_{0}(\Gamma)$. Remaining generators are Reeb chords of $\Gamma$ we write $CW^{\ast}_{+}(\Gamma)$ for the quotient complex $CW^{\ast}(\Gamma)/CW^{\ast}_{0}(\Gamma)$ and note that $CW^{\ast}_{+}$ is generated by Reeb chords.
 
Consider the link $\Lambda$ and let all components be decorated by minus, $\Lambda^{-}=\Lambda$.
Consider $CE^{\ast}(\Lambda)$ as a chain complex, generated by composable words
of Reeb chords with differential $d$ and with product $\cdot$ given by concatenation if the words are composable and zero otherwise. Let $\epsilon>0$ denote the size of the attaching region (i.e., the size of the tubular neighborhood of $\Lambda$ where $H$ is attached), we then have the following.

\begin{lem}\label{l:BEEchords=words}
For any $A>0$ there exists $\epsilon_0>0$ such that if $\epsilon<\epsilon_0$ then there is a natural
    one to one correspondence between the generators of $CW_{+}^{\ast}(\Gamma)$ (Reeb chords of $\Gamma$) of action $<A$ and the generators of $CE^{\ast}(\Lambda)$ (words of Reeb chords of $\Lambda$) of action $<A$. 
\end{lem}

\begin{proof}
	This is \cite[Theorem 1.2]{Esurgerycurves}.
\end{proof}	


We will next define the surgery map which is an $A_{\infty}$-morphism
\[ 
\Phi\colon CW^{\ast}(\Gamma)\to CE^{\ast}(\Lambda),
\]
that counts certain holomorphic disks. See Figure \ref{surgerymap}.

\begin{figure}[h!]
    \centering
    \begin{tikzpicture}[scale=1]
    \tikzset{->-/.style={decoration={ markings,
                mark=at position #1 with {\arrow{>}}},postaction={decorate}}}

        \draw[thick, red] (-1,0) arc (0:90:1);
        \draw[thick, red] (-2,1) arc (90:180:1);
        \draw[thick, blue] (-3,0) arc (180:240:1);
        \draw[thick, blue] (-2.5,{-sin(60)} ) arc (240:300:1);
        \draw[thick, blue] (-1.5,{-sin(60)} ) arc (300:360:1);

           \draw[thick, fill=white] (-2,1) circle(.05);
        \draw[thick, fill=white] (-1.5, {-sin(60)} ) circle(.05); 
        \draw[thick, fill=white] (-2.5, {-sin(60)} ) circle(.05);

        \draw[thick, fill=black] (-1,0) circle(.05);
        \draw[thick, fill=black] (-3,0) circle(.05);

     \begin{scope}[xshift=30, yshift=92]
        \draw [black, thick=1.5, dashed, scale=5] (0.5,-0.25) to[in=90,out=90] (1,-0.25); 
        \draw [black, thick=1.5, scale=5] (0.5,-0.25) to[in=270,out=270] (1,-0.25); 

        \draw [black, thick=1.5, scale=5] (0.5,-0.25) to[in=180,out=270] (0.75,-0.7); 
        \draw [black, thick=1.5, scale=5] (1,-0.25) to[in=0,out=270] (0.75,-0.7); 
       
       \draw[thick, fill=black, scale=5] (0.64,-0.67) circle(.01);

     \end{scope}

    \begin{scope}[yshift=100]
     
          \draw[red, thick=1, scale=5] (0.5,-0.25) to[in=120,out=270] (0.75, -0.7);

       \draw[blue, thick=1, scale=5] (0.5,-1.15) to (0.7, -1.3); 
      \draw[blue, thick=1, scale=5] (0.8,-1.3) to (1, -1.15); 
     
        \draw[blue, thick=1, scale=5] (1,-1.15) to[in=260,out=100] (0.85, -0.72); 
        \draw[blue, thick=1, scale=5] (0.5,-1.15) to[in=260,out=80] (0.75, -0.7); 
      \draw[blue, thick=1, scale=5] (0.7,-1.3) to[in=90,out=90] (0.8, -1.3);

 \draw [black, thick=1.5, scale=5] (0.5,-0.25) to[in=180,out=270] (0.75,-0.7); 
        \draw [black, thick=1.5, scale=5] (1,-0.25) to[in=0,out=270] (0.75,-0.7); 
        
         \draw[thick, fill=black, scale=5] (0.75,-0.7) circle(.01);

        \draw[red, thick=1, scale=5] (0.73,-0.37) to[in=120,out=270] (0.85, -0.72);

            \draw [black, thick=1.5, dashed, scale=5] (0.5,-0.25) to[in=90,out=90] (1,-0.25); 
        \draw [black, thick=1.5, scale=5] (0.5,-0.25) to[in=270,out=270] (1,-0.25); 

             \draw[red, thick=1, scale=5] (0.5,-0.25) to (0.73, -0.37); 
         
    \end{scope}

        \begin{scope}[yscale=-1, yshift=100] 
            \draw [black, thick=1.5, scale=5] (0.5,-0.25) to[in=90,out=90] (1,-0.25); 
        \draw [black, thick=1.5, dashed, scale=5] (0.5,-0.25) to[in=270,out=270] (1,-0.25); 

        \draw [black, thick=1.5, scale=5] (0.5,-0.25) to[in=180,out=270] (0.75,-0.7); 
        \draw [black, thick=1.5, scale=5] (1,-0.25) to[in=0,out=270] (0.75,-0.7);

        \end{scope}

    \draw[-stealth,decorate,decoration={snake,amplitude=3pt,pre length=2pt,post length=3pt}]
    (0,0.0) -- ++(1.5,0);

    \end{tikzpicture}
    \caption{A picture illustrating a curve contributing to $\Phi^1$ of
    the $A_\infty$-functor $\Phi$}
    \label{surgerymap}
\end{figure}

As in Section \ref{sec:mdlispaces}, consider the disk $D_{i+j+2}$ with two special punctures subdividing the boundary into an upper and a lower arc with $i$ and $j$ punctures respectively, and with a boundary numbering in the upper arc.  
Let $\mathbf{c}_{0}=c_{0;1}\dots c_{0;j}$ be a composable word of Reeb chords connecting
$\Lambda_{v}$ to $\Lambda_{w}$, and let $c_{i}\dots c_{1}$ be a word of generators of $CW^{\ast}(C)$. Consider the word of Reeb chords and intersection points 
\[ 
\mathbf{c}=c_{0;1}\dots c_{0;j}z^{w}c_{i}\dots c_{1}z^{v}.
\]

Define $\Phi_{i}\colon CW^{\ast}(C)^{\otimes_{i}}\to CE^{\ast}(\Lambda)$, 
\[ 
\Phi_i(\mathbf{c}')=\sum_{|\mathbf{c}_{0}|=|\mathbf{c}'|+i(n-2)}
|\mathcal{M}^{\co}(\mathbf{c})|\mathbf{c}_{0}.
\]
\begin{rem}\label{r:surgeryminimum}
Note that if $m^{v}$ is the minimum of the Morse function on $C^{v}$ as above then 
\[ 
\Phi_{1}(m^{v})=e_{v},
\]
because of the unique holomorphic disk corresponding to the flow line from the minimum to the intersection point between $C^{v}\cap L$, for the parallel copies this gives a triangle with corners at $m^{v}=C^{v}_{0}\cap C^{v}_{1}$, at $C^{v}_{0}\cap L$, and at $C^{v}_{1}\cap L$, and since there are no negative punctures the output is $e_{v}$. Also,
and if a word $\mathbf{c}'$ of generators of $CW^{\ast}(C)$ contains a generator $m^{v}$ and has length $i>1$ then 
\[ 
\Phi_{i}(\mathbf{c}')=0,
\]
as this corresponds, see Lemma \ref{l:diskquantumtreecorr}, to a holomorphic disk with a flow line from the minimum attached and such a configuration cannot be rigid unless the disk is constant.
\end{rem}

\begin{thm}\label{l:BEEdiagonal}
The maps $\Phi_i$ gives an $A_{\infty}$-map $CW^{\ast}(C)\to CE^{\ast}(\Lambda)$ which is an $A_{\infty}$ quasi-isomorphism.	
\end{thm}

\begin{proof} In order to see the $A_{\infty}$-relations we study the boundary of the moduli space $\mathcal{M}^{\co}(\mathbf{c})$. As usual the boundary numbering guarantees that there is no boundary splitting on $C$. The fact that there is no boundary splitting on $L$ follows from Stokes theorem: such a splitting would give a disk without positive puncture. The boundary of the moduli space thus consists of the following configurations. 
\begin{itemize}
\item[$(i)$] Two level curves with one level in the cobordism and one in either symplectization end.
\item[$(ii)$] Curves which split at the intersection point $C\cap L$.
\end{itemize}
Splitting $(i)$ corresponds to the map followed by the operation $d$ in
$CE^{\ast}(\Lambda)$ when the symplectization disk lies in the negative end and to an operation in $CW^{\ast}(C)$ followed by the map when the symplectization disk lies in the positive end. Splitting $(ii)$ corresponds to the tensor product of the map followed by the product operation $\cdot$ in $CE^{\ast}(\Lambda)$. The $A_{\infty}$-relations follow.  

To see that $\Phi$ is a quasi-isomorphism we argue as follows. We first fix an action cut off $A>0$ and show that $\Phi_{1}$ induces an isomorphism on homology below action $A$ by constructing algebraically one holomorphic disk interpolating between a Reeb chord of $\Gamma$ and the corresponding word of Reeb chords of $\Lambda$. For a complete proof we refer to \cite[Theorem 1.3]{Esurgerycurves}, the argument is roughly as follows. One starts from unique and uniformly transversely cut out such disks for single chord words obtained by a straightforward explicit geometric construction. Gluing such disks at their Lagrangian intersection punctures in $L\cap C$ and using small action to rule out all breakings except one, we find that there is algebraically one disk interpolating between a chord on $\Gamma$ and the corresponding word of chords of $\Lambda$. Together with Remark \ref{r:surgeryminimum} which shows $\Phi_{1}(m^{v})=e^{v}$, the existence of such disks implies that the map $\Phi_{1}$ has a triangular matrix with respect to the action filtration and hence is a chain isomorphism (compare \cite[Section 6.2]{BEE}): 
since $\Phi_{1}$ is an isomorphism on the subquotients of the action filtation (and the isomorphism on generators is given by the bijection given in Lemma \ref{l:BEEchords=words}), $\Phi_{1}$ is an isomorphism below action $A$, for any $A$. The $A_{\infty}$-isomorphism below action $A>0$ then follows from the usual spectral sequence argument.

To see that we get an isomorphism on the full complex we show that the isomorphisms discussed are compatible with action limits. More precisely, in order to increase the action limit $A>0$ for the 1-1 correspondence between Reeb chords of $\Gamma$ of action $<A$ and words of Reeb chords of $\Lambda$ of action $<A$ we must shrink the size $\delta>0$ of the handle attached. Consider attaching a handle of size $\delta>0$ to $\Lambda$ and denote the resulting new Weinstein manifold $X_{\delta}$ and the co-core disk $C_{\delta}\subset X_{\delta}$. 

If $\delta_{0}>\delta_{1}$ then by the isomorphism in Lemma \ref{l:CW=CWHam} and standard results for wrapped Floer cohomology, see e.g. \cite[Section 5.5]{EO}, there is a cobordism map
\[ 
CW^{\ast}_{+}(C_{\delta_{0}})\to CW^{\ast}_{+}(C_{\delta_{1}})
\]   
which is a quasi-isomorphism. Moreover, by the surgery description of chords for any $A>0$ there exists $\delta_{1}>0$ such that the above map has $\pm 1$ on the diagonal (with respect to the identification of generators in Lemma \ref{l:BEEchords=words}) for all chords of $\Gamma$ and words of chords of $\Lambda$ of action $<A$. 

Consider the directed system
\begin{equation}\label{eq:handlesyst}
CW^{\ast}_{+}(C_{\delta_{0}})\to CW^{\ast}_{+}(C_{\delta_{1}})\to \dots \to CW^{\ast}_{+}(C_{\delta_{j}})\to\dots,
\end{equation}
where $\delta_{j}\to 0$, and let
\[ 
\overline{CW}^{\ast}_{+}(C)=\underrightarrow{\lim}_{\delta}\,CW^{\ast}_{+}(C_{\delta}).
\] 
Then the homology $\overline{HW}^{\ast}_{+}(C)$ of $\overline{CW}^{\ast}_{+}(C)$ satisfies
\[ 
\overline{HW}^{\ast}_{+}(C) \ = \ \underrightarrow{\lim}_{\delta} \,HW^{\ast}_{+}(C_{\delta}) \quad = \quad HW^{\ast}_{+}(C_{\delta_{j}}),\text{ for any fixed }j.
\]
Here the last equality follows since all the arrows in the directed system of homology groups of \eqref{eq:handlesyst} are isomorphisms.

Consider next the Chekanov-Eliashberg algebra $CE^{\ast}(\Lambda)$ of $\Lambda$. We define the action truncated subcomplex $CE^{\ast}(\Lambda,a)$ generated by words of chords of total action $<a$. Viewing $CE^{\ast}(\Lambda)$ as chain complex generated by words of chords we then have
\[ 
CE^{\ast}(\Lambda)=\underrightarrow{\lim}_{a}\,CE^{\ast}(\Lambda,a).
\]

The surgery map gives for each $a_{j}$, $\delta_{j}>0$ such that the map
\[ 
\overline{CW}^{\ast}_{+}(C_{\delta_{j}})\to CE^{\ast}(\Lambda,a_{j})
\] 
is a chain isomorphism with $\pm 1$ on the diagonal below action $a_{j}$. By definition of surgery and cobordism maps the following diagram commutes
\begin{equation}\label{eq:surgerysequence} 
\begin{CD}
\overline{CW}^{\ast}_{+}(C_{\delta_{0}}) @>>> \overline{CW}^{\ast}_{+}(C_{\delta_{1}}) 
@>>> \dots @>>> \overline{CW}^{\ast}_{+}(C_{\delta_{j}}) @>>> \dots\\
@VVV @VVV  \dots @. @VVV\\
CE^{\ast}(\Lambda,a_{0}) @>>> CE^{\ast}(\Lambda,a_{1})
@>>> \dots @>>> CE^{\ast}(\Lambda,a_{j}) @>>> \dots,
\end{CD}
\end{equation}
where $\delta_{j+1}<\delta_{j}$ and $a_{j}<a_{j+1}$. Taking limits of the sequences we get a chain map
\begin{equation}\label{eq:chainsurgery}
\overline{CW}^{\ast}_{+}(C)\to CE^{\ast}(\Lambda).
\end{equation}
Taking the limits of the sequence \eqref{eq:surgerysequence} on the homology level and using that all vertical arrows are homology isomorphisms then gives homology isomorphisms in the limit and \eqref{eq:chainsurgery} is a quasi-isomorphism inducing an isomorphism
\[ 
\overline{HW}^{\ast}(C)\approx HCE^{\ast}(\Lambda).
\]

The above gives a homology isomorphism of chain complexes. To consider also products one uses the exact same argument. The product operation on $\overline{CW}^{\ast}(C)$ is induced from the action truncated version
\[ 
\overline{CW}^{\ast}(C,a_{1})\otimes\dots\otimes \overline{CW}^{\ast}(C,a_{m})\to
\overline{CW}^{\ast}(C,a_{1}+\dots+a_{m})
\]
and similarly on $CE^{\ast}$:
\[ 
CE^{\ast}(C,a_{1})\otimes CE^{\ast}(C,a_{2})\to
CE^{\ast}(C,a_{1}+a_{2}).
\]

\end{proof}

\begin{rem}
In the above proof we obtain the isomorphism by taking smaller and smaller handles. To see that such a procedure is necessary note that the correspondence between words of chords and chords is true only below an action limit determined by the size of the handle. For actions larger there are Reeb flows before the surgery that hits the neighborhood of the Legendrian without being close to a chord and which could give chords after the surgery. 
\end{rem}

\begin{rem}
There is also an ``upside-down'' perspective on the surgery just described. Namely, one can start from
    the contact manifold $Y$ and produce the contact manifold $Y_{0}$ by doing so called
    $+1$-surgery on $\Gamma$. In complete analogy with the above one shows that Reeb chords on
    $\Lambda$ are in natural one to one correspondence with words of Reeb chords on $\Gamma$ and one
    can construct an upside down surgery map of $A_\infty$-coalgebras:
\[ 
\Bar CW^{\ast}(C) \to LC_{\ast}(\Lambda).
\]
A similar argument also shows that this map is a quasi-isomorphism. Alternatively, one can prove this from the original surgery map using only algebra as follows. First write $CE^{\ast}(\Lambda)=\Omega LC_{\ast}(\Lambda)$ then
\[ 
\Bar CW^{\ast}(C) \simeq \Bar \Omega LC_{\ast}(\Lambda) \simeq  LC_{\ast}(\Lambda),
\]
since $LC_{\ast}(\Lambda)$ is co-nilpotent, see Section \ref{ssec:barcobaradj}.
\end{rem}

\subsection{Legendrian surgery and stopped wrapping}\label{ssec:CWpwrap}
In this section we outline a surgery approach to the computation of wrapped Floer cohomology in a
Weinstein manifold $X$ with wrapping stopped by a Legendrian $\Lambda$ in its boundary. We will use
the following model for the ambient manifold. Fix a tubular neighborhood of $\Lambda$ in the contact
boundary $Y$ of $X$ and attach a disk-bundle neighborhood of the $0$-section in
$T^{\ast}([0,\infty)\times\Lambda)$ along the boundary
$T^{\ast}([0,\infty)\times\Lambda)|_{0\times\Lambda}$ just like in Lagrangian handle attachment. We
use a Liouville vector field on this domain that agrees with the standard Liouville vector field
pointing outwards along fibers in the cotangent bundle over $[T,\infty)\times\Lambda$ for some
$T>0$. Let the components of $\Lambda$ be denoted $\Lambda_{v}$, $v\in Q_{0}$. Fix a base point
$p_v \in\Lambda_v$ for each $v$. Let $C^{v;\tau}$ denote the cotangent fiber $T^{\ast}_{(p_{v},\tau)}([0,\infty)\times\Lambda)$. We will compute the wrapped Floer cohomology of $C^{\tau}=\bigcup_{v\in Q_{0}}C^{v;\tau}$ for sufficiently large $\tau$ using a surgery approach. A straightforward monotonicity argument shows that the non-compactness of the cotangent bundle $T^{\ast}\Lambda\times[0,\infty)$ does not interfer with the compactness results for holomorphic curves used in the definition of wrapped Floer cohomology. 

We first consider the surgery map into the Chekanov-Eliashberg algebra with loop space coefficients.
Consider all components of $\Lambda$ decorated by a positive sign $\Lambda=\Lambda^{+}$ and consider $CE^{\ast}(\Lambda)$ which now involves, except for Reeb chords, also chains $C_{-\ast}(\Omega\Lambda)$ on the based loop space. We define an $A_{\infty}$- map
\[ 
\Phi\colon CW^{\ast}(C)\to CE^{\ast}(\Lambda),
\]
where $A_{\infty}$-structure on the right hand side is the standard DG-algebra structure induced by concatenation and the Pontryagin product (as defined in this paper).  
As in Section \ref{sec:mdlispaces}, consider a disk $D_{i+j+2}$ with two dividing punctures that subdivides the boundary into two arcs, lower and upper. Let the upper arc contain $i$ boundary punctures and is equipped with a decreasing boundary numbering $\kappa$ and the lower arc $j$ boundary punctures. Let $\mathbf{c}'=c_{i}\dots c_{1}$ be Reeb chords of $C$ and let $\mathbf{c}_{0}=c_{0;1}\dots c_{0;j}$ be Reeb chords of $\Lambda$. Let
\[ 
\mathbf{c}= c_{0;1}\dots c_{0;j} z^{v} c_{i}\dots c_{1}\dots z^{w}
\]  
and consider $\mathcal{M}^{\co}(\mathbf{c};\kappa)$, again as in Figure \ref{surgerymap}.

Theorems \ref{thm:mdlitv} and \ref{thm:mdlicmpct} show that this moduli space carries a fundamental
chain. We view this chain as parametrizing chains of paths in $\Lambda$ connecting the Reeb chord
endpoints in $\mathbf{c}_{0}$. We write $[\mathcal{M}^{\co}(\mathbf{c})]$ for the alternating word
of chains of loops and Reeb chords and view it as an element in $CE^{\ast}(\Lambda)$. Define the maps 
\[ 
\Psi_{i}\colon CW^{\ast}(C)^{\otimes_\k i}\to CE^{\ast}(\Lambda)
\]
as 
\[ 
\Psi_{i}(\mathbf{c}')=\sum_{\mathbf{c}_{0}}[\mathcal{M}^{\co}(\mathbf{c})].
\]

\begin{thm}\label{t:newsurgeryAinfty}
The map $\Psi\colon CW^{\ast}(C)\to CE^{\ast}(\Lambda)$ is an $A_{\infty}$-map.
\end{thm}

\begin{proof}
To see that the $A_{\infty}$-relation holds we look at the boundary of the moduli space $\mathcal{M}^{\co}(\mathbf{c})$ of dimension $d$. The codimension one boundary consists of three splittings:
\begin{itemize}
	\item[$(i)$] A one-dimensional curve splits off in the positive symplectization end.
	\item[$(ii)$] A curve splits off at the negative end.
	\item[$(iii)$] Splitting at one of the intersection points $z^{v}$. 
\end{itemize}
In order for splittings of the form $(i)$ to contribute to the codimension one boundary of the moduli space the part of the holomorphic building in $W$ consists of rigid disks with only positive punctures attached at one puncture to a negative puncture of the disk in the positive end and a disk of dimension $d-1$ in $\mathcal{M}^{\co}(\mathbf{b})$ attached at the remaining negative puncture. (Splittings where the dimension of the components of the holomorphic building are distributed differently have higher dimension along the disk $C$ and correspond to `hidden faces' from the point of view of $C_{\ast}(\Omega\Lambda)$.) Assembling the rigid disks and the one-dimensional disk we get a partial holomorphic disk building that contributes to the $A_{\infty}$-operations in $CW^{\ast}(L)$ followed by the map $\Psi$. Splittings of type $(ii)$ corresponds to the map $\Psi$ followed by the differential $\mu_{1}$ in $CE^{\ast}(\Lambda)$. Finally splittings of type $(iii)$ correspond to the map $\Psi$ followed by the product $\mu_{2}$ on $CE^{\ast}(\Lambda)$. We conclude that the terms contributing to $A_{\infty}$-relations express the codimension one boundary of $[\mathcal{M}^{\co}(\mathbf{c})]$ in two different ways and hence $\Psi$ is an $A_{\infty}$-map.
\end{proof}

\begin{rem}
In the boundary of the moduli space $\mathcal{M}^{\co}(\mathbf{c})$ considered in the proof of Theorem \ref{t:newsurgeryAinfty} there are also higher-dimensional curves splitting off in the positive symplectization end. Such splittings does neither contribute to the codimension one boundary of the chains of loops nor to the operations in the wrapped Floer cohomology and hence plays no role in the $A_{\infty}$-chain map equation.
\end{rem}

We will use slight generalizations of the map $\Psi$ below. More precisely, if $p_{j}$, $j=1,\dots,m$ are points in $[0,\infty)\times\Lambda$ and if $F_{j}$ is the cotangent fiber at $p_{j}$ then we have a similar surgery map
\[ 
\Psi^{p_ip_j}\colon CW^{\ast}(F_{i},F_{j})\to  CE^{\ast}_{p_ip_j}(\Lambda),
\]
which counts holomorphic disks with a positive Reeb chord connecting $F_{i}$ to $F_{j}$, two
Lagrangian intersection punctures at $p_{i}$ and at $p_{j}$, and a word of chains of loops in
$\Lambda$ and Reeb chords of $\Lambda$ as output, and where $CE^{\ast}_{ij}$ is directly analogous
to $CE^{\ast}$ but where the first chain of loops is in a word is replaced by a chain of paths from
$p_{i}$ to the base point and the last is replaced by a chain of paths from the base point to $p_{j}$.
In this setup the counterpart of the second component $\Psi_{2}$ is
\begin{equation}\label{eq:morefibersmap}
\Psi^{p_ip_jp_k}\colon CW^{\ast}(F_{j},F_{k})\otimes CW^{\ast}(F_{i},F_{j})\to CE^{\ast}_{p_ip_k}(\Lambda),
\end{equation}
and counts disks with two positive punctures at Reeb chords and two Lagrangian intersection punctures at $p_{i}$ and $p_{k}$. The counterpart of the $A_{\infty}$-equations in this setup is then
\begin{equation}\label{eq:morefibersprod} 
d\circ\Psi^{p_ip_k}+ \Psi^{p_jp_k}\cdot\Psi^{p_i p_j} + \Psi^{p_ip_jp_k}\circ(1\otimes\mu_{1}+\mu_{1}\otimes 1) + \Psi^{p_ip_k}\circ\mu_{2}=0, 
\end{equation}
where $d$ is the differential on $CE^{\ast}_{ik}$ and where $\cdot$ is the (Pontryagin) product
$CE^{\ast}_{p_j p_k}\otimes CE^{\ast}_{p_i p_j}\to CE^{\ast}_{p_ip_k}$. The proofs of these statements are word by word repetitions of the proof of Theorem \ref{t:newsurgeryAinfty}.

We next sketch a proof that the map $\Psi$ in Theorem \ref{t:newsurgeryAinfty} is in fact a
quasi-isomorphism, or, in other words, that its first component $\Psi_{1}$ induces an isomorphism on
homology. We filter $CW^{\ast}(C)$ by
action of its Reeb chord generators. To get a corresponding filtration on $CE^{\ast}(\Lambda)$ we
use action on Reeb chords in combination with the energy on the loops. We start with a discussion of the energy of loops, following \cite{Milnor}.

Equip $\Lambda$ with a Riemannian metric and  $\Omega=\Omega(\Lambda)$ denote the space of based loops in $\Lambda$ with the supremum norm: for two loops $\gamma,\beta\colon[0,1]\to \Lambda$,
\[
d^{\ast}(\gamma,\beta)=\sup_{t\in[0,1]} \rho(\gamma(t),\beta(t)),
\] 
where $\rho$ is the metric on $\Lambda$ induced by the Riemannian structure. Then the metric topology on $\Omega$ agrees with the standard compact open topology. 

Let $\Omega'=\Omega'(\Lambda)$ denote the space of piecewise smooth paths with metric
\[
d(\gamma,\beta) = d^{\ast}(\gamma,\beta)+ \int_{0}^{1}\left(|\dot\gamma|-|\dot{\beta}|\right)^{2}dt,
\]
where $\dot{\gamma}$ denotes the derivative of $\gamma$. The natural inclusion $\Omega'\to\Omega$ is
a homotopy equivalence, \cite[Theorem 17.1]{Milnor}. We will use finite dimensional approximations
to study $\Omega'$. The energy of a piece-wise smooth loop in $\Lambda$ is
\[
E(\gamma)= \int_0^{1}| \dot\gamma|^{2} dt.
\]
For $c>0$, let $\Omega^{c}\subset\Omega'$ denote the subset of loops of energy $E<c$. The space
$\Omega^{c}$ can be approximated by piece-wise geodesic loops. More precisely, fixing a subdivision
$0=t_{0}< t_{1}<\dots< t_{m}=1$ of $[0,1]$ we consider the space $B^{c}$ of loops of energy $E<c$
that are geodesic on each interval $[t_{i},t_{i+1}]$. Then \cite[Lemma 16.1]{Milnor} shows that for
all sufficiently fine subdivisions $B^{c}$ is a finite dimensional manifold (a submanifold of the
product $\Lambda^{\times m}$ in a natural way). Moreover \cite[Theorem 16.2]{Milnor}, all critical points of
$E|_{\Omega^c}$ lie in $B^{c}$ which is a deformation retract of $\Omega^{c}$ and for generic metric $E|_{B^{c}}$ is a Morse function. 

With these preliminaries established we turn to the actual proof. 
The first step will be to describe the Reeb chords of $C^{\tau}$. Let $g$ be a Riemannian metric on $\Lambda$ as above and let $f\colon [0,\infty)\to \R$ be a positive function with $f(0)=1$, $f'(0)=-1$, and $f'(t)<0$ monotone increasing. Define the metric $h$ on $\Lambda\times\R$ by
\[ 
h= dt^{2} + f(t) g.
\]

Then if $x$ and $y$ are points in $\Lambda$ and $\gamma\colon[0,s]\to\Lambda$ is a geodesic with $\gamma(0)=x$ and $\gamma(s)=y$ then there is a unique geodesic $(\gamma(t),r(t))\in\Lambda\times[0,\infty)$ such that
\begin{itemize}
	\item $(\gamma(0),0)=(x,0)$ and $(\gamma(s),r(s))=(y,0)$;
	\item $r(t)$ is Morse with has a unique maximum at an interior point $t=t_{0}$. 
\end{itemize}  

Note that the Reeb flow in the unit disk bundle is the natural lift of the geodesic flow. Assume next as above that the metric $g$ on $\Lambda$ is generic in the sense that the length functional for curves connecting any two Reeb chord endpoints in $\Lambda$ has only Morse critical points. Concretely, this means that the index form of any geodesic connecting two Reeb chord endpoints is non-degenerate. As in Lemma \ref{l:BEEchords=words}, this allows us to control the Reeb chords of $C^{\tau}$ below a given action for all sufficiently thin handles. More precisely, let $\epsilon$ denote the size of the tubular neighborhood of $\Lambda$ in $Y$ where we attach $T^{\ast}(\Lambda\times [0,\infty))$ and we introduce the following notion of a geodesic-Reeb chord word. A \emph{geodesic-Reeb chord word} is a word
\[ 
\gamma_{1}c_{1}\gamma_{2}c_{2}\dots c_{m}\gamma_{m},
\] 
where $\gamma_{1}$ is a geodesic from one of the base points $p_{v}$ to the start point of the
Reeb chord $c_{1}$, where $\gamma_{2}$ is a geodesic from the endpoint of $c_{1}$ to the start point
of $c_{2}$, etc, until finally $\gamma_{m}$ is a geodesic from the endpoint of the Reeb chord
$c_{m}$ to one of the base points $p_w$. We define the action of a geodesic-Reeb chord word to be the sum of actions of its Reeb chords and the energies of its geodesics.

\begin{lem}\label{l:newchords=words}
	For any $A>0$ there exists $\epsilon_{0}>0$ and $\tau_{0}>0$ such that for any $\epsilon<\epsilon_{0}$ and any $\tau>\tau_{0}$ there is a natural one to one correspondence between Reeb chords of $C^{\tau}$ of action $<A$ and geodesic-Reeb chord words of $\Lambda$ of action $<A$.
\end{lem}

{\bf Sketch of proof:} The proof uses the transversality of the Reeb chords and of the geodesics. The basic observation is that the point in the normal fiber of $\Lambda$ where the Reeb flow hits determines the direction of the geodesic in $\Lambda\times[0,\infty)$. After introducing a concrete smoothing of corners the lemma then follows from the finite dimensional inverse function theorem.

To show that $\Psi_{1}$ is a quasi-isomorphism we will show that it is represented by a triangular
matrix with ones on the diagonal with respect to the action/energy filtration. To this end we will
use the Morse theoretic finite dimensional model for the chain complex underlying the homology of the based loop space described above. In order to have $[\mathcal{M}^{\co}(\mathbf{c})]$ defined as a chain in this model we need to assure that the paths on the boundary of the holomorphic disk are sufficiently well behaved. We only sketch the construction. On holomorphic disk with unstable domains, we fix gauge using small spheres surrounding Reeb chord endpoint, compare \cite[Section A.2]{ES}. As in \cite[Section A.1]{ES} we use a configuration space for holomorphic curves consisting of maps with two derivatives in $L^{2}$. This means that the restriction to the boundary has $3/2$ derivatives in $L^{2}$ and in particular the projection to $\Lambda$ has bounded energy. Since the action of the positive puncture in a holomorphic disk contributing to the differential controls the norm of the solution it follows that we can use configuration spaces of bounded energy to study the disks in the differential: we approximate the boundary curves uniformly by a piecewise geodesic curve by introducing a uniformly bounded number of subdivision points and straight line homotopies in small charts. 

\begin{conj}\label{appconj}
The chain-map $\Psi_{1}\colon CW^{\ast}(C)\to CE^{\ast}(\Lambda)$ induces an isomorphism on homology. 
\end{conj}

{\bf Sketch of proof:} 
Consider a word of the form
\[ 
\gamma_{0}c_{1}\gamma_{1} c_{2}\dots c_{m}\gamma_{m},
\]
where $\gamma_{j}$ are geodesics in $\Lambda\times[0,\infty)$ and where $c_{j}$ are Reeb chords. We aim to construct algebraically one disk connecting the Reeb chord $a$ of $C$, corresponding to this word (see Lemma \ref{l:newchords=words}), to the word itself. We use an inductive argument and energy filtration. To start the argument we pick additional fiber disks $F_{c^{+}}$ and $F_{c_{-}}$ in $T^{\ast}(\Lambda\times[0,\infty))$ at $(c^{+},\epsilon)$ and $(c_{-},\epsilon)$ for very small $\epsilon>0$ near all Reeb chord end points $c_{+}$ and $c_{-}$ in $\Lambda$. We use the natural counterparts of the correspondence between mixed words of geodesics and Reeb chords before surgery and Reeb chords after,for mixed wrapped Floer cohomologies. For example, there is a straightforward analogue of Lemma \ref{l:newchords=words}: Reeb chord generators of $CW^{\ast}(F_{c^{+}},C)$ correspond before the surgery words of the form
\[
\gamma_{1}c_{1}\gamma_{2}\dots c_{m}\gamma_{m},
\]
where $\gamma$ is a geodesic connecting the base point of $F_{c^{+}}$ to the initial point of $c_{1}$,
$\gamma_{2}$ from the endpoint of $c_{1}$ to the start point of $c_{2}$, etc. To start the argument
we note that it is straightforward to construct holomorphic strips corresponding to the short
geodesics starting at $F_{c^{-}}$ followed by the chord $c$ and then the short geodesic to
$F_{c^{+}}$ and to show that they are unique. This corresponds to a generator of
$CW^*(F_{c^{-}},F_{c^{+}})$. Likewise, it is immediate to construct the holomorphic disk connecting a
Reeb chord generator of $CW^*(F_{c^{+}},C)$ corresponding to a geodesic, and show that it is unique, compare Lemma \ref{l:BEEdiagonal}. 

We now use these two to construct algebraically one disk from the Reeb chord generator of $CW^{\ast}(F_{c_{-}},C)$ corresponding to the short geodesic, the chord $c$, and a geodesic connecting the endpoint of $c$ to the base point of $C$. To this end we consider the natural map
\[
    \Psi^{p_v c^{+}c^{-}}\colon CW^*(F_{c^{+}},C)\otimes CW^*(F_{c^{-}},F_{c^{+}})\to CE^{\ast}(\Lambda),
\] 
see \eqref{eq:morefibersprod}. For the two Reeb chords $a$ connecting $F_{c^{-}}$ to $F_{c^{+}}$
corresponding to the chord $c$ of $\Lambda$, and $b$ connecting $F_{c^{+}}$ to $C$ corresponding to
the geodesic we then have (with $p_{v}$ denoting the base point):
\begin{align*}
    &d(\Psi^{p_{v} c^+ c^-}(b,a))+ (\Psi^{p_{v}c^-}(b))\cdot (\Psi^{c^{+}c^{-}}(a))  \\ 
    &+  \Psi^{p_{v} c^-}(\m_{2}(b,a)) + (-1^{|a|-1} \Psi^{p_v c^{+} c^-}(\m_{1}(b),a)
    +\Psi^{p_v c^{+}c^{-} }(b,\m_{1}(a)) =0. 
\end{align*}
Here we know that the terms containing $\m_{1}$ and $d$ involves non-trivial holomorphic disks or Morse flows in the finite dimensional approximation and hence lowers action/energy by an amount bounded below by some $\delta>0$ which we assume is much larger than $\epsilon>0$ above. Therefore, if we restrict attention to a small action window we find
\[ (\Psi^{p_{v}c^- }(b) \cdot \Psi^{c^{+}c^{-}}(a))+\Psi^{p_{v} c^-}(\m_{2}(b,a))=0.
\]
Here the first term is simply the Pontryagin product at the common endpoint of the curves, which is
homologous to the word $\epsilon'c\gamma$ of the small geodesic, the Reeb chord and then the longer
geodesic, by rounding the corner at $c^{+}$. It follows that $\m_{2}(a,b)=r$, where $r$ is a Reeb
chord with action between the sum of the actions of $a$ and $b$ and the action of $\epsilon'c\gamma$
and that $\Psi^{c^{-}p_{v}}(r)$ contains this word with coefficient $\pm 1$. Noting that there
is only one Reeb chord in the action window studied we find that the desired coefficient equals $\pm
1$. It is now clear how to continue the induction, in each step we add one more geodesic or Reeb
chord to any word. Using already constructed curves and \eqref{eq:morefibersprod} in a small action
window we find that the map $\Psi_{1}$ has a triangular action matrix with $\pm 1$ on the diagonal,
hence it is a quasi-isomorphism.

\end{document}